\documentclass[reqno, 10pt]{amsart} 

\usepackage[utf8]{inputenc}
\usepackage[T1]{fontenc}

\usepackage[english]{babel}

\usepackage[margin=1.2in]{geometry}
\usepackage{varwidth}
\usepackage{enumitem}
\usepackage{xcolor}
\usepackage{multirow}

\usepackage{amsmath, amsthm, amssymb, bbm}
\usepackage{mathtools}
\usepackage{unicode-math}
\usepackage[all]{xy}

\usepackage[bookmarks = true]{hyperref}

\usepackage{tikz}

\newtheorem{thm}{Theorem}[section]
\newtheorem*{thm*}{Theorem}
\newtheorem{lem}[thm]{Lemma}
\newtheorem{prop}[thm]{Proposition}
\newtheorem{cor}[thm]{Corollary}

\theoremstyle{definition}
\newtheorem{rmk}[thm]{Remark}

\newtheorem{defn}[thm]{Definition}
\newtheorem{ex}[thm]{Example}
\newtheorem{conj}[thm]{Conjecture}
\newtheorem{construction}[thm]{Construction}

\numberwithin{equation}{section}

\newcommand{\overbar}[1]{\mkern 1.5mu\overline{\mkern-1.5mu#1\mkern-1.5mu}\mkern 1.5mu}

\newcommand{\ob}{\overbar}

\newcommand{\mr}{\mathrm}

\newcommand{\wt}{\widetilde}

\newcommand{\Orb}{\operatorname{Orb}}
\newcommand{\dOrb}{\operatorname{∂Orb}}
\newcommand{\diag}{\operatorname{diag}}

\newcommand{\Int}{\operatorname{Int}}

\newcommand{\Aut}{\operatorname{Aut}}

\newcommand{\Hom}{\operatorname{Hom}}

\newcommand{\End}{\operatorname{End}}

\newcommand{\Lie}{\operatorname{Lie}}

\newcommand{\Spec}{\operatorname{Spec}}

\newcommand{\Spf}{\operatorname{Spf}}

\newcommand{\colim}{\operatorname{colim}}

\newcommand{\GL}{\operatorname{GL}}

\newcommand{\tensor}{\otimes}

\newcommand{\iso}{\cong}

\newcommand{\lr}{\longrightarrow}

\newcommand{\mbA}{\mathbb{A}}

\newcommand{\mbC}{\mathbb{C}}

\newcommand{\mbF}{\mathbb{F}}
\newcommand{\mbG}{\mathbb{G}}

\newcommand{\mbL}{\mathbb{L}}
\newcommand{\mbM}{\mathbb{M}}
\newcommand{\mbN}{\mathbb{N}}

\newcommand{\mbP}{\mathbb{P}}
\newcommand{\mbQ}{\mathbb{Q}}
\newcommand{\mbR}{\mathbb{R}}

\newcommand{\mbX}{\mathbb{X}}
\newcommand{\mbY}{\mathbb{Y}}
\newcommand{\mbZ}{\mathbb{Z}}

\newcommand{\bF}{\mathbf{F}}

\newcommand{\bM}{\mathbf{M}}
\newcommand{\bN}{\mathbf{N}}

\newcommand{\bV}{\mathbf{V}}

\newcommand{\bm}{\mathbf{m}}

\newcommand{\bs}{\mathbf{s}}
\newcommand{\bt}{\mathbf{t}}
\newcommand{\bu}{\mathbf{u}}
\newcommand{\bv}{\mathbf{v}}

\newcommand{\mcB}{\mathcal{B}}
\newcommand{\mcC}{\mathcal{C}}
\newcommand{\mcD}{\mathcal{D}}
\newcommand{\mcE}{\mathcal{E}}
\newcommand{\mcF}{\mathcal{F}}

\newcommand{\mcI}{\mathcal{I}}

\newcommand{\mcK}{\mathcal{K}}
\newcommand{\mcL}{\mathcal{L}}
\newcommand{\mcM}{\mathcal{M}}

\newcommand{\mcO}{\mathcal{O}}
\newcommand{\mcP}{\mathcal{P}}
\newcommand{\mcQ}{\mathcal{Q}}

\newcommand{\mcT}{\mathcal{T}}

\newcommand{\mcW}{\mathcal{W}}

\newcommand{\mcZ}{\mathcal{Z}}

\newcommand{\mfG}{\mathfrak{G}}

\newcommand{\mfm}{\mathfrak{m}}

\newcommand{\mfp}{\mathfrak{p}}

\long\def\[#1]#2{ \begin{#1}#2\end{#1}}
\makeatletter
\def\m{\@tut\@gobble\@tutt}
\def\@tut#1\@tutt#2{
\@ifnextchar,{\@tut #1 & #2\\\expandafter\expandafter\expandafter\@gobble\expandafter\@gobble\@gobble\@tutt}{
\@ifnextchar.{\begin{pmatrix}#1 &#2\end{pmatrix}\@gobble}{\@tut #1 & #2\@tutt}
}
}
\def\bm{\@tuut\@gobble\@tuutt}
\def\@tuut#1\@tuutt#2{
\@ifnextchar,{\@tuut #1 & #2\cr\expandafter\expandafter\expandafter\@gobble\expandafter\@gobble\@gobble\@tuutt}{
\@ifnextchar.{\bordermatrix{#1 &#2\cr}\@gobble}{\@tuut #1 & #2\@tuutt}
}
}
\makeatother

\newcommand{\del}{\operatorname{\partial Orb}}

\newcommand{\bOm}{\breve{\Omega}}
\newcommand{\bOmega}{\breve{\Omega}}

\begin{document}

\title{Arithmetic Transfer for inner forms of $GL_{2n}$}
\author{Qirui Li}
\author{Andreas Mihatsch}
\date{\today}

\begin{abstract}
We formulate Guo--Jacquet type fundamental lemma conjectures and arithmetic transfer conjectures for inner forms of $GL_{2n}$. Our main results confirm these conjectures for division algebras of invariant $1/4$ and $3/4$.
\end{abstract}

\maketitle

\setcounter{tocdepth}{1}
\tableofcontents

\newcommand{\Inv}{\mr{Inv}}
\newcommand{\quand}{\quad \text{and}\quad}
\newcommand{\Par}{\mr{Par}}
\newcommand{\Iw}{\mr{Iw}}
\newcommand{\simto}{\overset{\sim}{\to}}
\newcommand{\simlr}{\overset{\sim}{\lr}}

\section{Introduction}

\subsection{Background}
\label{ss:global}
Even though our paper will be purely local in nature, we begin by describing its global motivation. Let $F$ be a totally real number field, let $K/F$ be a quadratic CM extension, and let $D/K$ be a central simple algebra (CSA) of degree $2n$. Let $*:D\to D$ be an involution of the second kind. Recall that this means that $(xy)^* = y^*x^*$ and that $*$ restricts to the complex conjugation on $K$. We further assume that $*$ is positive in the sense that $\mr{trd}_{D/\mbQ}(x^*x) > 0$ for all $x\neq 0$. Let $β\in D^\times$ be a skew-hermitian element, meaning that $β^* = -β$. Then the algebraic group (over $\mbQ$)
\begin{equation}\label{eq:G}
G := \{x\in D^{\mr{op}, \times} \mid \text{$x^*βx = ν(x)β$ for some $ν(x) \in \mbG_m$}\}
\end{equation}
is an inner form of a unitary similitude group in $2n$ variables and the above data can be completed into a PEL type Shimura datum $(G, X_G)$. The resulting Shimura variety $\mr{Sh}_G$ can be described as a moduli space of polarized abelian varieties with $D$-action and level structure.

For example, we may take $(D, *)$ as the matrix algebra $M_{2n}(K)$ with transpose conjugation $x^* = \ob{x}^t$. Then $β \in GL_{2n}(K)$ can be any skew-hermitian matrix and $G$ is the corresponding unitary similitude group $GU(K^{2n}, β)$.

Coming back to the general situation, we assume that the signatures of $β$ are $(2n-1, 1)$ at a unique archimedean place and $(2n, 0)$ at all others. Then $\mr{Sh}_G$ is of dimension $2n-1$ and our interest lies in algebraic cycles in the arithmetic middle dimension $n-1 = \lfloor (2n-1)/2\rfloor$. By work of Li--Liu \cite{LL1, LL2} for $D = M_{2n}(K)$, it is known in many cases that the height pairings of such cycles are related to the leading terms of the Taylor expansions of certain $L$-functions. This relation can be understood as a generalization of the Gross--Zagier formula \cite{GZ, YZZ} to higher dimensions. It is moreover parallel to the arithmetic Gan--Gross--Prasad conjectures \cite{GGP}, and it also represents an instance of the Beilinson--Bloch height conjectures \cite{Beilinson, Bloch}.

We are thus lead to the problem of constructing algebraic cycles in arithmetic middle dimension on $\mr{Sh}_G$, and of relating their height pairings to the Taylor expansions of $L$-functions. One construction of such cycles, going back to W. Zhang, is given by imposing additional quadratic multiplication. The resulting cycles differ from the ones in \cite{LL1, LL2} which are instead closely related to Kudla--Rapoport divisors. We next describe this construction in more detail:

Let $E/F$ be a totally real quadratic extension and let $E\hookrightarrow D$ be an $F$-linear embedding such that $x^* = x$ for all $x\in E$. The centralizer $C = \mr{Cent}_D(E)$ then has center $EK$ and is preserved by $*$. If we further choose these data such that $β$ lies in $C^\times$, then we can define the algebraic group (over $\mbQ$)
\begin{equation}\label{eq:H}
H := \{x\in C^{\mr{op}, \times} \mid \text{$x^*βx = ν(x)β$ for some $ν(x) \in \mbG_m$}\}.
\end{equation}
This is an inner form of a unitary similitude group in $n$ variables for the quadratic extension $EK/E$. One may always find a PEL type Shimura datum $X_H$ for $H$ such that $(H, X_H)\to (G, X_G)$ is a morphism of Shimura data. Then we obtain a closed immersion $\mr{Sh}_H \hookrightarrow \mr{Sh}_G$ of Shimura varieties. Moreover, $\mr{Sh}_H$ is of the desired middle dimension $n-1$.

The existence of the Beilinson--Bloch height pairing is still conjectural. It is expected, however, that it decomposes into a sum of local heights whose non-archimedean terms are closely related to intersection numbers on integral models. For this reason, we consider integral models $\wt{\mr{Sh}_H}\hookrightarrow \wt{\mr{Sh}_G}$ of the given Shimura varieties. Our interest then lies in the intersection numbers $I(f) = \langle \wt{\mr{Sh}_H},\, f*\wt{\mr{Sh}_H} \rangle$ for varying Hecke operators $f\in C^\infty_c(G(\mbA_f))$, and their relations with leading terms of $L$-functions.

In this context, there is a relative trace formula (RTF) comparison approach due to Leslie--Xiao--Zhang \cite{LXZ} that goes back to work of Guo \cite{Guo} and Friedberg--Jacquet \cite{FJ}: One RTF is formulated for the given pair of groups $(G,\, H)$, the other one is formulated for the pair $(G',\, H')$ defined as $(GL_{2n},\, GL_n\times GL_n)$. The significance of the pair $(G',\, H')$ is that the $L$-functions of interest occur on the spectral side of its RTF. The problem thus becomes to relate the intersection numbers $I(f)$ with $(H'\times H')(\mbA)$-orbital integrals on $G'(\mbA)$.

This relation is made precise by factoring the global orbital integral into local ones, and analogously decomposing the global intersection number into place-by-place contributions. The latter relies on Rapoport--Zink (RZ) uniformization. One then, finally, arrives at a purely local question, namely that of expressing intersection numbers on moduli spaces of $p$-divisible groups (RZ spaces) in terms of local orbital integrals.

These ideas have recently lead to the discovery of new arithmetic fundamental lemma (AFL) identities: The linear AFL of the first author \cite{Li}, the variants from his joint work with B. Howard \cite{HL} and with the second author \cite{LM}, and the AFL for unitary groups of Leslie--Xiao--Zhang \cite{LXZ}. The terminology ``AFL'' here refers to the fact that all these are identities of intersection numbers on RZ spaces with \emph{good reduction} and central derivatives of local orbital integrals for \emph{spherical} Hecke functions. Put differently, they concern the places of good reduction of $\mr{Sh}_H$ and $\mr{Sh}_G$.

We mention that the above ideas were first proposed by W. Zhang in the context of the unitary arithmetic Gan--Gross--Prasad conjecture \cite{Z}. He deduced an AFL for unitary groups that he later proved with contributions of the second author and Z. Zhang \cite{Z19, M_loc_const, MZ, ZZ}. We refer to his ICM report \cite{Z_survey_18} for a survey.

Arithmetic transfer (AT) identities extend the realm of AFL identities in the sense that they express intersection numbers on RZ spaces with \emph{bad reduction} in terms of orbital integrals. Their role in the global setting is similar to that of AFL identities, but for the places of bad reduction of the Shimura varieties in question. This idea was first studied systematically by Rapoport--Smithling--Zhang \cite{RSZ1, RSZ2, RSZ3} in the context of the unitary arithmetic Gan--Gross--Prasad conjecture. Z. Zhang \cite{ZZ} generalized and solved one of their cases by proving AT identities in arbitrary dimension for maximal parahoric level for unramified quadratic extensions. His result has found global applications in the work of Disegni--Zhang \cite{DZ} who prove a $p$-adic variant of the arithmetic Gan--Gross--Prasad conjecture. Another application of AT in the global setting can be found in the work of C. Qiu \cite{Qiu} who proved AT identities for all places and level structures to show an analog of the Gross--Zagier formula over function fields.

\subsection{Arithmetic Transfer for Central Simple Algebras}

Consider again the above intersection numbers $I(f) = \langle \wt{\mr{Sh}_H},\, f*\wt{\mr{Sh}_H}\rangle$. The AT identities in the present paper provide an expression for the contributions to $I(f)$ from places of $F$ that are split in $K$ and inert in $E$. The completions of $G$ and $H$ at such places are essentially \emph{inner} forms of general linear groups. Correspondingly, we consider an intersection problem of EL type RZ spaces for central simple algebras (CSA). It is worth singling out two special cases:

If the CSAs in question split, then the resulting moduli spaces are Lubin--Tate spaces which leads to the intersection problem of the linear AFL mentioned above. If, however, the Hasse invariants of the CSAs in question are $1/2n$ and $1/n$, then the intersection problem is formulated for Drinfeld's half spaces and the cycles arise from Drinfeld's ``basic construction'' \cite[\S3]{Drinfeld}. This situation comes up when the two Shimura varieties have $p$-adic uniformization \cite[\S6.40]{RZ1}. 

We will now give a precise formulation of our AT conjecture and our results.

\subsection{The Fundamental Lemma Conjecture}

Let $F$ be a $p$-adic local field,\footnote{We allow local function fields in the main text with one exception: Section \ref{ss:displays} relies on the theory of $O_F$-displays from \cite{ACZ} which has only been developed for $p$-adic local fields.} let $E/F$ be an unramified quadratic field extension and let $η:F^\times \to \{\pm 1\}$ be the corresponding quadratic character. Let $G' = GL_{2n}(F)$ with subgroup $H' = GL_n(F\times F)$. For a test function $f'\in C^{\infty}_c(G')$, an element $γ\in G'$ that is regular semi-simple with respect to the $(H'\times H')$-action, and for a complex parameter $s\in \mbC$, we consider Guo's \cite{Guo} orbital integral
\begin{equation}\label{eq:orb_int_gamma_intro}
\Orb(γ, f', s) = \Omega(γ, s)\int_{\frac{H'\times H'}{(H'\times H')_γ}} f'(h_1^{-1} γ h_2) |h_1h_2|^s η(h_2)\ dh_1 dh_2.
\end{equation}
Here, the so-called transfer factor $\Omega(γ, s) \in \pm q^{\mbZ s}$ is defined in a way that ensures that $\Orb(γ, f', s)$ only depends on the double coset $H'γH'$; detailed definitions will be given in \S\ref{s:FL}. Our interest lies in the central value and the first derivative which we denote by
$$\Orb(γ, f') := \Orb(γ, f', 0),\qquad \del (γ, f') = \left.\frac{d}{ds}\right\vert_{s = 0} \Orb(γ, f', s).$$
Let $D/F$ be a CSA of degree $2n$ and let $E\subset D$ be a fixed embedding. The centralizer $C = \mr{Cent}_D(E)$ is again a CSA but over $E$. The Hasse invariants of $D$ and $C$ are related by $\mr{inv}_E(C) = 2\,\mr{inv}_F(D)$ (combine \cite[Corollary 9.1]{Draxl} with \cite[Proposition XIII.7]{Serre_lcft}). Define $(G, H) = (D^{\mr{op},\times}, C^{\mr{op},\times})$. The reason for passing to the opposite algebra here is that $D^{\mr{op}}$ is isomorphic to the ring of left endomorphisms of $D$ as left $D$-module. Given a test function $f\in C^\infty_c(G)$ and an element $g\in G$ that is regular semi-simple with respect to the $(H\times H)$-action, we consider the orbital integral
\begin{equation}\label{eq:orb_int_g_intro}
\Orb(g, f) := \int_{\frac{H\times H}{(H\times H)_g}} f(h_1^{-1} g h_2) dh_1 dh_2.
\end{equation}
It evidently only depends on the double coset $HgH$. Let $O_D\subseteq D$ be a maximal order such that $O_C := C\cap O_D$ is a maximal order of $C$ and put $f_D = 1_{O_D^\times}$. One can show that all such orders $O_D$ form a single $H$-conjugation orbit (Lemma \ref{lem:order_uniqueness}), so the orbital integrals $\Orb(g, f_D)$ do not depend on the choice of $O_D$. In dependence on the order $\ell$ of $D$ in the Brauer group of $F$, we will define a specific parahoric subgroup $K'\subseteq GL_{2n}(O_F)$. It corresponds to a parabolic of $(2n/\ell \times 2n/\ell)$-block upper triangular matrices and is chosen compatibly with the subgroup $H'\subseteq G'$. Let $f_D'^\circ = \mr{vol}(K'\cap H')^{-2} \cdot 1_{K'}$ and define $f'_D$ as a certain $H'$-translate of $f_D'^\circ$, see Definition \ref{def:test_function} for details. The orbital integrals $O(γ, f'_D, s)$ have a simple functional equation with respect to $s \longleftrightarrow -s$ (Proposition \ref{prop:functional_equation}).

Two regular semi-simple elements $γ\in G'$ and $g\in G$ are said to match if the invariants (in the sense of geometric invariant theory) of the orbits $H'γH'$ and $HgH$ agree. This notion defines an injection $[G_{\mr{rs}}]\hookrightarrow [G'_{\mr{rs}}]$ of the regular semi-simple orbits of $G$ into those of $G'$ which allows to compare orbital integrals on the two groups. Assuming all Haar measures are chosen compatibly, we conjecture that $f'_D$ is a smooth transfer of $f_D$ in the following sense:
\begin{conj}[Fundamental Lemma, \protect{Conjecture \ref{conj:FL}}]\label{conj:FL_intro}
Let $γ\in G'$ be an element that is regular semi-simple with respect to the $(H'\times H')$-action. Then
\begin{equation}\label{eq:FL_intro}
\Orb(γ, f'_D) = \begin{cases} \Orb(g, f_D) & \text{if there exists an element $g\in G$ that matches $γ$}\\
0 & \text{if there is no such $g$.}\end{cases}
\end{equation}
\end{conj}
The theorem about existence of smooth transfer of C. Zhang and H. Xue \cite{C_Zhang, Xue1} already states that there exists \emph{some} test function $f'\in C^\infty_c(G')$ such that $\Orb(γ, f')$ is given by the right hand side of \eqref{eq:FL_intro} for all $γ$. The new aspect of Conjecture \ref{conj:FL_intro} is that it provides $f'_D$ as a specific candidate. The choice of $f'_D$ is motivated by Z. Zhang's transfer result \cite{ZZ}. We explain this in more detail in Remark \ref{rmk:defn_test_function}.

Conjecture \ref{conj:FL_intro} specializes to the Guo--Jacquet fundamental lemma (FL) from \cite{Guo} if $D = M_{2n}(F)$. By work of N. Hultberg and the second author \cite[Theorem A]{HM}, Conjecture \ref{conj:FL_intro} is also known if the Hasse invariant of $D$ is $1/2$: The proof in this case is by reduction to the base change FL and to the Guo--Jacquet FL. Our main result on Conjecture \ref{conj:FL_intro} in the present paper is as follows.

\begin{thm}[\protect{Theorem \ref{thm:FL}}]\label{thm:FL_intro}
Assume that $D$ is a division algebra of degree $4$. Then Conjecture \ref{conj:FL_intro} holds.
\end{thm}

Note that if $D$ is a division algebra of degree $2n$, then $\Orb(g, f_D)$ is either $0$ or an integer divisor of $n$ (Proposition \ref{prop:orb_int_division}). In the case of Theorem \ref{thm:FL_intro} for example, $\Orb(γ, f'_D)$ is either $0$ or $1$ (Theorem \ref{thm:FL}).

\subsection{The Arithmetic Transfer Conjecture}
Denote by $\breve F$ the completion of a maximal unramified extension of $F$ and fix an embedding $E\subset \breve F$. Choose an isomorphism $\breve F\tensor_F C^{\mr{op}} \iso M_n(\breve F \times \breve F)$ that restricts to $(\mr{id}, τ) \tensor 1_n$ on the center $E$ of $C$, where $τ:E\to E$ is the Galois conjugation. Consider the conjugacy class of the ``Drinfeld type'' minuscule cocharacter
$$μ_H:\mbG_m \lr H_{\breve F} \iso (GL_n\times GL_n)_{\breve F},\quad t \longmapsto ((t, \ldots, t, 1),\, (t, \ldots, t)).$$
Let $μ_G: \mbG_m \to G_{\breve F}$ denote its composition with $H \to G$. Then every element $b\in B(H, μ_H)$ defines a morphism
\begin{equation}\label{eq:loc_Shdat}
(H,\, b,\, μ_H) \lr (G,\, b,\, μ_G)
\end{equation}
of local Shimura data in the sense of \cite{RV}. We denote by $H_b \to G_b$ the automorphism groups of the $H$-isocrystal (resp. $G$-isocrystal) defined by $b$. In analogy with our definition of $f_D$ as the characteristic function of $O_D^\times$, we consider integral models $\mcM_C$ and $\mcM_D$ for the local Shimura varieties for \eqref{eq:loc_Shdat} at levels $O_C^\times$ and $O_D^\times$. Their definition is as follows.

The elements of $B(H, μ_H)$ are in bijection with isogeny classes of pairs $(\mbY, ι)$ where $\mbY$ is a strict $O_F$-module of height $2n^2$ and dimension $n$ over the residue field $\mbF$ of $\breve F$, and where $ι:O_C\to \End(\mbY)$ is an $O_C$-action. The set $B(G, μ_G)$ is similarly in bijection with pairs $(\mbX, κ)$, where $\mbX/\mbF$ is a strict $O_F$-module of height $4n^2$ and dimension $2n$, and where $κ:O_D\to \End(\mbX)$ is an $O_D$-action. Under these bijections, the map $B(H, μ_H)\to B(G, μ_G)$ corresponds to the Serre tensor construction
\begin{equation}\label{eq:Serre_intro}
(\mbY, ι) \longmapsto (O_D\tensor_{O_C} \mbY,\,\, κ(x) := x\tensor \mr{id}_\mbY).
\end{equation}
Let $\mcM_C$ and $\mcM_D$ be the RZ spaces for $(\mbY, ι)$ and $(\mbX, κ) = O_D\tensor_{O_C} (\mbY, ι)$. These are certain EL type moduli spaces of $p$-divisible groups with $O_C$-action resp. $O_D$-action. They are formal schemes over $\Spf O_{\breve F}$ that are regular with semi-stable reduction and such that $\dim \mcM_D = 2\dim \mcM_C = 2n$. Furthermore, the groups $H_b$ and $G_b$ act from the right on $\mcM_C$ resp. $\mcM_D$ and there is a closed immersion
$$\mcM_C \lhook\joinrel\longrightarrow \mcM_D$$
that is equivariant with respect to $H_b\to G_b$. This closed immersion can be defined by a Serre tensor construction as in \eqref{eq:Serre_intro}.

\begin{defn}\label{def:int_intro}
Let $g\in G_b$ be a regular semi-simple element and let $Γ\subseteq H_b \cap g^{-1}H_b g$ be a free, discrete subgroup of covolume $1$. Define the intersection number
$$\Int(g) = \langle Γ\backslash \mcM_C,\,\, Γ\backslash g\cdot \mcM_C\rangle_{Γ\backslash \mcM_D} \in \mbZ.$$
\end{defn}
Taking the quotient by $Γ$ in this definition is the natural analog of taking the quotient by the stabilizer in the orbital integrals \eqref{eq:orb_int_gamma_intro} and \eqref{eq:orb_int_g_intro}. The restriction to regular semi-simple $g$ ensures that $Γ\backslash (\mcM_C\cap g\cdot \mcM_C)$ is a proper scheme over $\Spec O_{\breve F}$ with empty generic fiber. The definition is moreover independent of $Γ$.

The quantity $\Int(g)$ only depends on the $(H_b\times H_b)$-orbit of $g$, so we can use orbit matching to view it as a function on a subset of the $(H'\times H')$-orbits in $G'$. In this context, there is the following uniqueness and vanishing result.

\begin{prop}[\protect{\ref{prop:char_isogeny_class}}]\label{prop:b_intro}
\begin{enumerate}[wide, labelindent=0pt, labelwidth=!, label=(\arabic*), topsep=2pt, itemsep=2pt]
\item Let $γ\in G'$ be a regular semi-simple element. There exists at most one isogeny class $b\in B(H, μ_H)$ such that there exists an element $g\in G_b$ that matches $γ$.

\item Assume that such a pair $(b, g)$ exists. Then the sign of the functional equation of $\Orb(γ, f'_D, s)$ is negative, and in particular $\Orb(γ, f'_D) = 0$.
\end{enumerate}
\end{prop}

The second statement already hints that $\Int(g)$ might be related to the first derivative $\del (γ, f'_D)$ for matching $g$ and $γ$. This is made precise by the AT Conjecture in its explicit form:

\begin{conj}[ATC - Explicit Form]\label{conj:AT_intro}
There exists a correction function $f'_{\mr{corr}}\in C^\infty_c(G')$ such that for every regular semi-simple $γ\in G'_{\mr{rs}}$,
\begin{equation}\label{eq:ATC_explicit_intro}
\del(γ, f'_D) + \Orb(γ, f'_{\mr{corr}}) = \begin{cases} 2\,\Int(g)\log(q) & \text{\begin{varwidth}{\textwidth}if there exists some $b\in B(H, μ_H)$\\and some $g\in G_b$ that matches $γ$\end{varwidth}}\\[8pt]
0 & \text{otherwise.}
\end{cases}
\end{equation}
\end{conj}
Conjecture \ref{conj:AT_intro} builds on our FL Conjecture. There is also a weaker form that postulates the existence of a test function $f'\in C^\infty_c(G')$ whose orbital integral derivatives $\del (γ, f')$ agree with the intersection numbers $\Int(g)$ on the nose. We formulate and compare such variants in \S\ref{ss:ATC}.

The AT Conjecture is currently known in some low dimensional cases: Consider first the case of $D = M_{2n}(F)$. Then $f'_D = 1_{GL_{2n}(O_F)}$ and it is conjectured that one may take $f'_{\mr{corr}} = 0$. This is precisely the linear AFL conjecture from \cite{LM}. By the main results of \cite{Li_GL4} and \cite{LM}, it is known to hold whenever the connected part of $\mbY$ has height $\leq 4$.

Assume next that the Hasse invariant of $D$ is $1/2$. Then it is again expected that one may take $f'_{\mr{corr}} = 0$. The main result of \cite{HM} states that Conjecture \ref{conj:AT_intro} for such $D$ follows from the linear AFL.

The main result of the present paper, to be formulated in the next section, states that Conjecture \ref{conj:AT_intro} holds for division algebras of degree $4$. Here, $f'_\mr{corr}$ can be chosen from the Iwahori Hecke algebra. We speculate that this is a general phenomenon meaning that $f'_\mr{corr}$ can always be chosen as a linear combination of indicator functions of standard parahoric subgroups. In this context, we mention related work of He--Shi--Yang \cite{HSY} and He--Li--Shi--Yang \cite{HLSY}: There, the authors prove intersection number identities for Kudla--Rapoport divisors in the presence of bad reduction. Their result involves a unique characterization of certain occurring correction terms. It would be interesting to know if similar ideas apply in the context of Conjecture \ref{conj:AT_intro}.

\subsection{Invariant $1/4$ and $3/4$}
Assume from now on that $n = 2$ and that $D$ denotes a central division algebra (CDA) of Hasse invariant $λ\in \{1/4,\, 3/4\}$. Then $C$ is a quaternion division algebra over $E$. The test function $f_D'\in C^\infty_c(G')$ is an $H'$-translate of a scalar multiple $f'_\Iw$ of the characteristic function of an Iwahori in $G' = GL_4(F)$. We will also consider a test function $f'_\Par$ that is the characteristic function of a $(2\times 2)$-block parahoric in $GL_4(F)$.

The set $B(H, μ_H)$ is a singleton in this situation and the moduli space $\mcM_C$ is Drinfeld's half plane \cite{Drinfeld}. If $λ = 1/4$, then $\mcM_D$ is the four-dimensional Drinfeld half space. If $λ = 3/4$ however, then no explicit description of $\mcM_D$ is known. The two groups $H_b$ and $G_b$ are given by
$$H_b \iso GL_2(E),\qquad G_b \iso \begin{cases}
GL_4(F) & \text{if $λ = 1/4$}\\
GL_2(B) & \text{if $λ = 3/4$.}
\end{cases}$$
Here, $B/F$ denotes a quaternion division algebra.
\begin{thm}[\protect{Theorem \ref{thm:main_ATC}}]\label{thm:main_intro}
The AT conjecture holds for $D$. More precisely, let $f'_{\mr{corr}}$ be given by
\begin{equation}\label{eq:corr_function_intro}
f'_{\mr{corr}} = \begin{cases}
-4q\log(q)\cdot f'_\Par & \text{if $λ = 1/4$}\\
0 & \text{if $λ = 3/4$.}\end{cases}
\end{equation}
Then, for every regular semi-simple $γ\in G'_{\mr{rs}}$,
\begin{equation}\label{eq:AT_intro}
\del(γ, f'_D) + \Orb(γ, f'_{\mr{corr}}) = \begin{cases} 2\,\Int(g)\log(q) & \text{if there exists a $g\in G_b$ that matches $γ$}\\
0 & \text{otherwise.}
\end{cases}
\end{equation}
\end{thm}
Why and how it is that the two cases differ by a parahoric type orbital integral is a mystery to us. The difference only emerges during the proof of Theorem \ref{thm:main_intro} and is encoded in the $0$-dimensional embedded components of the intersection locus.

\subsection{Key Aspects}
\label{ss:aspects}

Our proof of Theorem \ref{thm:main_intro} is by determining explicitly and comparing both sides in \eqref{eq:AT_intro}. Key aspects are as follows:
\begin{enumerate}[wide, labelindent=0pt, labelwidth=!, label=(\arabic*), topsep=2pt, itemsep=2pt]
\item Concerning the orbital integral side, we combine three techniques to determine all occurring $\Orb(γ, f'_\Par)$, $\Orb(γ, f'_\Iw)$ and $\del(γ, f'_\Iw)$. First, we work out these orbital integrals for all hyperbolic orbits. In this situation, the computation can be reduced to a Levi subgroup and hence to $GL_2$. Second, we establish a germ expansion principle (Proposition \ref{prop:germ_exp}) that allows to write each orbital integral as a linear combination of a principal germ and a unipotent germ. The principal germ can be described explicitly in all situations. Third and finally, we use the results for hyperbolic orbits and the linear relations amongst the various germs to also determine the remaining orbital integrals. A summary of the final results can be found in \S\ref{s:main_analytic}. In particular, Proposition \ref{prop:derivative_teaser} gives a formula for the derivatives $\del(γ, f'_D)$ when $D$ is a division algebra of degree $4$. The basis for all mentioned results is a combinatorial expression for $\Orb(γ, f'_D)$ in terms of lattices, see \eqref{eq:orb_int_combinatorial}.

\item Concerning the intersection-theoretic side, we first prove a general formula for intersection numbers of surfaces in a $4$-dimensional space:
\end{enumerate}
\begin{prop}[see Corollary \protect{\ref{cor:intersection_simplified}}]\label{prop:intro_int_formula}
Let $X$ be a regular $4$-dimensional formal scheme that is locally formally of finite type over $\Spf O_{\breve F}$. Let $Y_1,\,Y_2\subseteq X$ be two regular $2$-dimensional closed formal subschemes. Assume that $Z = Y_1\cap Y_2$ is a proper scheme over $\Spec O_{\breve F}$ with empty generic fiber and of dimension $\leq 1$. Let $Z^{\mr{pure}},\, Z^{\mr{art}}\subseteq Z$ be its purely $1$-dimensional locus and the artinian subscheme of $0$-dimensional embedded components. Then
\begin{equation}\label{eq:intro_int_formula}
\langle Y_1, Y_2\rangle_X = \mr{len} (\mcO_{Z^{\mr{art}}})  - \deg (\det \mcC_1 \vert_{Z^{\mr{pure}}}) - \langle Z^{\mr{pure}}, Z^{\mr{pure}}\rangle_{Y_2}.
\end{equation}
Here, $\mcC_1$ is the conormal bundle of $Y_1\subseteq X$ and the intersection number on the very right is that of divisors on $Y_2$.
\end{prop}
\begin{enumerate}[wide, resume, labelindent=0pt, labelwidth=!, label=(\arabic*), topsep=2pt, itemsep=2pt]
\item In order to compute the occurring intersection numbers $\Int(g)$, we determine the three quantities on the right hand side of \eqref{eq:intro_int_formula}. The most important input here is Drinfeld's theorem \cite{Drinfeld} which provides an explicit description for $\mcM_C$ and, if $λ = 1/4$, for $\mcM_D$. This in particular enables us to compute the degree of the conormal bundle:
\end{enumerate}
\begin{prop}[see Propositions \protect{\ref{prop:conormal_1_4} and \ref{prop:conormal_3/4}}]\label{prop:intro_conormal}
Let $P\subseteq \mbF\tensor_{O_{\breve F}} \mcM_C$ be an irreducible component of the special fiber of $\mcM_C$. Let $\mcC$ be the conormal bundle of $\mcM_C\subseteq \mcM_D$. Then, for both Hasse invariants $λ\in \{1/4,\, 3/4\}$,
$$\deg(\mcC\vert_P) = q^2 - 1.$$
\end{prop}
Set $\mcI(g) = \mcM_C \cap g \cdot \mcM_C$. It is left to describe $\mcI(g)^{\mr{pure}}$ and $\mcI(g)^{\mr{art}}$, where the notation is meant in the sense of Proposition \ref{prop:intro_int_formula}.
\begin{enumerate}[wide, resume, labelindent=0pt, labelwidth=!, label=(\arabic*), topsep=2pt, itemsep=2pt]
\item The description of $\mcI(g)^{\mr{pure}}$ is in terms of the Bruhat--Tits stratification of the special fiber of $\mcM_C$: Each of its irreducible components is isomorphic to $\mbP^1$. Restricting to a fixed connected component $\mcM_C^0$ of $\mcM_C$, the dual graph of its special fiber is the Bruhat--Tits tree $\mcB$ for $PGL_{2,E}$. This is a $(q^2+1)$-regular tree whose vertices are in bijection with homothety classes of $O_E$-lattices $Λ\subseteq E^2$. Thus we may write
\begin{equation}\label{eq:intro_I_pure}
\mcM_C^0\cap \mcI(g)^{\mr{pure}} = \sum_{Λ\in \mr{Vert}(\mcB)} m(g, Λ) [P_Λ].
\end{equation}
One of our main auxiliary results (Theorem \ref{thm:classification_multiplicity_function}) describes the coefficient function $Λ\mapsto m(g, Λ)$ in terms of $g$.

\item The artinian locus $\mcI(g)^{\mr{art}}$ is where the two cases $λ = 1/4$ and $λ = 3/4$ are substantially different. If $λ = 1/4$, then $\mcM^0_C\cap \mcI(g)^{\mr{art}}$ consists of at most a single point of length $1$. If $λ = 3/4$, however, then every stratum $P_Λ$ that occurs in \eqref{eq:intro_I_pure} contributes artinian embedded components of total length $q$. The individual lengths and positions of these components depend on $g$ and $Λ$, see Table \ref{table:cases_I} for all possibilities. The total length of $Γ\backslash \mcI(g)^{\mr{art}}$ will, however, always match the orbital integral $4q\Orb(γ, f'_\Par)$ that makes up the difference of the two analytic sides for $λ = 1/4$ and $λ = 3/4$ in \eqref{eq:AT_intro}.

\item As mentioned before, there is no explicit description of $\mcM_D$ when $λ = 3/4$ which makes the computation of $\mcI(g)$ trickier than in the case $λ = 1/4$. We use a mix of Dieudonné theory, Cartier theory and display theory, as well as the previous results for the case of invariant $1/4$, to achieve a precise description, see the results in \S\ref{ss:inter_set_theoretic} and \S\ref{ss:superspecial} as well as Proposition \ref{prop:structure_3_4}.
\end{enumerate}

\subsection{Open Directions}

There is, of course, the question of how to prove Conjectures \ref{conj:FL} and \ref{conj:AT_intro} resp. the linear AFL \cite{LM} in general. We mention here three further problems of interest.

\begin{enumerate}[wide, labelindent=0pt, labelwidth=!, label=(\arabic*), topsep=2pt, itemsep=2pt]
\item The AT conjecture of the present article concerns inner forms of $GL_{2n}$ and their moduli spaces. An equally interesting question would be to study moduli spaces for $GL_{2n}$ (Lubin--Tate spaces), but with parahoric level structure. The global motivation from \S\ref{ss:global} applies verbatim to that situation.

\item Another question would be for an extension of our results to the biquadratic situation in the following sense: Instead of a single quadratic extension $E/F$, one considers two such extensions $E_1, E_2/F$, fixes embeddings $E_1, E_2\to D$, and defines intersection numbers from the corresponding cycles on $\mcM_D$. For the linear AFL, such a biquadratic generalization has been formulated by B. Howard and the first author \cite{HL}.

\item Finally, there is the problem of relating the local quantities of the present article with global intersection numbers and $L$-functions. We hope to return to this question in the future.
\end{enumerate}

\subsection{Layout of the paper}

We now give an overview of the contents of this paper. The paper consists of three parts.

In Part 1, we first give the group-theoretic setup (invariant theory, matching). We next define the orbital integrals of interest and formulate our FL conjecture. Then we introduce the moduli spaces of strict $O_F$-modules in question and the intersection numbers $\mr{Int}(g)$. We state and compare three variants of the AT conjecture.

In Part 2, we consider the case $G' = GL_4(F)$ and determine the quantities $\Orb(γ, f'_\Par)$, $\Orb(γ, f'_\Iw)$ and $\del (γ, f'_\Iw)$. A summary of our results can be found in \S\ref{s:main_analytic}, the main one being the proof of the FL conjecture for $f'_\Iw$. The further contents of Part 2 have been described at the beginning of \S\ref{ss:aspects}. 

In Part 3, we consider the case of a CDA $D/F$ of degree $4$ and compute the intersection numbers $\Int(g)$. Our main result is the proof of the AT conjecture in this situation. There are four main sections: In \S\ref{s:surface_intersection}, we prove the intersection number formula in Proposition \ref{prop:intro_int_formula}. In \S\ref{s:multiplicity_functions}, we study the functions $m(g,Λ)$ on $\mcB$ that will later describe the multiplicities of curves in the intersection locus, see \eqref{eq:intro_I_pure}. Then, in \S\ref{s:1_4}, we prove the AT conjecture for $λ = 1/4$. Because of Drinfeld's description of both $\mcM_C$ and $\mcM_D$ in this situation, this does not involve any $π$-divisible groups at all. Finally, in \S\ref{s:3_4}, we extend these results to $λ = 3/4$.

\subsection{Acknowledgements}

We are grateful to Michael Rapoport and Wei Zhang for their continued interest in this project and comments on an earlier version of our text. We furthermore thank Johannes Anschütz and Mingjia Zhang for helpful discussions. We also thank the referee for a careful reading of the article and suggestions for improvement.

\part{The Arithmetic Transfer Conjecture}

\section{Invariants}
\label{s:invariants}

Let $F$ be a field, let $E/F$ be an étale quadratic extension and let $D/F$ be a CSA of degree $2n$. Assume that there exists, and fix, an $F$-algebra embedding $E\to D$ that makes $D$ into a free $E$-module. (The latter condition is only relevant if $E\iso F\times F$.) Since $E/F$ is étale, $E\tensor_FE \iso E\times E$. So the left and right multiplication actions of $E$ on $D$ provide an eigenspace decomposition $D = D_+ \oplus D_-$ into $E$-linear and $E$-conjugate linear elements. That is, $D_+ = \mr{Cent}_D(E)$ which is the $E$-algebra $C$ from the introduction. Note that $\dim_F(D_+) = \dim_F(D_-)$. We denote the two components of an element $g\in D$ by $g_+$ and $g_-$.

Write $G = D^\times$ and $H = D_+^\times$ in the following. We consider the right-action
\begin{equation}\label{eq:group_action}
(H\times H) \times G \lr G,\quad (h_1,h_2)\cdot g = h_1^{-1}gh_2.
\end{equation}
An element $g\in G$ is called regular semi-simple if its $(H\times H)$-orbit is Zariski closed and if its stabilizer is of the minimal possible dimension. We denote these elements by $G_{\mr{rs}}$. The regular semi-simple orbits have been classified by Jacquet--Rallis \cite{JR} and Guo \cite[\S1]{Guo}. We work with the variant of their results that best suits our purposes.

\begin{defn}\label{def:invariant}
Let $g\in G$ be an element such that also $g_+$ lies in $G$. Then we define $z_g = g_+^{-1}g_-$ which lies in $D_-$. It can be thought of as the normalized conjugate-linear part of $g$. It is easily checked after base change to $\ob{F}$ that the reduced characteristic polynomial\footnote{Recall that the reduced characteristic polynomial of an element $x\in D$ is defined by $\mr{charred}_{D/F}(x; T) := \mr{char}_{D\tensor_F \ob F/ \ob F}(α(x\tensor 1); T)$ where $\ob F/F$ is an algebraic closure and $α:D\tensor_F \ob F\simto M_{2n}(\ob F)$ any choice of $\ob F$-algebra isomorphism.} of $z_g^2$ is a square, see \eqref{eq:matrix_example_central} below. We define the invariant of $g$ as its unique monic square root:
\begin{equation}\label{eq:def_invariant}
\Inv(g; T) = \mr{charred}_{D/F}\left(z_g^2;\ T\right)^{1/2} \in F[T].
\end{equation}
This is a monic polynomial of degree $n$. It satisfies $\Inv(g; 1) \neq 0$ because $1 + g_+^{-1}g_-$ and $1-g_+^{-1}g_-$ are both invertible, so $g_+^{-1}g_-$ does not have eigenvalues $\pm 1$. Indeed, $1+g_+^{-1}g_- = g_+^{-1}g$ is invertible by assumption. Let $ξ\in E^\times$ satisfy $\ob{ξ} = -ξ$. The identity $1-g_+^{-1}g_- = ξ(1 + g_+^{-1}g_-)ξ^{-1}$ implies that $1-g_+^{-1}g_-$ is invertible as well.
\end{defn}

It is clear by definition that $\Inv(g; T)$ only depends on the orbit $HgH$. The next lemma provides a converse for regular semi-simple elements.

\begin{lem}[Guo \protect{\cite[\S1]{Guo}}]\label{lem:invariant}
An element $g\in G$ is regular semi-simple if and only if both $g_+$, $g_-$ belong to $G$ and if the invariant $\Inv(g;T)$ is a separable polynomial. Moreover, two regular semi-simple elements $g_1$, $g_2\in G$ lie in the same $(H\times H)$-orbit if and only if their invariants agree.\qed
\end{lem}

\begin{ex}\label{ex:invariants}
Consider the split quadratic extension $F\times F$, the CSA $M_{2n}(F)$ and the embedding $ι:F\times F\to D$ that is given by $(a,b)\mapsto \diag(a1_n, b1_n)$. A $(2\times 2)$-block matrix $g = \left(\begin{smallmatrix} v & w \\ x & y\end{smallmatrix}\right)$ with blocks $v$, $w$, $x$, $y\in M_n(F)$ can only be regular semi-simple with respect to $ι$ if all its four blocks are invertible. In this case, we have
\begin{equation}\label{eq:matrix_example_central}
z_g^2 = \begin{pmatrix}
v^{-1}wy^{-1}x & \\ & y^{-1}xv^{-1}w
\end{pmatrix}.
\end{equation}
The invariant of $g$ is hence $\Inv(g; T) = \mr{char}(v^{-1}wy^{-1}x; T)$. Moreover, if $g$ is regular semi-simple and if $w\in GL_n(F)$ is any element with $\mr{char}(w; T) = \Inv(g; T)$, then
\begin{equation}\label{eq:simple_representative}
HgH = H\begin{pmatrix}
1 & w \\ 1 & 1
\end{pmatrix}H.
\end{equation}
\end{ex}

\begin{rmk}\label{rmk:comparison_Li}
A slightly different definition of invariant is given in \cite[Definition 1.1]{Li}. It is defined for every $g\in G$ and again a monic polynomial of degree $n$; let us call it $\Inv'(g;T)$. Assuming that $g_+$ lies in $G$, the two definitions are related by the Moebius transformation
\begin{equation}\label{eq:rel_invariants}
\Inv'(g; T) = T^n\, \Inv(g; 1)^{-1}\, \Inv\left(g; \frac{T-1}{T}\right).
\end{equation}
\end{rmk}

\begin{defn}\label{def:notation_rs}
Let $g\in G_{\mr{rs}}$ be regular semi-simple. We denote by $B_g\subseteq D$ the $F$-subalgebra that is generated by $E$ and $g^{-1}Eg$. We denote by $L_g \subset B_g$ its center. Up to isomorphism, these objects as well as $z_g$ only depend on the orbit $HgH$ because
\begin{equation}\label{eq:conjugacy_orbit}
B_{(h_1,h_2) g} = h_2^{-1}B_gh_2,\quad L_{(h_1,h_2) g} = h_2^{-1}L_gh_2,\quad z_{(h_1,h_2) g} = h_2^{-1}z_gh_2.
\end{equation}
\end{defn}

The next proposition summarizes their most important properties.
\begin{prop}\label{prop:quaternion_algebra}
Let $g\in G_{\mr{rs}}$ be a regular semi-simple element.
\begin{enumerate}[wide, labelindent=0pt, labelwidth=!, label=(\arabic*)]
\item The square $z_g^2$ lies in $L_g$. In fact, $L_g$ equals $F[z_g^2]$ and is, in particular, an étale $F$-algebra of degree $n$ that is isomorphic to $F[T]/(\Inv(g; T))$.
\item The composite $EL_g$ of $E$ and $L_g$ in $D$ is isomorphic to $E\tensor_F L_g$ and, in particular, an étale quadratic $L_g$-algebra.
\item The algebra $B_g$ equals $E L_g[z_g]$ and is, in particular, a quaternion algebra over $L_g$. It coincides with the centralizer $\mr{Cent}_D(L_g).$
\end{enumerate}
\end{prop}
\begin{proof}
This is a special case of \cite[Proposition 2.5.4]{HL}. Since the argument is short and instructive, and since our notation is slightly different from that in \cite{HL}, we include a proof for convenience.

Choose an $F$-algebra generator $ζ\in E$ to write $B_g = F[ζ, g^{-1}ζg]$. We make this choice with $\mr{tr}_{E/F}(ζ) = 1$, i.e. $\ob{ζ} = 1 - ζ$. Then we obtain (put $z = z_g$)
\begin{equation}\label{eq:get_to_linear}
\begin{aligned}
g^{-1}ζg \ & = (1 + z)^{-1} ζ (1 + z)\\
& = ζ + \frac{z}{1+z}(1-2ζ).
\end{aligned}
\end{equation}
It always holds that $1-2ζ \in E^\times$: This is clear if $E$ is a field and can be verified directly if $E \iso F\times F$. The fraction $z/(1+z)$ is a Moebius transformation that is defined at $z$. The inverse Moebius transformation is then defined at $z/(1+z)$, so we obtain $B_g = F[ζ, z]$. It is evident from definitions that $z^2$ commutes with both $ζ$ and $z$, and hence lies in the center $L_g$ of $B_g$.

\emph{Claim: The elements $1, z, \ldots, z^{2n-1}, ζ, ζz, \ldots, ζz^{2n-1}$ form an $F$-vector space basis of $B_g$.} This may be shown after base change to $\ob{F}$ which puts us into the situation of Example \ref{ex:invariants}. Then we may assume that $g$ is given as a block matrix $g = \left(\begin{smallmatrix} 1 & w \\ 1 & 1\end{smallmatrix}\right)$ as in \eqref{eq:simple_representative}. In this specific case we have
\begin{equation}
z_g = \begin{pmatrix}
 & w \\ 1 &
\end{pmatrix} \quand B_g = (F\times F)[z_g].
\end{equation}
The claim then follows from the fact that the characteristic polynomial of $w$, which equals $\Inv(g; T)$, is separable by Lemma \ref{lem:invariant} and hence agrees with the minimal polynomial. The identities $L_g = F[z_g^2]$ as well as $EL_g \iso E\tensor_FL_g$ and $B_g = EL_g[z_g]$ all follow from the claim.

It remains to show the statement $B_g = \mr{Cent}_D(L_g)$. We have already seen that $EL_g$ is an étale $F$-algebra of degree $2n = (\dim_F D)^{1/2}$. The only possibility for $D$ is then to be free as $EL_g$-module. This implies that $D$ is free as $L_g$-module so the centralizer $\mr{Cent}_D(L_g)$ is a quaternion algebra over $L_g$. It also contains $B_g$, however, and hence equals $B_g$.
\end{proof}

We call a polynomial $δ\in F[T]$ regular semi-simple if it is monic, separable and satisfies $δ(0)δ(1) \neq 0$. Example \ref{ex:invariants} shows that the regular semi-simple polynomials of degree $n$ are in bijection with the regular semi-simple $GL_n(F\times F)$-double cosets on $GL_{2n}(F)$. The following construction is taken from \cite[Proposition 2.5.6]{HL}.

\begin{defn}\label{def:universal_quaternion}
Let $δ\in F[T]$ be regular semi-simple. We define the two $F$-algebras
$$L_δ := F[z^2]/(δ(z^2))\quad \text{and}\quad  B_{δ} := (E\tensor_F L_δ)[z]$$
with commutator relation $(a\tensor b)z = z(\ob{a}\tensor b)$ for $a\in E$, $b\in L_δ$. Note that if $g\in G_{\mr{rs}}$ is a regular semi-simple element, then $B_g \iso B_δ$ by Proposition \ref{prop:quaternion_algebra}. We call $B_δ$ the universal quaternion algebra for invariant $δ$ because it detects orbits of invariant $δ$ in the following sense.
\end{defn}

\begin{cor}\label{cor:universal_quaternion}
Let $δ\in F[T]$ be regular semi-simple of degree $n$. The following three conditions are equivalent.
\begin{enumerate}[wide, labelindent=0pt, labelwidth=!, label=(\arabic*)]
\item There exists an element $g\in G_{\mr{rs}}$ of invariant $δ$.
\item There exists an $F$-algebra embedding $B_δ \to D$.
\item The identity $[B_δ] = [L_δ\tensor_F D]$ holds in the Brauer group of $L_δ$.
\end{enumerate}
\end{cor}
\begin{proof}
Assume that (1) holds and let $g\in G_{\mr{rs}}$ be such that $\Inv(g;T) = δ(T)$. Then Proposition \ref{prop:quaternion_algebra} states that $B_g\iso B_δ$, so (2) holds. Conversely, assume that there exists an embedding $ι:B_δ\to D$. Then $ι(E\tensor_F L_δ) \subset D$ is a commutative $F$-subalgebra of $F$-dimension $2n = [D:F]$. It is hence a maximal commutative subalgebra, so $D$ is a free $ι(E\tensor_F L_δ)$-module. In particular, $D$ is free both as $ι(E)$-module and as $ι(L_δ)$-module. It then follows from the Skolem--Noether Theorem that the given embedding $E\to D$ and $ι\vert_E:E\to D$ are conjugate. We may hence find $ι$ such that $ι\vert_E$ agrees with the given embedding of $E$. Then $g = 1 + ι(z)$ has the property that $g_+ = 1$ and $g_- = ι(z)$, and consequently that $z_g^2 = ι(z^2)$. Since $D$ is free over $ι(L_δ)$, it holds that
$$\mr{charred}_{D/F}(z^2; T) = \mr{char}_{L_δ/F}(z^2; T)^2$$
and hence that $\Inv(g; T) = δ(T)$. This shows that (2) implies (1).

We now prove the equivalence of (2) and (3) which holds more generally. Let $B$ be a quaternion algebra over an étale $F$-algebra $L$ of degree $n$. We claim the equivalence of
\begin{enumerate}[wide, labelindent=0pt, labelwidth=!, label=(\arabic*)]
\item[(2)] There exists an embedding $ι:B\to D$.
\item[(3)] It holds that $[B] = [L\tensor_FD]$ in the Brauer group of $L$.
\end{enumerate}
Assume that there exists an embedding $ι:B\to D$. Then $D$ is necessarily free as $ι(L)$-module and $ι(B) = \mr{Cent}_D(ι(L))$ for dimension reasons. The identity $[B] = [L\tensor_FD]$ follows from (a mild extension of) the centralizer theorem \cite[Theorem 9.6]{Draxl}. This shows that (2) implies (3).

Assume conversely that $[B] = [L\tensor_FD]$ holds. Let $M/L$ be a quadratic étale extension that splits $B$. Then it also holds that $M\tensor_F D \iso M_{2n}(M)$. Let
\begin{equation}
L = \prod_{i\in I} L_i,\quad M = \prod_{i\in I}M_i,\quad B = \prod_{i\in I}B_i
\end{equation}
denote the factorizations of $L$, $M$ and $B$ that correspond to the idempotents of $L$. Also pick an isomorphism $D \iso M_m(D_0)$ for a central division algebra (CDA) $D_0$. Each factor $M_i$ splits $D_0$, so $d = \dim_F(D_0)^{1/2}$ divides $\dim_F(M_i) = 2[L_i:F]$ by \cite[Corollary 9.4]{Draxl}. For every $i$, we define $D_i = M_{[M_i:F]/d}(D_0)$. By \cite[Corollary 9.3]{Draxl}, there exists an $F$-algebra embedding $M_i\to D_i$. Note that $\sum_{i\in I} [M_i:F]/d = m$, so we can form a block diagonal embedding
$$ι:M = \prod_{i\in I} M_i \hookrightarrow \prod_{i\in I} D_i \hookrightarrow D.$$
It follows from $[M:F] = \dim_F(D)^{1/2}$ that $D$ is free as $ι(M)$-module and hence also free over $ι(L)$. The centralizer $\mr{Cent}_D(ι(L))$ satisfies the identity $[\mr{Cent}_D(ι(L))] = [L\tensor_FD]$ and is hence isomorphic to $B$. Any choice of such an isomorphism defines an embedding $B\to D$. This shows that (3) implies (2).
\end{proof}

We will also require a definition of invariant for semi-simple $F$-algebras. Assume in what follows that $D/F$ is a finite-dimensional semi-simple $F$-algebra with center $Z$. Write $Z = \prod_{i\in I} Z_i$ as a product of fields and also decompose $D$ accordingly, $D = \prod_{i\in I} D_i$. We assume that $D$ is of total degree $2n$ in the sense that
\begin{equation}\label{eq:degree_ss}
2n = \sum_{i\in I} [Z_i:F]\cdot [D_i:Z_i].
\end{equation}
Finally, we assume the existence of, and fix, an embedding $E\to D$ such that each component $D_i$ becomes a free $E$-module. Then there is an eigenspace decomposition $D = D_+\oplus D_-$ as before and we continue to write $g = g_+ + g_-$ for the corresponding decomposition of elements $g\in D$. We also write $g_i$ for the $i$-th component of $g$.
\begin{defn}\label{def:invariant_not_central}
Let $g\in G$ be an element with $g_+\in G$. The invariant of $g$ is defined as
\begin{equation}\label{eq:invariant_ss}
\Inv(g; T) = \prod_{i\in I} \Inv(g_i; T) \in F[T].
\end{equation}
It is a monic polynomial of degree $n$. We call $g$ regular semi-simple if $\Inv(g;T)$ is regular semi-simple in the same sense as before.
\end{defn}

For example, the element $1 + z\in B_δ$ from Definition \ref{def:universal_quaternion} has invariant $\Inv(1+z; T) = δ$ with respect to the embedding $E\to B_δ$ that comes by construction.

Definition \ref{def:notation_rs} and the statements of Proposition \ref{prop:quaternion_algebra} apply and remain true without change for semi-simple $D$. In Corollary \ref{cor:universal_quaternion}, the equivalence of (1) and (2) remains true as well.

\section{Fundamental Lemma}
\label{s:FL}

\newcommand{\bdot}{\,\cdot\,}

\subsection{Setting}
\label{ss:setting}

We maintain the following setting throughout the paper.
\begin{enumerate}[wide, labelindent=0pt, labelwidth=!, label=(\arabic*), topsep=2pt, itemsep=2pt]
\item We denote by $F$ a non-archimedean local field with ring of integers $O_F$, uniformizer $π$ and residue cardinality $q$. We let $E/F$ be an unramified quadratic field extension with ring of integers $O_E$.

\item We let $K = F\times F$ denote the split quadratic extension of $F$. We view it as a subring of $M_{2n}(F)$ by the diagonal embedding $(a, b)\mapsto \diag(a1_n, b1_n)$ and we define
$$G' = GL_{2n}(F),\quad H' = GL_n(K).$$
Given $γ\in M_{2n}(F)$, we write $γ = γ_+ + γ_-$ for its decomposition into $K$-linear and conjugate-linear components. Whenever we speak of regular semi-simple elements of $G'$ or of their invariants, then this is meant with respect to the $(H'\times H')$-action. In fact, this is precisely the setting from Example \ref{ex:invariants}, albeit in different terminology.

\item We denote by $D$ a CSA of degree $2n$ over $F$. We fix an embedding $E\to D$ and use the the notations $D = D_+ \oplus D_-$ as well as $g = g_+ + g_-$ like before. We interchangeably write $C = D_+$, and define
$$G = D^\times,\quad H = C^\times.$$
(We will switch to opposed CSAs in \S4.) Whenever we speak of regular semi-simple elements of $G$ or of their invariants, then this is meant with respect to the $(H\times H)$-action.

\item Our normalization of the Hasse invariant is as follows. Let $F_{2n}/F$ be an unramified field extension of degree $2n$ with Frobenius $σ$. Then there is a unique integer $0\leq r < 2n$ such that $D$ is isomorphic to the cyclic $F$-algebra
$$F_{2n}[Π]/(Π^{2n} = π^r,\ Πa = σ(a)Π\ \text{for $a\in F_{2n}$}).$$
The Hasse invariant of $D$ is defined as $r/2n \in \mbQ/\mbZ$. For $λ\in \mbQ/\mbZ$, we write $D_λ$ to denote a CDA of Hasse invariant $λ$ over $F$.
\end{enumerate}

\subsection{Orbital Integrals}

Let $η:F^\times \to \{\pm 1\}$ be the non-trivial unramified quadratic character. In this section, we define and compare two kinds of orbital integrals. The first kind are $η$-twisted orbital integrals on $[G'] = H'\backslash G'/H'$. The second kind are orbital integrals on $[G] = H\backslash G/H$. In fact, the orbital integrals will only be defined on the regular semi-simple orbits $[G'_{\mr{rs}}]$ and $[G_{\mr{rs}}]$.

\subsubsection{Orbital integrals on $[G']$}
Given $γ\in G'$, we denote its stabilizer by
\begin{equation}\label{eq:stab_gamma}
(H'\times H')_γ := \{(h_1,h_2)\mid h_1^{-1}γh_2 = γ\}.
\end{equation}
The stabilizer of a regular semi-simple element $γ$ is isomorphic to the torus $L_γ^\times$, where $L_γ = L[z_γ^2]$ is the étale $F$-algebra of degree $n$ from Proposition \ref{prop:quaternion_algebra} (1). Indeed, we may rewrite \eqref{eq:stab_gamma} as
$$(H'\times H')_γ = \{(γhγ^{-1}, h) \mid h \in H' \cap γ^{-1}H'γ\}.$$
The intersection $H'\cap γ^{-1}H'γ$ is by definition the centralizer of $B_γ = F[K \cup γ^{-1}Kγ]$ in $G'$ (see Definition \ref{def:notation_rs}). Since $[L_γ:F] = n$ and since $B_γ/L_γ$ is a quaternion algebra, this centralizer equals the units of $L_γ \subset B_γ$.

We endow $(H'\times H')_γ$ with the Haar measures such that $O_{L_γ}^\times$ has volume $1$. We also normalize the Haar measure on $H'\times H'$ such that a maximal compact subgroup has volume $1$.

Let $|\cdot|:F^\times \to \mbR,\ x\mapsto q^{-v(x)}$ be the normalized absolute value on $F$. We define $η$ and $|\cdot |$ on $H'$ in the following way,
\begin{equation}\label{eq:def_characters}
η,\ |\cdot|:H'\lr \mbR,\quad η(\mr{diag}(a,b)) = η(\det(ab^{-1})),\quad |\mr{diag}(a,b)| = |\det(ab^{-1})|.
\end{equation}
\begin{defn}\label{def:orb_int_analytic}
For $γ\in G'_{\mr{rs}}$, a test function $f'\in C^\infty_c(G')$ and $s\in \mbC$, we define the orbital integral
\begin{equation}\label{eq:def_orb_int_GL}
O(γ,f',s) := \int_{\frac{H'\times H'}{(H'\times H')_γ}} f'(h_1^{-1} γ h_2) |h_1h_2|^s η(h_2)\ dh_1 dh_2.
\end{equation}
\end{defn}
The support of the integrand in \eqref{eq:def_orb_int_GL} is compact because the $(H'\times H')$-orbit of a regular semi-simple element is Zariski closed. This ensures convergence, and the resulting expression $O(γ, f', s)$ lies in $\mbC[q^s, q^{-s}]$. However, as a function of $γ$, the orbital integral does not yet descend to the orbit space $[G'_{\mr{rs}}]$ because it transforms by the character $η(\bdot)|\cdot|^s$ under the $(H'\times H')$-action. We next modify it in the simplest possible way that makes it $H'\times H'$-invariant.

\begin{defn}\label{def:transfer_factor}
Let $s\in \mbC$. Define the transfer factor $\Omega(\bdot, s) : G'_{\mr{rs}} \to \pm q^{\mbZ s}$ by
$$Ω\left(\left(\begin{smallmatrix} a & b \\ c & d\end{smallmatrix}\right),s\right)=η(\det(cd^{-1}))\cdot|\det(b^{-1}c)|^s.$$
\end{defn}
It satisfies $\Omega(h_1^{-1}γ h_2, s) = |h_1 h_2|^s \eta(h_2)\Omega(\gamma, s)$, so we can modify and rewrite \eqref{eq:def_orb_int_GL} as
\begin{equation}\label{eq:def_orb_int_GL_omega}
\Orb(γ, f', s) := \Omega(\gamma,s)\cdot O(γ,f',s) = \int_{\frac{H'\times H'}{(H'\times H')_γ}} f'(h_1^{-1} γ h_2) \cdot\Omega(h_1^{-1} γ h_2, s)\ dh_1 dh_2.
\end{equation}
Then $\Orb(γ, f', s)$ is $(H'\times H')$-invariant and descends to the orbit space $[G'_{\mr{rs}}]$. Note that we still have
\begin{equation}\label{eq:orb_int_trafo}
\Orb(γ, (h_1, h_2)^*(f'), s) = |h_1h_2|^{-s}η(h_2)\Orb(γ, f', s)
\end{equation}
for all $(h_1, h_2) \in H'\times H'$. We will mostly be interested in the central value and the central derivative of $\Orb(γ, f', s)$ which we denote by
\begin{equation}\label{eq:def_orb_int_central}
\Orb(γ, f') := \Orb(γ, f', 0),\quad\quad \del(γ, f') := \left.\frac{d}{ds}\right\vert_{s = 0}\Orb(γ,f',s).
\end{equation}

\subsubsection{Orbital integrals on $[G]$}
The definition of orbital integrals on $[G_{\mr{rs}}]$ is more straightforward because it does not involve any characters. Given $g\in G$, we denote its stabilizer by
$$(H\times H)_g = \{(h_1,h_2)\mid h_1^{-1}gh_2 = g\}.$$
As before, the stabilizer of a regular semi-simple element $g$ is isomorphic to the torus $L_g^\times$. Again we endow $H\times H$ and $(H\times H)_g$ with the Haar measures such that a maximal compact subgroup has volume $1$. Note that all maximal compact subgroups of $H\times H$ are conjugate, so this Haar measure is well-defined.

\begin{defn}\label{def:orb_int_geometric}
For $g\in G$ regular semi-simple and $f\in C^\infty_c(G)$, we define the orbital integral
\begin{equation}\label{geo-orb}
\Orb(g,f) = \int_{\frac{H\times H}{(H\times H)_g}} f(h_1^{-1} g h_2)\ dh_1 dh_2.
\end{equation}
\end{defn}
This function evidently descends to $[G_{\mr{rs}}]$.

\subsubsection{Transfer of orbital integrals}
We now compare orbital integrals on $[G'_{\mr{rs}}]$ and $[G_{\mr{rs}}]$.

\begin{defn}[\protect{\cite{Xue1}}]\label{def:matching}
(1) Two regular semi-simple elements $γ\in G'_{\mr{rs}}$ and $g\in G_{\mr{rs}}$ (resp. their orbits) are said to match if $\Inv(γ) = \Inv(g)$. Note that in this case also $L_γ \iso L_g$ so that we have chosen compatible Haar measures on the stabilizers $(H'\times H')_γ$ and $(H\times H)_g$.

\noindent (2) A test function $f'\in C^\infty_c(G')$ is called a transfer of $f\in C^\infty_c(G)$ if, for all $γ\in G'_{\mr{rs}}$,
\begin{equation}
\Orb(γ, f') = \begin{cases} \Orb(g, f) & \text{if there is a matching $g\in G_{\mr{rs}}$}\\
0 & \text{otherwise.}\end{cases}
\end{equation}
\end{defn}
Transfers in this sense exist by a result of C. Zhang \cite{C_Zhang}, also see \cite[Proposition 2.9]{Xue1}.

We note that Definition \ref{def:matching} is analogous to that of transfer in the context of the Jacquet--Rallis relative trace formula comparison, see e.g. \cite[\S2.4]{Zhang_FT} or \cite[Definition 2.2]{RSZ2}. However, a difference in our setting is that the matching relation does not yield a partitioning of $[G'_{\mr{rs}}]$ into sets of the form $[G_{\mr{rs}}]$. We illustrate this for $n = 2$:

\begin{ex}\label{ex:matching_n_equal_2}
Assume that $G' = GL_4(F)$ and let $γ\in G'_{\mr{rs}}$ be a regular semi-simple element with invariant $δ$. Our aim is to describe all possibilities for the CSA $D$ such that there exists a matching element $g\in G$. To this end, set $L_δ = F[z^2]/(δ(z^2))$ and let $B_δ = (E\tensor_F L_δ)[z]$ be the universal quaternion algebra for invariant $δ$ (with respect to $E/F$) that was constructed in Definition \ref{def:universal_quaternion}. By Corollary \ref{cor:universal_quaternion}, there exists an element $g\in G_{\mr{rs}}$ of invariant $δ$ if and only if there exists an embedding $B_δ\to D$. The following lists all the possibilities for this situation, each of which can occur.

\begin{table}[h!]
\centering
\def\arraystretch{1.3}
\begin{tabular}{|c|c|c|c|c|}
\hline
\multirow{2}{*}{$L_δ$} & \multirow{2}{*}{$B_δ$} & $D$ s.th. there is some $g\in G$ & $ε_0(δ)$ & $ε_{1/4}(δ)$\\
& & with $\Inv(g) = δ$ & $ε_{1/2}(δ)$ & $ε_{3/4}(δ)$\\
\hline
Field & $M_2(L)$ & $M_4(F)$ and $M_2(D_{1/2})$ & $+$ & $-$ \\
\hline
Field & Division & $D_{1/4}$ and $D_{3/4}$ & $-$ & $+$\\
\hline
$F\times F$ & $M_2(F)\times M_2(F)$ & $M_4(F)$ & $+$ & $-$\\
\hline
$F\times F$ & $D_{1/2}\times D_{1/2}$ & $M_2(D_{1/2})$ & $+$ & $-$\\
\hline
$F \times F$ & $M_2(F)\times D_{1/2}$ & none & $-$ & $+$\\
\hline
\end{tabular}
\medskip
\caption{Matching to $[G_{\mr{rs}}]$ for $n = 2$. Here, $D_λ$ denotes a CDA of Hasse invariant $λ$ over $F$. Moreover, $ε_λ(δ)$ denotes $ε_D(δ)$ for $D$ a CSA of degree $4$ and Hasse invariant $λ$. These are the signs of the functional equation for $f'_D$ and will be defined in \S\ref{ss:functional_eqn}.}
\label{table:matching}
\end{table}
\end{ex}

\subsection{The Fundamental Lemma Conjecture}
\label{ss:FL}
Recall that $C = D_+ = \mr{Cent}_D(E)$.

\begin{lem}\label{lem:order_uniqueness}
Let $O_1, O_2\subseteq D$ be two maximal orders that have the property that  $O_1\cap C$ and $O_2 \cap C$ are maximal orders in $C$. Then $O_1 \cap C = O_2\cap C$ implies $O_1 = O_2$. In particular, the maximal orders $O \subset D$ such that $O\cap C$ is also a maximal order form a single $C^\times$-conjugation orbit.
\end{lem}
\begin{proof}
The second statement follows directly from the first one because all maximal orders in $C$ are $C^\times$-conjugate. We focus on the first statement from now on.

Let $O_C = O_1 \cap C = O_2 \cap C$. Note that $O_E \subset O_C$ because it is the ring of integers of the center of $C$. Choose a suitable skew-field $Q$ and an isomorphism $D \iso M_m(Q)$. Let $Π\in Q$ be a uniformizer. Recall that $Q$ has a unique maximal order $O_Q$ and that the maximal orders in $D$ are precisely the subrings of the form $O_Λ = \End_{O^{\mr{op}}_Q}(Λ)$ where $Λ\subset Q^m$ is an $O_Q^{\mr{op}}$-lattice. Moreover, $O_Λ = O_{Λ'}$ if and only if $Λ' \in Λ Π^{\mbZ}$. In this way, classifying maximal orders in $D$ that contain $O_E$ is equivalent to classifying $R := O_E\tensor_{O_F} O_Q^{\mr{op}}$-stable lattices in $Q^m$, up to scaling by $Π^\mbZ$.

There are two cases for the ring $R$, depending on the parity of the degree $\ell = 2n/m$ of $Q$ over $F$. Let $M/F$ be an unramified field extension of degree $\ell$. With a suitable choice of $Π$ and for a suitable generator $τ\in \mr{Gal}(L/F)$, we may find a presentation of $O_Q$ as
\begin{equation}\label{eq:max_order_presentation}
O^{\mr{op}}_Q \iso O_M[Π],\quad Π^\ell = π,\quad Πa = τ(a)Π\ \text{for all $a\in M$.}
\end{equation}
If $\ell$ is odd, then $O_E\tensor_{O_F}O_M$ is the maximal order in an unramified field extension of $F$ of degree $2\ell$. Then $R \iso (O_E\tensor_{O_F}O_M) [Π]$ is again of the form \eqref{eq:max_order_presentation} and hence the maximal order in the skew field $E\tensor_FQ^{\mr{op}}$. The $Q^{\mr{op}}$ vector space $Q^m$ is isomorphic to $(E\tensor_FQ^{\mr{op}})^{m/2}$ as $(E\tensor_FQ^{\mr{op}})$-module, hence its $R$-stable lattices form a single orbit under $GL_R(Q^m) = C^\times$. This shows that any two maximal orders in $D$ that contain $O_E$ are $C^\times$-conjugate. In particular, we have proven the lemma whenever $\ell$ is odd.

The situation is a little different when $\ell$ is even. To simplify notation in the following, we make the further assumption that $E$ is contained in $M$. (This is meant with respect to the inclusions $M\subseteq Q \subseteq D$.) Starting from the presentation \eqref{eq:max_order_presentation} again, we then obtain
\begin{equation}\label{eq:order_l_even}
\begin{aligned} R &\ \overset{\sim}{\lr} (O_M[Π^2]\times O_M[Π^2])[Π], \quad Π(a, b) = (τ(b), τ(a))Π\ \ \text{for all $a,b\in O_M[Π^2]$.}\\
a\tensor m&\ \longmapsto (am, \bar{a}m)\end{aligned}
\end{equation}
Here, we have extended $τ$ to $M[Π^2]$ by $τ(Π^2) = Π^2$. The ring $O_M[Π^2]$ is the maximal order in the skew field $M[Π^2]$ and we can decompose $Q^m \iso V_0\times V_1$ as $M[Π^2] \times M[Π^2]$-module. The operator $Π$ is homogeneous of degree $1$ with respect to that decomposition, so $V_0$ and $V_1$ are both of dimension $m$ over $M[Π^2]$. Moreover, every $R$-stable lattice $Λ\subset Q^m$ is of the form
$Λ = Λ_0\times Λ_1$, where $Λ_i\subset V_i$ is an $O_M[Π^2]$-stable lattice and where $Λ_iΠ \subset Λ_{i+1}$ for both $i = 0,1$. Conversely, the direct sum of any such pair $(Λ_0, Λ_1)$ is an $R$-stable lattice in $Q^m$.

After this general description, we now prove the statement of the lemma. The centralizer $C$ acts diagonally on $V_0\times V_1$. Given an $R$-stable lattice $Λ = Λ_0\times Λ_1$, we find that
$$\mr{Stab}_C(Λ) = \mr{Stab}_C(Λ_0)\cap \mr{Stab}_C(Λ_1).$$
Moreover, since the $C$-action commutes with $Π$, we may write $\mr{Stab}_C(Λ_1) = \mr{Stab}_C(Λ_1Π)$. We see that $\mr{Stab}_C(Λ)$ is a maximal order in $C$ if and only if $Λ_0 \in Λ_1Π \cdot Π^{2\mbZ}$, which holds if and only if
\begin{equation}\label{eq:lattice_relation_max_order}
Λ_0 = Λ_1Π\quad \text{or}\quad Λ_1 = Λ_0Π.
\end{equation}
The set of $O_M[Π^2]$-lattices in $V_0$ form a single $C^\times$-orbit. Thus, the set of $R$-stable lattices $Λ$ such that $\mr{Stab}_C(Λ)$ is a maximal order form two $C^\times$-orbits that are distinguished by \eqref{eq:lattice_relation_max_order}. However, they are interchanged by multiplication by $Π\in Q^{\mr{op}}$ and, in particular, define the same $C^\times$-conjugation orbit of maximal orders in $C$ and $D$. The proof of the lemma is now complete.
\end{proof}

\begin{ex}\label{ex:max_order_D}
One byproduct of the above proof is the following statement: Assume that $\ell$ is odd and that $O\subset D$ is a maximal order that contains $O_E$. Then the intersection $O\cap C$ is a maximal order in $C$.

Consider, for example, an embedding $E\to M_{2n}(F)$. If $O_E\subset \End(Λ)$ for some $O_F$-lattice $Λ\subset F^{2n}$, then $Λ$ is an $O_E$-lattice and $\End_{O_F}(Λ)\cap \End_E(V) = \End_{O_E}(Λ)$ is a maximal order.

The statement does not hold true if $\ell$ is even: Let $Q = D_{1/2}$ be a quaternion division algebra over $F$ with uniformizer $Π$ and let $E\to Q$ be a fixed embedding. Let $D = M_2(Q)$ and let $E\to D$ be the diagonal embedding. The centralizer $C = \mr{Cent}_D(E)$ is then simply $M_2(E)$. Both the maximal orders $O_1 = M_2(O_Q)$ and $O_2 = \diag(Π, 1)^{-1}O_1\diag(Π,1)$ contain $O_E$. However, they do not both intersect $C$ in a maximal order:
$$O_1 \cap M_2(E) = M_2(O_E),\quad O_2 \cap M_2(E) = \begin{pmatrix}
O_E & O_E \\ (π) & O_E
\end{pmatrix}.$$
\end{ex}

Fix some maximal order $O_D\subset D$ such that $O_C = O_D\cap C$ is a maximal order in $C$ and consider the indicator function $f_D = 1_{O_D^\times}$. Lemma \ref{lem:order_uniqueness} implies that the orbital integrals $O(g, f_D)$ are independent of the choice of $O_D$. The purpose of the next definition is to provide a (conjectural) transfer $f'_D$ of $f_D$ in the sense of Definition \ref{def:matching}.

\begin{defn}\label{def:lattice_chains}
Let $λ = k/\ell$, with $(k,\ell) = 1$, be the Hasse invariant of $D$. Let $\mcL$ be the set of $O_K$-stable lattice chains in $F^{2n}$ that have the form
\begin{equation}\label{eq:lattice_chain}
Λ_\bullet = \big[Λ_0 \supset Λ_1 \supset \ldots \supset Λ_{\ell-1} \supset Λ_\ell = πΛ_0\big]
\end{equation}
and that furthermore satisfy the following property:
\begin{itemize}[wide, labelindent=0pt, labelwidth=!]
\item If $\ell$ is odd, then we demand that each quotient $Λ_i/Λ_{i+1}$ is a free $O_K/(π)$-module of rank $n/\ell$.
\item If $\ell$ is even, then we instead require $Λ_i/Λ_{i+1}$ is an $O_F/(π)$-vector space of dimension $2n/\ell$ and that the $O_K/(π)$-action on $Λ_i/Λ_{i+1}$ factors through the first projection $O_K = O_F\times O_F \to O_F$ if $i$ is even, and through the second projection if $i$ is odd.
\end{itemize}
\end{defn}

The stabilizer in $G'$ of a lattice chain $Λ_\bullet \in \mcL$ is by definition the subgroup
$$\mr{Stab}_{G'}(Λ_\bullet) = \{γ\in G' \mid γ Λ_i = Λ_i\text{ for all $i$}\}.$$
The group $H'$ acts transitively on $\mcL$ by translation, so these stabilizers form a single $H'$-conjugation orbit. However, because of the character $η(\bdot)|\cdot|^s$ in the definition of orbital integrals on $G'$, we have to be more specific about our desired test function than in the case of $G$.

\begin{defn}\label{def:test_function}
Pick any lattice chain $Λ^{\mr{std}}_\bullet\in \mcL$ such that $Λ^{\mr{std}}_0 = O_F^{2n}$ and define
\begin{equation}
f_D'^\circ = \mr{vol}(\mr{Stab}_{H'}(Λ^{\mr{std}}_\bullet))^{-2} 1_{\mr{Stab}_{G'}(Λ^{\mr{std}}_\bullet)}.
\end{equation}
Any two choices for $Λ^{\mr{std}}_\bullet$ differ by $\mr{Stab}(O_F^{2n}) \cap H' = GL_n(O_F)\times GL_n(O_F)$, so $f_D'^\circ$ is defined up to conjugation by $GL_n(O_F)\times GL_n(O_F)$. Also note that if $G = GL_{2n}(F)$, then $\ell = 1$ and the only possible standard chain is $O_F^{2n} \supset πO_F^{2n}$ --- we recover $f_D'^\circ = 1_{GL_{2n}(O_F)}$ as in Guo's case. Let $h_1\in H'$ be any element with $|h_1|^{-s} = q^{-2ns/\ell}$ and define $f'_D \in C^\infty_c(G')$ by
\begin{equation}
f_D'(γ) := \begin{cases}
f_D'^\circ & \text{if $\ell$ is odd}\\
f_D'^\circ(h_1^{-1} γ) & \text{if $\ell$ is even.}
\end{cases}
\end{equation}
By \eqref{eq:orb_int_trafo}, the orbital integrals of $f_D'^\circ$ and $f_D'$ are related by
\begin{equation}\label{eq:norm_orb_int}
\Orb(γ, f_D', s) = \Orb(γ, f_D'^\circ, s)\cdot \begin{cases}
1 & \text{if $\ell$ is odd}\\
q^{-2ns/\ell} & \text{if $\ell$ is even}
\end{cases}
\end{equation}
and, in particular, have the same central value. The advantage of the normalized function $f_D'$ is that its functional equation is completely symmetric, cf. Proposition \ref{prop:functional_equation} below. Examples for $Λ_\bullet^{\mr{std}}$ and $f_D'^\circ$ when $n = 2$ and $\ell \in \{2, 4\}$ can be found at the beginning of \S\ref{s:main_analytic}.
\end{defn}

\begin{conj}[Fundamental Lemma for CSAs]\label{conj:FL}
The function $f'_D$ is a transfer of $f_D$ in the sense of Definition \ref{def:matching}. That is, for regular semi-simple $γ\in G'_{\mr{rs}}$,
\begin{equation}\label{eq:conj_FL}
\Orb(γ, f'_D) = \begin{cases}
\Orb(g, f_D) & \text{if there exists a matching $g\in G_{\mr{rs}}$}\\
0 & \text{otherwise.}\end{cases}
\end{equation}
\end{conj}
Conjecture \ref{conj:FL} complements the Guo--Jacquet Fundamental Lemma which is formulated for the case $D = M_{2n}(F)$ and for the full Hecke algebra. We recall it here for comparison:
\begin{conj}[Guo--Jacquet Fundamental Lemma \protect{\cite[(1.12)]{Guo}}]\label{conj:FL_GJ}
Assume that $D = M_{2n}(F)$ and that the embedding $E\to D$ satisfies $O_E\subset M_{2n}(O_F)$. Then every $\GL_{2n}(O_F)$-biinvariant compactly supported function $f$ is a transfer of itself: For regular semi-simple $γ\in G'_{\mr{rs}}$,
\begin{equation}\label{eq:conj_FL_GJ}
\Orb(γ,f) = \begin{cases}
\Orb(g, f) & \text{if there exists a matching $g\in G_{\mr{rs}}$}\\
0 & \text{otherwise.}\end{cases}
\end{equation}
\end{conj}
Conjecture \ref{conj:FL} and Conjecture \ref{conj:FL_GJ} precisely overlap for the unit Hecke function $1_{GL_{2n}(O_F)}$. This is also the case that was proved by Guo, see \cite[(1.12)]{Guo}. We mention that Guo's formulation does not involve the transfer factor. Instead, he works with an orbit representative of the form
$$γ = \begin{pmatrix}
1 & w\\ 1 & 1
\end{pmatrix},$$
where each entry is an $(n\times n)$-matrix (see the line after \cite[(1.10)]{Guo}). Such a representative satisfies $\Omega(γ, 0) = 1$ which gives the link of his result with our formulation.

For Hasse invariant $λ = 1/2$, the CSA $D$ is isomorphic to $M_n(D_{1/2})$. In this case, Conjecture \ref{conj:FL} can be reduced to Guo's result and is hence known; we refer to \cite{HM}.

The following is our main result in this setting. Its proof will be given as Theorem \ref{thm:FL} below.
\begin{thm}\label{thm:FL_teaser}
Conjecture \ref{conj:FL} holds whenever $D$ is a division algebra of degree $4$.
\end{thm}
Note that the orbital integrals on the right hand side of \eqref{eq:conj_FL} have a particularly simple form if $D$ is a division algebra:
\begin{prop}\label{prop:orb_int_division}
Assume that $D$ is a division algebra; denote by $v_D:D^\times \to \mbZ$ its normalized valuation. Assume that $g\in G_\mr{rs}$ is regular semi-simple and denote by $f(L_g/F)$ the inertia degree of $L_g/F$. Then
\begin{equation}\label{eq:orb_int_division}
\Orb(g, f_D) = \begin{cases} f(L_g/F) & \text{if $v_D(g)\in 2\mbZ$}\\
0 & \text{otherwise.}
\end{cases}
\end{equation}
\end{prop}
\begin{proof}
In the given situation, $f_D$ is the indicator function of the units $O_D^\times$ of the unique maximal order in $D$. The centralizer $C = \mr{Cent}_D(E)$ is a CDA of degree $n$ over $E$, so $v_D(C^\times) = 2\mbZ$. It moreover holds that $v_D(D_-\backslash \{0\}) = 2\mbZ + 1$ and hence follows that $(O_C^\times\, g\,O_C^\times) \cap O_D^\times \neq \emptyset$ if and only if $v_D(g) \in 2\mbZ$. By the triangle inequality, this is equivalent to $v_D(g_-) > v_D(g_+)$ which is equivalent to $v_D(1 + z_g) = 0$. We obtain from Definition \ref{def:orb_int_geometric} that
$$\Orb(g, f_D) = \int_{L_g^\times \backslash C^\times} 1_{O_D^\times}(c(1+z_g)c^{-1})\ dc = \mr{vol}(L_g^\times \backslash C^\times) 1_{O_D}(z_g).$$
The Haar measures were defined such that $\mr{vol}(O_C^\times) = \mr{vol}(O_{L_g}^\times) = 1$ and, in particular, satisfy $\mr{vol}(L_g^\times \backslash C^\times) = f(L_g/F)$. This proves \eqref{eq:orb_int_division}.
\end{proof}

It is possible that the FL conjecture for $D$ a division algebra is related to Kottwitz's Euler--Poincaré functions \cite{Kottwitz}. Our proof of Theorem \ref{thm:FL_intro} is not along such lines, however, but rather a byproduct of our calculation of $\del O(γ, f'_D)$ when $n = 2$.

\begin{rmk}\label{rmk:defn_test_function}
The original motivation for our definition of $f'_D$ was the following. Let $\breve F$ be the completion of a maximal unramified field extension of $F$. Denote by $O_{\breve F}$ its ring of integers and by $\mbF$ its residue field. The scalar extension $\breve F\tensor_FD$ is isomorphic to $M_{2n}(\breve F)$ and under any such isomorphism, $\breve O_D = O_{\breve F}\tensor_{O_F}O_D$ gets identified with the stabilizer of a lattice chain
$$\breve{Λ}_\bullet = [\breve{Λ}_0 \supset \breve{Λ}_1 \supset \ldots \supset \breve{Λ}_{\ell-1}\supset \breve{Λ}_\ell = π\breve{Λ}_0]$$
such that $\dim_\mbF(\breve{Λ}_i /\breve{Λ}_{i+1}) = 2n/\ell$. The action of $R = O_{\breve F}\tensor_{O_F}O_E \subset \breve O_D$ on the quotients $\breve{Λ}_i/\breve{Λ}_{i+1}$ has the characteristics from Definition \ref{def:lattice_chains}: If $\ell$ is odd, then every quotient $\breve{Λ}_i/\breve{Λ}_{i+1}$ is free over $R/(π)$ of rank $n/\ell$. If $\ell$ is even, then the $R$-action on $\breve{Λ}_i/\breve{Λ}_{i+1}$ alternatingly factors through one of the two projections $R\to O_{\breve F}$.

A similar phenomenon occurs for the parahoric level fundamental lemma of Z. Zhang \cite[Theorem 4.1]{ZZ} in the Gan--Gross--Prasad setting  (also see \cite[Conjecture 10.3]{RSZ2} for an earlier formulation in a special case). The two group-theoretic data there that define the two test functions also have the property that they become isomorphic after scalar extension to $\breve F$.
\end{rmk}

\subsection{Functional equation for $\Orb(γ, f'_D, s)$}
\label{ss:functional_eqn}

Our aim in this section is to prove a functional equation for $\Orb(γ, f'_D, s)$. The motivation for this is twofold: First, it will imply the vanishing part of the fundamental lemma (Conjecture \ref{conj:FL}) in many cases. Second, it will imply that the derivatives that will occur in our AT conjecture are indeed the leading terms of the Taylor expansion of the orbital integral in question. We begin by defining the sign of the functional equation.

\begin{defn}\label{def:sign}
Let $λ \in (2n)^{-1}\mbZ/\mbZ$ be the Hasse invariant of $D$ and let $δ\in F[T]$ be regular semi-simple of degree $n$. Let $L_δ = F[z^2]/(δ(z^2))$ and $B_δ = (E\tensor_F L_δ)[z]$ be the universal algebras for $E$ and $δ$ from Definition \ref{def:universal_quaternion}. Write $L_δ = \prod_{i\in I} L_i$ for the decomposition of $L_δ$ into fields and let $B_δ = \prod_{i\in I} B_i$ be the corresponding decomposition of $B_δ$. Denoting by $β_i\in 2^{-1}\mbZ/\mbZ$ the Hasse invariant of $B_i/L_i$, we define
\begin{equation}
ε_D(δ) := nλ + \sum_{i\in I} β_i \in 2^{-1}\mbZ/\mbZ \iso \{\pm 1\}.
\end{equation}
We define $ε_D(γ) = ε_D(\Inv(γ))$ and $ε_D(g) = ε_D(\Inv(g))$ whenever $γ\in G'$ and $g\in G$ are regular semi-simple.
\end{defn}

\begin{lem}\label{lem:sign_alternative}
An equivalent description of $ε_D(δ)$ is given as follows. Let $δ_0\in F^\times$ be the constant coefficient of $δ$ and let $ε'_D = nλ \in \{\pm 1\}$. Then
\begin{equation}\label{eq:sign_comparison}
ε_D(δ) = η(δ_0)\cdot ε'_D.
\end{equation}
In particular, if $δ = \Inv(γ;T)$ for some $γ\in G'_{\mr{rs}}$, then $ε_D(γ) = η(\det_F(z_γ))ε'_D$.
\end{lem}
\begin{proof}
With notation as in Definition \ref{def:sign}, we need to see that $\sum_{i\in I} β_i = η(δ_0)$. Let $z_i$ denote the component of $z\in B_δ$ in the factor $B_i$. Then $(-1)^nδ_0 = \prod_{i\in I} N_{L_i/F}(z_i^2)$, so it suffices to show that $β_i = η(N_{L_i/F}(z_i^2))$. This follows directly from the compatibility of the local reciprocity map with the norm of field extensions, see for example \cite[\S2.4]{Serre_lcft}.

If $δ = \Inv(γ;T)$, then $δ_0$ is (by definition) a square root of $\det_F(z_γ^2)$ and hence the last formula holds.
\end{proof}

\begin{lem}\label{lem:sign_central_value}
Let $g\in G_{\mr{rs}}$ be a regular semi-simple element. Then $ε_D(g) = 1$.
\end{lem}
\begin{proof}
Put $δ = \Inv(g)$. By Corollary \ref{cor:universal_quaternion}, the existence of an element $g\in G_{\mr{rs}}$ of invariant $δ$ is equivalent to the identity $[B_δ] = [L_δ\tensor_F D]$ in the Brauer group of $L_δ$. Writing $L_δ = \prod_{i\in I} L_i$ as a product of fields as before and taking the sum of the Hasse invariants on both sides, we obtain
$$\sum_{i\in I} β_i = \sum_{i\in I} [L_i:F] λ = nλ.$$
\end{proof}
\begin{rmk}\label{rmk:sign_matching}
Table \ref{table:matching} illustrates that the converse to Lemma \ref{lem:sign_central_value} does not hold. Its rows 3 and 4 show cases where the sign $ε_D(δ)$ is positive for $D = M_4(F)$ or $M_2(D_{1/2})$, but where there is no $g\in G_{\mr{rs}}$ of invariant $δ$. Row 5 shows the existence of such cases when $D$ is a division algebra of degree $4$.
\end{rmk}
\begin{prop}\label{prop:functional_equation}
The orbital integrals of $f'_D$ satisfy the functional equation
\begin{equation}\label{eq:functional_eqn}
\Orb(γ, f'_D, -s) = ε_D(γ) \Orb(γ, f'_D, s).
\end{equation}
\end{prop}
The proof will be given at the end of this section. We first establish some auxiliary results that provide a combinatorial expression for $\Orb(γ, f'_D, s)$. Everything relies on the following simple observation: Assume that $h_1$, $h_2\in H'$ are two elements and that $Λ_{i,\bullet} = h_i Λ^{\mr{std}}_\bullet$ are the two corresponding lattice chains in $\mcL$. Then
\begin{equation}\label{eq:translation_stabilizer_to_lattice}
h_1^{-1}γh_2 \in \mr{Stab}_{G'}(Λ_\bullet^{\mr{std}})\quad \Longleftrightarrow\quad γΛ_{2,\bullet} = Λ_{1,\bullet}.
\end{equation}
\begin{defn}\label{def:lattice_chains_gamma}
Motivated by \eqref{eq:translation_stabilizer_to_lattice}, we make the following two definitions. First, we let
\begin{equation}
\mcL(γ) := \{Λ_\bullet \in \mcL \mid γΛ_\bullet \in \mcL\}.
\end{equation}
Second, for every lattice $Λ\subset F^{2n}$ such that both $Λ$ and $γΛ$ are $(O_F\times O_F)$-stable, we put
\begin{equation}\label{eq:transfer_lattice}
\Omega(γ, Λ, s) := \Omega(h_1^{-1}γh_2, s),
\end{equation}
where $h_1$, $h_2\in H'$ are chosen such that $Λ = h_2\cdot O_F^{2n}$ and and $γΛ = h_1\cdot O_F^{2n}$. (The lattice $O_F^{2n}$ comes up here because we have normalized the test function $f'_D$ by the requirement $Λ_0^{\mr{std}} = O_F^{2n}$.)
\end{defn}
Assume that $γ\in G'_{\mr{rs}}$ is regular semi-simple. The torus $L_γ^\times \subset G'$ acts on $\mcL(γ)$ by multiplication and we write
$$\mr{Stab}(Λ_\bullet) = \{x\in L_γ^\times \mid xΛ_i = Λ_i\text{ for all $i = 0,\ldots,\ell-1$}\}$$
for the stabilizer a lattice chain $Λ_\bullet \in \mcL(γ)$. Taking into account the volume factor in the definition of $f'_D$, see Definition \ref{def:test_function}, as well as the normalization in \eqref{eq:norm_orb_int}, we can then write the orbital integral of $f'_D$ as
\begin{equation}\label{eq:orb_int_combinatorial}
\Orb(γ, f'_D, s) = q^{-ms} \sum_{Λ_\bullet \in L_γ^\times \backslash \mcL(γ)} [O_{L_γ}^\times : \mr{Stab}(Λ_\bullet)]\ \Omega(γ, Λ_0, s)
\end{equation}
where $m = 0$ if $\ell$ is odd and $m = 2n/\ell$ if $\ell$ is even. The next few lemmas study this expression in more detail.
\begin{lem}\label{lem:lattice_combinatorics}
Let $γ\in G'_{\mr{rs}}$ be a regular semi-simple element and let $z = z_γ$.
\begin{enumerate}[wide, labelindent=0pt, labelwidth=!, label=(\arabic*)]
\item Let $Λ$ be an $O_K$-lattice. Then $γΛ$ is an $O_K$-lattice as well if and only if $z/(1+z)\cdot Λ \subseteq Λ$. If $z$ is topologically nilpotent, then this is furthermore equivalent to $zΛ\subset Λ$.
\item Assume that $\ell$ is odd and that $z$ is topologically nilpotent. Then
\begin{equation}\label{eq:lattice_chains_gamma_odd}
\mcL(γ) = \{Λ_\bullet \in \mcL \mid zΛ_i \subseteq Λ_i\text{ for all $i = 0,\ldots,\ell-1$}\}.
\end{equation}
\item Assume that $\ell$ is even. Then $\mcL(γ)\neq \emptyset$ only for $γ$ such that $z$ is topologically nilpotent. More precisely,
\begin{equation}\label{eq:lattice_chains_gamma_even}
\mcL(γ) = \{Λ_\bullet \in \mcL \mid z Λ_i \subseteq Λ_{i+1}\text{ for all $i = 0,\ldots,\ell-1$}\}.
\end{equation}
\end{enumerate}
\end{lem}
\begin{proof}
(1) Write $O_K = O_F[ζ]$ where $ζ$ satisfies $\bar{ζ} = 1 - ζ$. Then $γ^{-1}ζγ = ζ + z/(1+z)\cdot (1-2ζ)$ as in \eqref{eq:get_to_linear}. Hence, given an $O_K$-lattice $Λ$, the lattice $γΛ$ is $ζ$-stable if and only if
$$(ζ+z/(1+z)\cdot (1-2ζ))Λ\subseteq Λ.$$
It is checked directly that $1-2ζ\in O_K^\times$, so this inclusion holds if and only if $z/(1+z)\cdot Λ \subseteq Λ$. If $z$ is moreover topologically nilpotent, then $O_F[z] = O_F[z/(1+z)]$ and this condition becomes equivalent to $zΛ\subset Λ$.

(2) Assume that $\ell$ is odd. Any lattice chain $Λ_\bullet \in \mcL(γ)$ has the property that each $Λ_i$ is both $O_K$-stable and $γ^{-1}O_Kγ$-stable. By Part (1) and under our assumption that $z$ is topologically nilpotent, this is equivalent to $z\cdot Λ_i \subseteq Λ_i$ which proves the relation $\subseteq$ in \eqref{eq:lattice_chains_gamma_odd}. Assume conversely that $Λ_\bullet \in \mcL$ has the property that each $Λ_i$ is $z$-stable. We need to show and claim that each $(γΛ_i)/(γΛ_{i+1})$ is a free $O_K/(π)$-module. Since $γ_+$ is $O_K$-linear, this is equivalent to each quotient
$$(γ_+^{-1}γΛ_i)/(γ_+^{-1}γΛ_{i+1}) = ((1+z)Λ_i)/((1+z)Λ_{i+1})$$
being a free $O_K/(π)$-module. But it was assumed that $z$ is topologically nilpotent and that each $Λ_i$ is $z$-stable, so $(1+z)Λ_i = Λ_i$ for all $i$ and the claim follows because $Λ_\bullet \in \mcL$.

(3) Assume that $\ell$ is even and that $Λ_\bullet\in \mcL(γ)$ is any lattice chain. Then by definition of $\mcL(γ)$, the action of $O_K$ on $(γΛ_i)/(γΛ_{i+1})$ factors over the first factor of $O_K = O_F\times O_F$ if $i$ is even and over the second factor if $i$ is odd. Equivalently (apply the isomorphism $γ$), the $γ$-conjugated action of $O_K$ on $Λ_i/Λ_{i+1}$ factors over the first factor if $i$ is even and over the second factor if $i$ is odd. This is yet equivalent to $ζ$ and $γ^{-1}ζγ = ζ + z/(1+z)\cdot (1-2ζ)$ defining the same endomorphism of $Λ_i/Λ_{i+1}$. Since $1-2ζ\in O_K^\times$, this happens if and only if $z/(1+z)\cdot Λ_i\subseteq Λ_{i+1}$. Given that this holds for all $i$ and that $Λ_\ell = πΛ_0$, we deduce that $z/(1+z)$ is topologically nilpotent. Then $z$ is topologically nilpotent as well as claimed in the lemma. Identity \eqref{eq:lattice_chains_gamma_even} follows easily from the given arguments.
\end{proof}

\begin{lem}\label{lem:additional_symmetry}
(1) The following operator $Z_γ$, defined on lattice chains in $F^{2n}$, defines an automorphism of $\mcL(γ)$:
\begin{equation}\label{eq:def_Z_gamma}
Z_γ \cdot [Λ_0\supset Λ_1 \supset \ldots \supset Λ_\ell] := [z_γΛ_1 \supset z_γΛ_2 \supset \ldots \supset z_γΛ_\ell \supset πz_γΛ_1].
\end{equation}

\noindent (2) Moreover, $Z_γ$ commutes with the $L_γ^\times$-action on $\mcL(γ)$ and satisfies
\begin{equation}\label{eq:omega_Z_gamma}
\Omega(γ, (Z_γ Λ_\bullet)_0, s) = ε_D(γ) \Omega(γ, Λ_0, -s)\cdot \begin{cases}
1 & \text{if $\ell$ is odd}\\
q^{4ns/\ell} & \text{if $\ell$ is even.}
\end{cases}
\end{equation}
\end{lem}
\begin{proof}
(1) A direct computation shows that $γz_γγ^{-1} = γ_+z_γγ_+^{-1}$, so both elements $z_γ$ and $γz_γγ^{-1}$ are $K$-conjugate linear elements of $G'$. It follows that if a lattice $Λ$ has the property that both $Λ$ and $γΛ$ are $O_K$-stable, then also $z_γΛ$ and $γz_γΛ$ are $O_K$-stable. Thus, given any $Λ_\bullet\in \mcL(γ)$, the new chains $Z_γΛ_\bullet$ and $γZ_γΛ_\bullet$ are again chains of $O_K$-lattices. Taking into account the shift by one in \eqref{eq:def_Z_gamma}, both $Z_γΛ_\bullet$ and $γZ_γΛ_\bullet$ again satisfy the eigenvalue condition in the definition of $\mcL$ (see Definition \ref{def:lattice_chains}). Hence $Z_γΛ_\bullet \in \mcL(γ)$ as claimed.

(2) Proposition \ref{prop:quaternion_algebra} states that $L_γ = F[z_γ^2]$ which implies that multiplication by $z_γ$ and by elements from $L_γ^\times$ commute. It is left to prove Identity \eqref{eq:omega_Z_gamma}. It is easily checked that both sides of that identity are invariant under left-multiplication of $H'$ on $γ$. So we may assume that $γ = 1 + z$ with $z = z_γ$. This implies that $γ$ and $z$ commute which will simplify some expressions below. It furthermore allows for a more convenient description of $\Omega(γ, Λ, s)$. In its formulation, we write $Λ = Λ_+ \oplus Λ_-$ for the decomposition of an $O_K$-lattice $Λ$ into its $O_K$-eigenspaces.
\begin{lem}\label{lem:translation_omega}
Assume that $γ = 1 + z$ with $z = z_γ$. Assume that $Λ$ is an $O_K$-lattice such that also $γΛ$ is an $O_K$-lattice. Then
\begin{equation}\label{eq:translation_omega}
\Omega(γ, Λ, s) = (-1)^{[(γΛ)_- : zΛ_+] + [(γΛ)_-: Λ_-]} q^{([(γΛ)_+:zΛ_-] - [(γΛ)_- : zΛ_+])s}.
\end{equation}
\end{lem}
\begin{proof}
This follows directly from the definition of $\Omega(γ, Λ, s)$: Assume that $Λ = h_2O_F^{2n}$ and that $γΛ = h_1O_F^{2n}$. Let $\left(\begin{smallmatrix} a & b \\ c & d \end{smallmatrix}\right) = h_1^{-1}γh_2$. Recall that now by \eqref{eq:transfer_lattice} and by Definition \ref{def:transfer_factor},
\begin{equation}\label{eq:recap_omega}
\Omega(γ, Λ, s) = \Omega\left(\left(\begin{smallmatrix} a & b \\ c & d \end{smallmatrix}\right),s\right) = (-1)^{v(c) + v(d)}q^{(v(b)-v(c))\cdot s},
\end{equation}
where $v:F^\times\to \mbZ$ is the normalized valuation. Translating to $Λ$, we have
\begin{equation}\label{eq:as_easy_as_abcd}
\begin{array}{rlcrl}
v(a) \!\!\!\! & = [(γΛ)_+ : Λ_+] & \ \ & v(b) \!\!\!\! & = [(γΛ)_+ : zΛ_-]\\
v(c) \!\!\!\! & = [(γΛ)_- : zΛ_+] & \ \ & v(d) \!\!\!\! & = [(γΛ)_- : Λ_-].
\end{array}
\end{equation}
Substituting \eqref{eq:as_easy_as_abcd} in \eqref{eq:recap_omega} proves the lemma.
\end{proof}
Let $Λ$ be an $O_K$-lattice such that also $γΛ$ is an $O_K$-lattice. Since $z$ is $K$-conjugate linear and furthermore commutes with $γ$, it holds that
$$(zΛ)_\pm = zΛ_\mp,\quad (γzΛ)_{\pm} = z(γΛ)_{\mp}.$$
We obtain from \eqref{eq:translation_omega} that
\begin{equation}\label{eq:translation_omega_II}
\Omega(γ, zΛ, s) = (-1)^{[z(γΛ)_+ : z^2 Λ_-] + [z(γΛ)_+ : z Λ_+]} q^{([z(γΛ)_- : z^2 Λ_+] - [z(γΛ)_+ : z^2 Λ_-])s}.
\end{equation}
The exponents of the signs of \eqref{eq:translation_omega} and \eqref{eq:translation_omega_II} are related by
\begin{equation}\label{eq:expo_rel_I}
[z(γΛ)_+ : z^2 Λ_-] + [z(γΛ)_+:zΛ_+] = [γΛ : zΛ] - [(γΛ)_- : zΛ_+] + [γΛ : Λ] - [γΛ_-: Λ_-],
\end{equation}
those for the $q$-powers by
\begin{equation}\label{eq:expo_rel_II}
[z(γΛ)_- : z^2 Λ_+] - [z(γΛ)_+ : z^2 Λ_-] = - [(γΛ)_+ : zΛ_-] + [(γΛ)_- : zΛ_+].
\end{equation}
Note that $[γΛ:zΛ] + [γΛ:Λ] \equiv [Λ:zΛ]$ mod $2$ in \eqref{eq:expo_rel_I}, so we obtain
\begin{equation}\label{eq:omega_Z_gamma_pre_I}
\Omega(γ, zΛ, s) = η(\det(z)) \Omega(γ, Λ, - s).
\end{equation}
It is left to take care of the shift in \eqref{eq:def_Z_gamma}. Assume that $Λ'\subseteq Λ$ is a sublattice that also has the property that both $Λ'$ and $γΛ'$ are $O_K$-stable. Define integers $a_{\pm}$ and $b_{\pm}$ by the identities
$$a_\pm = [Λ_\pm : Λ'_\pm] \quand b_\pm = [(γΛ)_\pm : (γΛ')_\pm].$$
Then we obtain from \eqref{eq:translation_omega} that
\begin{equation}\label{eq:omega_Z_gamma_pre_II}
\Omega(γ, Λ', s) = (-1)^{a_+ + a_-}q^{(a_- - b_+ - a_+ + b_-)s}\Omega(γ, Λ, s).
\end{equation}
Apply this to the two lattices $zΛ_1 \subset zΛ_0$ that arise from $zΛ_\bullet$ with $Λ_\bullet\in \mcL(γ)$. Depending on the parity of $\ell$, the following two cases occur. If $\ell$ is odd, then $zΛ_0/zΛ_1$ is free over $O_K/(π)$ so $a_+ = a_-$ and $b_+ = b_-$. We obtain that
\begin{equation}\label{eq:omega_Z_gamma_pre_III}
\Omega(γ, zΛ_1, s) = \Omega(γ, zΛ_0, s).
\end{equation}
If $\ell$ is even, then $zΛ_0/zΛ_1$ and $γzΛ_0/γzΛ_1$ are both free over $O_F/(π)$ of rank $2n/\ell$ with $O_K$ acting via the second projection. (Indeed, $O_K$ acts via the first projection on $Λ_0/Λ_1$ and $(γΛ_0)/(γΛ_1)$ but $z$ is $O_K$-conjugate linear.) In particular, $a_+ = b_+ = 0$ and $a_- = b_- = 2n/\ell$. Identity \eqref{eq:omega_Z_gamma_pre_II} then specializes to
\begin{equation}\label{eq:omega_Z_gamma_pre_IV}
\Omega(γ, zΛ_1, s) = (-1)^{2n/\ell} q^{4ns/\ell}\Omega(γ, zΛ_0, s).
\end{equation}
Combining \eqref{eq:omega_Z_gamma_pre_I}, \eqref{eq:omega_Z_gamma_pre_III} and \eqref{eq:omega_Z_gamma_pre_IV}, it follows that
\begin{equation}\label{eq:omega_Z_gamma_pre_V}
\Omega(γ, (Z_γ Λ_\bullet)_0, s) = (-1)^{2n/\ell}η(\det(z)) \Omega(γ, Λ_0, -s)\cdot \begin{cases}
1 & \text{if $\ell$ is odd}\\
q^{4ns/\ell} & \text{if $\ell$ is even.}
\end{cases}
\end{equation}
Recall that the constant coefficient of $\Inv(γ)$ is a square root of $\det(z^2)$. The sign $(-1)^{2n/\ell}η(\det(z))$ from \eqref{eq:omega_Z_gamma_pre_V} hence equals $ε_D(γ)$ by Lemma \ref{lem:sign_alternative}, and the proof of \eqref{eq:omega_Z_gamma} is complete.
\end{proof}

\begin{proof}[Proof of the functional equation (Proposition \ref{prop:functional_equation}).] Let $m = 0$ if $\ell$ is odd and $m = 2n/\ell$ if $\ell$ is even. Using the combinatorial description \eqref{eq:orb_int_combinatorial} together with Lemma \ref{lem:additional_symmetry}, we have
$$\begin{aligned}
\Orb(γ, f'_D, -s) &\ = \ q^{ms} \sum_{Λ_\bullet \in L_γ^\times\backslash \mcL(γ)} [O_{L_γ}^\times : \mr{Stab}(Λ_\bullet)]\ \Omega(γ, Λ_0, -s)\\
&\ =\ ε_D(γ) q^{ms} q^{-2ms}\sum_{Λ_\bullet \in L_γ^\times\backslash \mcL(γ)} [O_{L_γ}^\times : \mr{Stab}(Λ_\bullet)]\ \Omega(γ, (Z_γΛ_\bullet)_0, s)\\
&\ =\ ε_D(γ) \Orb(γ, f'_D, s)
\end{aligned}$$
as was to be shown.
\end{proof}

\section{Arithmetic Transfer}
\label{s:ATC}

The setting is the same as in \S\ref{ss:setting} except that we from now on take
\begin{equation}\label{eq:opposed}
G = D^{\mr{op}, \times},\quad H = C^{\mr{op}, \times}.
\end{equation}
Note that $G$ and $G^{\mr{op}}$ have the same underlying topological space which implies that $C^\infty_c(G) = C^\infty_c(G^{\mr{op}})$. Moreover, $H$ and $H^{\mr{op}}$ have the same underlying topological space as well and the definitions of $g\in G$ being regular semi-simple, of the invariant $\mr{Inv}(g;T)$, and of the orbital integral $\Orb(f, g)$ are all unchanged when taking them for opposed CSAs.

\subsection{Local Shimura Data}

Let $\breve F$ be the completion of a maximal unramified extension of $F$. Let $O_{\breve F}$ denote its ring of integers and let $\mbF$ be its residue field. The Frobenius automorphism of $\breve F$ is the unique $F$-automorphism inducing $q$-Frobenius $x\mapsto x^q$ on $\mbF$; we denote it by $σ:\breve F\to \breve F$. Let $v_{\breve F}:\breve F^\times\to \mbZ$ be the normalized valuation.

By $F$-isocrystal, or simply isocrystal, we mean a pair $\bN = (N, \bF)$ that consists of a finite-dimensional $\breve F$-vector space $N$ and a $σ$-linear automorphism $\bF$. The Verschiebung of $\bN$ is defined as $\bV = π\bF^{-1}$. Height, dimension and slope of $\bN$ are all meant in the relative sense with respect to $F$: The height $\mr{ht}(\bN)$ is the $\breve F$-dimension of $N$, the dimension $\dim(\bN)$ is the integer $v_{\breve F}(\det \bV)$, and the slope is their ratio $\dim(\bN)/\mr{ht}(\bN)$. Note that $\dim(\bN)$ might be negative.

The Dieudonné--Manin classification \cite{DM} states that the ($F$-linear) category of isocrystals is semi-simple and that the isomorphism classes of its simple objects are in bijection with $\mbQ$: For every $μ = r/s$, where $(r,s) = 1$, there is a unique (up to isomorphism) simple isocrystal $\bN_μ$ of height $s$ and dimension $r$. The endomorphism ring $\End(\bN_μ)$ is a CDA over $F$ of Hasse invariant $μ$.

\begin{defn}\label{def:isocrystals}
(1) By $C$-isocrystal, we mean a pair $(\bN_+, ι)$ that consists of an isocrystal $\bN_+$ and an $F$-linear $C$-action $ι:C\to \End(\bN_+)$ with the following numerical conditions: The height of $\bN_+$ is $2n^2 = \dim_F(C)$, the dimension of $\bN_+$ is $n$, and the slopes of all subisocrystals of $\bN_+$ lie in the interval $[0,1]$.

\noindent (2) By $D$-isocrystal, we mean a pair $(\bN, κ)$ that consists of an isocrystal $\bN$ and an $F$-linear $D$-action $κ:D\to \End(\bN)$ with the following numerical conditions: The height of $\bN$ is $4n^2 = \dim_F(D)$, the dimension of $\bN$ is $2n$, and the slopes of all subisocrystals of $\bN$ lie in the interval $[0,1]$.
\end{defn}

\begin{rmk}\label{rmk:isocrystal_numerical}
Recall that by covariant Dieudonné theory $p$-divisible groups over $\mbF$ together with quasi-homomorphisms are equivalent to $\mbQ_p$-isocrystals that have the slopes of all subisocrystals in the interval $[0,1]$. Under this equivalence, height and dimension of the $p$-divisible equal height and dimension of the corresponding $\mbQ_p$-isocrystal. The analogous statement holds for strict $O_F$-modules over $\mbF$ (see Definition \ref{def:pi_div_group}) and $F$-isocrystals. This motivates the slope condition in Definition \ref{def:isocrystals}.
\end{rmk}

The Serre tensor construction defines a functor
\begin{equation}\label{eq:Serre_tensor}
\begin{aligned}
\{\text{$C$-isocrystals}\} &\ \lr \{\text{$D$-isocrystals}\}\\
(\bN_+, ι) &\ \longmapsto (\bN = D\tensor_C \bN_+,\ κ(x) = x\tensor \mr{id}_{\bN_+}).
\end{aligned}
\end{equation}

\begin{lem}\label{lem:isocrystals}
(1) Two $C$-isocrystals (resp. two $D$-isocrystals) are isomorphic if and only if the underlying isocrystals are isomorpic. In particular, the functor \eqref{eq:Serre_tensor} defines an injective map on isomorphism classes.

\noindent (2) A $D$-isocrystal $(\bN, κ)$ lies in the essential image of \eqref{eq:Serre_tensor} if and only if there exists an $F$-algebra map $E\to \End_D(\bN, κ)$.
\end{lem}
\begin{proof}
(1) Let $\bN$ be any isocrystal. By the Dieudonné--Manin classification, $\End(\bN)$ is a product of CSAs over $F$. It then follows from the Skolem--Noether Theorem applied factor by factor that any two $F$-algebra homomorphisms $C\to \End(\bN)$ are conjugate. In other words, there is at most one way (up to $C$-linear isomorphism) to define a $C$-action on $\bN$. The same argument applies to $D$-isocrystals. The injectivity of \eqref{eq:Serre_tensor} on isomorphism classes follows directly because $D\tensor_C \bN_+ \iso \bN_+^{\oplus 2}$ as isocrystal.

(2) The category of $C$-isocrystals is $E$-linear because the center of $C$ is $E$. The Serre tensor construction is functorial, so every object in its image has a $D$-linear $E$-action. Explicitly, $E$ acts on $(\bN, κ) = D\tensor_C (\bN_+, ι)$ by
$$ι\vert_E:E\lr \End_D(\bN, κ),\quad a \longmapsto 1\tensor ι(a).$$
Assume conversely that there exists an embedding $ι:E\to \End_D(\bN,κ)$. Then $\bN$ has an action by $κ(E)\tensor_F ι(E)$ and thus decomposes into $C$-stable eigenspaces $\bN = \bN_+ \oplus \bN_-$. Let us write $ι$ for the resulting $C$-action $C \to \End(\bN_+)$ on the first factor. The natural $D$-linear map $D\tensor_C (\bN_+, ι) \to \bN$ is the desired isomorphism.
\end{proof}

We next place the above definitions into a group-theoretic context, following the EL formalism in \cite[Definition 3.18]{RZ1}. For this we consider $H$ and $G$ as algebraic groups over $F$. Our convention is that $\End_D(D)$ acts on the left of $D$ and is hence isomorphic to $D^{\mr{op}}$. In light of \eqref{eq:opposed}, we have an isomorphism
$$G \overset{\iso}{\lr} \End_D(D)^\times,\quad g \longmapsto [x \longmapsto xg].$$
Recall that the Kottwitz set of $G$ is defined as the set of $σ$-conjugacy classes in $G(\breve F)$:
$$B(G) = G(\breve F)/\{b \sim gbσ(g)^{-1}\}.$$
It is a standard fact (see \cite[\S1.7]{RZ1}) that this set is in bijection with isomorphism classes of isocrystals of height $\dim_F(D)$ with $D$-action:
\begin{equation}\label{eq:Kottwitz_bijection}
\begin{aligned}
B(G) &\ \overset{\iso}{\lr} \left\{(\bN, κ) \left\vert \text{\begin{varwidth}{\textwidth}\centering $\bN$ an $F$-isocrystal of height $\dim_F(D)$\\
$κ:D\to \End(\bN)$ a $D$-action\end{varwidth}}\right\}/\iso\right.\\
[b] &\ \longmapsto (\bN_b, κ) := (\breve F\tensor_F D, σ\tensor b).
\end{aligned}
\end{equation}
After scalar extension to $\ob{F}$, there is an isomorphism $α:\ob{F}\tensor_F D^{\mr{op}} \iso M_{2n}(\ob{F})$ of $\ob{F}$-algebras from which one obtains an isomorphism $G_{\ob{F}} \iso GL_{2n,\ob{F}}$. Consider the $\ob{F}$-conjugacy class of the cocharacter
\begin{equation}\label{eq:minuscule_cochar_G}
μ_G:\mbG_m\lr G_{\ob{F}} \overset{\iso}{\underset{α}{\lr}} GL_{2n,\ob{F}},\quad t\longmapsto \mr{diag}(t, \ldots, t, 1).
\end{equation}
Via \ref{def:isocrystals}, the subset $B(G, μ_G) \subset B(G)$ of $μ$-admissible elements is in bijection with the isomorphism classes of $D$-isocrystals from Definition \ref{def:isocrystals}. (For the purposes of our article, the reader may take that as the definition of $B(G, μ_G)$.) Namely, by the Dieudonné classification from Remark \ref{rmk:isocrystal_numerical}, every $D$-isocrystal $(\bN,κ)$ is the isocrystal of a strict $O_F$-module $X$ over $\mbF$ together with a rational action $κ:D\to F\tensor_{O_F} \End(X)$. By \cite[\S3.19]{RZ1} which also holds for $F$-isocrystals, this implies that $(\bN, κ)$ lies in $B(G, μ_G)$. Conversely, an $F$-isocrystal with $D$-action $(\bN, κ) \in B(G, μ_G)$ is necessarily $μ$-weakly admissible. By definition, see \cite[Definition 1.18 and \S1.3]{RZ1}, the latter is equivalent to $\bN$ being of dimension $2n$ and with the slopes of all subisocrystals in $[0,1]$.

To make the analogous definitions for $H$, we need to fix an embedding $E \subset \ob{F}$. Also choose an isomorphism $β:\ob{F}\tensor_F C^{\mr{op}} \iso M_n(\ob{F})\times M_n(\ob{F})$ such that $β(a) = (a, \bar a)$ for all $a\in E$. Consider the $H(\ob{F})$-conjugacy class of
\begin{equation}\label{eq:minuscule_cochar_H}
μ_H:\mbG_m\lr H_{\ob{F}} \overset{\iso}{\underset{β}{\lr}} GL_{n,\ob{F}} \times GL_{n,\ob{F}},\quad t\longmapsto \mr{diag}((t, \ldots, t, 1),\ (t,\ldots,t)).
\end{equation}
Then $B(H, μ_H)$ is in bijection with the isomorphism classes of $C$-isocrystals in the sense of Definition \ref{def:isocrystals}. Moreover, the natural map $B(H, μ_H) \to B(G, μ_G)$ is given by \eqref{eq:Serre_tensor}.

\begin{defn}\label{def:b}
For $b\in H(\breve F)$, we denote by $(\bN_{b, +}, ι)$ the $C$-isocrystal given by $(\breve F\tensor_F C, σ\tensor b)$ with its natural $C$-action. We write $C_b = \End_C(\bN_{b, +}, ι)$ and $H_b = C_b^\times$. We further define
$$(\bN_b, κ) := D\tensor_C (\bN_{+, b}, ι)$$
as well as $D_b = \End_D(\bN_b, κ)$ and $G_b = D_b^\times$. Note that there is an inclusion $H_b\to G_b$, $g\mapsto \mr{id}_D\tensor g$ by functoriality of the Serre tensor construction.
\end{defn}

It follows from the Dieudonné--Manin classification that $D_b$ is a semi-simple $F$-algebra of total degree $2n$ in the sense of \eqref{eq:degree_ss}. By construction, there is an embedding $E\to D_b$ and $C_b = \mr{Cent}_{D_b}(E)$. The more precise description of $D_b$ is as follows: Let $\bN_b \iso \bigoplus_{μ\in [0, 1]} \bN^{n_μ}_μ$ be the slope decomposition, where $\bN_μ$ denotes a simple isocrystal of slope $μ$. Then
$$D_b \iso \prod_{μ\in [0,1]} M_{m_μ}(D_{μ - λ})$$
where $λ$ is the Hasse invariant of $D$, where $D_{μ - λ}$ denotes a CDA over $F$ of Hasse invariant $μ - λ$, and where $m_μ$ is characterized by
$$M_{m_μ}(D_{μ-λ}) \tensor_F D_λ \iso M_{n_μ}(D_μ).$$
It is a well-known and curious phenomenon that the number of elements of $B(H, μ_H)$ and $B(G, μ_G)$ strongly depends on $λ$. We give some examples:

\begin{ex}\label{ex:isogeny_classes}
(1) Assume that $D = M_{m}(D_0)$ where $D_0$ is a CDA over $F$. Then, by Morita equivalence, the category of isocrystals with $D$-action is equivalent to that of isocrystals with $D_0$-action: To a pair $(\bN_0,\ κ_0:D_0\to \End(\bN_0))$, one associates the $m$-th power $\bN_0^m$ with its natural extension of $κ_0$ to $M_m(D_0)$. Under this equivalence, $D$-isocrystals in the sense of Definition \ref{def:isocrystals} correspond to isocrystals with $D_0$-action $(\bN_0, κ_0)$ such that $\bN_0$ is of height $\dim_F(D)/m = m\cdot \dim_F(D_0)$, of dimension $[D:F]/m = [D_0:F]$, and has all its slopes within $[0, 1]$. (Note that $[D:F] = m\cdot [D:F_0]$ explaining why $\bN_0$ is required to have dimension $[D_0:F]$.)

\noindent (2) Consider the special case $D = M_{2n}(F)$. By (1), the Kottwitz set $B(G, μ_G)$ is in bijection with isomorphism classes of isocrystals of height $2n$, dimension $1$ and with all slopes within $[0,1]$. The slope vector of such an isocrystal is of the form $(0^{(2n-n_0)}, 1/n_0)$ for a unique integer $1\leq n_0 \leq 2n$, and every such $n_0$ can occur. The endomorphism ring in this case is isomorphic to $M_{2n-n_0}(F)\times D_{1/n_0}$ which admits an embedding of $E$ if and only if $n_0$ is even. This characterizes the image of the map $B(H, μ_H) \to B(G, μ_G)$ by Lemma \ref{lem:isocrystals}.

\noindent (3) Assume that $λ \in \{1/2n, (n+1)/2n\}$. Then the Hasse invariant (over $E$) of $C$ is $2λ = 1/n$. In this case, $B(H, μ_H)$ consists of a single element $[b]$, cf. \cite[Lemma 3.60]{RZ1}, which is known as the Drinfeld case. The corresponding isocrystals $\bN_{+,b}$ and $\bN_b$ are isoclinic of slope $1/2n$. This applies in particular when $n = 2$ and $λ \in \{1/4, 3/4\}$ which is the main case of interest of the paper.

\noindent (4) Assume that $n = 3$ and $λ \in \{1/3, 5/6\}$. Then the Hasse invariant of $C$ is $2/3$ and $B(H, μ_H)$ consists of two elements. By Lemma \ref{lem:isocrystals}, they may be characterized uniquely by the slope vector of the underlying isocrystal. One possibility is $(1/6, 1/6, 1/6)$, which is the basic case, the other is $(1/12, 1/3, 1/3)$.
\end{ex}

Recall that we have given a definition of invariant for double cosets $H_b\backslash G_b /H_b$, see \eqref{eq:invariant_ss}.

\begin{prop}\label{prop:char_isogeny_class}
Let $δ\in F[T]$ be a regular semi-simple invariant of degree $n$. Then there is at most one $[b]\in B(H, μ_H)$ such that there exists an element $g\in G_b$ of invariant $δ$. In case of existence, all such elements $g$ form a single $H_b \times H_b$-orbit. Furthermore, in this case $ε_D(δ) = -1$.
\end{prop}

Note that the statement about the set of such $g$ forming a single orbit is non-trivial because $D_b$ may not be simple, so the Skolem--Noether Theorem does not immediately apply.

\begin{cor}\label{cor:vanishing_geometric_matching}
Let $γ\in G'_{\mr{rs}}$ be a regular semi-simple element such that there exists an isogeny class $[b]\in B(H, μ_H)$ and an element $g\in G_b$ that matches $γ$. Then $\Orb(γ, f'_D) = 0$.
\end{cor}

\begin{proof}[Proof of the corollary.] Proposition \ref{prop:char_isogeny_class} states that the sign in the functional equation of $\Orb(γ, f'_D, s)$ is negative. (See Proposition \ref{prop:functional_equation} for that functional equation.)
\end{proof}

\begin{proof}[Proof of Proposition \ref{prop:char_isogeny_class}.]
We write $B = B_δ$ and $L = L_δ$ in the following. Let $L = \prod_{i\in I}L_i$ denote the decomposition of $L$ into fields and let $B = \prod_{i\in I}B_i$ be the corresponding decomposition of $B$. The tensor product $P = D\tensor_F B$ has center $L$ and a similar decomposition $P = \prod_{i\in I} P_i$. The $i$-th factor $P_i$ is a CSA of degree $4n$ over $L_i$. Let $ρ_i \in (4n)^{-1}\mbZ/\mbZ$ be its Hasse invariant, and write $n_i = [L_i:F]$. Let $β_i$ be the Hasse invariant of $B_i/L_i$. These invariants are related by $ρ_i = n_i λ + β_i$ because
\begin{equation}\label{eq:Hasse_relations}
\begin{aligned}
\mr{inv}_{L_i}((D\tensor_FL_i)\tensor_{L_i} B_i) & = \mr{inv}_{L_i}(D\tensor_FL_i) + \mr{inv}_{L_i}(B_i)\\
& = [L_i:F]\cdot λ + β_i.
\end{aligned}
\end{equation}

Let $b\in B(G, μ_G)$ be any isomorphism class; denote by $(\bN_b, κ)$ the corresponding $D$-isocrystal. Giving an embedding $E\to D_b$ and an orbit $H_bgH_b \subset G_b$ of invariant $δ$ is the same as lifting $κ$ to a faithful action $\wt{κ}:P\to \End(\bN)$ up to $G_b$-conjugacy by \eqref{eq:conjugacy_orbit} and because necessarily $B_g = B_δ$ (via $z_g \mapsto z$) if such an orbit exists (Proposition \ref{prop:quaternion_algebra}). The uniqueness of the pair $(b,\ H_b g H_b)$ is thus equivalent to the uniqueness (up to $P$-linear isomorphism) of an isocrystal $\bN$ of height $4n^2$, of dimension $2n$, with all slopes within $[0,1]$ and with a faithful $P$-action $\wt{κ}$.

Assume $(\bN, \wt{κ})$ is such a pair. Then $\bN$ decomposes, $\bN = \prod_{i\in I}\bN_i$, into a product of isocrystals with faithful $P_i$-action. The $i$-th factor $P_i$ is a CSA over $L_i$ of degree $4n$ so the height of $\bN_i$ has to be an integer multiple of $4nn_i$. Since $\sum_{i\in I} 4nn_i = 4n^2 = \mr{ht}(\bN)$, the height of $\bN_i$ has to be exactly $4nn_i$. Furthermore, $\bN_i$ is necessarily isoclinic, say of slope $μ_i = d_i/4nn_i$. The uniqueness of $(\bN, \wt{κ})$ up to $P$-linear isomorphism is then equivalent to the vector $(d_i)_{i\in I}$ being uniquely determined by $P$.

The $i$-th endomorphism ring $\End(\bN_i)$ is a CSA over $F$ of Hasse invariant $μ_i$ and degree $4nn_i$. Since $[P_i:F][L_i:F] = 4nn_i$, it follows from the centralizer theorem that $\wt{κ}(P_i)$ equals the centralizer of $\wt{κ}(L_i)$ in $\End(\bN_i)$. This implies that
\begin{equation}\label{eq:key_rho}
ρ_i = n_iμ_i = d_i/4n.
\end{equation}
Here, the integer $d_i$ is the dimension of $\bN_i$ and we know that $\sum_{i\in I} d_i = 2n$. Because all slopes are assumed to lie within the interval $[0,1]$, we in particular obtain that $0\leq d_i \leq 2n$ and hence see that $(d_i)_{i\in I}$ is uniquely determined by $P$. This shows the uniqueness of $b$ and the orbit $H_bgH_b$.

It is still left to prove that $ε_D(δ) = -1$ if $b$ and such an orbit exist. To this end, we take up the identity $ρ_i = β_i + n_i λ$ from the beginning of the proof. Combining with \eqref{eq:key_rho}, we see that $β_i + n_iλ = d_i/4n$. Taking the sum over all $i\in I$, it follows that
\begin{equation}\label{eq:sign_matching_RZ}
ε_D(δ) = \sum_{i\in I} \frac{d_i}{4n} = \frac{1}{2} \in 2^{-1}\mbZ/\mbZ \iso \{\pm 1\}
\end{equation}
as claimed.
\end{proof}

\subsection{Moduli Spaces}

Let $[b]\in B(H, μ_H)$ be an isogeny class. Then $(H, b, μ_H)$ and $(G, b, μ_G)$ are local Shimura data triples. Our aim in this section is to define integral models of the corresponding local Shimura varieties (over $\breve F$) at maximal level.

Fix an embedding $E\to \breve F$ as well as maximal orders $O_C\subset C$ and $O_D\subset D$. By definition, a $\Spf O_{\breve F}$-scheme is an $O_{\breve F}$-scheme $S$ such that $π\in \mcO_S$ is locally nilpotent.

\begin{defn}\label{def:pi_div_group}
Let $S$ be a $\Spf O_{\breve F}$-scheme.
\begin{enumerate}[wide, labelindent=0pt, labelwidth=!, label=(\arabic*)]
\item Assume that $F$ is of characteristic $0$. A strict $O_F$-module over $S$ is a pair $(X, α)$ that consists of a $p$-divisible group $X$ over $S$ and an action $α:O_F\to\End(X)$ that is strict in the sense that $\Lie(α(a)) = a$ for all $a\in O_F$. Height and slope of a strict $O_F$-module are meant in the relative sense, meaning $[F:\mbQ_p]\cdot \mr{ht}(X)$ is the height of $X$ as $p$-divisible group.
\item Assume that $F \iso \mbF_q(\!(π)\!)$ is of characteristic $p$. A strict $O_F$-module over $S$ is a $π$-divisible group $(X, α)$ over $S$ in the sense of \cite[Definition 7.1]{HS} such that $\Lie(α(a)) = a$ for all $a\in O_F$. (In other words, we demand $d = 1$ in part (iv) of \cite[Definition 7.1]{HS}.) Height, dimension and slope are defined as in \cite[\S7]{HS}.
\end{enumerate}
\end{defn}

\begin{defn}\label{def:special}
Let $S$ be a $\Spf O_{\breve F}$-scheme.
\begin{enumerate}[wide, labelindent=0pt, labelwidth=!, label=(\arabic*)]
\item A special $O_C$-module over $S$ is a pair $(Y, ι)$ that consists of a strict $O_F$-module $Y$ and an $O_C$-action $ι:O_C\to \End(Y)$ such that the following conditions are satisfied. The height of $Y$ is $2n^2$, its dimension is $n$, and the $O_C$-action is special in the sense that for all $x\in O_C$,
\begin{equation}\label{eq:special_C}
\mr{char}(ι(x)\mid \Lie(Y); T) = \mr{charred}_{C/E}(x; T).
\end{equation}
Here, the right hand side is considered as an element of $\mcO_S[T]$ via the fixed embedding $E\subset \breve F$ and the structure map $O_{\breve F}\to \mcO_S$.

\item A special $O_D$-module over $S$ is a pair $(X, κ)$ that consists of a strict $O_F$-module $X$ and an $O_D$-action $κ:O_D\to \End(X)$ such that the following conditions are satisfied. The height of $X$ is $4n^2$, its dimension is $2n$, and the $O_D$-action is special in the sense that for all $x\in O_D$,
\begin{equation}\label{eq:special_D}
\mr{char}(κ(x)\mid \Lie(X); T) = \mr{charred}_{D/F}(x; T).
\end{equation}
Here, the right hand side is considered as an element of $\mcO_S[T]$ via the structure map $O_{\breve F}\to \mcO_S$.
\end{enumerate}
\end{defn}

\begin{rmk}\label{rmk:special_condition}
By \cite[3.58]{RZ1}, an equivalent way to formulate \eqref{eq:special_C} and \eqref{eq:special_D} is as follows. Let $L/E$ be an unramified field extension of degree $n$ and fix an embedding $O_L\to O_C$. Then, given an action $ι:O_C\to \End(Y)$ or $κ:O_D\to \End(X)$, the Lie algebra $\Lie(Y)$ resp. $\Lie(X)$ becomes an $O_L\tensor_{O_F}\mcO_S$-module. Since $S$ was assumed to be an $O_{\breve F}$-scheme, there is an eigenspace decomposition
$$\Lie(Y) = \bigoplus_{φ\in \Hom_E(L, \breve F)} \Lie(Y)_φ\quad \text{resp.}\quad \Lie(X) = \bigoplus_{φ\in \Hom_F(L, \breve F)} \Lie(X)_φ.$$
Then \eqref{eq:special_C} resp. \eqref{eq:special_D} holds for all $x\in O_C$ (resp. all $x\in O_D$) if and only if each summand $\Lie(Y)_φ$ (resp. each summand $\Lie(X)_φ$) is locally free of rank $1$ as $O_S$-module.
\end{rmk}

\begin{rmk}[Morita Equivalence]\label{rmk:reformulation}
It is possible to reformulate Definition \ref{def:special} in terms of division algebras only. For brevity, we only consider the case of $D$: Assume that $O_D = M_m(O_{D_0})$ where $O_{D_0}$ denotes the maximal order in a CDA $D_0$. Then special $O_D$-modules over $S$ are equivalent to pairs $(X_0, κ_0)$ where $X_0$ is a strict $O_F$-module over $S$ of height $m\cdot\dim_F(D_0)$, dimension $[D_0:F]$, and where $κ_0:O_{D_0}\to \End(X_0)$ is special in the sense that for all $x\in O_{D_0}$,
$$\mr{char}(κ_0(x)\vert \Lie(X_0); T) = \mr{charred}_{D_0/F}(x; T).$$
The equivalence is given by $(X_0, κ_0) \mapsto (X_0^m, M_m(κ_0))$.
\end{rmk}

We fix a special $O_C$-module $(\mbY, ι)$ and a special $O_D$-module $(\mbX, κ)$ over $\mbF$ for the next definition. These are the so-called framing objects.

\begin{defn}\label{def:RZ}
The RZ moduli space $\mcM_C$ is defined as the following functor on the category of schemes over $\Spf O_{\breve F}$,
$$\mcM_C(S) = \left\{(Y, ι, ρ) \left\vert \text{\begin{varwidth}{\textwidth}\centering $(Y,ι)$ a special $O_C$-module over $S$\\
$ρ:\ob{S}\times_{\Spec \mbF} \mbY \lr \ob{S}\times_S Y$ an $O_C$-linear quasi-isogeny\end{varwidth}}\right.\right\}.$$
We define a moduli space of $\mcM_D$ in the exact same way,
$$\mcM_D(S) = \left\{(X, κ, ρ) \left\vert \text{\begin{varwidth}{\textwidth}\centering $(X,κ)$ a special $O_D$-module over $S$\\
$ρ:\ob{S}\times_{\Spec \mbF} \mbX \lr \ob{S}\times_S X$ an $O_D$-linear quasi-isogeny\end{varwidth}}\right.\right\}.$$
\end{defn}

\begin{prop}\label{prop:representability}
The functors $\mcM_C$ and $\mcM_D$ are representable by formal schemes that are locally formally of finite type over $\Spf O_{\breve F}$. The irreducible components of the maximal reduced subschemes of $\mcM_C$ and $\mcM_D$ are projective over $\Spec \mbF$. Both formal schemes are regular with semi-stable reduction over $\Spf O_{\breve F}$. Moreover, $\mcM_C$ has dimension $n$ and $\mcM_D$ has dimension $2n$.
\end{prop}
\begin{proof}
The representability of $\mcM_C$ and $\mcM_D$ by a locally formally finite type formal scheme, and the fact that the irreducible components of their reduced loci are projective over $\Spec \mbF$ are general properties of (P)EL type RZ spaces, see \cite[Theorem 3.25]{RZ1}. (The analogous result in the equal characteristic setting is \cite[Theorem 4.18]{AH}.) The regularity follows with the standard local model argument: \cite[Proposition 3.33]{RZ1} states that regularity of the two spaces follows from that of the local models for the data $(H, μ_H, O_C^\times)$ and $(G, μ_G, O_D^\times)$. After base extension to $O_{\breve F}$, these local models are isomorphic to parahoric type local models for $GL_n$ resp. $GL_{2n}$ with cocharacter $μ = (1,\ldots, 1,0)$. It is well-known that these are of dimension $n$ (resp. $2n$) with semi-stable reduction \cite[Theorem 5.6]{HPR}. (The assumption $p\neq 2$ is not needed for this part of the theorem. The result is originally due to Drinfeld \cite{Drinfeld}.)
\end{proof}

\subsection{Quadratic CM Cycles}
\label{ss:quad_cycles}

Assume from now on that the maximal orders $O_C$ and $O_D$ are chosen such that $O_C = C\cap O_D$. (See Lemma \ref{lem:order_uniqueness} for a uniqueness statement in this context.) Let $S$ be a scheme over $\Spf O_{\breve F}$ and let $(Y, ι)$ be a strict $O_F$-module with $O_C$-action $ι$ over $S$. Then $O_D\tensor_{O_C} Y$ is a strict $O_F$-module over $S$ with a natural $O_D$-action. It will be useful to have a more explicit description of this construction.

The ring $O_D$ is an $O_E\tensor_{O_F}O_E$-module via left and right multiplication and decomposes into eigenspaces with respect to this action: $O_D = O_C \oplus (D_-\cap O_D)$. We have used here that $E$ is unramified over $F$. The space of conjugation linear element $D_-\cap O_D$ is an $O_C$-module via left multiplication and takes the form $O_C\cdot Π$ for some generator $Π$. With respect to the $O_C$-basis $(1, Π)$, there is then the presentation
\begin{equation}\label{eq:presentation_Serre_tensor_I}
O_D\tensor_{O_C} Y = Y\oplus ΠY
\end{equation}
where $ΠY$ is our notation for the summand $Π\tensor Y$ which we identify with $Y$. Note that $Π^2\in O_C$ and that $Π^{-1}O_CΠ = O_C$. The $O_D$-action on $O_D\tensor_{O_C}Y$ has the matrix description
\begin{equation}\label{eq:O_D_action_explicit}
\begin{aligned}
O_D &\ \lr\ \End(Y \oplus ΠY)\\
a + bΠ &\ \longmapsto \ \begin{pmatrix} a & bΠ^2 \\ Π^{-1}bΠ & Π^{-1}aΠ\end{pmatrix}.
\end{aligned}
\end{equation}
\begin{lem}\label{lem:special_Serre_compatible}
Let $(Y, ι)$ be a special $O_C$-module over a $(\Spf O_{\breve F})$-scheme $S$. Then the Serre tensor construction $(O_D\tensor_{O_C}Y,\ κ(x) = x\tensor \mr{id}_Y)$ is a special $O_D$-module.
\end{lem}
\begin{proof}
Let $L$ denote an unramified extension of degree $2n$ of $F$; fix an embedding $E\to L$. We claim that for any choice of two $E$-linear embeddings $i_1,i_2:O_L\to O_C$, the two images $i_1(O_L)$ and $i_2(O_L)$ are $O_C^\times$-conjugate. To prove this, we consider the decomposition
$$O_C = \bigoplus_{φ\in \mr{Gal}(L/E)} Λ_φ$$
into eigenspaces with respect to the action $i_1\tensor i_2$ of $O_{L}\tensor_{O_F}O_{L}$ by left and right multiplication. Our task is to show that there exists an index $φ$ and an element $x\in O_C^\times \cap Λ_φ$. Namely, any such element satisfies $x^{-1}\circ i_1 \circ x = i_2\circ φ$ and hence $x^{-1}i_1(O_L)x = i_2(O_L)$.

Each $Λ_φ$ is an $i_1(O_L)$-module of rank $1$ and $Λ_φΛ_ψ \subseteq Λ_{φ+ψ}$. It follows that every non-zero homogeneous element $x_φ\in Λ_φ$ lies in $C^\times$ and that if $x_φ\in Λ_φ$ is topologically nilpotent and $0\neq x_ψ\in Λ_ψ$ any other homogeneous element, then $x_ψx_φx_ψ{}^{-1}$ is again topologically nilpotent. Thus, given any two topologically nilpotent homogeneous elements $x_φ, x_ψ$ in degrees $φ$ and $ψ$, say, their product $x_φx_ψ$ is again topologically nilpotent. Since $O_C$ contains elements that are not topologically nilpotent, it follows that there also exists a homogeneous element $x_φ\in Λ_φ$ that is not topologically nilpotent. Then $x_φ^{n+1} \in i_1(O_L)^\times x_φ$ implies that $x_φ\in O_C^\times$ and we have proved the claim.

We now come to the main arguments. Let $X = O_D\tensor_{O_C} Y$. The above claim implies that there exists an embedding $i:O_L\to O_C$ and a choice of $Π$ in \eqref{eq:presentation_Serre_tensor_I} such that $Πi(O_L) = i(O_L)Π$. Let $ψ\in \mr{Gal}(L/F)$ be defined by $Π^{-1}\circ i \circ Π = i\circ ψ$. Note that $ψ$ satisfies $ψ\vert_E \neq \mr{id}_E$ because $Π$ is $E$-conjugate linear and consider the decompositions of $\Lie(Y)$ and $\Lie(X)$ into $O_L$-eigenspaces,
$$\Lie(Y) = \bigoplus_{φ\in \Hom_E(L, \breve F)} \Lie(Y)_φ,\quad \Lie(X) = \bigoplus_{φ\in \Hom_F(L, \breve F)} \Lie(X)_φ.$$
Then \eqref{eq:presentation_Serre_tensor_I} and \eqref{eq:O_D_action_explicit} for our specific choices of $i$ and $Π$ imply that
\begin{equation}\label{eq:Lie_alg_relation}
\Lie(X)_φ \iso \begin{cases} \Lie(Y)_φ & \text{if $φ\vert_E = \mr{id}_E$}\\
\Lie(Y)_{φψ^{-1}} & \text{if $φ\vert_E \neq \mr{id}_E$.}
\end{cases}
\end{equation}
It follows that if $\Lie(Y)_φ$ is of rank $1$ for all $φ\in \Hom_E(L, \breve F)$, then $\Lie(X)_φ$ is of rank $1$ for all $φ\in \Hom_F(L, \breve F)$. By Remark \ref{rmk:reformulation}, this precisely means that $X$ as a special $O_D$-module if $Y$ is a special $O_C$-module.
\end{proof}

From now on, we assume that the two framing objects are related by the Serre tensor construction
\begin{equation}\label{eq:framing}
(\mbX, κ) = O_D \tensor_{O_C} (\mbY, ι).
\end{equation}
In this situation, Lemma \ref{lem:special_Serre_compatible} states that there is a morphism of formal schemes given by
\begin{equation}\label{eq:basic_construction}
\begin{aligned}
\mcM_C &\ \lr\ \mcM_D\\
(Y, ι, ρ) &\ \longmapsto\ (O_D\tensor_{O_C} Y,\ κ(x) = x\tensor \mr{id}_Y,\ \mr{id}_{O_D} \tensor ρ).
\end{aligned}
\end{equation}
Given any subset $T \subseteq \End^0(\mbY)$ of the quasi-endomorphisms of $\mbY$, we define a subfunctor $\mcZ(T) \subseteq \mcM_C$ by
\begin{equation}\label{eq:def_cycle}
\mcZ(T)(S) := \{(Y, ι, ρ) \in \mcM_C(S) \mid ρTρ^{-1} \subseteq \End(Y)\}.
\end{equation}
Here, the condition is meant in the sense that $ρTρ^{-1}$ is always a subset of the \emph{quasi}-endomorphisms $\End^0(Y)$ because $ρ$ is a quasi-isogeny. The functor $\mcZ(T)$ is representable by a closed formal subscheme of $\mcM_C$, see \cite[Proposition 2.9]{RZ1}. In exactly the same way, we define a closed formal subscheme $\mcZ(T)\subseteq \mcM_D$ whenever $T\subseteq \End^0(\mbX)$. We apply this construction to the subring $ι(O_E)\subseteq \End(\mbX)$ to obtain the closed formal subscheme $\mcZ(ι(O_E))\subset \mcM_D$.

Consider an $S$-valued point $(X, κ, ρ)\in \mcZ(ι(O_E))(S)$. Then $X$ is equipped with the two commuting $O_E$-actions $κ\vert_{O_E}$ and $ρ\circ ι \circ ρ^{-1}$. Since $E/F$ is unramified, $(\mr{id}\cdot \mr{id}, \mr{id} \cdot τ):O_E\tensor_{O_F} O_E \overset{\sim}{\to} O_E\times O_E$. We denote by $X = X_+ \oplus X_-$ the resulting eigenspace decomposition of $X$. In particular, $X_+$ is the summand on which the two $O_E$-actions agree.

The purpose of the above definitions was that we can now give a description of the image of $\mcM_C\to \mcM_D$.

\begin{prop}\label{prop:closed_immersion}
The morphism $\mcM_C\to \mcM_D$ is a closed immersion. Its image consists of all those points $(X, κ, ρ)\in \mcZ(ι(O_E))$ with the following two additional properties. Let $X = X_+ \oplus X_-$ be the eigenspace decomposition as explained before.
\begin{enumerate}[wide, labelindent=0pt, labelwidth=!, label=(\arabic*)]
\item The $κ(O_C)$-action on $X_+$ is special in the sense of \eqref{eq:special_C}.
\item The endomorphism $κ(Π)$ defines an isomorphism $κ(Π):X_+ \overset{\iso}{\lr} X_-$.
\end{enumerate}
\end{prop}
\begin{proof}
Let $\mcZ\subseteq \mcZ(ι(O_E))$ be the subfunctor defined by the conditions (1) and (2). Condition (1) is a Zariski closed condition on $\mcZ(ι(O_E))$ because it is given by the equality of the two polynomials in \eqref{eq:special_C}. Condition (2) is an open and closed condition: The map $κ(Π):X_+ \to X_-$ is always an isogeny because $ρ^{-1}κ(Π)ρ:\mbX_+\to \mbX_-$ is an isogeny. Condition (2) then describes the locus where the height of $κ(Π)$ is $0$, which is open and closed. We conclude that $\mcZ$ is a closed formal subscheme of $\mcZ(ι(O_E))$.

It is clear from definitions that the map $\mcM_C \to \mcM_D$ factors through $\mcZ$. Conversely, given a point $(X, κ, ρ)\in \mcZ(S)$, let $(X_+, κ\vert_{O_C}, ρ_+)$ be the direct summand where the $κ(O_E)$ and $(ρι(O_E)ρ^{-1})$-actions coincide. Then $(X_+, κ\vert_{O_C}, ρ_+) \in \mcM_C(S)$ because of Condition (1). Condition (2) ensures that
$$O_D\tensor_{O_C}(X_+, κ\vert_{O_C}, ρ_+) \overset{\iso}{\lr} (X, κ, ρ)$$
via the natural $O_D$-linear map $O_D\tensor_{O_C} X_+ \to X$. This constructs an inverse $\mcZ\to \mcM_C$.
\end{proof}

\subsection{Intersection Numbers}
\label{ss:intersection_numbers}

Let $[b] \in B(H, μ_H)$ be the isogeny class defined by the framing object $(\mbY, ι)$. After a suitable choice of identification, we may simply write (or redefine) $H_b = \End^0_C(\mbY, ι)^\times$ and $G_b = \End^0_D(\mbX, κ)$. Then $H_b$ and $G_b$ act from the right on $\mcM_C$ resp. $\mcM_D$ by composition in the framing. The closed immersion $\mcM_C\to \mcM_D$ is equivariant with respect to $H_b\to G_b$.

\begin{defn}\label{def:intersection_locus}
Let $g\in G_{b,\mr{rs}}$ be a regular semi-simple element. The intersection locus for $g$ is
$$\mcI(g) := \mcM_C \cap (g\cdot \mcM_C).$$
\end{defn}
Let $g \in G_{b,\mr{rs}}$ be regular semi-simple. Recall from Definition \ref{def:notation_rs} that $B_g \subset D_b$ denotes the subring $F[ι(E), g^{-1} ι(E)g]$ and that $L_g\subset B_g$ denotes its center. In particular, it holds that $L_g = C_b \cap g^{-1} C_b g$ and we obtain the following lemma.
\begin{lem}\label{lem:Ig_stable_L}
The action of $L_g^\times \subset G_b$ on $\mcM_D$ preserves both $\mcM_C$ and $g\cdot \mcM_C$. In particular, it preserves $\mcI(g)$.
\end{lem}
One consequence is that $\mcI(g)$ is never quasi-compact if it is non-empty. (Consider e.g. the action of $π^\mbZ \subset L_g^\times$.) However, taking the quotient by $L_g^\times$ solves this issue:

\begin{prop}\label{prop:proper}
Assume that $F$ is $p$-adic. Let $g\in G_{b,\mr{rs}}$ be regular semi-simple and let $Γ\subset L_g^\times$ be a discrete cocompact subgroup with $L_g^\times = Γ\times O_{L_g}^\times$. Then $\mcI(g)$ is a scheme and the quotient $Γ\backslash \mcI(g)$ is proper over $\Spec O_{\breve F}$.
\end{prop}

\begin{rmk}\label{rmk:F_p_adic}
The only reason for the restriction to $p$-adic $F$ is that our proof relies on \cite[Lemma 4.3.15]{CS} which is only stated for $p$-divisible groups. The statement should also be true when $F = \mbF_q(\!(π)\!)$, however, and we assume this for later definitions.
\end{rmk}

Note that any $Γ$ as in Proposition \ref{prop:proper} acts without fixed points on $\mcM_D$ by \cite[Corollary 2.35]{RZ1}. The quotient $Γ\backslash \mcI(g)$ can be constructed in the following way. First choose a finite index subgroup $Γ'\subset Γ$ that acts properly discontinuously on $\mcM_D$. The quotient $Γ'\backslash \mcI(g)$ can be constructed in the Zariski topology. Then pass to $(Γ/Γ')\backslash (Γ'\backslash \mcI(g))$ which is a quotient of a formal scheme by a finite group that acts without fixed points.

In the following we write $L = L_g$ and $B = B_g$.

\begin{proof}
The proof will even show that $\mcZ(ι(O_E))\cap g\mcZ(ι(O_E))$ is a scheme and that the quotient $Γ\backslash (\mcZ(ι(O_E))\cap g\cdot \mcZ(ι(O_E)))$ is quasi-compact. It is based on the observation that
\begin{equation}\label{eq:int_identity}
\mcZ(ι(O_E) \cap g\cdot \mcZ(ι(O_E)) = \mcZ(ι(O_E) \cup g^{-1} ι(O_E)g) = \mcZ(R),
\end{equation}
where $R$ is defined as the ring $R = O_F[ι(O_E), g^{-1} ι(O_E)g] \tensor_{O_F}O_D$. This is an order in the semi-simple $F$-algebra $P = B\tensor_F D$.

It follows that the generic fiber of $\mcI(g)$ is empty: The algebra $P$ is a CSA of degree $4n$ over $L$. It can only act faithfully on an étale $π$-divisible $O_F$-module of height $4n^2$ if $P \iso M_{4n}(L)$. But this would mean that $B$ splits $D$ which would imply that $ε_D(g) = 1$ by Lemma \ref{lem:sign_central_value}. However, this is excluded by Proposition \ref{prop:char_isogeny_class}.

We next prove that $Γ\backslash \mcI(g)$ is quasi-compact. This is equivalent to proving that the set of closed points $\mcZ(R)(\mbF)$ is bounded modulo $Γ$ in the following sense. The set $\mcZ(R)(\mbF)$ identifies with the set of $R$-stable Dieudonné lattices $\bM$ in the isocrystal $\bN = (N, \bF)$ of $\mbX$ that are special. Let $\bM(\mbX)\subset \bN$ be the Dieudonné lattice defined by $\mbX$. We need to see that there exists an integer $c \geq 0$ such that for every $\bM \in \mcZ(R)(\mbF)$, there is some $x\in Γ$ with $π^c\bM(\mbX) \subseteq x\bM \subseteq π^{-c}\bM(\mbX)$. (The condition of points in $\mcZ(R)(\mbF)$ being special will not play a role for the argument.)

Let $L = \prod_{i\in I}L_i$ and $P = \prod_{i\in I}P_i$ be the decompositions that correspond to the idempotents in $L$. Then $P_i$ is a CSA over $L_i$ of degree $4n$ and the corresponding summand $\bN_i$ of $\bN$ has height $4n[L_i:F]$. In this situation, the set of Dieudonné lattices $\bM_i\subseteq \bN_i$ that are stable under some choice of maximal order $O_{P_i}\subset P_i$ is bounded modulo $π^\mbZ$. (Indeed, $\breve O_{P_i} = O_{\breve F}\tensor_{O_{L_i}} O_{P_i}$ is an order in $M_{4n}(\breve F)$. The set of $\breve O_{P_i}$-stable lattices in $\breve F^{4n}$ is bounded.) Thus the set of $O_P = \prod_{i\in I} O_{P_i}$-stable Dieudonné lattices in $\bN$ is bounded modulo $Γ$. For every $R$-stable Dieudonné lattice $\bM$, the lattice $O_P\cdot \bM$ is $O_P$-stable and the index $[O_P\cdot \bM : \bM]$ is bounded in terms the index $[O_P:R]$. It follows that $Γ\backslash \mcZ(R)(\mbF)$ is bounded as claimed.

It is left to show that $\mcI(g)$ is a scheme. A priori, it is known to be a locally noetherian formal scheme. We thus need to see that for every noetherian, adic, $π$-adically complete $O_{\breve F}$-algebra $A$, every morphism $f:\Spf A\to \mcI(g)$ extends to a map $\Spec A\to \mcI(g)$. Equivalently, we need to see that for every such $f$, the ideal $J(f) = f^{-1}(\mcO_{\mcI(g)}^{\circ\circ})A$ that is generated by the inverse images of all topologically nilpotent elements in $\mcO_{\mcI(g)}$ is nilpotent. We claim that it suffices to consider the case of a DVR: Indeed, assume that there exists a prime ideal $\mfp\subset A$ such that $J(f)\not\subset \mfp$ and let $\mfm$ be a maximal ideal containing $\mfp$. It necessarily holds that $J(f)\subseteq \mfm$ because $J(f)$ is nilpotent. Since $A$ is noetherian, there exists a complete DVR $B$ and a map $g:\Spec B \to \Spec A$ such that the generic point of $\Spec B$ maps to $\mfp$ and the special point to $\mfm$. In particular $J(f\circ g) = g^{-1}(J(f))B$ would be a non-trivial ideal of $B$. This proves the claim and allows us to henceforth assume that $A$ is a DVR. We will even assume that $A$ is complete with algebraically closed residue field. We already know that $\mcI(g)$ has empty generic fiber, so it holds that $πA = 0$. we write $\Spec A = \{s, η\}$ where $s$ is the special and $η$ the generic point.

Let $(X, κ, ρ)$ be the point that defines the morphism $f:\Spf A\to \mcI(g)$. The datum $X$ algebraizes to a strict $O_F$-module over $\Spec A$ with $R$-action. Our task is to show that $ρ$ algebraizes as well.
The key observation for this is that, by Proposition \ref{prop:char_isogeny_class}, the geometric isogeny class of the generic fiber $X_η$ is uniquely determined by $\Inv(g;T)$ and hence equal to that of $\mbX$. In other words, the point-wise slope vector of $X$ on $\Spec A$ is constant. The perfection $A^{\mr{perf}} = \colim_{x\mapsto x^p} A$ of $A$ is again strictly henselian.
\begin{lem}[\protect{\cite[Lemma 4.3.15]{CS}}]\label{lem:CS}
Let $k$ be the residue field of $A^\mr{perf}$. The functor $Y\mapsto k\tensor_{A^\mr{perf}} Y$ from strict $O_F$-modules $Y$ up to isogeny over $A^\mr{perf}$ to strict $O_F$-modules up to isogeny over $k$ is an equivalence.
\end{lem}
\begin{proof}
The cited lemma states this when $F = \mbQ_p$. The general case follows immediately because strict $O_F$-modules are nothing but $p$-divisible groups with strict $O_F$-action.
\end{proof}
By Lemma \ref{lem:CS}, there exists a quasi-isogeny $ρ':A^{\mr{perf}}\tensor_k \mbX \to X$ such that $k\tensor_{A^{\mr{perf}}} ρ' = k\tensor_A ρ$. By the rigidity of quasi-isogenies \cite[(2.1)]{RZ1}, this implies $ρ' = A^{\mr{perf}}\tensor_A ρ$ which shows that $A^{\mr{perf}}\tensor_A ρ$ is algebraic. The map $A\to A^{\mr{perf}}$ is faithfully flat, so this implies that $ρ$ is algebraic.
\end{proof}

\begin{defn}\label{def:intersection_number}
For $g\in G_{b,\mr{rs}}$, we define
$$\Int(g) := χ\left(Γ\backslash \mcI(g),\ \  \mcO_{Γ\backslash \mcM_C} \tensor^{\mbL}_{\mcO_{Γ\backslash \mcM_D}} \mcO_{Γ\backslash g\cdot \mcM_C}\right)\in \mbZ.$$
Here, by the regularity of $\mcM_C$ and $\mcM_D$ from Proposition \ref{prop:representability}, the complex on the right hand side is perfect. It is supported on $Γ\backslash \mcI(g)$ which is a projective $O_{\breve F}$-scheme with $π^N\mcO_{\mcI(g)} = 0$ for $N\gg 0$ by Proposition \ref{prop:proper}. This explains why $\Int(g)$ is well-defined. Note that the passage to the quotient by $Γ$ is completely analogous to taking the quotient by the stabilizer in the definition of the orbital integrals in \S3.1.
\end{defn}

We end this section with some auxiliary results about the intersection locus $\mcI(g)$ that will be useful in later sections.

\begin{lem}\label{lem:nilpotent_reduction}
Let $\bN$ be an isocrystal with $E$-action and let $g = g_+ + g_- \in \End(\bN)^\times$ be an automorphism such that $g_+$ lies again in $\End(\bN)^\times$, where $g_+$ and $g_-$ denote the $E$-linear resp. $E$-conjugate linear components of $g$. Assume that there exists an $O_E$-stable Dieudonné lattice  $\bM\subset \bN$ such that $g\bM$ is $O_E$-stable as well. Assume furthermore that the Verschiebung $\bV$ on $\bN$ is topologically nilpotent and that the $O_E$-action on both $\bM/\bV\bM$ and $g\bM/\bV(g\bM)$ is strict. Then $z_g = g_+^{-1}g_-$ is topologically nilpotent.
\end{lem}
\begin{proof}
Considering $g_+^{-1}g\bM$ instead, we may assume that $g$ is of the form $g = 1 + z$ with $z = z_g$. Let $\bM = \bM_+ \oplus \bM_-$ and $g\bM = \bM_+' \oplus \bM_-'$ be the bigradings that come from the $O_E$-action.
\emph{Claim: It holds that $z \bM_+\subseteq \bM_-$.} Assume this claim holds. The strictness condition for $\bM$ means that $\bV\bM_+ = \bM_-$. Thus we obtain $z\bM_+\subseteq \bV\bM_+$ and hence $z^{2n} \bM_+ \subseteq \bV^{2n}\bM_+$ for every $n\geq 0$. The Verschiebung is topologically nilpotent by assumption, so it follows that $z$ is topologically nilpotent as claimed. It is only left to prove the claim.

\emph{Proof of the Claim.} First note that
\begin{equation}\label{eq:ident_Dieud}
\bM_+' = \bM_+ + z\bM_-\quand \bM_-' = \bM_- + z\bM_+.
\end{equation}
Moreover, the $O_E$-action on $g\bM$ was assumed to be strict as well, meaning that $\bV\bM_+' = \bM_-'$. Substituting this in \eqref{eq:ident_Dieud}, it follows that
$$\bM_- + z\bM_+ = \bM_- + zV\bM_-.$$
In particular, $z\bM_+\subseteq \bM_- + z\bV^2 \bM_+$ and hence, for all $i\geq 1$,
\begin{equation}\label{eq:ident_Dieud_II}
z\bV^{2i-2}\bM_+ \subseteq \bM_- + z\bV^{2i}\bM_+.
\end{equation}
Since $\bV$ is topologically nilpotent by assumption, there exists an integer $i$ such that $z\bV^{2i}\bM_+ \subset \bM_-$. Descending induction on $i$ with the help of \eqref{eq:ident_Dieud_II} proves that $z\bM_+\subseteq \bM_-$ as claimed.
\end{proof}

By definition, for every regular semi-simple $g\in G_b$, the element $z_g = g_+^{-1}g_-$ lies in $\End^0_D(\mbX, κ)$. So the definition in \eqref{eq:def_cycle} applies and defines a closed formal subscheme $\mcZ(z_g)\subseteq \mcM_D$.

\begin{prop}\label{prop:nilpotent_reduction}
(1) Assume that $[b]\in B(H, μ_H)$ is such that $\bN_b$ has no étale part and assume that $g\in G_{b, \mr{rs}}$ is regular semi-simple. If $\mcI(g)\neq \emptyset$, then $z_g$ is topologically nilpotent.

\noindent (2) Let $[b]\in B(H, μ_H)$ be any and let $g\in G_{b, \mr{rs}}$ be an element such that $z_g$ is topologically nilpotent. Then $\mcI(g) = \mcM_C \cap \mcZ(z_g)$.
\end{prop}
\begin{proof}
(1) The assumption that $\bN_b$ has no étale part precisely says that $\bV$ is topologically nilpotent. Then any point $(X, κ, ρ)\in \mcI(g)(\bM)$ defines a Dieudonné lattice $\bM\subset \bN_b$ that satisfies the conditions of Lemma \ref{lem:nilpotent_reduction}.

(2) Let $ζ\in O_E^\times$ be an $O_F$-algebra generator of trace $1$. It always holds that $1-2ζ\in O_E^\times$ because $E/F$ is unramified. Using that $z_g$ is topologically nilpotent, identity \eqref{eq:get_to_linear} then implies the following equality of subrings of $\End^0_D(\mbX)$,
\begin{equation}\label{eq:order}
R = O_F[ι(O_E), g^{-1}ι(O_E)g] = O_E[z_g].
\end{equation}
We obtain from \eqref{eq:int_identity} that $\mcM_C \cap g \cdot \mcM_C \subseteq \mcM_C \cap \mcZ(z_g)$. Assume conversely that $(X, κ, ρ)\in \mcM_C\cap \mcZ(z_g)$. We need to show that $(X, κ, ρ)\in g\cdot \mcM_C$. Equivalently, by the $H_b$-equivariance of the embedding $\mcM_C\to \mcM_D$, we need to show that $(X, κ, ρ)\in (g_+^{-1}g)\mcM_C = (1 + z_g)\mcM_C$. So we may assume that $g = 1+z_g$ from now on. Using that $(X, κ, ρ)\in \mcZ(z_g)$ and also that $z_g$ is topologically nilpotent by assumption, $ρgρ^{-1} = 1+ρz_gρ^{-1}$ defines an automorphism of $X$. Thus $(X, κ, ρ)$ is a $g$-fixed point of $\mcM_D$ that also lies in $\mcM_C$, and hence lies in $g\cdot \mcM_C$.
\end{proof}

\subsection{The Arithmetic Transfer Conjecture}
\label{ss:ATC}

We can now formulate our AT conjecture. Recall that $f_D = 1_{O_D^\times}\in C^\infty_c(G)$ denotes the standard test function on the CSA side, see \S\ref{ss:FL}.
\begin{conj}[ATC]\label{conj:ATC}
There exists a transfer $f''_D\in C^\infty_c(G')$ of $f_D$ in the sense of Definition \ref{def:matching} with the following additional property. For every regular semi-simple element $γ\in G'_{\mr{rs}}$,
\begin{equation}\label{eq:ATC}
\del(γ, f''_D) = \begin{cases} 2\,\Int(g)\log(q) & \text{\begin{varwidth}{\textwidth}if there exists some $[b]\in B(H, μ_H)$\\and some $g\in G_{b, \mr{rs}}$ that matches $γ$\end{varwidth}}\\[8pt]
0 & \text{otherwise.}
\end{cases}
\end{equation}
\end{conj}

\begin{conj}[ATC -- Equivalent Form]\label{conj:ATC_equiv}
For every transfer $f''_D\in C^\infty_c(G')$ of $f_D$ in the sense of Definition \ref{def:matching}, there exists a correction function $f''_{\mr{corr}}\in C^\infty_c(G')$ such that for every regular semi-simple element $γ\in G'_{\mr{rs}}$,
\begin{equation}\label{eq:ATC_equiv}
\del(γ, f''_D) + \Orb(γ, f''_{\mr{corr}}) = \begin{cases} 2\,\Int(g)\log(q) & \text{\begin{varwidth}{\textwidth}if there exists some $[b]\in B(H, μ_H)$\\and some $g\in G_{b, \mr{rs}}$ that matches $γ$\end{varwidth}}\\[8pt]
0 & \text{otherwise.}
\end{cases}
\end{equation}
\end{conj}

\begin{proof}[Proof of the equivalence of Conjectures \ref{conj:ATC} and \ref{conj:ATC_equiv}.]
One direction is completely elementary: Let $ϕ'\in C^\infty_c(G')$ be any test function. Fix some $h\in H'$ that satisfies $|h|^{-s} = q^s$. The function $θ(ϕ') := ϕ' - (h, 1)^*(ϕ')$ then satisfies $\Orb(γ, θ(ϕ'), s) = (1 - q^s)\Orb(γ, ϕ', s)$ for all $γ\in G'_{\mr{rs}}$, see \eqref{eq:orb_int_trafo}, and hence
$$\Orb(γ, θ(ϕ')) = 0,\quad \del(γ, θ(ϕ')) = -\Orb(γ, ϕ') \log(q).$$
So if $f''_D$ and $f''_{\mr{corr}}$ are as in Conjecture \ref{conj:ATC_equiv}, then $f''_D - θ(f''_{\mr{corr}})/\log(q)$ has all the properties required in Conjecture \ref{conj:ATC}.

The converse direction relies on the density principle for orbital integrals on $G'$ which is due to H. Xue \cite[Theorem 8.3]{Xue2}. It states that any test function $f'\in C^\infty_c(G')$ such that $\Orb(γ, f') = 0$ for all $γ\in G'_{\mr{rs}}$ lies in the space
$$V = \{ϕ' - η(h_2)(h_1, h_2)^*(ϕ') \mid ϕ'\in C^\infty_c(G'),\ h_1,h_2\in H'\}.$$
We apply this as follows: Assume that $f''_D$ has all the properties that are required in Conjecture \ref{conj:ATC} and assume that $f'\in C^\infty_c(G')$ is any transfer of $f_D$. Then $f''_D - f'$ lies in $V$. Since
$$\del (γ, ϕ' - η(h_2)(h_1, h_2)^*(ϕ')) = \Orb(γ, ϕ') \log |h_1h_2|,$$
we deduce that there exists a correction function $f''_{\mr{corr}}$ with $\Orb(γ, f''_{\mr{corr}}) = \del(γ, f''_D - f')$ for all $γ\in G'_{\mr{rs}}$. Then $(f', f''_{\mr{corr}})$ has all the properties that are required in Conjecture \ref{conj:ATC_equiv}.
\end{proof}

Taking into account our FL (Conjecture \ref{conj:FL}), we have the following explicit form of the AT:

\begin{conj}[ATC -- Explicit Form]\label{conj:ATC_explicit}
Let $f'_D$ be the test function from Definition \ref{def:test_function}. There exists a correction function $f'_{\mr{corr}}\in C^\infty_c(G')$ such that for every regular semi-simple element $γ\in G'_{\mr{rs}}$,
\begin{equation}\label{eq:ATC_explicit}
\del(γ, f'_D) + \Orb(γ, f'_{\mr{corr}}) = \begin{cases} 2\,\Int(g)\log(q) & \text{\begin{varwidth}{\textwidth}if there exists some $[b]\in B(H, μ_H)$\\and some $g\in G_{b, \mr{rs}}$ that matches $γ$\end{varwidth}}\\[8pt]
0 & \text{otherwise.}
\end{cases}
\end{equation}
\end{conj}

The status of Conjecture \ref{conj:ATC_explicit} is as follows:
\begin{enumerate}[wide, labelindent=0pt, labelwidth=!, label=(\arabic*)]
\item Consider the case that $D \iso M_{2n}(F)$. Then it is conjectured that one may take $f'_{\mr{corr}} = 0$ (AFL conjecture). The AFL conjecture first appeared in \cite{Li}\footnote{The version in \cite{LM} includes a correction that is related to the counting of connected components of $Γ\backslash \mcI(g)$.} and has been verified for $n = 1$ and $n = 2$ in \cite{Li} and \cite{Li_future}.

For general $n$, at least the vanishing part of \eqref{eq:ATC_explicit} is known by \cite[Corollary 2.14]{LM}. Furthermore, \cite[Theorem 1.2]{LM} states that it is enough to consider \eqref{eq:ATC_explicit} for all basic isogeny classes.

\item Consider next the case that $D\iso M_n(D_{1/2})$. Then \cite[Theorem B]{HM} reduces Conjecture \eqref{conj:ATC_explicit} to the linear AFL conjecture for $M_{2n}(F)$. In particular, the case $D\iso M_2(D_{1/2})$ is known by \cite{Li_future}.

\item The main result of the present paper is a verification of Conjecture \ref{conj:ATC_explicit} for $D \iso D_{1/4}$ and $D\iso D_{3/4}$. In particular, Conjecture \ref{conj:ATC_explicit} is known in all cases with $n\leq 2$.
\end{enumerate}

\part{Orbital integrals for $GL_4$}

\section{Main Results}
\label{s:main_analytic}
We now specialize to the case $n = 2$, i.e. $G' = GL_4(F)$. Consider the following two subgroups of $GL_4(O_F)$,
\begin{equation}
\Par := \begin{pmatrix}
GL_2(O_F) & π\,M_2(O_F)\\
M_2(O_F) & GL_2(O_F)
\end{pmatrix},\quad\quad \Iw := \begin{pmatrix}
O_F^\times & (π) & (π) & (π)\\
O_F & O_F^\times & O_F & (π)\\
O_F & (π) & O_F^\times & (π)\\
O_F & O_F & O_F & O_F^\times
\end{pmatrix}.
\end{equation}
The first is the stabilizer of the lattice chain
$$O_F^{\oplus 2}\oplus O_F^{\oplus 2}\ \supset\ (π)^{\oplus 2}\oplus O_F^{\oplus 2},$$
the second is the stabilizer of
$$\begin{aligned}
O_F^{\oplus 2}\oplus O_F^{\oplus 2} & \ \supset\ (π) \oplus O_F \oplus O_F \oplus O_F\\
& \ \supset\ (π) \oplus O_F\oplus (π)\oplus O_F\\
& \ \supset\ (π) \oplus (π)\oplus (π)\oplus O_F.
\end{aligned}$$
These are standard lattice chains in the sense of Definition \ref{def:test_function}. Set $f'_\Par = 1_\Par$ and $f'_\Iw = (q+1)^4\,1_\Iw.$ These define the functions $f_D'^\circ$ from Definition \ref{def:test_function}, i.e.
\begin{equation}
f_D'^\circ = \begin{cases}
1_{GL_4(O_F)} & \text{if $D \iso M_4(F)$}\\
f'_\Par & \text{if $D\iso M_2(D_{1/2})$}\\
f'_\Iw & \text{if $D$ division.}
\end{cases}
\end{equation}
In this section, $D$ will always be a division algebra of degree $4$ and $f'_D$ the corresponding test function. The relation of $f'_D$ and $f_D'^\circ$ from \eqref{eq:norm_orb_int} specializes to
\begin{equation}\label{eq:norm_orb_int_Iw}
\Orb(γ, f'_D, s) = q^{-s}\Orb(γ, f'_\Iw, s).
\end{equation}
The aim of this chapter is to compute the central values and the central derivatives $\Orb(γ, f'_\Par)$, $\Orb(γ, f'_\Iw)$ and $\del(γ, f'_\Iw)$. Our results on $\Orb(γ, f'_\Iw)$ will, in particular, prove the FL conjecture for $D$. The results about $\Orb(γ, f'_\Par)$ and $\del(γ, f'_\Iw)$ in turn will be used to verify the AT conjecture later.

We now define the so-called \emph{numerical invariant} of an element $γ\in G'_{\mr{rs}}$ or $g\in G_{\mr{rs}}$. It simplifies the invariant $\Inv(γ;T)$ resp. $\Inv(g;T)$ in the sense that it only records a certain valuation and a certain conductor. Its significance lies in the fact that all orbital integrals and all intersection numbers in this article only depend on the numerical invariant.

Recall the definition of the conductor: Assume $L/F$ is an étale quadratic extension and $O\subset L$ an $O_F$-order. The conductor $\mr{cond}(O)$ is the unique integer $c\geq 0$ such that $O = O_F + π^cO_L$.

\begin{defn}\label{def:orbit_invariants}
Let $δ = T^2 + δ_1 T + δ_0\in F[T]$ be a regular semi-simple invariant of degree $2$. Recall that this means that $δ$ is separable with $δ(0)δ(1)\neq 0$. The numerical invariant of $δ$ is the triple $(L, r, d)$ where
\begin{equation}\label{eq:num_inv_I}
L := F[T]/(δ(T)),\quad r = v(δ_0),\quad d = \mr{cond}(O_F[π^k\cdot t]) - r/2 - k.
\end{equation}
Here, $t := T$ mod $(δ(T))$ is the image of $T$ in $L$. The étale quadratic $F$-algebra $L$ is only considered up to isomorphism. In fact, everything will only depend on whether $L/F$ is inert, ramified or split. Moreover, the integer $k$ in \eqref{eq:num_inv_I} is chosen sufficiently large so that $π^k t\in O_L$; the definition of $d$ is independent of this choice.

The numerical invariant of a regular semi-simple element $γ\in G'_{\mr{rs}}$ or $g\in G_{\mr{rs}}$ is the numerical invariant of $\Inv(γ; T)$ resp. $\Inv(g;T)$. For example, the numerical invariant of an element $γ\in G'_{\mr{rs}}$ may also be written as
\begin{equation}\label{eq:invariants_gamma_concrete}
\left(L_γ,\ \ v(\det(z_γ)),\ \ \mr{cond}(O_F[π^kz_γ^2]) - v(\det(z_γ))/2 - k\right),\quad k\gg 0.
\end{equation}
\end{defn}
Note that Lemma \ref{lem:sign_alternative} expresses the sign of the functional equations of $f'_\Par$ and $f'_\Iw$ directly in terms of $r$:
\begin{equation}\label{eq:sign_concrete}
ε_{M_2(D_{1/2})}(γ) = (-1)^r,\quad ε_D(γ) = (-1)^{r+1}.
\end{equation}
The following three are our main results in this chapter and will all be proved in \S\ref{s:orb_ints}.
\begin{prop}\label{prop:orb_int_para_teaser}
Let $γ\in G'_{\mr{rs}}$ be regular semi-simple with numerical invariant $(L, r, d)$. The parahoric orbital integral $\Orb(γ, f'_\Par)$ vanishes if $r$ is odd, or if $r \leq 0$, or if $r/2 + d \leq 0$. In all other cases, it is given by
\begin{equation}\label{eq:orb_int_para_teaser}
\begin{cases}
\phantom{2(}1 + q^2 +  \ldots + q^{r/2 - 2} & \text{if $L$ ramified and $r\in 4\mbZ$}\\
\phantom{2}(1 + q^2 +  \ldots + q^{r/2 - 3}) + \phantom{2}(q^{r/2-1} + q^{r/2} + \ldots + q^{r/2 + d - 1}) & \text{if $L$ ramified and $r\in 2 + 4\mbZ$}\\
2(1 + q^2 + \ldots + q^{r/2-2}) & \text{if $L$ inert and $r\in 4\mbZ$}\\
2(1 + q^2 + \ldots + q^{r/2-3}) + 2(q^{r/2-1} + q^{r/2} + \ldots + q^{r/2 + d - 2}) + q^{r/2 + d - 1} & \text{if $L$ inert and $r\in 2 + 4\mbZ$}\\
\phantom{2(}0 & \text{if $L$ split and $r\in 4\mbZ$}\\
\phantom{2(}q^{r/2 + d - 1} & \text{if $L$ split and $r\in 2 + 4\mbZ$.}
\end{cases}
\end{equation}
\end{prop}

\begin{thm}\label{thm:Iw_central_teaser}
The fundamental lemma (Conjecture \ref{conj:FL}) holds. In other words, for every regular semi-simple $γ\in G'_{\mr{rs}}$,
\begin{equation}\label{eq:FL_intermediate_teaser}
\Orb(γ, f'_\Iw) = \begin{cases} \Orb(g, f_D) & \text{if there exists a matching $g\in G$}\\
0 & \text{otherwise.}
\end{cases}
\end{equation}
\end{thm}
Let $(L, r, d)$ be the numerical invariant of an element $γ\in G'_{\mr{rs}}$ and let $δ = \Inv(γ;T)$. We remark that by Corollary \ref{cor:universal_quaternion} the matching element $g$ in \eqref{eq:FL_intermediate_teaser} exists if and only if $B_δ$ (constructed for $E/F$) is a division algebra, which is if and only if $L_δ\tensor E$ is a field and $z^2 \in L_δ$ not a norm from $L_δ\tensor_FE$, which is if and only if $L/F$ is a ramified field extension and $r$ odd. (Recall that $L \iso L_δ$.)

\begin{prop}\label{prop:derivative_teaser}
Let $γ\in G'_{\mr{rs}}$ be regular semi-simple with numerical invariants $(L, r, d)$. Assume first that $r$ is odd, meaning that the sign $ε_D(γ)$ of the functional equation of $\Orb(γ, f'_\Iw, s)$ is positive. Then
\begin{equation}\label{eq:derivative_trivial_case_teaser}
\del(γ, f'_D) = 0\quand \dOrb(γ, f'_\Iw) = \Orb(γ, f'_\Iw) \log(q).
\end{equation}
Assume now that $r$ is even which implies $\del(γ, f'_D) = \del(γ, f'_\Iw)$. If $r\leq 0$, then $\dOrb(γ, f'_\Iw) = 0$. If $r > 0$, it is given by
\begin{equation}\label{eq:derivative_main_teaser}
\del(γ, f'_\Iw) = 4q\log(q)\,\Orb(γ, f'_\Par) + \log(q)\,\begin{cases}
r & \text{if $L$ ramified}\\
2r & \text{if $L$ inert}\\
0 & \text{if $L$ split}.
\end{cases}
\end{equation}
\end{prop}

\section{Hyperbolic Orbits}
\label{s:hyperbolic}

We call a regular semi-simple element $γ\in G'_{\mr{rs}}$ hyperbolic if $L_γ\iso F\times F$. In this situation, the orbital integrals $\Orb(γ, f'_\Iw, s)$ and $\Orb(γ, f'_\Par, s)$ can be expressed in terms of much simpler orbital integrals for the Levi that is defined by $L_γ$. 

In the following we fix a hyperbolic element $γ\in G'_{\mr{rs}}$ of the form $γ = 1+z_γ$; set $z = z_γ$. We also fix an isomorphism $L_γ \iso F\times F$. Recall from \S\ref{ss:setting} that $K = \{\diag(a, a, b, b) \in M_4(F) \mid a, b\in F\}$ denotes the diagonal copy of $F\times F$, and recall from Proposition \ref{prop:quaternion_algebra} that $V := F^4$ is free as $K\tensor_FL_γ$-module.

Let $V = V^0\oplus V^1$ be the eigenspace decomposition as $L_γ$-module. It is preserved by $γ$ because $γ$ and $z$ commute under our assumption $γ = 1+z$. It also has the property that both $V^0$ and $V^1$ are free $K$-modules of rank $1$. Thus, we are precisely in the setting of the Levi reduction formula from \cite{LM} and we begin by recalling the relevant results from \cite{LM}.

\subsection{Lattice Decomposition}

The reduction to the Levi is based on the fact that there is a bijection of lattices $X\subset V$ and the set
\begin{equation}\label{eq:lattice_bijection}
\left\{(X^0, X^1, s) \,\left\vert\, \text{\begin{varwidth}{\textwidth} \center $X^0\subset V^0,\ X^1\subset V^1$ both $O_F$-lattices\\
$s:X^1\to V^0/X^0$ any $O_F$-linear map\end{varwidth}}\right\}\right..
\end{equation}
It is given by sending $X$ to $(X^0, X^1, s)$ where
\begin{equation}\label{eq:lattice_bijection_map}
X^0 = X\cap V^0,\quad X^1 = (X+V^0)/V^0,\quad s = [X^1\to X \to V^0/X^0].
\end{equation}
Here, the map $X\to V^0/X^0$ is the projection to the first component and the map $X^1\to X$ is defined by any choice of splitting for $X\twoheadrightarrow X^1$. Moreover, there is a criterion for lattice inclusions. Assume that $X^0\subseteq Y^0\subset V^0$ and $X^1\subseteq Y^1\subset V^1$ are sublattices and that $s_Y:Y^1\to V^0/Y^0$ resp. $s_X:X^1\to V^0/X^0$ are maps as in \eqref{eq:lattice_bijection}. Let $X, Y\subset V$ be the corresponding lattices in $V$. Then
\begin{equation}\label{eq:inclusion_criterion}
X\subseteq Y\quad \Longleftrightarrow\quad \text{the diagram}\quad
\begin{minipage}{5cm}
\xymatrix{
Y^1 \ar[r]^{s_Y} & V^0/Y^0 \\
X^1 \ar[r]^{s_X} \ar@{^{(}->}[u] & V^0/X^0 \ar@{->>}[u]}\end{minipage}
\quad\text{commutes.}
\end{equation}
The following lemma is immediately clear and stated here for later application.
\begin{lem}\label{lem:inclusion_triviality}
Consider two lattices $X^0\subset Y^0$ and $X^1 \subset Y^1$ as in Diagram \eqref{eq:inclusion_criterion}.
\begin{enumerate}[wide, labelindent=0pt, labelwidth=!, label=(\arabic*), topsep=2pt, itemsep=2pt]
\item Assume that $X^0 = Y^0$. Then for every map $s_Y:Y^1\to V^0/Y^0$, there is a unique map $s_X$ such that \eqref{eq:inclusion_criterion} commutes.
\item Assume that $X^1 = Y^1$. Then for every map $s_X:X^1\to V^0/X^0$, there is a unique map $s_Y$ such that \eqref{eq:inclusion_criterion} commutes.
\end{enumerate}
\end{lem}

In the situation of the fixed hyperbolic element $γ$, there is the following numerical result. Write $γ^0 = γ\vert_{V^0}$ and $γ^1 = γ\vert_{V^1}$ for the two components. Then $γ^0$ and $γ^1$ are regular semi-simple (in the sense of \S\ref{s:invariants}) as endomorphisms of the $K$-modules $V^0$ and $V^1$, respectively, and
$$\Inv(γ; T) = \Inv(γ^0; T)\Inv(γ^1; T).$$
Thus, if we define $α^j\in F$ by $\Inv(γ^j; T) = T-α^j$, then $α^0,α^1\not\in \{0,1\}$ and $α^0 \neq α^1$ by regular semi-simpleness of $γ$. We also define $z^j$ as the $j$-component of $z$. Equivalently, $z^j = z_{γ^j}$.

\begin{prop}\label{prop:count_lattice_gluing}
Assume that $X^0\subset V^0$ and $X^1\subset V^1$ are two $O_K$-lattices that are $z^j$-stable. Then there are $|α^0-α^1|^{-1}$ many lattices $X\subset V$ such that
\begin{enumerate}[wide, labelindent=0pt, labelwidth=!, label=(\arabic*), topsep=2pt, itemsep=2pt]
\item $X\cap V^0 = X^0$ and $(X+V^0)/V^0 = X^1$,
\item $X$ is $O_K$-stable and $z$-stable.
\end{enumerate}
\end{prop}
This is essentially a very special case of \cite[Proposition 4.5]{LM}. There are, however, some boundary cases which are not covered by that result (especially if the residue cardinality is $2$) which is why we include a short proof.
\begin{proof}
Fix $O_K$-linear isomorphisms $O_K\iso X^j$ for both $j = 0,1$. Via these coordinates, we understand $z^0$ and $z^1$ as $O_K$-conjugate linear endomorphisms of $O_K$. By \eqref{eq:inclusion_criterion}, the lattices $X\subset V$ that satisfy the conditions (1) and (2) are in bijection with the set
\begin{equation}\label{eq:lattice_gluing_hom_set}
\left\{s\in \Hom_{O_K}(O_K, K/O_K) \mid z^0 s = s z^1\right\}.
\end{equation}
Also fix an isomorphism $O_K\iso O_F\times O_F$. In this basis, $z^0$ and $z^1$ are given by anti-diagonal matrices because they are $O_K$-conjugate linear, say
$$z^0 = \begin{pmatrix}
 & a\\ b &
\end{pmatrix},\quad z^1 = \begin{pmatrix}
 & c\\ d &
\end{pmatrix}.$$
Here, $a, b, c$ and $d$ all lie in $O_F$ while $ab = α^0$ and $cd = α^1$. An element $(s_+, s_-)\in \Hom_{O_F}(O_F, F/O_F)^2$ lies in the set \eqref{eq:lattice_gluing_hom_set} if and only if
\begin{equation}\label{eq:lat_concrete}
as_+ = cs_-,\quad bs_- = ds_+.
\end{equation}
Thus we need to count the solutions $(s_+, s_-)\in (F/O_F)^2$ of \eqref{eq:lat_concrete}. By symmetry of the expression, we may assume that $a$ is the coefficient with minimal valuation. Dividing \eqref{eq:lat_concrete} by $a$, we first note that the solutions to
$$s_+ = a^{-1}cs_-,\quad a^{-1}bs_- = a^{-1}ds_+$$
are precisely the pairs of the form $(a^{-1}cs_-, s_-)$ with $(a^{-1}b - a^{-2}cd)s_- = 0$. There are $|a^{-1}b - a^{-2}cd|^{-1}$ many such pairs. It follows that the solution count for \eqref{eq:lat_concrete} is
$$|a^2|^{-1} |a^{-1}b - a^{-2}cd|^{-1} = |ab - cd|^{-1} = |α^0 - α^1|^{-1}$$
and the proposition is proved.
\end{proof}
%

\subsection{Lattice Chains}

We now extend Proposition \ref{prop:count_lattice_gluing} to the lattice chains from Definitions \ref{def:lattice_chains} and \ref{def:lattice_chains_gamma} when $n \leq 2$.
\begin{defn}\label{def:lattice_chains_GL4}
Define the following sets of lattice chains.
\begin{enumerate}[wide, labelindent=0pt, labelwidth=!, label=(\arabic*), topsep=2pt, itemsep=2pt]
\item Let $\mcP$ be the set of chains of $O_K$-lattices in $V$ of the form
$$Λ_0 \supset (π, 1) Λ_0 \supset πΛ_0.$$
Here, $(π, 1)$ is meant as the element in $O_K$. Moreover, define
\begin{equation}\label{eq:lattice_chains_Par}
\mcP(γ) = \{Λ_\bullet \in \mcP \mid zΛ_i \subseteq Λ_{i+1}\text{ for $i = 0,1$}\}.
\end{equation}
By Lemma \ref{lem:lattice_combinatorics} (3), the set $\mcP(γ)$ agrees (up to notation) with the set defined in Definition \ref{def:lattice_chains_gamma} for $(n, \ell) = (2, 2)$.

\item Let $\mcL$ be the set of chains of $O_K$-lattices in $V$ of the form
$$
Λ_0\supset Λ_1\supset Λ_2\supset Λ_3 \supset πΛ_0
$$
and such that $O_K$ acts on $Λ_i/Λ_{i+1}$ via the first projection $O_K\to O_F$ if $i = 0,2$, resp. via the second projection if $i = 1,3$. Furthermore, let
\begin{equation}\label{eq:lattice_chains_Iw}
\mcL(γ) = \{Λ_\bullet \in \mcL \mid zΛ_i \subseteq Λ_{i+1}\text{ for all $i = 0,\ldots,3$}\}.
\end{equation}
By Lemma \ref{lem:lattice_combinatorics} (3), the set $\mcL(γ)$ is the set defined in Definition \ref{def:lattice_chains_gamma} for $(n, \ell) = (2, 4)$.

\item For $j = 0,1$, let $\mcL_+^j$ be the set of chains of $O_K$-lattices in $V^j$ of the form
$$
Λ^j_0\supset (π, 1) Λ^j_0 \supset πΛ^j_0.
$$
Let $\mcL_-^j$ be the set of chains of $O_K$-lattices in $V^j$ of the form
$$
Λ^j_0\supset (1, π) Λ^j_0 \supset πΛ^j_0.
$$
Denote by $z^0$ and $z^1$ the two components of $z$ and define
\begin{equation}\label{eq:lattice_chains_Levi}
\mcL^j_\pm(γ^j) = \{Λ^j_\bullet \in \mcL^j_\pm(γ^j) \mid z^jΛ^j_i\subseteq Λ^j_{i+1}\text{ for $i = 0,1$}\}.
\end{equation}
By Lemma \ref{lem:lattice_combinatorics} (3), the set $\mcL^j_+(γ^j)$ agrees (up to notation) with the set defined in Definition \ref{def:lattice_chains_gamma} for $(n, \ell) = (1, 2)$. The set $\mcL^j_-(γ^j)$ is a variant.
\end{enumerate}
\end{defn}

Next, let $Λ_\bullet$ lie in $\mcP$ or $\mcL$. Applying the map \eqref{eq:lattice_bijection_map} to each term, we construct a pair of lattice chains in $V^0$ and $V^1$ by
\begin{equation}\label{eq:lattice_chain_decomp}
Λ_\bullet^0 := Λ_\bullet \cap V^0 \quand Λ_\bullet^1 := (Λ_\bullet + V^0)/V^0.
\end{equation}
The situation is straightforward for $\mcP$: If $Λ_\bullet \in \mcP$, then in particular $Λ_1 = (π, 1) Λ_0$ and hence also $Λ_1^j = (π, 1)Λ^j_0$ for both $j = 0,1$. It follows that $(Λ_\bullet^0, Λ_\bullet^1) \in \mcL^0_+ \times \mcL^1_+$.

The situation is more subtle for $\mcL$: For each index $i = 0, \ldots, 3$, precisely one out of the following two possibilities occurs,
\begin{equation}\label{eq:lattice_chain_decomp_cases}
\begin{cases}
[Λ_i^0 : Λ_{i+1}^0] = 1 &\text{and}\quad Λ_i^1 = Λ_{i+1}^1\\
Λ_i^0 = Λ_{i+1}^0 &\text{and}\quad [Λ_i^1 : Λ_{i+1}^1] = 1.
\end{cases}
\end{equation}
We define the type of $Λ_\bullet$ as the vector $t(Λ_\bullet) \in \{0,1\}^4$ with $t(Λ_\bullet)_i = 0$ precisely if the first case occurs in \eqref{eq:lattice_chain_decomp_cases}. Since $Λ_4 = πΛ_0$ and since $Λ_0$, $Λ_0^0$ and $Λ_0^1$ are all free over $O_K$, the type $t(Λ_\bullet)$ can take the four values
\begin{equation}\label{eq:type}
(0, 0, 1, 1),\ (1, 1, 0, 0),\ (0, 1, 1, 0)\ \text{and}\ (1, 0, 0, 1).
\end{equation}
In particular, for each $Λ_\bullet \in \mcL$, each case in \eqref{eq:lattice_chain_decomp_cases} occurs precisely twice. So there is a natural way to view $Λ_\bullet^0$ and $Λ_\bullet^1$ as $2$-term lattice chains. With this indexing convention,
\begin{equation}\label{eq:lattice_chain_decomp_cases_II}
(Λ_\bullet^0, Λ_\bullet^1) \in \begin{cases}
\mcL_+^0\times \mcL_+^1 & \text{if $t(Λ_\bullet) = (0, 0, 1, 1)$ or $(1, 1, 0, 0)$}\\
\mcL_+^0\times \mcL_-^1 & \text{if $t(Λ_\bullet) = (0, 1, 1, 0)$}\\
\mcL_-^0\times \mcL_+^1 & \text{if $t(Λ_\bullet) = (1, 0, 0, 1)$}.
\end{cases}
\end{equation}

\begin{lem}\label{lem:lattice_projection}
The map in \eqref{eq:lattice_chain_decomp} restricts to a surjection
\begin{equation}\label{eq:surjection_Par}
\mcP(γ)\ \relbar\joinrel\twoheadrightarrow\ \mcL_+^0(γ^0) \times \mcL_+^1(γ^1)
\end{equation}
all of whose fibers have cardinality $q^{-1}|α^0 - α^1|^{-1}$. Similarly, the map in \eqref{eq:lattice_chain_decomp_cases_II} restricts to a surjection
\begin{equation}\label{eq:surjection_Iw}
\mcL(y)\ \relbar\joinrel\twoheadrightarrow\ \mcL_+^0(γ^0)\times \mcL_+^1(γ^1) \ \sqcup\ \mcL_-^0(γ^0)\times \mcL_+^1(γ^1)\ \sqcup\ \mcL_+^0(γ^0)\times \mcL_-^1(γ^1)
\end{equation}
such that fibers over $\mcL_+^0(γ^0)\times \mcL_+^1(γ^1)$ have cardinality $2\,|α^0-α^1|^{-1}$ and fibers over its complement have cardinality $|α^0-α^1|^{-1}$. Moreover, both \eqref{eq:surjection_Par} and \eqref{eq:surjection_Iw} commute with the action of $L_γ^\times$.
\end{lem}
\begin{proof}
It is clear that if $Λ_\bullet$ is $z$-stable in the sense of \eqref{eq:lattice_chains_Par} or \eqref{eq:lattice_chains_Iw}, then $Λ_\bullet^j$ is $z^j$-stable in the sense of \eqref{eq:lattice_chains_Levi}. In other words, the two maps \eqref{eq:surjection_Par} and \eqref{eq:surjection_Iw} are defined as claimed. Moreover, the inclusion and projection maps $V^0\hookrightarrow V \twoheadrightarrow V/V^0$ are $L_γ$-linear by definition, so it is clear that both maps \eqref{eq:surjection_Par} and \eqref{eq:surjection_Iw} commute with the $L_γ^\times$-action. Our main task is to prove the claims on their fiber cardinalities. We will assume $v(α^0), v(α^1)> 0$ for this because otherwise both the sources and targets in \eqref{eq:surjection_Par} and \eqref{eq:surjection_Iw} are empty.

We begin with the case of $\mcP(γ)$. Define the auxiliary operator $\wt{z} = (π,1)^{-1}z$. Then the pairs $(Λ^0_\bullet, Λ^1_\bullet) \in \mcL^0_+(γ^0)\times \mcL^1_+(γ^1)$ are in bijection with pairs of $O_K[\wt{z}]$-lattices $(Λ^0_0, Λ^1_0)$ in $V^0$ and $V^1$. (Indeed, an $O_K[\wt z]$-lattice $Λ^j_0$ in $V^j$ can be uniquely extended to the lattice chain $(Λ^j_0, (π, 1)Λ^j_0)\in  \mcL^j_+(γ^j)$.) Note that $\wt{z}^2 = π^{-1}z^2$, so the eigenvalues of $\wt{z}^2$ are $π^{-1}α^0$ and $π^{-1}α^1$. By Proposition \ref{prop:count_lattice_gluing}, there are $q^{-1}|α^0 - α^1|^{-1}$ many $O_K[\wt z]$-stable lattices $Λ_0$ such that
$$Λ_0\cap V^0 = Λ_0^0\quand (Λ_0\cap V^0)/V^0 = Λ_0^1.$$
For any of these possibilities, $Λ_0 \supset (π, 1) Λ_0 \supset πΛ_0$ defines a unique extension to an element $Λ_\bullet \in \mcP(γ)$. This $Λ_\bullet$ is then a preimage of $(Λ_\bullet^0, Λ_\bullet^1)$ and all claims about \eqref{eq:surjection_Par} are proved.

Now we turn to $\mcL(γ)$. Assume we are given a pair $(Λ_\bullet^0, Λ_\bullet^1)$ in the right hand side of \eqref{eq:surjection_Iw} as well as a type $t\in \{0,1\}^4$ that is compatible with the pair in the sense of \eqref{eq:lattice_chain_decomp_cases_II}. We claim that there are $|α^0-α^1|^{-1}$ many lattice chains $Λ_\bullet\in \mcL(γ)$ of type $t$ that map to $(Λ_\bullet^0, Λ_\bullet^1)$ under \eqref{eq:surjection_Iw}. We first prove this for the type $t = (0, 0, 1, 1)$. By \eqref{eq:inclusion_criterion}, the set of four term $O_K[z]$-lattice chains $Λ_\bullet$ such that $Λ_\bullet \cap V^0 = Λ_\bullet^0$ and $(Λ_\bullet + V^0)/V^0 = Λ^1_\bullet$ is in bijection with the set of tuples $(s_0, s_1, s_2, s_3)$ of $z$-linear maps that give rise to a commutative ladder of the form
\begin{equation}\label{eq:lattice_gluing_ladder}
\begin{minipage}{\textwidth}
\xymatrix{
Λ^1_0 \ar[d]_{s_0} & Λ^1_0 \ar@{=}[l] \ar[d]_{s_1} & Λ^1_0 \ar@{=}[l] \ar[d]_{s_2} & Λ^1_1 \ar@{_{(}->}[l] \ar[d]_{s_3} & πΛ^1_0 \ar@{_{(}->}[l] \ar[d]_{s_0}\\
V^0/Λ^0_0 & V^0/Λ^0_1 \ar@{->>}[l] & V^0/(πΛ^0_0) \ar@{->>}[l] & V^0/(πΛ^0_0) \ar@{=}[l] & V^0/(πΛ^0_0). \ar@{=}[l]
}
\end{minipage}
\end{equation}
By Lemma \ref{lem:inclusion_triviality} (or by direct inspection), these tuples are in bijection with just the datum of the $O_K[z]$-linear map $s_2$. By Proposition \ref{prop:count_lattice_gluing}, there are precisely $|α^0-α^1|^{-1}$ many of those. Moreover, the corresponding chain $Λ_\bullet$ necessarily lies in $\mcL(γ)$ because for every $i$, there are $j$ and $k$ with $Λ_i/Λ_{i+1} \iso Λ^j_k/Λ^j_{k+1}$. So the condition $zΛ_i/Λ_{i+1} = 0$ follows from the assumption that the same kind of condition holds for $Λ^0_\bullet$ and $Λ^1_\bullet$. Our claim is now proved for $t = (0, 0, 1, 1)$. For the other three possibilities for $t$, the argument is completely identical. Namely, all of the involved lattice chains may be extended (uniquely) to a $π^\mbZ$-periodic chain and the four possibilities for the type $t$ differ by rotation permutations, see \eqref{eq:type}.
\end{proof}

\subsection{Orbital integrals}

We next define suitably normalized orbital integrals on $GL(V^0)$ and $GL(V^1)$. Choose $K$-bases $V^j\iso F\times F$ such that
$$O_F^4 \cap V^0 = O_F^2\quad \text{and}\quad (O_F^4 + V^0)/V^0 = O_F^2.$$
Assume that $Λ^j\subset V^j$ is a lattice such that both $Λ^j$ and $γ^jΛ^j$ are $O_K$-stable. Specializing the definition in \eqref{eq:transfer_lattice} to this situation, we have defined
$$\Omega(γ^j, Λ^j, s) = \Omega((h_1^j)^{-1}γh_2^j, s)$$
where $h_1^j$, $h_2^j\in K^\times$ are such that $h_2^j O_F^2 = Λ^j$ and $h_1^j O_F^2 = γ^jΛ^j$. It is immediate that whenever $(Λ^0, Λ^1) = (Λ\cap V^0, (Λ+V^0)/V^0)$ for some $O_K$-lattice $Λ\subset F^4$ such that also $γΛ$ is $O_K$-stable, then
\begin{equation}\label{eq:rel_omega_Levi}
\Omega(γ, Λ, s) = \Omega(γ^0, Λ^0, s)\Omega(γ^1, Λ^1, s).
\end{equation}
Consider the two standard parahoric subgroups of $GL_2(F)$,
\begin{equation}
I_+ = \begin{pmatrix}
O^\times_F & (π) \\ O_F & O_F^\times
\end{pmatrix} \quand I_- = \begin{pmatrix}
O^\times_F & O_F \\ (π) & O_F^\times
\end{pmatrix}.
\end{equation}
Let $ϕ_\pm$ denote the indicator function of $I_\pm$. These two functions are related by the conjugation $I_- = \diag(π, 1)^{-1}\cdot I_+\cdot \diag(π, 1)$ so their orbital integrals satisfy
\begin{equation}\label{eq:rel_phi_pm}
\Orb(γ', ϕ_-, s) = (-q^{-2s}) \Orb(γ', ϕ_+, s).
\end{equation}
Here, $γ'\in GL_2(F)_{\mr{rs}}$ denotes a regular semi-simple element. Let us write $\mcL_\pm(γ')$ for the analog of \eqref{eq:lattice_chains_Levi} for the vector space $F^2$. The combinatorial interpretation of orbital integrals \eqref{eq:orb_int_combinatorial} specializes to
\begin{equation}\label{eq:orb_int_phi}
\Orb(γ', ϕ_\pm, s) = \sum_{Λ_\bullet \in F^\times \backslash \mcL_\pm(γ')} \Omega(γ', Λ_0, s).
\end{equation}

\begin{prop}\label{prop:orb_int_n_equal_1}
Let $γ'\in GL_2(F)_{\mr{rs}}$ be a regular semi-simple element of invariant $T-α$. Put $X = -q^{-2s}$. Then the orbital integrals of $ϕ_+$ is given by
\begin{equation}\label{eq:GL2_par_int}
\Orb(γ', ϕ_+, s) = \begin{cases} 0 & \text{if $v(α) \leq 0$}\\
|α|^{-s} \frac{X^{v(α)} - 1}{X - 1} & \text{if $v(α) > 0$}.
\end{cases}
\end{equation}
\end{prop}
\begin{proof}
The orbital integral only depends on the orbit of $γ'$, so we may choose $γ' = \left(\begin{smallmatrix} 1 & α \\ 1 & 1\end{smallmatrix}\right)$ for the computation. Let $h_1 = \diag(a, 1)$ and $h_2 = \diag(c, d)$. Then
$$h_1^{-1} γ' h_2 \in I_+ \quad \Longleftrightarrow \quad v(d) = 0,\ v(c) \geq 0,\ v(a) = v(c),\ v(α) > v(a).
$$
Assuming these equivalent conditions are met, the transfer factor is given by $\Omega(γ', s) = q^{v(α)s}$ and the character that occurs in the integrand is
$$|h_1h_2|^s η(h_2) = (-1)^{v(c)}q^{-2v(c)s}.$$
A direct evaluation of the definition in \eqref{eq:def_orb_int_GL_omega} now gives
$$\Orb(γ', ϕ_+, s) = q^{v(α)s} \sum_{i = 0}^{v(α)-1} (-q^{-2s})^i$$
as was to be shown.
\end{proof}

\begin{prop}\label{prop:hyperbolic}
Let $γ\in G'_{\mr{rs}}$ be regular semi-simple and hyperbolic with $\Inv(γ; T) = (T-α)(T-β)$. Set $X = -q^{-2s}$.
\begin{enumerate}[wide, labelindent=0pt, labelwidth=!, label=(\arabic*), topsep=2pt, itemsep=2pt]
\item The parahoric orbital integral $\Orb(γ, f'_\Par, s)$ vanishes if $v(α)\leq 0$ or $v(β)\leq 0$. Otherwise, it is given by
\begin{equation}\label{eq:Par_orb_int_hyper}
\Orb(γ, f'_\Par, s)\, =\, q^{-1} |α - β|^{-1}\ |αβ|^{-s}\ \frac{(X^{v(α)} - 1)(X^{v(β)} - 1)}{(X-1)^2}.
\end{equation}
\item The Iwahori orbital integral $\Orb(γ, f'_\Iw, s)$ vanishes if $v(α)\leq 0$ or $v(β)\leq 0$. Otherwise, it is given by
\begin{equation}\label{eq:Iw_orb_int_hyper}
\Orb(γ, f'_\Iw, s)\, =\, 2q\,(X+1)\Orb(γ, f'_\Par, s).
\end{equation}
\end{enumerate}
\end{prop}
\begin{proof}
If $v(α)\leq 0$ or $v(β)\leq 0$, then $z = z_γ$ is not topologically nilpotent. It follows that $\mcL(γ) = \emptyset$ and $\mcP(γ) = \emptyset$, so both orbital integrals vanish by \eqref{eq:orb_int_combinatorial}.

Now assume that $v(α),v(β)> 0$. We first deal with the function $f'_\Par$. Let $Γ\subset L_γ^\times$ be the subgroup $(π,1)^\mbZ \times (1, π)^\mbZ$. It has the property that $\mr{vol}(L_γ^\times /Γ) = 1$, so we may rewrite the combinatorial description of the orbital integral \eqref{eq:orb_int_combinatorial} as
$$\Orb(γ, f'_\Par, s) = \sum_{Λ_\bullet \in Γ\backslash \mcP(γ)} \Omega(γ, Λ_0, s).$$
By \eqref{eq:surjection_Par} and \eqref{eq:rel_omega_Levi}, this equals
$$\sum_{(Λ_\bullet^0 \times Λ_\bullet^1) \in \left(F^\times \backslash \mcL^0_+(γ^0)\right)\times \left(F^\times \backslash \mcL^1_+(γ^1)\right)}
\,q^{-1}|α-β|^{-1}\ \Omega\left(γ^0, Λ_0^0, s\right)\ \Omega\left(γ^1, Λ_0^1, s\right)$$
which coincides with $q^{-1}|α-β|^{-1}\Orb(γ^0, ϕ_+, s)\Orb(γ^1, ϕ_+, s)$ by \eqref{eq:orb_int_phi}. Substituting \eqref{eq:GL2_par_int} yields \eqref{eq:Par_orb_int_hyper}.

Now we turn to the test function $f'_\Iw$. Arguing as before and using \eqref{eq:orb_int_combinatorial}, we obtain the combinatorial formula
$$\Orb(γ, f'_\Iw, s) = \sum_{Λ_\bullet \in Γ\backslash \mcL(γ)} \Omega(γ, Λ_0, s).$$
By \eqref{eq:surjection_Iw} and \eqref{eq:rel_omega_Levi}, this expression may be rewritten as the following sum of three terms:
$$\begin{aligned}
\phantom{+} & \sum_{
(Λ_\bullet^0 \times Λ_\bullet^1) \in \left(F^\times \backslash \mcL^0_+(γ^0)\right)\times \left(F^\times \backslash \mcL^1_+(γ^1)\right)}
2\,|α-β|^{-1}\ \Omega\left(γ^0, Λ_0^0, s\right)\ \Omega\left(γ^1, Λ_0^1, s\right)
\\
 + & \sum_{
(Λ_\bullet^0 \times Λ_\bullet^1) \in \left(F^\times \backslash \mcL^0_+(γ^0)\right)\times \left(F^\times \backslash \mcL^1_-(γ^1)\right)}
\phantom{2\,}|α-β|^{-1}\ \Omega\left(γ^0, Λ_0^0, s\right)\ \Omega\left(γ^1, Λ_0^1, s\right)
\\
+ & \sum_{
(Λ_\bullet^0 \times Λ_\bullet^1) \in \left(F^\times \backslash \mcL^0_-(γ^0)\right)\times \left(F^\times \backslash \mcL^1_+(γ^1)\right)}
\phantom{2\,}|α-β|^{-1}\ \Omega\left(γ^0, Λ_0^0, s\right)\ \Omega\left(γ^1, Λ_0^1, s\right).
\end{aligned}$$
So we obtain from \eqref{eq:rel_phi_pm} and \eqref{eq:orb_int_phi} that
$$\Orb(γ, f'_\Iw, s) = 2\,|α-β|^{-1}\ (X+1)\, \Orb(γ^0, ϕ_+, s)\, \Orb(γ^1, ϕ_+, s).$$
Substituting \eqref{eq:GL2_par_int} again and comparing the result with the formula for $\Orb(γ, f'_\Par, s)$ yields \eqref{eq:Iw_orb_int_hyper}.
\end{proof}

\section{Germ Expansion}

In this section, we prove a germ expansion principle for the parahoric and Iwahori orbital integrals. Together with the results from \S\ref{s:hyperbolic}, this will allow us to compute $\Orb(γ, f'_\Par, s)$ and $\Orb(γ, f'_\Iw, s)$ in all cases. We will first focus on the case of $f'_\Iw$; the case of $f'_\Par$ is much simpler and will be treated in \S\ref{ss:parahoric}.

\subsection{Simplified lattice counting}
\label{ss:simplified_lattices}
Our first aim is to give a more concrete description of the set $\mcL(γ)$ from Definition \ref{def:lattice_chains_GL4}. (Note that its definition does not require $γ$ to be hyperbolic.) Throughout, $w\in GL_2(F)$ denotes an element such that $γ(w) := \left(\begin{smallmatrix} 1 & 1 \\ w & 1 \end{smallmatrix}\right)$ lies in $G'_{\mr{rs}}$. In other words, we assume that $L := F[w]$ is a quadratic étale extension of $F$ and that $w, 1-w \in L^\times$. In this situation, $z_{γ(w)}^2 = \diag(w,w)$, so $L_{γ(w)}$ can be identified with $L$. More precisely, $L_{γ(w)}$ equals the image of the diagonal embedding $L\to M_2(F)\times M_2(F) \subset M_4(F)$.

\begin{defn}\label{def:lc_simplified}
Let $\mcL(w)$ be the set of quadruples $(Λ_0, Λ_0^\flat, Λ_1, Λ_1^\flat)$ of $O_F$-lattices in $L$ that have the following three properties.
\begin{enumerate}[wide, labelindent=0pt, labelwidth=!, label=(\arabic*), topsep=2pt, itemsep=2pt]
\item It holds that $Λ_0^\flat \subset Λ_0$ and $Λ_1^\flat \subset Λ_1$, and each of these two inclusions is of index $1$.
\item It holds that $O_L\cdot Λ_0 = O_L$.
\item The four lattices fit into the diagram
\begin{equation}\label{eq:lc_simplified}
\begin{array}{ccccccc}
Λ_0 & \supset & Λ_0^\flat & \supseteq & Λ_1       & \supseteq & wΛ_0\\
    &         & \cup      &           & \cup      &           & \cup\\
    &         & πΛ_0      & \supseteq & Λ_1^\flat & \supseteq & wΛ_0^\flat.
\end{array}
\end{equation}
\end{enumerate}
\end{defn}
For the next statement, choose a free discrete subgroup $Γ\subset L^\times$ such that $\mr{vol}(L^\times/Γ) = 1$. Concretely, take $Γ = \varpi^{\mbZ}$ with some uniformizer $\varpi \in L$ if $L$ is a field or $Γ \iso (\varpi_1^\mbZ, \varpi_2^\mbZ)$ for two uniformizers $\varpi_1, \varpi_2\in F$ if $L\iso F\times F$.

Also let $(L, r, d)$ be the numerical triple of $γ(w)$ from \ref{def:orbit_invariants}: The quadratic $F$-algebra $L = F[w]$ is already given while
\begin{equation}\label{eq:invariants_w}
r = v(N_{L/F}(w))\quand d = \mr{cond}(O_F[π^kw]) - k - r/2,\quad k\gg 0.
\end{equation}
\begin{lem}\label{lem:lc_bijection}
Given $(Λ_0, Λ_0^\flat, Λ_1, Λ_1^\flat)\in \mcL(w)$, the following lattice chain lies in $\mcL(γ(w))$:
$$Λ_0\oplus Λ_1\ \supset\ Λ_0^\flat \oplus Λ_1\ \supset\ Λ_0^\flat \oplus Λ_1^\flat\ \supset\ πΛ_0 \oplus Λ_1^\flat\ \supset\ π(Λ_0 \oplus Λ_1).$$
This assignment defines a bijection
\begin{equation}\label{eq:lc_bijection}
ψ:\mcL(w) \overset{\sim}{\lr} Γ\backslash \mcL(γ(w))
\end{equation}
which has the property that
\begin{equation}\label{eq:lc_bijection_omega}
\Omega(γ(w),\ ψ(Λ_0, Λ_0^\flat, Λ_1, Λ_1^\flat),\ s) = (-q^s)^{-r(w)} (-q^{2s})^{[Λ_0:Λ_1]}.
\end{equation}
In particular, the Iwahori orbital integral has the expressions
\begin{equation}\label{eq:lc_orb_int}
\Orb(γ(w), f'_\Iw, s) \ =\ (-q^s)^{-r(w)} \sum_{(Λ_0, Λ_0^\flat, Λ_1, Λ_1^\flat)\in \mcL(w)} (-q^{2s})^{[Λ_0:Λ_1]}.
\end{equation}
\end{lem}
\begin{proof}
The well-definedness and bijectivity of \eqref{eq:lc_bijection} follows directly from definitions. For the description of the transfer factor \eqref{eq:lc_bijection_omega}, we first note that $\mcL(w)\neq \emptyset$ implies that $w$ and hence $z_γ$ are topologically nilpotent. The element $γ = γ(w)$ has the form $γ = 1 + z_γ$, so if $Λ\subset F^4$ is both $O_K$-stable and $γ^{-1}O_Kγ$-stable, then already $γΛ = Λ$ by Lemma \ref{lem:lattice_combinatorics}. In this case, \eqref{eq:translation_omega} simplifies to
$$\Omega(γ, Λ, s) = (-1)^{[Λ_- : z Λ_+]} q^{([Λ_+:zΛ_-] - [Λ_-: zΛ_+])s}.$$
Then \eqref{eq:lc_bijection_omega} is obtained from evaluating this with $Λ_+ = Λ_0\oplus 0$ and $\,Λ_- = 0 \oplus Λ_1$ with $z = \left(\begin{smallmatrix} & 1\\ w & \end{smallmatrix}\right)$. Finally, \eqref{eq:lc_orb_int} follows from the previous two statements and \eqref{eq:orb_int_combinatorial}.
\end{proof}

\subsection{The Germ Expansion}
\label{ss:germ_exp}

Any $O_F$-lattice $Λ\subset L$ defines an order $O_Λ = \{x\in L \mid xΛ\subseteq Λ\}$. We define the conductor of $Λ$ as the conductor of its order,
$$\mr{cond}(Λ) := \mr{cond}(O_Λ).$$

\begin{defn}\label{def:germs}
We define the principal germ orbital integral of $w$ as
\begin{equation}\label{eq:principal_germ}
P(w, s) := \sum_{(Λ_0, Λ_0^\flat, Λ_1, Λ_1^\flat) \in \mcL(w),\ \mr{cond}(Λ_0)\ = \ \mr{cond}(Λ_0^\flat)\ =\ 0} (-q^{2s})^{[Λ_0:Λ_1]}.
\end{equation}
Note that $\mr{cond}(Λ_0) = 0$ just says $Λ_0 = O_L$. Let $i(L)$ denote the index $[O_L^\times : (O_F + πO_L)^\times]$. Define the unipotent germ orbital integral as
\begin{equation}\label{eq:unipotent_germ}
U(w, s) := i(L)^{-1}\sum_{(Λ_0, Λ_0^\flat, Λ_1, Λ_1^\flat) \in \mcL(w),\ (\mr{cond}(Λ_0), \mr{cond}(Λ_0^\flat))\ \neq\ (0,0)} (-q^{2s})^{[Λ_0:Λ_1]}.
\end{equation}
\end{defn}
\begin{prop}\label{prop:germ_exp}
The Iwahori orbital integral, the principal germ, and the unipotent germ are related by the following identity. For every $w \in GL_2(F)$ such that $γ(w)\in G'_{\mr{rs}}$,
\begin{equation}\label{eq:germ_exp}
\Orb(γ(w), f'_\Iw, s) = (-q^s)^{-r(w)} \left[P(w,s) + i(L) U(w,s)\right].
\end{equation}
\end{prop}
\begin{proof}
This follows directly from the combinatorial description in \eqref{eq:lc_orb_int} and the definition of the two germs.
\end{proof}

\begin{prop}\label{prop:germ_independence}
The principal germ $P(w,s)$ depends only on the triple $(L, r(w), d(w))$. The unipotent germ is independent of $L$ and only depends on the pair $(r(w), d(w))$.
\end{prop}

\begin{proof}
The claim about the principal germ will be proved as part of Proposition \ref{prop:principal_germ} below. We only deal with the unipotent germ here which relies on the following lemmas.

\begin{lem}[Classification of numerical invariants]\label{lem:std_form_r_d} Let $r = r(w)$ and $d = d(w)$ in the following.
\begin{enumerate}[wide, labelindent=0pt, labelwidth=!, label=(\arabic*), topsep=2pt, itemsep=2pt]
\item Assume that $d \geq 0$ and let $ζ \in O_L$ be an $O_F$-algebra generator. Then $r$ is even, $d\in \mbZ$, and $π^{-r/2}w \in O_L^\times$. Moreover, $w$ is of the form $π^{r/2}(a + π^dbζ)$ for suitable elements $a\in O_F$ and $b\in O_F^\times$.

\item Assume that $d < 0$. Then $L$ is either ramified or split.

\item Assume that $d < 0$ and that $L$ is ramified. Then $r$ is odd, $d = -1/2$ and $π^{(r-1)/2}w$ is a uniformizer of $L$.

\item Assume that $d < 0$ and that $L = F\times F$. Then $w = (x, y)$ for two element $x,y\in F^\times$ such that
$$v(x) + v(y) = r\quand |v(x) - v(y)| = -2d.$$
\end{enumerate}
\end{lem}
\begin{proof}
This is proved by an easy case-by-case analysis.

(a) Assume that $L/F$ is an unramified field extension. Then $r = 2v_L(w)$ is even. The element $w_0 = π^{-r/2}w$ lies in $O_L^\times$. By definition, see \eqref{eq:invariants_w}, $d = \mr{cond}(O_F[w_0])$ is $\geq 0$. Moreover, given a generator $O_L = O_F[ζ]$, there obviously are $a, b\in O_F^\times$ such that $w_0 = a + π^dbζ$. This proves (2), as well as (1) whenever $L$ is an unramified field extension.

(b1) Assume that $L/F$ is a ramified field extension. Then $r = v_L(w)$. First assume that $r$ is even. Then $w_0 = π^{-r/2}w$ lies in $O_L^\times$ and, just as in (a), $d = \mr{cond}(O_F[w_0])$ is $\geq 0$. Given a generator $O_L = O_F[ζ]$, it is again clear that there are $a, b\in O_F^\times$ such that $w_0 = a + π^dbζ$. This shows (1) whenever $L$ is ramified.

(b2) Now assume $r$ is odd. Then $w = π^{(r-1)/2}\varpi$ for a uniformizer $\varpi \in L$. We find that
$$\begin{aligned}
d & = \mr{cond}(O_F[π^{-(r-1)/2}w]) + (r-1)/2 - r/2\\
& = \mr{cond}(O_L) - 1/2 = -1/2.
\end{aligned}$$
This proves (3).

(c1) Now assume that $L = F\times F$ and $w = (x,y)$ with $v(x) \leq v(y)$. The identity $r = v(x) + v(y)$ is clear from definition. Set $w_0 = π^{-v(x)}w$. Then
\begin{equation}\label{eq:r_d_split}
\begin{aligned}
d & = \mr{cond}(O_F[w_0]) + v(x) - (v(x) + v(y))/2\\
& = \mr{cond}(O_F[w_0]) + (v(x) - v(y))/2.
\end{aligned}
\end{equation}
Assume that $v(x) = v(y)$. Then we get $d = \mr{cond}(O_F[w_0]) \geq 0$. Given a generator $ζ\in O_L$, there are $a,b\in O_F^\times$ such that $w_0 = a + π^dbζ$. Since $v(x) = r/2$, this shows (1).

(c2) Finally, assume that $v(x) < v(y)$. Then $O_F[w_0] = O_L$, so \eqref{eq:r_d_split} gives $-2d = v(y) - v(x)$ which shows (4).
\end{proof}

\begin{lem}\label{lem:changing_L}
Let $L_1$ and $L_2$ be two quadratic étale $F$-algebras. Let $R_{i,c} = O_F + π^cO_{L_i}$ denote the order of conductor $c$ in $L_i$. Let furthermore $w_1\in L_1^\times \setminus F$ and $w_2\in L_2^\times\setminus F$ be elements with
$$r(w_1) = r(w_2)\quand d(w_1) = d(w_2).$$
Then there exists an $F$-linear map $ϕ:L_1\to L_2$ that has the following property. For every $c\geq 0$, both
$$ϕ(R_{1,c}) = R_{2,c}\quand ϕ(w_1R_{1,c}) = w_2R_{2,c}.$$
\end{lem}
\begin{proof}
Given $x,y\in O_F^\times$, there is a unique $F$-linear map $ϕ_{x,y}:L_1\to L_2$ such that
$$ϕ_{x,y}(1) = x\quand ϕ_{x,y}(w_1) = yw_2.$$
We will show that there are $x,y\in O_F$ such that $ϕ = ϕ_{x,y}$ has the properties claimed in the lemma. (In fact, $ϕ$ is necessarily of such a form.) Note that for all $x,y$ as above, $ϕ_{x,y}(O_F) = O_F$ and $ϕ_{x,y}(w_1O_F) = w_2O_F$. Since $R_{i,c} = O_F + π^cO_{L_i}$, our task is thus to find $x$ and $y$ such that also $ϕ_{x,y}(O_{L_1}) = O_{L_2}$ and $ϕ_{x,y}(w_1 O_{L_1}) = w_2 O_{L_2}$. Put $r = r(w_1)$ and $d = d(w_1)$ in the following.

\emph{First assume that $d \geq 0$.} Let $ζ_i\in O_{L_i}$ be two $O_F$-algebra generators. Using Lemma \ref{lem:std_form_r_d} (1), there are $a_i\in O_F$ and $b_i\in O_F^\times$ such that
$$w_1 = π^{r/2}(a_1 + π^d b_1 ζ_1),\quad w_2 = π^{r/2}(a_2 + π^d b_2 ζ_2).$$
Consider first the case $d = 0$. Then we claim that $ϕ = ϕ_{1,1}$ satisfies the assertion of the lemma. Indeed,
$$ϕ(ζ_1) = b_1^{-1}(π^{-r/2} w_2 - a_1) = b_1^{-1}(a_2 - a_1) + b_1^{-1}b_2 ζ_2$$
is an $O_F$-algebra generator of $O_{L_2}$. Furthermore, $w_iO_{L_i} = π^{r/2}O_{L_i}$ because $d\geq 0$, so it also follows that $ϕ(w_1O_{L_1}) = w_2O_{L_2}$.

Consider now the case $d > 0$. Then $a_1, a_2\in O_F^\times$ and we claim that $ϕ = ϕ_{x,y}$ with $x = a_1^{-1}a_2$ and $y = 1$ satisfies the assertion of the lemma. Indeed, we obtain
$$ϕ(ζ_1) = π^{-d}b_1^{-1}(π^{-r/2} w_2 - a_2) = b_1^{-1}b_2 ζ_2$$
which is again an $O_F$-algebra generator of $O_{L_2}$. Hence $ϕ(O_{L_1}) = O_{L_2}$ and the desired statement $ϕ(w_1O_{L_1}) = w_2O_{L_2}$ is obtained just as in the previous case. This proves the lemma in case $d \geq 0$.

\emph{Now assume that $d < 0$.} No matter which of the two cases (3) and (4) of Lemma \ref{lem:std_form_r_d} occur for $w_1$ and $w_2$, it always holds that $ζ_i = π^{-r/2-d}w_i$ is an $O_F$-algebra generator of $O_{L_i}$. So take $ϕ = ϕ_{1,1}$. Then $ϕ(ζ_1) = ζ_2$ and hence $ϕ(O_{L_1}) = O_{L_2}$. Furthermore, the element $π^{-2d}/ζ_i$ is again an $O_F$-algebra generator of $O_{L_i}$, see Lemma \ref{lem:std_form_r_d} (4). We obtain
$$ϕ(w_1O_{L_1}) = π^{r/2 + d} ϕ(ζ_1 (O_F + π^{-2d}/ζ_1 O_F)) = π^{r/2+d} (ζ_2O_F + π^{-2d}O_F) = w_2O_{L_2}$$
as desired and the proof of the lemma is complete.
\end{proof}

We now come to the main part of the proof of Proposition \ref{prop:germ_independence}. Let $\mcB$ be the $PGL$-building for $L$ viewed as $F$-vector space. Concretely, $\mcB$ is the graph (a tree in fact) with the following description. Its vertices are in bijection with $O_F$-lattices $Λ\subset O_L$ such that $O_L/Λ$ is a cyclic $O_F$-module. Its edges are given by unordered pairs $\{Λ, Λ'\}$ such that one lattice is a sublattice of index $1$ of the other. We write $d(Λ, Λ')$ for the distance between two vertices $Λ$ and $Λ'$. Consider the subtree $\mcE \subset \mcB$ that is spanned by the vertices
\begin{equation}\label{eq:def_mcE}
\mcE = \{ Λ\in \mcB \mid O_L\cdot Λ = O_L\}.
\end{equation}
(An equivalent description is as follows: Let $\mcC = \{Λ\in \mcB\mid O_L\cdot Λ = Λ\}$ be the subtree spanned by all $O_L$-lattices. It equals $\{O_L\}$, the edge $\{O_L, \varpi O_L\}$ or the apartment $\{(π^k, 1)O_L,\, (1, π^k)O_L \mid k \geq 0\}$, depending on whether $L$ is inert, ramified with uniformizer $\varpi$, or isomorphic to $F\times F$. Then $\mcE$ consists of all points of $\mcB$ whose shortest path to $O_L$ does not contain any other point of $\mcC$.)
Let $\mcE_c = \{Λ\in \mcE \mid d(O_L, Λ) = c\}$. The following statements hold in this situation:
\begin{enumerate}[wide, labelindent=0pt, labelwidth=!, label=(\arabic*), topsep=2pt, itemsep=2pt]
\item For any vertex $Λ\in \mcE$, we have that $\mr{cond}(Λ) = d(O_L, Λ)$.
\item The group $O_L^\times$ acts transitively on $\mcE_c$ with stabilizer $R_c^\times$. Here, $R_c = O_F + π^cO_L$ again denotes the order of conductor $c$. In particular, $\#\mcE_c = [O_L^\times : R_c^\times]$ which equals $i(L) q^{c-1}$ if $c\geq 1$.
\item Since $\mcE$ is a tree, $O_L^\times$ then also acts transitively on the set of edges of $\mcE$ with distance $c$ from $O_L$. In particular, every edge of distance $c$ from $O_L$ is an $O_L^\times$-translate of the edge $\{R_c, R_{c+1}\}$.
\end{enumerate}
Consider now a pair $(Λ_0, Λ_0^\flat)$ of lattices in $L$ with $Λ_0\in \mcE$ and such that $Λ_0^\flat \subset Λ_0$ with index $1$. Assume that $\mr{cond}(Λ_0, Λ_0^\flat) \neq (0, 0)$. This is equivalent to $\{Λ_0, Λ_0^\flat\}\not\subset \mcC$ which means that $\{Λ_0, Λ_0^\flat\}$ is an edge of $\mcE$. In particular, $Λ_0$ and $Λ_0^\flat$ have different conductor. So if $c = \min\{\mr{cond}(Λ_0), \mr{cond}(Λ_0^\flat)\}$, then
\begin{equation}\label{eq:conductors}
\mr{cond}(Λ_0, Λ_0^\flat) \in \{ (c, c+1),\ (c+1, c)\}.
\end{equation}
In the first case, $Λ_0^\flat$ lies in $\mcB$. In the second case, it is instead $π^{-1}Λ_0^\flat$ that lies in $\mcB$. Depending on which case occurs, $\{Λ_0, Λ_0^\flat\}$ or $\{Λ_0, π^{-1}Λ_0^\flat\}$ defines an edge of $\mcB$ that lies in $\mcE$ (because $Λ_0\in \mcE$) and that has distance $c$ from $O_L$.

For a fixed pair $(Λ_0, Λ_0^\flat)$, let $\mcL(w; Λ_0, Λ_0^\flat)$ denote the set of quadruples $(Λ_0, Λ_0^\flat, Λ_1, Λ_1^\flat)\in \mcL(w)$. Assuming that $Λ_0^\flat \subseteq πO_L$, there is the symmetry
$$\begin{aligned}
\mcL(w; Λ_0, Λ_0^\flat) &\ \overset{\sim}{\lr} \ \mcL(w; π^{-1}Λ_0^\flat, Λ_0)\\
(Λ_0, Λ_0^\flat, Λ_1, Λ_1^\flat) &\ \longmapsto\ (π^{-1}Λ_0^\flat, Λ_0, π^{-1}Λ_1^\flat, Λ_1).
\end{aligned}$$
We deduce that no matter which case occurs in \eqref{eq:conductors}, $\#\mcL(w; Λ_0, Λ_0^\flat) = \#\mcL(w; R_c, R_{c+1}).$
It follows that we can rewrite the definition of the unipotent germ in \eqref{eq:unipotent_germ} as
\begin{equation}\label{eq:unipotent_during_proof}
U(w, s) \ =\ \ 2\ \sum_{c\geq 0} \ \ q^c \sum_{(R_c, R_{c+1}, Λ_1, Λ_1^\flat) \in \mcL(w)} (-q^{2s})^{[R_c:Λ_1]}.
\end{equation}
It now follows from Lemma \ref{lem:changing_L} that the outer diagram in \eqref{eq:lc_simplified},
$$
\begin{array}{ccccccc}
R_c & \supset & R_{c+1}  &\supseteq & \ldots & \supseteq & wR_c\\
    &         & \cup     & &  &           & \cup\\
    &         & πR_c     &\supseteq & \ldots & \supseteq & wR_{c+1},
\end{array}
$$
only depends on the invariants $r(w)$ and $d(w)$ up to $F$-linear isomorphism. We conclude that the Expression \eqref{eq:unipotent_during_proof} only depends on $(r(w), d(w))$ and not on $L$, as was to be shown.
\end{proof}

\subsection{The Principal Germ}

The purpose of this section is to explicitly compute the principal germ. The relative position $(M_0 : M_1) \in \mbZ^2$ of two lattices $M_0, M_1\subset F^2$ is, by definition, the pair $(a, b)$ with $a\leq b$ that consists of the valuations of their elementary divisors (Cartan decomposition). For example, $(M_0 : M_1) = (0, k)$ with $k\geq 0$ if and only if $M_0\supseteq M_1$ with cyclic quotient $M_0/M_1$ of length $k$. By lattice pair in $F^2$, we mean a pair of lattices $(M, M^\flat)$ such that $(M : M^\flat) = (0, 1)$.

For non-negative integers $0\leq a \leq b$ and a third integer $0\leq k \leq a + b$, we define the quantity
\begin{equation}\label{eq:xi}
\Xi_k(a, b) = 1 + 2q + \ldots + 2q^{\min\{k, a, a+b-k\}}.
\end{equation}
The boundary cases here are understood as $\Xi_k(a,b) = 1$ whenever $\min\{k, a, a+b-k\} = 0$. We also need a slight modification of $\Xi_k(a,b)$ which will only be considered for $a\geq 1$:
\begin{equation}\label{eq:xi_prime}
\Xi'_k(a, b) = \begin{cases}
1 + 2q + \ldots + 2q^{\min\{k, a + b -k\}} & \text{if $k < a$ or $b < k$}\\
1 + 2q + \ldots + 2q^{a-1} + q^a & \text{if $a \leq k \leq b$}
\end{cases}
\end{equation}
with boundary cases $\Xi'_k(a, b) = 1$ whenever $k \in \{0, a+b\}$.

\begin{lem}\label{lem:lattice_count}
Let $(M_0, M_0^\flat)$ and $(M_1, M_1^\flat)$ be two lattice pairs in $F^2$ such that $M_0\supseteq M_1$ and $M_0^\flat \supseteq M_1^\flat$. Let $0\leq a \leq b$ be such that $(M_0:M_1) = (a : b)$. Then the number of lattice pairs $(Λ, Λ^\flat)$ that fit into the diagram
\begin{equation}\label{eq:standard_lattice_diagram}
\begin{array}{ccccc}
M_0 & \supseteq & Λ & \supseteq & M_1\\
\cup & & \cup & & \cup\\
M_0^\flat & \supseteq & Λ^\flat & \supseteq & M_1^\flat
\end{array}
\end{equation}
and furthermore satisfy $[M_0:Λ] = k$ is given by
\begin{equation}\label{eq:count_xi}
\begin{cases}
\Xi_k(a, b) & \text{if $(M_0^\flat : M_1^\flat) = (a : b)$}\\
\Xi_k'(a, b) & \text{if $(M_0^\flat : M_1^\flat) = (a-1 : b+1)$}\\
\Xi_k'(a+1, b-1) & \text{if $a+2 \leq b$ and $(M_0^\flat : M_1^\flat) = (a+1 : b-1)$}.
\end{cases}
\end{equation}
\end{lem}

\begin{proof}[Part 1 of the proof: Auxiliary results.]
We begin with a few easier counting formulas. Let $0 \leq a \leq b$ and $0\leq k \leq a+b$ be integers and let $M_0\supseteq M_1$ be two lattices of relative position $(a :b)$.  Put
\begin{equation}
\Phi^{\mr{prim}}_k(a,b) = \#\{M_0 \supseteq Λ\supseteq M_1 \mid [M_0:Λ] = k,\ M_0/Λ\text{ cyclic}\}.
\end{equation}
We have
$$
\#\mr{Surjections}\big(O_F/(π)^a \oplus O_F/(π)^b,\ O_F/(π)^k\big) = \begin{cases}
1 & \text{if $k = 0$}\\
q^{2k} - q^{2k-2} & \text{if $1 \leq k\leq a$}\\
q^a(q^k - q^{k-1}) & \text{if $a < k \leq b$}\\
0 & \text{if $b<k$}.\end{cases}
$$
Since $\#\Aut\big(O_F/(π)^k\big) = q^k - q^{k-1},$ it follows that
\begin{equation}\label{eq:count_phi_prim}
\Phi^{\mr{prim}}_k(a,b) = \begin{cases}
1 & \text{if $k = 0$}\\
q^{k-1} + q^k & \text{if $1 \leq k \leq a$}\\
q^a & \text{if $a < k\leq b$}\\
0 & \text{if $b<k$}.\end{cases}
\end{equation}
We next consider the quantities
\begin{equation}\label{eq:def_phi}
\Phi_k(a,b) = \#\{M_0 \supseteq Λ\supseteq M_1 \mid [M_0:Λ] = k\}.
\end{equation}
It holds that $\Phi_0(a,b) = 1$ and that
\begin{equation}
\Phi_1(a,b) = \begin{cases} 0 & \text{if $a = b = 0$}\\
1 & \text{if $a = 0$ and $1 \leq b$}\\
1 + q  & \text{if $1 \leq a$.}
\end{cases}
\end{equation}
A sublattice $Λ\subseteq M_0$ has the property that $M_0/Λ$ is not cyclic if and only if it is contained in $πM_0$. So there is a recursion formula for $2\leq k$:
\begin{equation}
\Phi_k(a,b) = \Phi_{k-2}(a-1, b-1) + \Phi^{\mr{prim}}_k(a,b).
\end{equation}
It follows from this and \eqref{eq:count_phi_prim} that
\begin{equation}\label{eq:count_phi}
\Phi_k(a,b) = 1 + q + \ldots + q^{\min\{k, a, a+b-k\}}.
\end{equation}

We next count lattice pairs that lie between $M_0$ and $M_1$. More precisely, we consider the quantity
\begin{equation}\label{eq:def_psi}
\Psi_k(a,b) = \#\{M_0 \supseteq Λ\supset Λ^\flat \supseteq M_1 \mid [M_0:Λ] = k,\ [Λ:Λ^\flat] = 1\}.
\end{equation}
If we are given $M_0 \supseteq Λ^\flat \supseteq M_1$, then there are either $1$ or $1+q$ possibilities for finding a lattice $Λ$ as in \eqref{eq:def_psi}, depending on whether $M_0/Λ^\flat$ is cyclic or not. We obtain that
$$\Psi_k(a,b) = Φ^{\mr{prim}}_{k+1}(a, b) + (1+q)Φ_{k-1}(a-1, b-1).$$
Evaluating this expression with \eqref{eq:count_phi}, it follows that
\begin{equation}\label{eq:count_psi}
\Psi_k(a, b) = \begin{cases}
1 + 2q + \ldots + 2q^k + q^{k+1} & \text{if $k < a$}\\
1 + 2q + \ldots + 2q^a & \text{if $a \leq k < b$}\\
1 + 2q + \ldots + 2q^{a+b-k-1} + q^{a+b-k} & \text{if $b \leq k \leq a+b-1$.}
\end{cases}
\end{equation}

\noindent \emph{Part 2 of the proof: Main result.} We now come back to the setting of the lemma. That is, $(M_0, M_0^\flat)$ and $(M_1, M_1^\flat)$ denote lattice pairs with $(M_0 : M_1) = (a:b)$ and $M_0^\flat\supseteq M_1^\flat$; our aim is to show \eqref{eq:count_xi}. We begin by noting that diagrams of the form \eqref{eq:standard_lattice_diagram} have the symmetry
$$
\begin{array}{ccccc}
M_0 & \supseteq & Λ & \supseteq & M_1\\
\cup & & \cup & & \cup\\
M_0^\flat & \supseteq & Λ^\flat & \supseteq & M_1^\flat
\end{array}
\quad\longmapsto\quad
\begin{array}{ccccc}
M_0^\flat & \supseteq & Λ^\flat & \supseteq & M^\flat_1\\
\cup & & \cup & & \cup\\
πM_0 & \supseteq & πΛ & \supseteq & πM_1.
\end{array}
$$
Thus the third case in \eqref{eq:count_xi} follows directly from the second one and will not be considered anymore.

Next, we settle the case $a = 0$: Then $M_0/M_1$ is cyclic. The second case cannot occur (and the third case has already been dealt with), so we are in the first case meaning that the quotient $M_0^\flat/M_1^\flat$ is cyclic as well. It follows that for every $0\leq k \leq b$ there is a unique lattice pair $(Λ, Λ^\flat)$ that fits \eqref{eq:standard_lattice_diagram} and satisfies $[M_0:Λ] = k$. This fits the special case $\Xi_k(0, b) = 1$ in \eqref{eq:xi}.

From now on we can and do assume that $1\leq a$ which implies that $M_0^\flat \supset M_1$. The set of lattice pairs $(Λ, Λ^\flat)$ in question is then in bijection with the following disjoint union. (The condition $[M_0:Λ] = k$ is understood without explicit mentioning.)
\begin{equation}\label{eq:lattice_counting_three_conditions}
\{M_0 \supseteq Λ \supseteq M_1 \mid M_0^\flat \not\supseteq Λ\}\ \ \sqcup\ \ 
\{M_0^\flat \supseteq Λ^\flat \supseteq M_1^\flat \mid Λ_0^\flat\not\supseteq M_1\}\ \ \sqcup\ \ 
\{M^\flat_0 \supseteq Λ \supset Λ^\flat \supseteq M_1\}.
\end{equation}
Indeed, for every lattice $Λ$ from the first set or $Λ^\flat$ from the second set, there is a unique way to complete the diagram \eqref{eq:standard_lattice_diagram}. It is given by setting $Λ^\flat = Λ\cap M_0^\flat$ or $Λ = Λ^\flat + M_1$, respectively. Furthermore, the cardinalities of all three sets in \eqref{eq:lattice_counting_three_conditions} are easily expressed in terms of $\Phi$ and $\Psi$:
\begin{equation}\label{eq:counting_formulas_xi}
\begin{aligned}
\#\{M_0 \supseteq Λ \supseteq M_1 \mid M_0^\flat \not\supseteq Λ\} &\ = \#\{M_0 \supseteq Λ \supseteq M_1\} - \#\{M_0^\flat \supseteq Λ \supseteq M_1\}\\
&\ = Φ_k(a, b) - Φ_{k-1}(M_0^\flat:M_1)\\
\#\{M_0^\flat \supseteq Λ^\flat \supseteq M_1^\flat \mid Λ^\flat \not \supseteq M_1\} &\ = \#\{M_0^\flat\supseteq Λ^\flat \supseteq M_1^\flat\} - \#\{M_0^\flat \supseteq Λ^\flat \supseteq M_1\}\\
&\ = Φ_k(M_0^\flat : M_1^\flat) - Φ_k(M_0^\flat:M_1)\\
\#\{M_0^\flat \supseteq Λ\supset Λ^\flat\supseteq M_1\} &\ = Ψ_{k-1}(M_0^\flat:M_1).
\end{aligned}
\end{equation}
More precisely, we obtain from \eqref{eq:counting_formulas_xi} that the number of $(Λ, Λ^\flat)$ in question is
\begin{equation}\label{eq:Xi_general}
Φ_k(a, b) + Φ_k(M_0^\flat : M_1^\flat) - Φ_{k-1}(M_0^\flat : M_1) - Φ_k(M_0^\flat : M_1) + Ψ_{k-1}(M_0^\flat : M_1).
\end{equation}
We have already reduced to the first and second case in \eqref{eq:count_xi}, so there are the following three possibilities left, none of which poses any difficulties:
$$\begin{cases}
(M_0^\flat : M_1^\flat) = (a, b),\mspace{74mu} (M_0^\flat : M_1) = (a, b - 1),\ \ a \leq b-1 & \text{\quad Case (1)}\\
(M_0^\flat : M_1^\flat) = (a, b),\mspace{74mu} (M_0^\flat : M_1) = (a - 1, b) & \text{\quad Case (2)}\\
(M_0^\flat : M_1^\flat) = (a - 1, b + 1),\ \ (M_0^\flat : M_1) = (a - 1, b) & \text{\quad Case (3)}.
\end{cases}$$
Let $Φ_k^\mr{total}$ denote the sum of the four $Φ$-terms in \eqref{eq:Xi_general}. Using \eqref{eq:count_phi}, one sees that in Case (1)
$$Φ_k^{\mr{total}} = \begin{cases}
q^k & \text{if $0\leq k\leq a$}\\
0 & \text{if $a < k < b$}\\
q^{a+b-k} & \text{if $b\leq k \leq a + b$,}
\end{cases}$$
in Case (2)
$$Φ_k^{\mr{total}} = \begin{cases}
q^k & \text{if $0\leq k < a$}\\
2q^a & \text{if $a \leq k \leq b$}\\
q^{a+b-k} & \text{if $b < k \leq a + b$,}
\end{cases}$$
and in Case (3)
$$Φ_k^{\mr{total}} = \begin{cases}
q^k & \text{if $0\leq k < a$}\\
q^a & \text{if $a \leq k \leq b$}\\
q^{a+b-k} & \text{if $b < k \leq a + b$.}
\end{cases}$$
Adding these to $Ψ_{k-1}(a, b-1)$ in Case (1) resp. to $Ψ_{k-1}(a-1, b)$ in Cases (2) and (3), and using \eqref{eq:count_psi}, proves \eqref{eq:count_xi}.
\end{proof}

\begin{prop}\label{prop:principal_germ}
Let $w\in M_2(F)$ and $P(w,s)$ be as in Definition \ref{def:germs}. Let $(L, r, d)$ be the numerical invariants of $w$, see \eqref{eq:invariants_w}. For integers $0\leq a \leq b$ and $0\leq k\leq a+b$, let $\Xi_k(a, b)$ and $\Xi'_k(a, b)$ denote the quantities from \eqref{eq:xi} and \eqref{eq:xi_prime}; set $X = -q^{-2s}$.
\begin{enumerate}[wide, labelindent=0pt, labelwidth=!, label=(\arabic*), topsep=2pt, itemsep=2pt]
\item If $L$ is an unramified field extension or if $r \leq 0$, then $P(w, s) = 0$.
\item Assume that $L$ is ramified and $r \geq 1$. Then either $d \geq 0$ and $r$ is even or $d = -1/2$ and $r$ is odd. The principal germ $P(w, s)$ only depends on $(r,d)$ and equals
\begin{equation}\label{eq:princ_germ_ram}
P(L, r, d, s) := \begin{cases}
\sum_{k = 0}^{r-1} \Xi_{k}(r/2-1,r/2) X^{-k-1} & \text{if $d \geq 0$}\\
\sum_{k = 0}^{r-1} \Xi_{k}(r/2+d,r/2-d-1) X^{-k-1} & \text{if $d = -1/2$.}
\end{cases}
\end{equation}

\item Assume that $L\iso F\times F$ and $r \geq 1$. If $d \geq 0$, then $r$ is even. The principal germ $P(w, s)$ only depends on $(r, d)$ and equals
\begin{equation}\label{eq:princ_germ_split}
P(F\times F, r, d, s) := 2 \cdot \begin{cases}
\sum_{k = 0}^{r-1} \Xi_k(r/2-1,r/2) X^{-k-1} & \text{if $d \geq 0$}\\
\sum_{k = 0}^{r-1} \Xi'_k(r/2+d,r/2-d-1) X^{-k-1} & \text{if $d < 0$.}
\end{cases}
\end{equation}
\end{enumerate}
\end{prop}
\begin{proof}
Recall first the definition of the principal germ from \eqref{eq:principal_germ}:
\begin{equation}\label{eq:princ_germ_recap}
P(w, s) = \sum_{(O_L, Λ_0^\flat, Λ_1, Λ_1^\flat) \in \mcL(w),\ \mr{cond}(Λ_0^\flat) = 0} X^{-[O_L:Λ_1]}.
\end{equation}
The task is thus to count the set of lattice diagrams of the form $(O_L, Λ_0^\flat, Λ_1, Λ_1^\flat)\in \mcL(w)$ with $\mr{cond}(Λ_0^\flat) = 0$ and $[O_L:Λ_1] = k$. This counting problem was the content of Lemma \ref{lem:lattice_count} and it is only left to evaluate this lemma in dependence on $(L, r, d)$.

If $L$ is an unramified field extension, then $P(w, s) = 0$ for the trivial reason that there are no lattices $Λ_0^\flat\subset O_L$ of index $1$ and conductor $0$. If $r \leq 0$, then $w$ is not topologically nilpotent which implies $\mcL(w) = \emptyset$ and hence $P(w, s) = 0$. This proves Part (1). We also note that the case distinctions for $(r, d)$ in Parts (2) and (3) were already stated in Lemma \ref{lem:std_form_r_d}, so it only left to prove \eqref{eq:princ_germ_ram} and \eqref{eq:princ_germ_split}.

Consider first the case of a ramified extension $L$ and of $r \geq 1$. Let $\varpi\in L$ denote a uniformizer. Then $Λ_0^\flat = \varpi O_L$ is the unique sublattice of $O_L$ of index $1$ and conductor $0$. Define $0\leq a \leq b$ by
$$(a, b) = (\varpi O_L : w O_L) = (π O_L : w\varpi O_L).$$
If $d \geq 0$, then $r$ is even and $(a, b) = (r/2-1, r/2)$. Lemma \ref{lem:lattice_count} states that there are $\Xi_{k-1}(r/2-1, r/2)$ many choices $(Λ_1, Λ_1^\flat)$ such that $(O_L, \varpi O_L, Λ_1, Λ_1^\flat) \in \mcL(w)$ and such that $[O_L:Λ_1] = k$. Specializing \eqref{eq:princ_germ_recap} to this case we precisely obtain the first identity in \eqref{eq:princ_germ_ram}.

We use the same arguments for $d = -1/2$. In this case, $r$ is odd and $a = b = (r-1)/2$. Lemma \ref{lem:lattice_count} states that there are $\Xi_{k-1}((r-1)/2, (r-1)/2)$ many tuples $(O_L, \varpi O_L, Λ_1, Λ_1^\flat) \in \mcL(w)$ with $[O_L:Λ_1] = k$, and we obtain the second identity in \eqref{eq:princ_germ_ram}. This completes the proof of Part (2).

Consider now the case of a split extension $L \iso F\times F$ and of $r\geq 1$. Write $w = (w_1, w_2)$  with $v(w_1)\leq v(w_2)$ for a fixed choice of such an isomorphism. There are two sublattices of $O_L$ of index $1$ and conductor $0$, namely $M_1 = (π, 1)O_L$ and $M_2 = (1, π)O_L$.

Assume first that $d \geq 0$. Then $v(w_1) = v(w_2) = r/2$ and
$$(M_i: wO_L) = (πO_L : wM_i) = (r/2-1, r/2)$$
for both possibilities $i \in \{1,2\}$. Lemma \ref{lem:lattice_count} states that there are $\Xi_{k-1}(r/2 - 1, r/2)$ many tuples $(O_L, \varpi O_L, Λ_1, Λ_1^\flat) \in \mcL(w)$ with $[O_L:Λ_1] = k$ and one obtains the first identity in \eqref{eq:princ_germ_split} in the same way as before.

Assume now that $d < 0$ and put $a := v(w_1)$ as well as $b := v(w_2)$. Then $(a,b) = (r/2 + d, r/2 - d)$ and one easily checks the identities
$$\begin{aligned}
(M_1 : wO_L) = (a-1, b),&\quad (πO_L : wM_1) = (a, b-1)\\
(M_2 : wO_L) = (a, b-1),&\quad (πO_L : wM_2) = (a-1, b).
\end{aligned}$$
It always holds that $a < b$. Applying the second and third identity in \eqref{eq:count_xi}, we find that the number of tuples $(O_L, M_i, Λ_1, Λ_1^\flat)\in \mcL(w)$ with $[O_L:Λ_1] = k$ is given by $\Xi'_{k-1}(a, b-1)$ for both $i = 1$ and $2$. The second identity in \eqref{eq:princ_germ_split} follows directly from \eqref{eq:princ_germ_recap} which finishes the proof of Part (3) and of the proposition.
\end{proof}

\subsection{The Parahoric Case}
\label{ss:parahoric}

The exact same ideas can be used to define a germ expansion for the parahoric orbital integral $\Orb(γ, f'_\Par, s)$ and to give a formula for the principal germ. Throughout, $w\in GL_2(F)$ is an element such that $γ(w) = \left(\begin{smallmatrix}
1 & 1 \\ w & 1
\end{smallmatrix}\right)$ is regular semi-simple. Let $(L, r, d)$ be its numerical triple, see \eqref{eq:invariants_w}. Define $\mcP(w)$ as the set of pairs $(Λ_0, Λ_1)$ of $O_F$-lattices in $L$ such that $O_L\cdot Λ_0 = O_L$ and such that
\begin{equation}\label{eq:lattice_set_para}
Λ_0\ \supset\ πΛ_0\ \supseteq\ Λ_1\ \supseteq\ wΛ_0.
\end{equation}
The set $\mcP(w)$ takes the role of $\mcL(w)$, but for the parahoric test function $f'_\Par$:
\begin{lem}\label{lem:lc_bijection_par}
Given $(Λ_0, Λ_1) \in \mcP(w)$, the following lattice chain lies in $\mcP(γ(w))$:
$$Λ_0 \oplus Λ_1\ \supset\ πΛ_0 \oplus Λ_1\ \supset\ π(Λ_0 \oplus Λ_1).$$
This assignment defines a bijection
\begin{equation}\label{eq:lc_bijection_par}
ψ:\mcP(w) \overset{\sim}{\lr} Γ\backslash \mcP(γ(w))
\end{equation}
with the property
\begin{equation}\label{eq:lc_bijection_omega_par}
\Omega(γ(w),\ ψ(Λ_0, Λ_1),\ s) = (-q^s)^{-r(w)}(-q^{2s})^{[Λ_0:Λ_1]}.
\end{equation}
In particular, the parahoric orbital integral has the expression
\begin{equation}\label{eq:lc_orb_int_par}
\Orb(γ(w), f'_\Par, s) \ =\ (-q^s)^{-r(w)} \sum_{(Λ_0, Λ_1)\in \mcP(w)} (-q^{2s})^{[Λ_0:Λ_1]}.
\end{equation}
\end{lem}

We again decompose the sum in \eqref{eq:lc_orb_int_par} into principal and unipotent germ:
\begin{equation}\label{eq:def_germs_par}
\begin{aligned}
P_{\Par}(w, s) &\ = \sum_{(O_L, Λ_1)\in \mcP(w)} (-q^{2s})^{[Λ_0:Λ_1]},\\
U_{\Par}(w,s) &\ = i(L)^{-1}\sum_{(Λ_0, Λ_1)\in \mcP(w),\ Λ_0\neq O_L} (-q^{2s})^{[Λ_0:Λ_1]}.
\end{aligned}
\end{equation}
The relation of the two germs with the orbital integral is again given by
\begin{equation}\label{eq:germ_exp_par}
O(γ(w), f'_\Par, s) = (-q^s)^{-r(w)}[P_\Par(w,s) + i(L) U_\Par(w,s)].
\end{equation}

\begin{prop}\label{prop:germ_indep_par}
Both the principal germ $P_\Par(w,s)$ and the unipotent germ $U_{\Par}(w, s)$ depend only on $(r,d)$ and not on $L$.
\end{prop}
\begin{proof}
Let $R_c = O_F + π^cO_L$ again denote the order of conductor $c$ in $L$. By Lemma \ref{lem:changing_L}, the relative position $(R_c:w R_c)$ only depends on $(r,d)$. Moreover, for every $c \geq 1$, the number
$$i(L)^{-1} \#\{Λ_0\subseteq O_L\mid \mr{cond}(Λ_0) = c,\ O_L\cdot Λ_0 = O_L\}$$
equals $q^{c-1}$ and is hence independent of $L$. Combining these facts with the definition of $\mcP(w)$ and \eqref{eq:def_germs_par} proves the proposition.
\end{proof}

\begin{prop}\label{prop:principal_germ_par}
For integers $0\leq a \leq b$ and $0\leq k\leq a+b$, let $Φ_k(a, b)$ denote the quantity from \eqref{eq:count_phi}; set $X = -q^{-2s}$.
\begin{enumerate}[wide, labelindent=0pt, labelwidth=!, label=(\arabic*), topsep=2pt, itemsep=2pt]
\item If $r\leq 0$, then $\Orb(γ(w), f'_\Par, s) = 0$.
\item If $r > 0$, then $\Orb(γ(w), f'_\Par, s)$ only depends on $(r,d)$ and equals
\begin{equation}\label{eq:princ_germ_par}
P_\Par(r,d,s) := \begin{cases}
\sum_{k = 0}^{r-2} Φ_k(r/2-1, r/2 - 1) X^{-k-2} & \text{if $d \geq 0$}\\
\sum_{k = 0}^{r-2} Φ_k(r/2 + d - 1, r/2 - d - 1) X^{-k-2} & \text{if $d < 0$.}
\end{cases}
\end{equation}
\end{enumerate}
\end{prop}
\begin{proof}
The vanishing statement in (1) holds because $\mcP(w) = \emptyset$ if $r \leq 0$. For (2), we note that no matter what $L$ is, the relative position of $O_L$ and $wO_L$ is given by
$$(O_L: wO_L) = \begin{cases}
(r/2, r/2) & \text{if $d \geq 0$}\\
(r/2 + d, r/2 - d) & \text{if $d < 0$}.
\end{cases}$$
It follows from the definition of $Φ_k(a,b)$ in \eqref{eq:def_phi} and that of $\mcP(w)$ in \eqref{eq:lattice_set_para} that the number of pairs $(O_L, Λ_1)\in \mcP(w)$ such that $[O_L: Λ_1] = k$ equals
$$\begin{cases}
Φ_{k-2}(r/2 - 1, r/2 - 1) & \text{if $d \geq 0$}\\
Φ_{k-2}(r/2 + d - 1, r/2 - d - 1) & \text{if $d < 0$}.
\end{cases}$$
Substituting these quantities in \eqref{eq:def_germs_par} proves the proposition.
\end{proof}

\section{$\Orb(γ, f'_\Par)$, $\Orb(γ, f'_\Iw)$ and $\dOrb(γ, f'_\Iw)$}
\label{s:orb_ints}

We have shown in Propositions \ref{prop:germ_independence} and \ref{prop:germ_indep_par} that the principal and unipotent germs for both $f'_\Par$ and $f'_\Iw$ only depend on the triple $(L, r, d)$ resp. $(r, d)$. Accordingly, we will write $P(L, r, d, s)$ for the Iwahori principal germ for such a numerical triple. We will similarly write $U(r, d, s)$ for the Iwahori unipotent germ as well as $P_\Par(r,d,s)$ and $U_\Par(r, d, s)$ for the parahoric germs.

\subsection{The Central Values}

\begin{prop}\label{prop:orb_int_para}
Let $γ\in G'_{\mr{rs}}$ be regular semi-simple with numerical invariants $(L, r, d)$, see \eqref{eq:invariants_gamma_concrete}.
\begin{enumerate}[wide, labelindent=0pt, labelwidth=!, label=(\arabic*), topsep=2pt, itemsep=2pt]
\item If $r$ is odd, or if $r \leq 0$, or if $r/2 + d \leq 0$, then $\Orb(γ, f'_\Par) = 0$.
\item In all other cases, the parahoric orbital integral $\Orb(γ, f'_\Par)$ is given by
\begin{equation}\label{eq:orb_int_para}
\begin{cases}
\phantom{2(}1 + q^2 +  \ldots + q^{r/2 - 2} & \text{if $L$ ramified and $r\in 4\mbZ$}\\
\phantom{2}(1 + q^2 +  \ldots + q^{r/2 - 3}) + \phantom{2}(q^{r/2-1} + q^{r/2} + \ldots + q^{r/2 + d - 1}) & \text{if $L$ ramified and $r\in 2 + 4\mbZ$}\\
2(1 + q^2 + \ldots + q^{r/2-2}) & \text{if $L$ inert and $r\in 4\mbZ$}\\
2(1 + q^2 + \ldots + q^{r/2-3}) + 2(q^{r/2-1} + q^{r/2} + \ldots + q^{r/2 + d - 2}) + q^{r/2 + d - 1} & \text{if $L$ inert and $r\in 2 + 4\mbZ$}\\
\phantom{2(}0 & \text{if $L$ split and $r\in 4\mbZ$}\\
\phantom{2(}q^{r/2 + d - 1} & \text{if $L$ split and $r\in 2 + 4\mbZ$.}
\end{cases}
\end{equation}
\end{enumerate}
\end{prop}
\begin{proof}
The sign of the functional equation of $\Orb(γ, f'_\Par, s)$ is $(-1)^r$, see Lemma \ref{lem:sign_alternative} and Proposition \ref{prop:functional_equation}. So $\Orb(γ, f'_\Par) = 0$ whenever $r$ is odd. Assume that $r \leq 0$ or $r/2 + d \leq 0$. Then $z_γ$ is not topologically nilpotent, and hence $\Orb(γ, f'_\Par, s) = 0$ by Lemma \ref{lem:lattice_combinatorics} (3). (If $r/2 + d \leq 0$ and $r  >0$, then necessarily $d < 0$ which implies that one eigenvalue of $z_γ$ has valuation $r/2 + d$, see Lemma \ref{lem:std_form_r_d} cases (3) and (4).) So we can henceforth assume that $r$ is even, that $r> 0$ and that $r/2 + d > 0$.

First, we consider the case of a split extension $L$. Let $α,β\in F$ be the two eigenvalues of $z_γ$. Note that $v(α)$ and $v(β)$ are both positive and of the same parity under our assumptions on $(r, d)$. If their parity is even, which is equivalent to $r\in 4\mbZ$, then $\Orb(γ, f'_\Par) = 0$ by Identity \eqref{eq:Par_orb_int_hyper}. If the parity of $v(α)$ and $v(β)$ is odd instead, then $v(α - β) = r/2 + d$ and \eqref{eq:Par_orb_int_hyper} shows $\Orb(γ, f'_\Par) = q^{r/2 + d - 1}$. This proves \eqref{eq:orb_int_para} in case $L$ is split.

Assume from now on that $L$ is a field. With the standing assumption that $r$ is even, it will necessarily hold that $d\geq 0$. Moreover, the germ expansion identity \eqref{eq:germ_exp_par} specializes to
\begin{equation}\label{eq:germ_exp_par_value}
\Orb(γ, f'_\Par) = P_\Par(r, d, 0) + i(L) U_\Par(r, d, 0).
\end{equation}
We can now use our knowledge of the hyperbolic orbital integrals to compute the unipotent germ. Let $\wt{γ} \in G'_{\mr{rs}}$ be an auxiliary element with numerical invariants $(F\times F, r, d)$. The value at $s = 0$, equivalently at $X = -1$, of the parahoric principal germ \eqref{eq:princ_germ_par} for $d\geq 0$ is given by the geometric series
\begin{equation}\label{eq:para_princ_germ_value}
P_\Par(r, d, 0) = 1 - q + q^2 - \ldots + (-q)^{r/2-1}.
\end{equation}
Substituting \eqref{eq:para_princ_germ_value} in the germ expansion \eqref{eq:germ_exp_par_value} for $\wt{γ}$ shows that the unipotent germ is either a geometric series or a sum of two such series:
\begin{equation}\label{eq:para_uni_germ_value}
U_\Par(r, d, 0) = \begin{cases}
\phantom{(}1 + q^2 + \ldots + q^{r/2-2} & \text{if $r\in 4\mbZ$}\\
(1 + q^2 + \ldots + q^{r/2 - 3}) + (q^{r/2-1} + q^{r/2} + \ldots + q^{r/2 + d - 2}) & \text{if $r\in 2 + 4\mbZ$.}
\end{cases}
\end{equation}
It is left to substitute \eqref{eq:para_princ_germ_value} and \eqref{eq:para_uni_germ_value} in \eqref{eq:germ_exp_par_value} with $i(L) = q$ for $L$ ramified and $i(L) = q+1$ for $L$ inert. This proves the proposition in the remaining four cases.
\end{proof}

\begin{thm}\label{thm:FL}
The central value of the Iwahori orbital integral is given by
\begin{equation}\label{eq:fund_lem_gl4}
\Orb(γ, f'_\Iw) = \begin{cases} 1 & \text{if $L$ ramified and $r \geq 1$ odd}\\
0 & \text{otherwise.}
\end{cases}
\end{equation}
In particular, the fundamental lemma (Conjecture \ref{conj:FL}) holds in case $D$ is a division algebra of degree $4$.
\end{thm}
\begin{proof}
We first compute the right hand side of the FL Identity \eqref{eq:conj_FL}. Let $g\in G_{\mr{rs}}$ be regular semi-simple. Proposition \ref{prop:quaternion_algebra} states that $D$ contains the $F$-algebra $E\tensor_F L_g$ which hence has to be a field. It follows that $L_g$ is a ramified field extension of $F$. Then Proposition \ref{prop:orb_int_division} states that
$$\Orb(g, f_D) = \begin{cases}
1 & \text{if $v_D(g) \in 2\mbZ$}\\
0 & \text{otherwise.}
\end{cases}
$$
Here, $v_D:D^\times\to \mbZ$ denotes the normalized valuation of $D$. The condition $v_D(g) \in 2\mbZ$ holds if and only if $v_D(z_g) \geq 1$. Moreover, $v_D(z_g)$ is always odd and thus $r = v_F(N_{L_g/F}(z_g^2)) \in 2\mbZ + 1$. In this way, the FL for $f'_\Iw$ becomes Identity \eqref{eq:fund_lem_gl4}.

We turn to the computation of $\Orb(γ, f'_\Iw)$. The sign of the functional equation of $\Orb(γ, f'_\Iw, s)$ is $(-1)^{r+1}$, see \eqref{eq:sign_concrete} and Proposition \ref{prop:functional_equation}. It follows that $\Orb(γ, f'_\Iw) = 0$ whenever $r$ is even. Moreover, Proposition \ref{prop:hyperbolic} in particular implies that $(X+1)$ divides $\Orb(γ, f'_\Iw, s)$ if $L$ is split, and hence that $\Orb(γ, f'_\Iw) = 0$ in all such cases. Since $r$ is always even when $L$ is an unramified field extension, the only remaining possibility for $\Orb(γ, f'_\Iw)$ to be non-zero is when $L$ is ramified and $r$ odd. In this case $d = -1/2$ by Lemma \ref{lem:std_form_r_d} (3). If $r \leq 0$, then \eqref{eq:fund_lem_gl4} holds for the trivial reason that $\Orb(γ, f'_\Iw, s) = 0$ by Lemma \ref{lem:lattice_combinatorics} (3). Thus it is left to consider the case $L$ ramified, $r> 0$ odd, $d = -1/2$.

Our first aim is to determine $U(r, -1/2, 0)$. To this end, we evaluate the principal germ for $F\times F$ from \eqref{eq:princ_germ_split} at $s = 0$, equivalently at $X = -1$:
\begin{equation}\label{eq:princ_germ_split_value}
P(F\times F, r, -1/2, 0) = -2[1 - 2q + 2q^2 + \ldots + 2(-q)^{(r-3)/2} + (-q)^{(r-1)/2}].
\end{equation}
Using the vanishing of $\Orb(\bdot, f'_\Iw)$ in all hyperbolic cases (see above), we obtain from the germ expansion \eqref{eq:germ_exp} that
\begin{equation}\label{eq:uni_germ_value}
\begin{aligned}
U(r, -1/2, 0) &\ = (1-q)^{-1}P(F\times F, r, -1/2, 0)\\
&\ = -2[1 - q + q^2 - \ldots + (-q)^{(r-3)/2}].
\end{aligned}
\end{equation}
Let $L$ be a ramified field extension. The principal germ for the case $(L, r, -1/2)$, given by \eqref{eq:princ_germ_ram}, specializes to
\begin{equation}\label{eq:princ_germ_ram_value}
P(L, r, -1/2, 0) = -[1 - 2q + 2q^2 - \ldots +2(-q)^{(r-1)/2}].
\end{equation}
Let $γ\in G'_{\mr{rs}}$ have numerical invariants $(L, r, d)$. We substitute \eqref{eq:uni_germ_value} and \eqref{eq:princ_germ_ram_value}, with $i(L) = q$ and $r$ odd in the germ expansion \eqref{eq:germ_exp} for $(L, r, 1/2)$ and obtain
$$\Orb(γ, f'_\Iw) = -[P(L, r, -1/2, 0) + q\, U(r, -1/2, 0)] = 1.$$
This is precisely Identity \eqref{eq:fund_lem_gl4} and the proof of the theorem is complete.
\end{proof}

\subsection{The Central Derivative}

\begin{prop}\label{prop:derivative}
Let $γ\in G'_{\mr{rs}}$ be regular semi-simple with numerical invariants $(L, r, d)$. Assume first that $r$ is odd, meaning that the sign $ε_D(γ)$ of the functional equation of $\Orb(γ, f'_\Iw, s)$ is positive. Then
\begin{equation}\label{eq:derivative_trivial_case}
\del(γ, f'_D) = 0\quand \dOrb(γ, f'_\Iw) = \Orb(γ, f'_\Iw) \log(q).
\end{equation}
Assume now that $r$ is even which implies $\del(γ, f'_D) = \del(γ, f'_\Iw)$. If $r\leq 0$, then $\dOrb(γ, f'_\Iw) = 0$. If $r > 0$, it is given by
\begin{equation}\label{eq:derivative_main}
\del(γ, f'_\Iw) = 4q\log(q)\,\Orb(γ, f'_\Par) + \log(q)\,\begin{cases}
r & \text{if $L$ ramified}\\
2r & \text{if $L$ inert}\\
0 & \text{if $L$ split}.
\end{cases}
\end{equation}
\end{prop}
\begin{proof}
Identity \eqref{eq:derivative_trivial_case} follows immediately from the functional equation (Proposition \ref{prop:functional_equation}): If its sign $ε_D(γ)$ is positive, then $\Orb(γ, f'_D, s)$ has an even functional equation, so $\del (γ, f'_D) = 0$. Applying \eqref{eq:norm_orb_int_Iw}, we also have the functional equation
$$\Orb(γ, f'_\Iw, -s) = q^{-2s}ε_D(γ)\Orb(γ, f'_\Iw, s).$$
Taking the derivative of both sides at $s = 0$ and assuming $ε_D(γ) = 1$ gives the other identity in \eqref{eq:derivative_trivial_case}.

Moreover, if $r \leq 0$, then $z_γ$ is not topologically nilpotent, so $\mcL(γ) = \emptyset$, and hence $\Orb(γ, f'_\Iw, s) = 0$ by Lemma \ref{lem:lattice_combinatorics}. From now on we assume that $r$ is even and that $r>0$. In particular, $ε_D(γ) = -1$ and hence $\Orb(γ, f'_D) = 0$. Then \eqref{eq:norm_orb_int_Iw} shows that $\del(γ, f'_D) = \del(γ, f'_\Iw)$, as claimed in the proposition. We now come to the main part of the proof.

Consider first the case that $L\iso F\times F$ is split. The factor $(X+1)$ in \eqref{eq:Iw_orb_int_hyper} has the property that $(X+1)\vert_{s = 0} = 0$ and $(d/ds)_{s = 0}(X + 1) = 2\log(q)$. Thus the derivative of \eqref{eq:Iw_orb_int_hyper} at $s = 0$ is given by
\begin{equation}\label{eq:deriv_easy}
\dOrb(γ, f'_\Iw) = 4q \Orb(γ, f'_\Par) \log(q)
\end{equation}
which proves \eqref{eq:derivative_main} when $L$ is split.

Our next aim is to compute the central derivatives
$$\partial P(L, r, d) = \left.\frac{d}{ds}\right\vert_{s = 0} P(L, r, d, s)\quand \partial U(r, d) = \left.\frac{d}{ds}\right\vert_{s = 0} U(r, d, s).$$
We are ultimately interested in the case of a field extension $L$, and here $r$ even implies $d\geq 0$. So we only compute $\partial P$ and $\partial U$ with this restriction. Directly from \eqref{eq:princ_germ_split}, we find
\begin{equation}\label{eq:princ_germ_split_deriv}
\partial P(F\times F, r, d) = 4[r/2 - (r-2)q + (r-4)q^2 - \ldots + 2(-q)^{r/2 - 1}] \log(q).
\end{equation}
Let $\wt{γ}$ be an auxiliary hyperbolic element with numerical invariants $(r,d)$. We obtain from the germ expansion \eqref{eq:germ_exp} and our previous result \eqref{eq:deriv_easy} that
\begin{equation}\label{eq:uni_germ_deriv}
\partial U(r, d) = (q-1)^{-1}[4q \Orb(\wt{γ}, f'_\Par)\log(q) - \partial P(F\times F, r, d)].
\end{equation}
The orbital integral $\Orb(\wt{γ}, f'_\Par)$ here is either $0$ or $q^{r/2+d-1}$ which depends on whether $r \in 4\mbZ$ or $r\in 2+4\mbZ$, see Proposition \ref{prop:orb_int_para}. Substituting this in \eqref{eq:uni_germ_deriv} yields
\begin{equation}\label{eq:uni_germ_deriv_explicit}
\partial U(r,d) = \begin{cases}
4 \big[\frac r2 - (\frac r2 -2)q + (\frac r2 -2)q^2 -\ldots + 2(-q)^{r/2-3} + 2(-q)^{r/2-2}\big]\log(q)\\
4 \big[\frac r2 - (\frac r2 -2)q + (\frac r2 -2)q^2 -\ldots + 3(-q)^{r/2-4} + 3(-q)^{r/2-3}\\
\phantom{4\big[\frac r2 - (\frac r2 -2)q + (\frac r2 -2)q^2 -\ldots} - q^{r/2-2} + q^{r/2-1}+\ldots+q^{r/2+d-1}\big]\log(q).
\end{cases}
\end{equation}
Here, the first line occurs if $r\in 4\mbZ$ and the second if $r\in 2 + 4\mbZ$. It is left to evaluate the expression
$$\dOrb(γ, f'_\Iw) = \partial P(L, r, d) + i(L)\partial U(r,d).$$
If $L$ is ramified, then $i(L) = q$ and
$$P(L, r, d, s) = P(F\times F, r, d, s)/2$$
by Proposition \ref{prop:principal_germ}. Thus we may reuse \eqref{eq:princ_germ_split_deriv} and we obtain
\begin{equation}\label{eq:deriv_finale_ram}
\dOrb(γ, f'_\Iw) = \log(q) \begin{cases}
r + 4q + 4q^3 + \ldots + 4q^{r/2 - 1}\\
r + 4q + 4q^3 + \ldots + 4q^{r/2 - 1} + 4q^{r/2} + \ldots + 4q^{r/2 + d},
\end{cases}
\end{equation}
where the first case is for $r\in 4\mbZ$ and the second for $r\in 2+4\mbZ$.

Consider now the case where $L$ is an unramified field extension of $F$. Then $i(L) = q+1$ and $P(L, r, d, s) = 0$ by Proposition \ref{prop:principal_germ}, so simply $\dOrb(γ, f'_\Iw) = (q+1)\partial U(r, d)$. This equals
\begin{equation}\label{eq:deriv_finale_inert}
\log(q) \begin{cases}
2r + 8q + 8q^3 + \ldots + 8q^{r/2 - 1}\\
2r + 8q + 8q^3 + \ldots + 8q^{r/2} + 8q^{r/2 + 1} + \ldots + 8q^{r/2 + d - 1} + 4q^{r/2 + d -1}.
\end{cases}
\end{equation}
Here, again, the first case is for $r\in 4\mbZ$ and the second for $r\in 2 + 4\mbZ$.

Comparing \eqref{eq:deriv_finale_ram} and \eqref{eq:deriv_finale_inert} with \eqref{eq:orb_int_para} shows \eqref{eq:derivative_main}, and the proof of the proposition is complete.
\end{proof}

\part{Intersection numbers on $\mcM_{1/4}$ and $\mcM_{3/4}$}

In this third part, we establish AT for the two division algebras of Hasse invariant $1/4$ and $3/4$. This is the main result of our paper and we formulate it upfront, cf. Theorem \ref{thm:main_ATC}. The proof will be completed in \S\ref{ss:intersection_numbers_1_4} for invariant $1/4$ and in \S\ref{ss:intersection_numbers_3_4} for invariant $3/4$.

The layout is as follows. After formulating the result in \S\ref{s:main_results}, there will be two short sections that equally concern both Hasse invariants. The first (\S\ref{s:surface_intersection}) provides a formula for intersection numbers of regular surfaces in regular $4$-space. The second (\S\ref{s:multiplicity_functions}) provides a description of certain multiplicity functions on the Bruhat--Tits tree of $PGL_{2,E}$. These will later describe the multiplicities of the $1$-dimensional components in the intersection loci $\mcI(g)$.

Subsequently, we will first complete the proof of AT for invariant $1/4$ in \S\ref{s:1_4}. The key point here is that Drinfeld's theorem \cite{Drinfeld} provides an explicit linear algebra description of $\mcM_D$ for $D = D_{1/4}$. So the proofs in \S\ref{s:1_4} will, in fact, not involve any $π$-divisible groups.

In section \S\ref{s:3_4}, we will use deformation-theoretic arguments to extend the results from Hasse invariant $1/4$ to invariant $3/4$.

\section{Main Results}
\label{s:main_results}

The notation will be the same as in \S\ref{s:ATC}. We assume, however, that $n = 2$ and that $D$ is a division algebra of Hasse invariant $λ\in \{1/4, 3/4\}$. The centralizer $C = \mr{Cent}_D(E)$ is then a quaternion division algebra over $E$. We also denote by $B = D_{1/2}$ a quaternion division algebra over $F$. Recall that $O_D\subset D$ denotes a maximal order such that $O_C = C\cap O_D$ is again maximal.

The set $B(H, μ_H)$ has a single element $[b]$ in this situation (Example \ref{ex:isogeny_classes} (3)). The corresponding $C$-isocrystal $\bN_{b,+}$ is of height $8$, dimension $2$ and isoclinic of slope $1/4$. We choose framing objects: Let $(\mbY,ι)$ denote a special $O_C$-module over $\Spec \mbF$ and put $(\mbX, κ) = O_D\tensor_{O_C} (\mbY, ι)$. We identify the isocrystal of $(\mbY, ι)$ with $\bN_{b,+}$. In particular, we view $C_b$ and $D_b$ as the groups of quasi-automorphisms of $(\mbY, ι)$ and $(\mbX, κ)$. Then
\begin{equation}\label{eq:endo_rings_CDA_deg_4}
C_b \iso M_2(E)\quand D_b \iso \begin{cases} M_4(F) & \text{if } λ = 1/4\\
M_2(B) & \text{if } λ = 3/4.
\end{cases}
\end{equation}
As before, we put $H_b = C_b^\times$ and $G_b = D_b^\times$. These act from the right on the moduli spaces $\mcM_C$ and $\mcM_D$ whose definitions we briefly recall. First, $\mcM_C$ is the formal scheme over $\Spf O_{\breve F}$ with functor of points
\begin{equation}\label{eq:def_cycle_GL4}
\mcM_C(S) = \left\{(Y, ι, ρ) \left\vert \text{\begin{varwidth}{\textwidth}\centering $(Y,ι)/S$ a special $O_C$-module\\
$ρ:\ob{S} \times_{\Spec \mbF} \mbY \to \ob{S}\times_S Y$ an $O_C$-linear quasi-isogeny\end{varwidth}}\right\}\right..
\end{equation}
This is the (base chang to $O_{\breve F}$) of the Drinfeld half-plane for $O_E$. It is a two-dimensional, regular, $π$-adic formal scheme whose description will be recalled in \S\ref{ss:Drinfeld_Theorem} below. Second, $\mcM_D$ is the formal scheme over $\Spf O_{\breve F}$ with functor of points
\begin{equation}\label{eq:def_ambient_CDA_deg_4}
\mcM_D(S) = \left\{(X, κ, ρ) \left\vert \text{\begin{varwidth}{\textwidth}\centering $(X,κ)/S$ a strict $O_D$-module\\
$ρ:\ob{S} \times_{\Spec \mbF} \mbX \to \ob{S}\times_S X$ an $O_D$-linear quasi-isogeny\end{varwidth}}\right\}\right..
\end{equation}
If $λ = 1/4$, then $\mcM_D$ is Drinfeld's $4$-space and, in particular, a $π$-adic formal scheme. Its description will also be given in \S\ref{ss:Drinfeld_Theorem}. If $λ = 3/4$ however, then there is no known explicit description of $\mcM_D$.

For every regular semi-simple $g\in G_{b, \mr{rs}}$, we have defined the intersection locus $\mcI(g) = \mcM_C \cap g \mcM_C$ and an intersection number $\Int(g)\in \mbZ$ in \S\ref{ss:intersection_numbers}. We formulate our main result:

\begin{thm}\label{thm:main_ATC}
The AT conjecture holds for $D$. More precisely, let $f'_{\mr{corr}}$ be given by
\begin{equation}\label{eq:corr_function}
f'_{\mr{corr}} = \begin{cases}
-4q\log(q)\cdot f'_\Par & \text{if $λ = 1/4$}\\
0 & \text{if $λ = 3/4$.}\end{cases}
\end{equation}
Then, for every regular semi-simple $γ\in G'_{\mr{rs}}$,
\begin{equation}\label{eq:AT}
\del(γ, f'_D) + \Orb(γ, f'_{\mr{corr}}) = \begin{cases} 2\,\Int(g)\log(q) & \text{if there exists a $g\in G_b$ that matches $γ$}\\
0 & \text{otherwise.}
\end{cases}
\end{equation}
\end{thm}
This theorem will be proved as Theorems \ref{thm:main_1_4} and \ref{thm:main_3_4} at the end of sections \S\ref{s:1_4} and \S\ref{s:3_4}. We mention here that our proof of Theorem \ref{thm:main_ATC} is not fully complete when $F$ is of equal characteristic and $λ = 3/4$. The reason is that \S\ref{ss:displays} relies on the $O_F$-display theory of \cite{ACZ} which was only developed for $p$-adic local fields. Completing the proof requires a duplication of \S\ref{ss:displays} but for local shtuka. We will not carry out these arguments in order to keep the article at a reasonable length.

\section{Surface Intersections}
\label{s:surface_intersection}

The aim of this section is to derive a general formula for intersection numbers of regular surfaces in a regular $4$-dimensional space.
Let first $Y$ be a regular formal scheme, pure of dimension $2$ and let $Z = V(\mcI)\subseteq Y$ be a closed formal subscheme of dimension $\leq 1$.

\begin{defn}\label{def:purification}
(1) Let $Z^{\mr{pure}}\subseteq Z$ be the maximal effective Cartier divisor on $Y$ that is contained in $Z$. More precisely, we define $Z^{\mr{pure}} = V(\mcI^{\mr{pure}})$ where for every affine open $U = \Spf A$ of $Y$, say $U\cap Z = \Spf A/I$,
$$\mcI^{\mr{pure}}(U) = \left\{f\in A \left\vert \text{\begin{varwidth}{\textwidth} \centering $f = 0$ in $(A/I)_\eta$ for all generic points $η$ of\\
$1$-dimensional irreducible components of $\Spec A/I$\end{varwidth}}\right\}\right..$$
(2) Set $Z^{\mr{art}} = V(\mcI^{\mr{art}})$ with
$\mcI^{\mr{art}} = \{f\in \mcO_X\mid f\mcI^{\mr{pure}}\subseteq \mcI\}.$
Note that $\mcO_{Z^{\mr{art}}}$ is isomorphic to the quotient $\mcI^{\mr{pure}}/\mcI.$
\end{defn}

Let $X$ be a regular formal scheme, pure of dimension $4$, and let $Y_1,Y_2\subseteq X$ be regular closed formal subschemes, both pure of dimension $2$. Then $Y_1$ and $Y_2$ are locally defined by a regular sequence of length $2$ in $\mcO_X$. Put $Z = Y_1\cap Y_2$ and assume $\dim Z \leq 1$.
The conormal bundle of $Y_i$ in $X$ is $\mcC_i = (\mcI_i/\mcI_i^2)\vert_{Y_i}$, where $Y_i = V(\mcI_i)$. Let $D = Z^{\mr{pure}}$ be the purely $1$-dimensional locus of $Z$ as in Definition \ref{def:purification}; it is independent of whether it is defined with respect to $Z\subseteq Y_1$ or $Z\subseteq Y_2$.
\begin{prop}\label{prop:tor_terms_for_surfaces}
The following identities hold for the Tor-terms $T_i := Tor_i^{\mcO_X}(\mcO_{Y_1}, \mcO_{Y_2})$.
\begin{enumerate}[wide, labelindent=0pt, labelwidth=!, label=(\arabic*), topsep=2pt, itemsep=2pt]
\item $T_0 = \mcO_Z$.
\item $T_1 = (\det \mcC_1)\vert_D \tensor_{\mcO_D} \mcO_{Y_2}(D)\vert_D$.
\item $T_2 = 0$.
\end{enumerate}
\end{prop}
\begin{proof}
The claim on $T_0$ is immediate. To prove the statements about the higher Tor-terms, we first assume that $Y_1 = V(f_1,f_2)$ is the vanishing locus of two elements and that $D = V(g)$ for some $g\in \mcO_{Y_2}$. Then the Koszul complex
$$K_{(f_1,f_2)} := \left[
\xymatrix{
\mcO_X \ar[r]^{\left(\begin{smallmatrix} f_2 \\ -f_1\end{smallmatrix}\right)}
& \mcO_X^{\oplus 2} \ar[r]^{(f_1,f_2)}
& \mcO_X}\right]$$
is quasi-isomorphic to $\mcO_{Y_1}$ and
$$T_i = H^{-i}(\mcO_{Y_2}\tensor_{\mcO_X} K_{(f_1,f_2)}).$$
Let $\overbar{f}_i$ denote the image of $f_i$ in $\mcO_{Y_2}$. Then
$$\mcO_{Y_2}\tensor_{\mcO_X} K_{(f_1,f_2)} = \left[
\xymatrix{\mcO_{Y_2} \ar[r]^{\left(\begin{smallmatrix} \overbar f_2 \\ -\overbar f_1\end{smallmatrix}\right)}
& \mcO_{Y_2}^{\oplus 2} \ar[r]^{(\overbar f_1,\overbar f_2)}
& \mcO_{Y_2}}\right].$$
Set $f'_i = g^{-1} \overbar f_i \in \mcO_{Y_2}$. Then $(f'_1,f'_2)$ forms a regular sequence in $\mcO_{Y_2}$ because $V(f'_1,f'_2)$ is artinian by definition of $D$. In particular,
$$T_2 = \ker \left(\begin{smallmatrix} \overbar f_2 \\ -\overbar f_1\end{smallmatrix}\right) = \ker \left(\begin{smallmatrix} f'_2 \\ -f'_1\end{smallmatrix}\right) = 0.$$
Moreover,
$$T_1 = H^{-1}(\mcO_{Y_2}\tensor_{\mcO_X} K_{(f_1,f_2)}) = \mcO_{Y_2} \cdot \left(\begin{smallmatrix} f'_2 \\ -f'_1\end{smallmatrix}\right) / \mcO_{Y_2}\cdot \left(\begin{smallmatrix} \overbar f_2 \\ -\overbar f_1\end{smallmatrix}\right).$$
is a line bundle on $D$ and our choices provide the specific generator
$$c_{f_1,f_2,g} := \left(\begin{smallmatrix} f'_2 \\ -f'_1\end{smallmatrix}\right) \text{ mod }\mcO_{Y_2}\cdot \left(\begin{smallmatrix} \overbar f_2 \\ -\overbar f_1\end{smallmatrix}\right).$$ 

Now we turn to the general situation. The given local argument already implies that $T_2 = 0$. We claim that the above construction glues to a map (and hence isomorphism)
\begin{equation}\label{eq:iso_tor_1}
\begin{aligned}
(\det \mcC_1)\vert_D \tensor_{\mcO_D} \mcO_{Y_2}(D)\vert_D & \lr & T_1\\
\left(\overbar f_1 \wedge \overbar f_2\right) \tensor g^{-1} & \longmapsto & c_{f_1,f_2,g}.
\end{aligned}
\end{equation}
(Here, $f_1$, $f_2$ and $g$ denote any local generators as before.) It is clear that if $g$ is replaced by $ug$ with $u\in \mcO_{Y_2}^\times$, then $f'_i$ gets replaced by $u^{-1}f'_i$. We find that $c_{f_1,f_2,ug} = u^{-1}c_{f_1,f_2,g}$ which shows the independence of \eqref{eq:iso_tor_1} from the chosen trivialization $g^{-1}\in \mcO_{Y_2}(D)$.

Now assume $(h_1, h_2) = (f_1, f_2) A$ for some $A\in GL_2(\mcO_X)$. Then $A$ defines an isomorphism of complexes
\begin{equation}
\xymatrix{
\mcO_{Y_2} \ar[r]^{\left(\begin{smallmatrix} \overbar f_2 \\ -\overbar f_1\end{smallmatrix}\right)} & \mcO_{Y_2}^{\oplus 2} \ar[r]^{(\overbar f_1,\overbar f_2)} & \mcO_{Y_2}\\
\mcO_{Y_2} \ar[r]^{\left(\begin{smallmatrix} \overbar h_2 \\ -\overbar h_1\end{smallmatrix}\right)} \ar[u]^{\det A} & \mcO_{Y_2}^{\oplus 2} \ar[r]^{(\overbar h_1,\overbar h_2)} \ar[u]^A & \mcO_{Y_2} \ar@{=}[u]}
\end{equation}
and one easily checks the relation
$$A \cdot c_{h_1,h_2,g} = \det A \cdot c_{f_1, f_2, g}.$$
Since $h_1\wedge h_2 = \det A \cdot f_1\wedge f_2$, this precisely says that \eqref{eq:iso_tor_1} is also independent of the choice of trivialization of $\mcC_1$.
\end{proof}
Assume now that $X$ is a $\Spf W$-scheme of locally formally finite type where $W$ is a complete DVR, and that $Z\to \Spec W$ is a proper scheme with empty generic fiber. Then we define the intersection number of $Y_1$ and $Y_2$ as
$$\langle Y_1, Y_2\rangle_{X} = χ(\mcO_{Y_1}\tensor^{\mbL}_{\mcO_X} \mcO_{Y_2})\in \mbZ.$$
As a corollary to Proposition \ref{prop:tor_terms_for_surfaces}, this has the following more concrete description.
\begin{cor}\label{cor:intersection_simplified}
With all notation as before,
$$\begin{aligned}
\langle Y_1, Y_2\rangle_X =\ & \chi(\mcO_{Z^{\mr{art}}}) + \chi(\mcO_D) - \chi\big(\det \mcC_1\vert_D \tensor_{\mcO_D} \mcO_{Y_2}(D)\vert_D\big)\\
=\ & \mr{len} (\mcO_{Z^{\mr{art}}})  - \deg (\det \mcC_1 \vert_D) - \langle D,D\rangle_{Y_2}.
\end{aligned}$$
\end{cor}
\begin{proof}
The first equality is Proposition \ref{prop:tor_terms_for_surfaces}. To obtain the second, we applied the Riemann--Roch identity
$$\chi(\mcO_D) - \chi(\mcL) = - \deg(\mcL),\quad \mcL \in \mr{Pic}(D),$$
and rewrote $\deg \mcO_{Y_2}(D)\vert_D$ as the self-intersection number of $D$ on $Y_2$. We refer to \cite[Tag 0AYQ]{Stacks} for the notion of degree in this possibly non-reduced context.
\end{proof}

\section{Multiplicity functions}
\label{s:multiplicity_functions}

Let $W$ be a $2$-dimensional $E$-vector space and let $\mcB$ denote the Bruhat--Tits building of the projective linear group $PGL_E(W)$. Recall that $\mcB$ is a $(q^2+1)$-regular tree whose vertices are the homothety classes of $O_E$-lattices in $W$. Two vertices are connected by an edge if and only if the two homothety classes have representatives $Λ_0$ and $Λ_1$ with $πΛ_0\subset Λ_1 \subset Λ_0$.

Let $z\in GL_F(W)$ be an $E$-conjugate linear endomorphism. The aim of this section is to give a precise description of the shape of the function
\begin{equation}\label{eq:def_multiplicity_function}
n(z,-):\mr{Vert}(\mcB)\lr \mbZ,\quad Λ\longmapsto \max\{k \in \mbZ\mid zΛ \subseteq π^kΛ\}.
\end{equation}
More precisely, we give a description for all $z$ such that $1+z$ lies in $GL_F(W)$ and is regular semi-simple with respect to $E\subseteq \End_F(W)$. We begin with a simple classification lemma over the residue field whose proof we omit.
\begin{lem}\label{lem:classif_emanating_edges}
Denote by $σ$ the Galois conjugation of $\mbF_{q^2}/\mbF_q$. Let $\bar{Λ}$ be a $2$-dimensional $\mbF_{q^2}$-vector space and let $0\neq \bar{z}\in \End_{\mbF_q}(\bar{Λ})$ be a $σ$-linear endomorphism. Precisely one of the following six statements applies to $\bar{z}$:
\begin{enumerate}[wide, labelindent=0pt, labelwidth=!, label=(\arabic*), topsep=2pt, itemsep=2pt]
\item It is nilpotent, i.e. $\bar{z}^2 = 0$. In this case, there is a unique line $\ell \subset \bar{Λ}$ such that $\bar{z} \ell \subseteq \ell$, namely $\ell = \bar{z}\bar{Λ}$. In a suitable basis we have
$$\bar{z} = \left(\begin{matrix} & 1\\ &  \end{matrix}\right)σ.$$
\item It is neither invertible nor nilpotent. Then there are precisely two lines $\ell_1,\ell_2 \subseteq \bar{Λ}$ such that $\bar{z}\ell_i \subset \ell_i$, namely $\bar{z}\bar{Λ}$ and $\ker(\bar{z})$. In a suitable basis we have
$$\bar{z} = \left(\begin{matrix}λ & \\ &  \end{matrix}\right)σ$$
for some scalar $0\neq λ\in \mbF_{q^2}$.
\item It is invertible and there is precisely one line $\ell \subset \bar{Λ}$ with $\bar{z}\ell = \ell$. Then there are $λ,μ\in \mbF_{q^2}^\times$ and a basis of $\bar{Λ}$ such that $λμ^q + λ^qμ \neq 0$ and such that $\bar{z}$ is given by
$$\bar{z} = \left(\begin{matrix}λ & μ\\ & λ  \end{matrix}\right)σ.$$
\item It is invertible and there are precisely two lines $\ell_1, \ell_2\subset W$ with $\bar{z}\ell_i = \ell_i$. Let $0 \neq v_i\in \ell_i$ be any two vectors and define $λ_i$ by $\bar{z}v_i = λ_iv$. Then $λ_1^{q+1} \neq λ_2^{q+1}$ and $\bar{z}$ is given in that basis by
$$\bar{z} = \left(\begin{matrix}λ_1 & \\ & λ_2 \end{matrix}\right)σ.$$
\item It is invertible and there are precisely $q+1$ lines $\ell\subset V$ with $\bar{z}\ell = \ell$. In a suitable basis and for a suitable scalar $0\neq λ\in \mbF_{q^2}$, we have
$$\bar{z} = \left(\begin{matrix}λ & \\ & λ\end{matrix}\right)σ.$$
\item It is invertible and there is no $\bar{z}$-stable line.\qed
\end{enumerate}
\end{lem}

We return to $O_E$-lattices in $W$ and the function $n(-,-)$.
\begin{lem}\label{lem:properties_multiplicity_function}
The function $n(-,-)$ enjoys the following properties.
\begin{enumerate}[wide, labelindent=0pt, labelwidth=!, label=(\arabic*), topsep=2pt, itemsep=2pt]
\item $n(πz, Λ) = n(z, Λ) + 1$
\item $|n(z, Λ) - n(z, Λ')| \leq 1$ whenever $Λ$ and $Λ'$ are neighbors in $\mcB$.
\item The function $n(z,-)$ is bounded above by $v(\det_E(z^2))/4$ and, in particular, takes a maximum.
\item Let $Λ'' \in [Λ,Λ']$ be a lattice on the unique shortest path connecting $Λ$ and $Λ'$ in $\mcB$. Then
$$n(z, Λ'') \geq \min\{n(z, Λ),\ n(z,Λ')\}.$$
\end{enumerate}
\end{lem}
\begin{proof}
The first three claims follow directly from the definition. For the last one, choose $Λ$ and $Λ'$ in their homothety classes such that $Λ'\subseteq Λ$ with $Λ/Λ'$ cyclic. Then $Λ''$ is the homothety class of one of the lattices $Λ' + π^iΛ$, $i\geq 0$, and the claim follows from the definition of $n(-,-)$.
\end{proof}

We next analyze the local properties of $n(z,-)$ in conjunction with Lemma \ref{lem:classif_emanating_edges}. Given a lattice $Λ$, we obtain a non-zero $σ$-linear endomorphism $\bar{z}_Λ := (π^{-n(z,Λ)}z$ mod $πΛ)$ in $\End_{\mbF_q}(\bar{Λ})$, where $\bar{Λ} = Λ/πΛ$. Let $\ell = Λ/Λ' \subseteq \bar{Λ}$ be the line corresponding to the neighbor lattice $πΛ\subset Λ'\subset Λ$.
\begin{lem}\label{lem:local_properties_multiplicities}
The following cases occur:
$$n(z, Λ') = \begin{cases}
n(z,Λ)-1 & \text{if $\bar{z}_Λ(\ell) \not\subseteq \ell$}\\[8pt]
n(z,Λ)+1 & \text{\begin{varwidth}{\textwidth}if $\bar{z}_Λ(\ell)\subseteq \ell$, if $\bar{z}_Λ$ falls into case (1) of Lemma \ref{lem:classif_emanating_edges},\\
\leavevmode\phantom{if $\bar{z}_Λ(\ell)\subseteq \ell$,} and if $v(\det_E(π^{-2n(z, Λ)}z^2)) \geq 4$\end{varwidth}}\\[8pt]
n(z,Λ) & \text{if $\bar{z}_Λ(\ell)\subseteq \ell$, but $\bar{z}_Λ$ does not satisfy the previous further two conditions.}
\end{cases}$$
\end{lem}
\begin{proof}
By the scaling invariance from Lemma \ref{lem:properties_multiplicity_function} (1), it suffices to consider the case $n(z,Λ) = 0$. We also put $\bar{z} = \bar{z}_Λ$. It is clear that $n(z, Λ') \geq 0$ if and only if $\bar{z}\ell \subseteq \ell$. If $v_F(\det_F(z)) = 0$, then necessarily also $n(z,Λ') = 0$. This happens if and only if $\bar{z}$ does not fall into cases (1) and (2) of Lemma \ref{lem:classif_emanating_edges}. In the remaining two cases, we may find an $\mbF_{q^2}$-basis $(e_1,e_2)$ of $\bar{Λ}$ such that
\begin{equation}\label{eq:mod_pi_cases_z}
\bar{z} = \left(\begin{matrix} & 1\\ &  \end{matrix}\right)σ\quad \text{or}\quad\bar{z} = \left(\begin{matrix}λ & \\ &  \end{matrix}\right)σ,\ \ λ\neq 0,
\end{equation}
and $\ell = \mbF_{q^2}e_1$ in the first case or $\ell \in \{\mbF_{q^2}e_1,\ \mbF_{q^2}e_2\}$ in the second. Lifting $(e_1,e_2)$ to an $O_E$-basis of $Λ$, we obtain a matrix presentation
$$z = \left(\begin{matrix}a & b\\c & d \end{matrix}\right)σ$$
that reduces modulo $πΛ$ to \eqref{eq:mod_pi_cases_z}. Depending on whether $\ell = \mbF_{q^2}e_1$ or $\ell = \mbF_{q^2}e_2$, we obtain that $z\vert_{Λ'}$ has a matrix presentation as
\begin{equation}\label{eq:mod_pi_cases_z_II}
z = \left(\begin{matrix}a & πb\\π^{-1}c & d \end{matrix}\right)σ\ \ \ \text{or}\ \ \ y = \left(\begin{matrix}a & π^{-1}b\\πc & d \end{matrix}\right)σ.
\end{equation}
Then $n(z, Λ') = n(z, Λ) + 1$ occurs if and only if all four entries in \eqref{eq:mod_pi_cases_z_II} have valuation $\geq 1$. This never happens in the second case of \eqref{eq:mod_pi_cases_z} because here $v(a) = 0$. It is left with the first case in \eqref{eq:mod_pi_cases_z} and find that $n(z, Λ') = n(z, Λ) + 1$ if and only if $v(π^{-1}c) \geq 1$, which is equivalent to the stated condition $v(\det_E(z^2)) \geq 4$.
\end{proof}

\begin{defn}\label{def:Tz}
Let $\mcT(z)$ denote the set of homothety classes of $O_E$-lattices in which $n(z,-)$ takes its maximum.
\end{defn}

Property (4) of Lemma \ref{lem:properties_multiplicity_function} shows that $\mcT(z)$ is connected. Lemmas \ref{lem:classif_emanating_edges} and \ref{lem:local_properties_multiplicities} imply that each of its vertices has valency $0$, $1$, $2$ or $q+1$.
\begin{prop}\label{prop:multiplicities}
Denote by $d(Λ, \mcT(z))$ the distance of $Λ$ from $\mcT(z)$. Then
$$n(z,Λ) = \max n(z,-) - d(Λ, \mcT(z)).$$
\end{prop}
\begin{proof}
The claim is tautologically true for $Λ\in \mcT(z)$. For $Λ$ with $d(Λ, \mcT(z)) = 1$, it follows from Statement (2) of Lemma \ref{lem:properties_multiplicity_function}.

Now assume $d(Λ, \mcT(z)) \geq 2$. Let $Λ'$ denote the unique neighbor of $Λ$ on the shortest path towards $\mcT(z)$. By induction on $d(-, \mcT(z))$, we find that $Λ'$ has a neighbor $Λ''$ with $n(z,Λ'') = n(z,Λ') + 1$, namely the subsequent lattice on the shortest path towards $\mcT(z)$. This implies by Lemma \ref{lem:local_properties_multiplicities} that $\bar{z}_{Λ'}$ falls into Case (1) of Lemma \ref{lem:classif_emanating_edges}, and hence that $n(z,Λ) = n(z,Λ') - 1$ as claimed.
\end{proof}

This reduces us to describe $\mcT(z)$. We assume from now on that $1+z$ is regular semi-simple in the sense of \S\ref{s:invariants}. By definition, this means that $\Inv(1+z; T) = \mr{char}_E(z^2;T)$ is a separable polynomial with $\Inv(1+z; 0)\Inv(1+z; 1) \neq 0$. The description of $\mcT(z)$ will be in terms of the numerical invariant $(L, r, d)$ of $1+z$ from Definition \ref{def:orbit_invariants}:
$$L = F[z^2],\quad r = v(N_{L/F}(z^2)),\quad d = \mr{cond}(O_F[π^kz^2]) - k - r/2,\ \ k\gg 0.$$
Note that $L$ is an étale quadratic extension of $F$.

\begin{lem}\label{lem:maximum_of_multiplicity_function}
The maximum of $n(z,-)$ is given by
$$\max \left\{k \in \mbZ \mid (π^{-k}z)^2 \in O_L\right\} = \begin{cases}
\lfloor r/4\rfloor & \text{if $d \geq 0$}\\
\lfloor r/4 + d/2 \rfloor & \text{if $d < 0$}.
\end{cases}$$
\end{lem}
\begin{proof}
Considering all multiples $π^\mbZ z$, the claim is equivalent to the following statement: There exists a lattice $Λ$ with $zΛ \subseteq Λ$ if and only if $z^2\in O_L$.

The ``only if'' direction is clear because $zΛ\subseteq Λ$ implies that $z^2$ has an integral characteristic polynomial. To prove the ``if'' direction, observe that $W$ is a free module of rank $1$ over $E\tensor_FL$ by Proposition \ref{prop:quaternion_algebra} (2). Pick any lattice $Λ' \subseteq W$ that is stable under $O_E\tensor_{O_F} O_L$. If $z^2\in O_L$, then $Λ = Λ' + zΛ'$ is preserved by $z$.
\end{proof}

Since $\mcT(π^kz) = \mcT(z)$ for all $k\in \mbZ$, it suffices to describe $\mcT(z)$ whenever $z^2\in O_L\setminus π^2O_L$, and in this case $\mcT(z) = \{Λ \mid zΛ\subseteq Λ\}$. We first treat the case of units.

\begin{prop}\label{prop:lattic_set_center_classification_unit}
Assume that $z^2 \in O_L^\times$. Then there exists an $L$-linear, $E$-conjugate linear involution $τ$ on $W$ that commutes with $z$ such that
$$O_E[z] = O_E[τ, z^2].$$
In particular, $\mcT(z)$ is the set of $O_E$-scalar extensions of $z^2$-stable $O_F$-lattices in $W^{τ = \mr{id}}$.
\end{prop}
\begin{proof}
Let $R = O_F[z^2]$ and denote by $\mfm$ its Jacobson radical. The norm map
$$N_{E/F}:\big(O_E\tensor_{O_F} (R/\mfm)\big)^\times \lr (R/\mfm)^\times$$
is surjective because $O_E/O_F$ is étale. It equals the map on $(R/\mfm)$-points of the smooth morphism $N_{E/F}:\mr{Res}_{O_E/O_F} \mbG_m \to \mbG_m$ of smooth $O_F$-group schemes. (The norm morphism is smooth because $O_E/O_F$ is étale.) Using completeness and a deformation argument, it follows that the map
$$N_{E/F}:\big(O_E\tensor_{O_F} R\big)^\times \lr R^\times$$
is surjective. Hence there exists an element $t\in (O_E\tensor_{O_F} R)^\times$ with $N_{E/F}(t) = z^2$. Then $τ := t^{-1}z$ lies in $O_E[z]$ and satisfies $τ^2 = \mr{id}$. The identity $O_E[z] = O_E[τ, z^2]$ follows directly.
\end{proof}

\begin{prop}\label{prop:lattice_set_center_classification_nonunit}
Assume that $z^2\in O_L \setminus π^2O_L$ is not a unit. Then $\mcT(z)$ takes the following shape.
\begin{enumerate}[wide, labelindent=0pt, labelwidth=!, label=(\arabic*), topsep=2pt, itemsep=2pt]
\item If $L$ is a field extension, then $\mcT(z)$ consists of a single edge.
\item If $L \iso F\times F$, then $\mcT(z)$ consists of an apartment.
\end{enumerate}
\end{prop}
\begin{proof}
Let $Λ$ be any lattice with $zΛ \subseteq Λ$, existence being ensured by Proposition \ref{prop:multiplicities}. Then $\ob{z}_Λ$ has to fall into Case (1) or (2) of the local classification Lemma \ref{lem:classif_emanating_edges}. Case (1) occurs precisely if $z$ is topologically nilpotent, which under the assumption $z^2\in O_L\setminus (π^2O_L \cup O_L^\times)$ is equivalent to $L$ being a field. (If $L$ is a field and $z^2\in O_L\setminus O_L^\times$, then $z^2$ is topologically nilpotent. Conversely, assume $L = F\times F$ and write $z^2 = (x,y)$. The quaternion algebra $(E\tensor L)[z]$ embeds into $\End_F(W)$ because this is the current setting, so is isomorphic to $M_2(L)$. Thus $z^2$ lies in the image of the norm map $E\tensor_FL\to L$ which means $v(x),v(y) \in 2\mbZ$. Hence $z^2 \notin π^2O_L$ implies that $z^2$ is not topologically nilpotent.)

Consider first Case (1). Then $Λ$ has precisely one neighbor in $\mcT(z)$, say $Λ'$. Then $\ob{z}_{Λ'}$ is again of Case (1) because the property of $L$ being a field (or $z$ being topologically nilpotent) is independent of the lattice. Thus $Λ'$ also has a unique neighbor in $\mcT(z)$ and hence $\mcT(z) = \{Λ, Λ'\}$ as claimed.

Consider now Case (2). Then $\ell_1 = \bigcap_{i\geq 0} z^iΛ$ and
$$\ell_2 = \{λ\in Λ \vert z^iλ \to 0\text{ as }i\to \infty\}$$
are complementary $z$-stable direct summand $O_E$-modules of $Λ$ of rank $1$. Picking non-zero $e_i\in \ell_i$, we see that every lattice $π^aO_Ee_1 \oplus π^bO_Ee_2$ is stable under $z$. These provide all elements of $\mcT(z)$ because any lattice in $\mcT(z)$ has exactly two neighbors in $\mcT(z)$ by the local classification Lemma \ref{lem:classif_emanating_edges} (2).
\end{proof}

\begin{rmk}\label{rmk:occuring_cases_GL4}
We observe that not all triples $(L, r, d)$ may occur. Namely the cyclic $L$-algebra $L[E,z]$ has center $L$ and embeds into $\End_F(W)$, so has to be isomorphic to $M_2(L)$. It follows that $z^2\in L$ is always a norm from $E\tensor_FL$:
\begin{enumerate}[wide, labelindent=0pt, labelwidth=!, label=(\arabic*), topsep=2pt, itemsep=2pt]
\item If $L$ is ramified, then this means that $r = v_L(z^2) \in 2\mbZ$. In particular, it will always be the case that $d \geq 0$ by Lemma \ref{lem:std_form_r_d} (3).

\item If $L = F\times F$ is split with, say, $z^2 = (z_1,z_2)$, then $v(z_1),\, v(z_2)\in 2\mbZ$. In particular, $r = v(z_1) + v(z_2) \in 4\mbZ$.

\item If $L$ is inert, then there is no such restriction on $z^2$.
\end{enumerate}
These possibilities are the ones that lead to rows 1 and 3 in Table \ref{table:matching} (take $δ = \Inv(1+z;T)$ which equals $\mr{char}_{L/F}(z;T)$ and gives $B_δ \iso L[E, z]$).
\end{rmk}

\begin{thm}\label{thm:classification_multiplicity_function}
The set $\mcT(z)$ takes the following shape, depending on the numercial invariant $(L, r, d)$ of $z$:
\begin{enumerate}[wide, labelindent=0pt, labelwidth=!, label=(\arabic*), topsep=2pt, itemsep=2pt]
\item If $L$ is inert and $r\equiv 0$ mod $4$, then $\mcT(z)$ is a $(q+1)$-regular ball of radius $d$ around a vertex.
\item If $L$ is inert and $r\equiv 2$ mod $4$, then $\mcT(z)$ is an edge.
\item If $L$ is ramified and $r\equiv 0$ mod $4$, then $\mcT(z)$ is a $(q+1)$-regular ball of radius $d$ around an edge.
\item If $L$ is ramified and $r\equiv 2$ mod $4$, then $\mcT(z)$ is an edge.
\item If $L \iso F\times F$, and if $z^2 = (z_1, z_2)$ has the property $v_F(z_1) = v_F(z_2)$, then $\mcT(z)$ is a $(q+1)$-regular ball of radius $d$ around an apartment.
\item If $L \iso F\times F$, and if $z^2 = (z_1, z_2)$ has the property $v_F(z_1) \neq v_F(z_2)$, then $\mcT(z)$ is an apartment.
\end{enumerate}
\end{thm}

\begin{figure}[h!]
\centering
\begin{minipage}{6.5cm}
\hspace{1cm}
\begin{tikzpicture}[scale=0.3]
  \node[ draw,circle,fill=black] (center) at (0,0) {}; 
  \foreach \i/\label in {1/1, 2/2, 3/3, 4/4, 5/5}{
    \pgfmathsetmacro\anglei{-90+72*\i}
    \pgfmathsetmacro\xcoordi{5*cos(\anglei)}
    \pgfmathsetmacro\ycoordi{5*sin(\anglei)}
  \ifnum\i=2\relax
      \draw[black,ultra thick] (center) -- (\xcoordi, \ycoordi);
      \node[draw,circle,fill=black] (node\i) at (\xcoordi,\ycoordi) {};
    \else\ifnum\i=3\relax
      \draw[black,ultra thick] (center) -- (\xcoordi, \ycoordi);
      \node[draw,circle,fill=black] (node\i) at (\xcoordi,\ycoordi) {};
    \else\ifnum\i=5\relax
      \draw[black,ultra thick] (center) -- (\xcoordi, \ycoordi);
      \node[draw,circle,fill=black] (node\i) at (\xcoordi,\ycoordi) {};
    \else\iftrue
      \draw[black] (center) -- (\xcoordi, \ycoordi);
      \node[draw,circle,fill=white] (node\i) at (\xcoordi,\ycoordi) {};
    \fi\fi\fi\fi
    \foreach \j/\labelj in {1/1, 2/2, 3/3, 4/4}{
      \pgfmathsetmacro\anglej{180+\anglei+72*\j}
      \pgfmathsetmacro\xcoordj{\xcoordi+2*cos(\anglej)}
      \pgfmathsetmacro\ycoordj{\ycoordi+2*sin(\anglej)}
      \node[draw, circle] (node\i\j) at (\xcoordj, \ycoordj) {};
      \draw (node\i) -- (node\i\j);
    } 
  }
\end{tikzpicture}
\end{minipage}
\begin{minipage}{8cm}
\hspace{2mm}
\begin{tikzpicture}[scale=0.8]
  \node (center) at (0,0) {}; 
  
  \foreach \i/\label in {1/1, 2/2}{
    \pgfmathsetmacro\anglei{180*\i}
    \pgfmathsetmacro\xcoordi{1.8*cos(\anglei)}
    \pgfmathsetmacro\ycoordi{1.8*sin(\anglei)}
    \node[draw,circle,fill=black] (node\i) at (\xcoordi,\ycoordi) {};
    \foreach \j/\labelj in {1/1, 2/2, 3/3, 4/4}{
      \pgfmathsetmacro\anglej{180+\anglei+72*\j}
      \pgfmathsetmacro\xcoordj{\xcoordi+2*cos(\anglej)}
      \pgfmathsetmacro\ycoordj{\ycoordi+2*sin(\anglej)}
     \ifnum\j=2\relax
      \node[draw, circle,fill=black] (node\i\j) at (\xcoordj, \ycoordj) {};
      \draw[black,ultra thick] (node\i) -- (node\i\j);
     \else\ifnum\j=3\relax
      \node[draw, circle,fill=black] (node\i\j) at (\xcoordj, \ycoordj) {};
      \draw[black,ultra thick] (node\i) -- (node\i\j);
     \else\iftrue
      \node[draw, circle] (node\i\j) at (\xcoordj, \ycoordj) {};
      \draw (node\i) -- (node\i\j);
     \fi\fi\fi

      \foreach \k/\labelk in {1/1, 2/2, 3/3, 4/4}{
        \pgfmathsetmacro\anglek{180+\anglej+72*\k}
        \pgfmathsetmacro\xcoordk{\xcoordj + 0.7*cos(\anglek)}
        \pgfmathsetmacro\ycoordk{\ycoordj + 0.7*sin(\anglek)}
        \node[draw,circle] (node\i\j\k) at (\xcoordk, \ycoordk) {};
        \draw (node\i\j) -- (node\i\j\k);
      }

    } 
  }
 \draw[black,ultra thick] (node1) -- (node2);
\end{tikzpicture}
\end{minipage}
\caption{Left: Case (1) of Theorem \ref{thm:classification_multiplicity_function} for $d = 1$ and $q = 2$. The set $\mcT(z)$ consists of a single vertex of valency $q+1$ and $q+1$ vertices of valency $1$ (black vertices). The ambient $(q^2+1)$-regular tree $\mcB$ is sketched (white vertices). Right: Similar sketch for case (3) of Theorem \ref{thm:classification_multiplicity_function} for $d = 1$ and $q = 2$.}
\label{fig:T}
\end{figure}
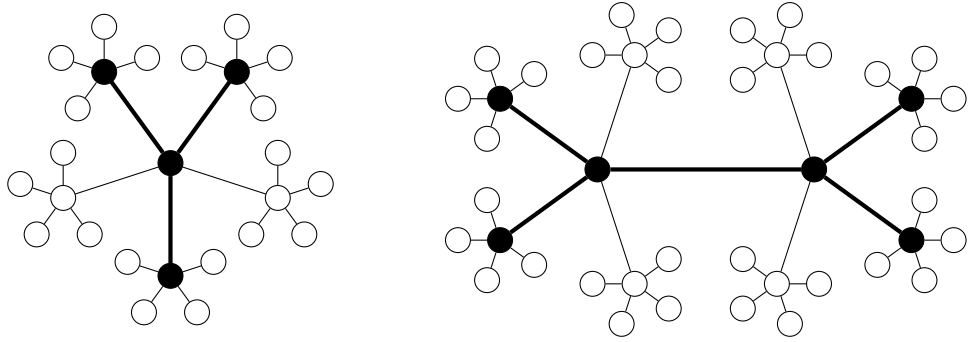
\begin{proof}
First consider the two cases when $L$ is a field and $r\equiv 0$ mod $4$. Then $(π^{-r/4}z)^2 \in O_L^\times$ and Proposition \ref{prop:lattic_set_center_classification_unit} states that $\mcT(z)$ is the set of homothety classes of $π^{-r/2}z^2$-stable $O_F$-lattices in a $1$-dimensional $L$-vector space. This set is well-known to be a $(q+1)$-regular ball around a vertex (resp. a ball around an edge) of radius equal to the conductor of $O_F[π^{-r/2}z^2]$. This conductor equals $d$ so (1) and (3) are proven.

Next, stick with the case that $L$ is a field but assume $r\equiv 2$ mod $4$. Then $(π^{(2-r)/4}z)^2$ is not a unit but lies in $O_L \setminus π^2O_L$. In this situation $\mcT(z)$ is edge by Proposition \ref{prop:lattice_set_center_classification_nonunit} (1). This proves (2) and (4).

Finally, assume $L \iso F\times F$ and define $k\in \mbZ$ through $(π^kz)^2 \in O_L \setminus π^2O_L$. Write $z^2 = (z_1, z_2)$ as in the proposition. If $v_F(z_1) = v_F(z_2)$, then $(π^kz)^2 \in O_L^\times$ and Proposition \ref{prop:lattic_set_center_classification_unit} states that $\mcT(z)$ is the set of homothety classes of $(π^kz)^2$-stable $O_F$-lattices in a free $L$-module of rank $1$. This set is well-known to be a $(q+1)$-regular ball around an apartment of radius equal to the conductor $d$ of $O_F[π^{-r/2}z^2]$ as claimed. This settles (5).

If, however, $v_F(z_1)\neq v_F(z_2)$, then $(π^kz)^2 \notin O_L^\times$ and Proposition \ref{prop:lattice_set_center_classification_nonunit} (2) states that $\mcT(z)$ is an apartment which proves (6).
\end{proof}

\section{Invariant $1/4$}
\label{s:1_4}

The aim of this section is to prove Theorem \ref{thm:main_ATC} for $λ = 1/4$. To this end, we first recall Drinfeld's linear algebra description of the map $\mcM_C \to \mcM_D$ in \S\ref{ss:Drinfeld_Theorem} and \S\ref{ss:basic_construction}. We will subsequently use this to compute all intersection numbers in question, our final result being the simple formulas in Theorem \ref{thm:main_1_4}.

\subsection{Drinfeld's Theorem}
\label{ss:Drinfeld_Theorem}
Let, for a moment, $D$ be a CDA over $F$ of Hasse invariant $1/n$. The main result of Drinfeld's paper \cite{Drinfeld} states that each connected component of the RZ space $\mcM_D$ from Definition \ref{def:RZ} is isomorphic to Deligne's formal scheme $\bOmega_F^{n-1}$. We will now formulate this result in more detail. Our main reference is \cite[\S3.54]{RZ1}, to which we also refer for more background.

Let $W$ be an $n$-dimensional $F$-vector space. By lattice chain in $W$, we mean a non-empty set $\mcL$ of $O_F$-lattices in $W$ that satisfies the following two conditions:
\begin{enumerate}[wide, labelindent=0pt, labelwidth=!, label=(\arabic*), topsep=2pt, itemsep=2pt]
\item $Λ,Λ'\in \mcL$ implies $Λ\subseteq Λ'$ or $Λ'\subseteq Λ$.
\item $Λ\in \mcL$ implies $π^\mbZ Λ\subseteq \mcL$.
\end{enumerate}
Denote the set of lattice chains in $W$ by $\mcW$. Given $\mcL \in \mcW$ and any lattice $Λ_0\in \mcL$, we may consider all lattices of $\mcL$ that are contained in $Λ_0$ and contain $πΛ_0$, say these are
\begin{equation}\label{eq:representing_chain}
πΛ_0 \subset Λ_k \subset \ldots \subset Λ_1 \subset Λ_0.
\end{equation}
Then $\mcL = \{π^\mbZ Λ_i,\ i = 0,\ldots,k\}$, so we call $Λ_k\subset \ldots \subset Λ_0$ a representing chain for $\mcL$. We next define a $π$-adic affine formal scheme $U_\mcL$ over $\Spf O_F$. Choosing a practical approach, we give the less canonical definition in terms of a representing chain \eqref{eq:representing_chain}. With this convention, the points $U_\mcL(S)$ for a $\Spf O_F$-scheme $S$ are the commutative diagrams of line bundle quotients
\begin{equation}\label{eq:U_L}
\xymatrix{
\mcO_S\tensor_{O_F} Λ_0 \ar[r]^{\mr{id}\tensor π} \ar@{->>}[d]^{φ_0} &
\mcO_S\tensor_{O_F} Λ_k \ar[r] \ar@{->>}[d]^{φ_k} &
\ldots \ar[r] &
\mcO_S\tensor_{O_F} Λ_1 \ar[r] \ar@{->>}[d]^{φ_1} &
\mcO_S\tensor_{O_F} Λ_0 \ar@{->>}[d]^{φ_0} \\
\mcL_0 \ar[r]^{α_k} &
\mcL_k \ar[r]^{α_{k-1}} &
\ldots \ar[r]^{α_1} &
\mcL_1 \ar[r]^{α_0} &
\mcL_0
}
\end{equation}
up to isomorphism in the pairs $(\mcL_i, α_i)$, and such that the following condition holds: The section $φ_i(λ_i)$ is invertible whenever $λ_i \in Λ_i \setminus Λ_{i+1}$ (for $i = 0,1,\ldots,k-1$) resp. $λ_k\in Λ_k\setminus πΛ_0$ (for $i = k$). A diagram of the form \eqref{eq:U_L} may be extended in a natural way to the full chain $\mcL$ and, in this way, the definition becomes independent of the chosen representing chain. The resulting $U_\mcL$ is isomorphic to a principal open subset of $\Spf O_F\langle T_0,\ldots, T_{n-1}\rangle/(T_0\cdots T_{n-1} - π)$. In particular, it is $π$-adic, $n$-dimensional, and regular with semi-stable reduction over $\Spf O_F$.

There are open immersions $U_{\mcL'} \subseteq U_\mcL$ for all inclusions of lattice chains $\mcL' \subseteq \mcL$. Their definition is based on the following simple observation. Assume that $Λ_2\subset Λ_1 \subset Λ_0$ are lattices and that we are given $(\mcL_2, φ_2)$, $(\mcL_0,φ_0)$ and $α_0\circ α_1$ in the following diagram
\begin{equation}\label{eq:U_L_chart_inclusion}
\xymatrix{
\mcO_S\tensor_{O_F} Λ_2 \ar[r] \ar@{->>}[d]^{φ_2} &
\mcO_S\tensor_{O_F} Λ_1 \ar[r] \ar@{-->>}[d]^{φ_1} &
\mcO_S\tensor_{O_F} Λ_0 \ar@{->>}[d]^{φ_0}\\
\mcL_2 \ar[r]^{α_1} &
\mcL_1 \ar[r]^{α_0} &
\mcL_0.
}
\end{equation}
Assume further that the outer square commutes and that $φ_0(λ)$ is invertible for all $λ\in Λ_0\setminus Λ_2$. Then there is a unique way (up to isomorphism) to fill in $(\mcL_1,φ_1)$ and to factor $α_0\circ α_1$ as depicted. Namely let $λ\in Λ_1\setminus Λ_2$. Then $α_0(φ_1(λ)) = φ_0(λ)$ has to be invertible, so $α_0$ has to be an isomorphism. Thus we may put $\mcL_1 = \mcL_0$, $α_0 = \mr{id}$ and $φ_1 = φ_0\vert_{\mcO_S\tensor Λ_1}$.
We leave it to the reader to extend this construction to lattice chains and diagrams as in \eqref{eq:U_L}.

Assume that $Λ_k\subset \ldots \subset Λ_0$ represents $\mcL$ as above and that $\mcL'\subseteq \mcL$ is a subchain. Let $I\subseteq \{0,\ldots, k\}$ be such that $Λ_i\in \mcL'$ if and only if $i\in I$. The above-constructed map $U_{\mcL'}\to U_\mcL$ identifies $U_{\mcL'}$ with the subfunctor of all those diagrams \eqref{eq:U_L} that have the property that $α_{i-1}$ is an isomorphism if $i\notin I$. The maps $U_{\mcL'} \to U_\mcL$ are hence open immersions. Uniqueness of the construction in \eqref{eq:U_L_chart_inclusion} ensures that the family $(U_{\mcL'}\to U_\mcL)_{\mcL'\subseteq \mcL}$ satisfies the cocycle condition.

Furthermore, every isomorphism $φ:W\to W'$ of $F$-vector spaces provides a compatible family of isomorphisms $φ:U_{\mcL}\overset{\sim}{\rightarrow} U_{φ(\mcL)}$. In particular, an element $g\in GL_F(W)$ defines a compatible family of isomorphisms
\begin{equation}\label{eq:action_on_Drinfeld_charts}
g:U_{\mcL} \overset{\sim}{\lr} U_{g\mcL}.
\end{equation}
In the following we write $GL_F(W)^0 = \{g\in GL_F(W)\mid v(\det(g)) = 0\}$. If, for example, $g\in GL_F(W)^0$ and $g\mcL = \mcL$, then this means that every lattice of $\mcL$ is $g$-stable. In this case, the $g$-action on $U_{\mcL}$ is the natural action of $g$ on diagrams of the form \eqref{eq:U_L}.

\begin{defn}\label{def:Deligne_formal_scheme}
Let $\Omega_F(W)$ denote the formal scheme that is obtained from the gluing datum $(U_{\mcL'}\to U_\mcL)_{\mcL'\subseteq \mcL}$. Let $GL_F(W)$ act on $\Omega_F(W)$ by the action that is chart-wise given by \eqref{eq:action_on_Drinfeld_charts}. We also write $\Omega_F^{n-1} := \Omega_F(F^n)$ in the case $W = F^n$.
\end{defn}

Let $\mcM_D^i \subset \mcM_D$ denote the open and closed formal subscheme of triples $(X, κ, ρ)$ such that the height of $ρ$ is $i$. Note that $(\mbX, κ, \mr{id})\in \mcM_D(\mbF)$, so $\mcM_D^0\neq \emptyset$. Furthermore, an element $g\in G_b \iso GL_n(F)$ provides an isomorphism
\begin{equation}\label{eq:connected_comp_Drinfeld}
g:\mcM_D^i \overset{\sim}{\lr} \mcM_D^{i + 4v_F(\det(g))}.
\end{equation}
Finally, a simple Dieudonné module argument shows that $\mcM_D^{i}= \emptyset$ if $i\notin 4\mbZ$. In this way, the following result provides a complete description of $\mcM_D$.

\begin{thm}[Drinfeld \cite{Drinfeld}]\label{thm:Drinfeld}
There is a $GL_n(F)^0$-equivariant isomorphism
\begin{equation}\label{eq:Drinfeld_iso}
O_{\breve F}\widehat{\tensor}_{O_F} \Omega_F^{n-1} \overset{\iso}{\lr} \mcM_D^{0}.
\end{equation}
Here, we let $G_b$ act from the left of $\mcM_D$ (instead of as from the right) by $g\mapsto g^{-1}$. In particular, $M_D^0$ is connected.
\end{thm}

The special fiber $\mbF_q\tensor_{O_F} \Omega_F(W)$ is a reduced scheme. Its set of irreducible components is in bijection with the homothety classes of lattices $Λ \subset W$. The irreducible component associated to $Λ$ is a blow up of the projective projective spaces $\mbP(\bar{Λ})$ centered in the union of all $\mbF_q$-rational hyperplanes of $\mbP(\bar{Λ})$. In the case $n = 2$, since a blow up does not affect smooth curves, the irreducible components of $\mbF_q\tensor_{O_F}\Omega^1_F$ are of the form $\mbP(\bar{Λ}) \iso \mbP^1_{\mbF_q}$.

In light of \eqref{eq:Drinfeld_iso}, we will mostly be interested in the base change of $\Omega_F(W)$ to $O_{\breve F}$. For this reason, we introduce the notation
\begin{equation}\label{eq:notation_Drinfeld}
\breve{\Omega}_F(W) := O_{\breve F}\widehat{\tensor}_{O_F} \Omega_F(W),\quad \breve U_{\mcL} := O_{\breve F}\widehat{\tensor}_{O_F} U_{\mcL}.
\end{equation}

\subsection{The Basic Construction}
\label{ss:basic_construction}
We now specialize to the situation of a $2$-dimensional $E$-vector space $W$. It is simultaneously a $4$-dimensional $F$-vector space. If $\mcL$ is a chain of $O_E$-lattices in $W$, then we write $U_{E,\mcL}\subseteq \Omega_E(W)$ for the corresponding chart. We also put
$$
\breve{\Omega}_E(W) := O_{\breve F}\widehat{\tensor}_{O_E} \Omega_E(W),\quad \breve U_{E,\mcL} := O_{\breve F}\widehat{\tensor}_{O_E} U_{E,\mcL}.
$$
Let $ζ\in O^\times_E$ be some fixed generator. It may be viewed as an element of $GL_F(W)^0$ and hence defines an automorphism of $\Omega_F(W)$. The isomorphism in Theorem \ref{thm:Drinfeld} is $GL_F(W)^0$-equivariant, so restricts to an isomorphism of $ζ$-fixed points
$$\mcM_D^{0, ζ} \iso \breve{\Omega}_F(W)^ζ.$$
By Proposition \ref{prop:closed_immersion}, $\mcM_C^{0}$ is contained in the fixed points $\mcM_D^{0, ζ}$. Our aim is to describe its image in $\breve{\Omega}_F(W)^ζ$.

\begin{prop}\label{prop:Drinfeld_fixed_points}
Precisely two of the connected components of $\breve{\Omega}_F(W)^ζ$ are flat over $\Spf O_{\breve F}$. Each of these is isomorphic to $\breve{\Omega}_E(W)$. The image of $\mcM_C^{0}$ along \eqref{eq:Drinfeld_iso} equals one of them.
\end{prop}
\begin{proof}
The fixed points $\Omega_F(W)^ζ$ are contained in the union of the charts $U_\mcL$ for $\mcL$ that satisfy $ζ\mcL = \mcL$. Since $v(\det(ζ)) = 0$, the condition $ζ\mcL = \mcL$ means that $ζ$ fixes each lattices of $\mcL$ individually, i.e. that $\mcL$ is a chain of $O_E$-lattices. Given an $O_E$-lattice $Λ$, there is a natural decomposition
\begin{equation}\label{eq:def_Lambda_pm}
O_E\tensor_{O_F} Λ = Λ^+ \oplus Λ^-
\end{equation}
because $E/F$ is unramified. Here, the notation is such that $Λ^+$ (resp. $Λ^-$) is the set of elements on which the two $O_E$-actions coincide (resp. differ by Galois conjugation). For a $\Spf O_E$-scheme $S$, a quotient line bundle
\begin{equation}\label{eq:line_bundle_quot}
\xymatrix{ \mcO_S\tensor_{O_F} Λ \ar@{->>}[r]^-{φ} & \mcL}
\end{equation}
is $ζ$-stable if and only if the quotient map factors over the projection to $\mcO_S\tensor_{O_E} Λ^+$ or over the projection to $\mcO_S\tensor_{O_E} Λ^-$.

Let $\mcL\in \mcW$ be a chain of $O_E$-lattices represented by $πΛ_0\subset Λ_1 \subset Λ_0$. Let $S$ be a $\Spf O_E$-scheme and consider a point of $U_\mcL^ζ(S)$ represented by
\begin{equation}\label{eq:U_L_over_E_to_F_I}
\xymatrix{
\mcO_S\tensor_{O_E} (Λ^+_0\oplus Λ^-_0) \ar[r]^{\mr{id}\tensor π} \ar@{->>}[d]^{(φ_0^+, φ_0^-)} &
\mcO_S\tensor_{O_E} (Λ^+_1\oplus Λ^-_1) \ar[r] \ar@{->>}[d]^{(φ_1^+, φ_1^-)} &
\mcO_S\tensor_{O_E} (Λ^+_0 \oplus Λ^-_0) \ar@{->>}[d]^{(φ_0^+, φ_0^-)}\\
\mcL_0 \ar[r]^{α_1} &
\mcL_1 \ar[r]^{α_0} &
\mcL_0.
}
\end{equation}
Then one can define a decomposition $S = S^+ \sqcup S^- \sqcup S^{\neq}$ into open and closed subschemes in the following way: $S^+$ is the locus where both $φ_0^-$ and $φ_1^-$ vanish. Similarly, $S^-$ is the locus where both $φ_0^+$ and $φ_1^+$ vanish. Finally, $S^{\neq}$ is the complement. This decomposition is functorial and hence defines a decomposition
$$O_E \widehat{\tensor}_{O_F} U_{\mcL}^ζ \ =\ U_{\mcL}^{ζ,+}\, \sqcup\, U_{\mcL}^{ζ,-}\, \sqcup\, U_{\mcL}^{ζ, \neq}.$$
It is furthermore compatible with gluing maps and stable under the $GL_E(W)$-action, and in particular defines a decomposition
$$O_E \widehat{\tensor}_{O_F} \Omega_F(W)^ζ \ = \ \Omega_F(W)^{ζ,+}\, \sqcup\, \Omega_F(W)^{ζ,-}\, \sqcup\, \Omega_F(W)^{ζ, \neq}.$$
The subscheme $\Omega_F(W)^{ζ, \neq}$ lies above the special point $\Spec \mbF_{q^2} \subset \Spf O_E$ and is hence nowhere flat. Indeed, assume for example that $φ_0^- = 0$ and $φ_1^+ = 0$. Then $φ_1^-$ is both a surjection onto a line bundle and $πφ_0^- = 0$ is divided by $φ_0^- = 0$. It follows that $π = 0$. The symmetric argument applies if $φ_0^+ = 0$ and $φ_1^- = 0$.

Recall from Theorem \ref{thm:Drinfeld} that $\mcM^0_C$ is a flat and connected $O_{\breve F}$-scheme. We conclude that the proof of the proposition will be complete if we can show that the two formal schemes $\Omega_F(W)^{ζ, \pm}$ are both isomorphic to $\Omega_E(W)$. To this end, first note that every $E$-conjugate linear element $τ\in GL_F(W)$ defines an isomorphism
\begin{equation}\label{eq:auto_Drinfeld}
τ:\Omega_F(W)^{ζ, +}\overset{\iso}{\lr} \Omega_F(W)^{ζ, -}.
\end{equation}
It hence suffices to describe an isomorphism of $\Omega_E(W)$ with $\Omega_F(W)^{ζ, +}$, say. Let $\mcL$ be a chain of $O_E$-lattices that is represented by $πΛ_0\subset Λ_1 \subset Λ_0$. Let $S$ be a $\Spf O_E$-scheme and consider an $S$-valued point of the chart $U_{E,\mcL}(S) \subset \Omega_E(W)(S)$ represented by
\begin{equation}\label{eq:U_L_over_E}
\xymatrix{
\mcO_S\tensor_{O_E} Λ^+_0 \ar[r]^{\mr{id}\tensor π} \ar@{->>}[d]^{φ_0} &
\mcO_S\tensor_{O_E} Λ^+_1 \ar[r] \ar@{->>}[d]^{φ_1} &
\mcO_S\tensor_{O_E} Λ^+_0 \ar@{->>}[d]^{φ_0}\\
\mcL_0 \ar[r]^{α_1} &
\mcL_1 \ar[r]^{α_0} &
\mcL_0.
}
\end{equation}
Map this datum to the following point of $U_{\mcL}^{ζ,+}(S)$:
\begin{equation}\label{eq:U_L_over_E_to_F_II}
\xymatrix{
\mcO_S\tensor_{O_E} (Λ^+_0\oplus Λ^-_0) \ar[r]^{\mr{id}\tensor π} \ar@{->>}[d]^{(φ_0,0)} &
\mcO_S\tensor_{O_E} (Λ^+_1\oplus Λ^-_1) \ar[r] \ar@{->>}[d]^{(φ_1,0)} &
\mcO_S\tensor_{O_E} (Λ^+_0 \oplus Λ^-_0) \ar@{->>}[d]^{(φ_0,0)}\\
\mcL_0 \ar[r]^{α_1} &
\mcL_1 \ar[r]^{α_0} &
\mcL_0.
}
\end{equation}
It is not difficult to check that this definition is compatible with gluing maps and defines an $GL_E(W)$-equivariant isomorphism $\Omega_E(W) \overset{\iso}{\to} \Omega_F(W)^{ζ, +}$; we omit these details. The proof of the proposition is now complete.
\end{proof}

Which of the two flat components of $\breve{\Omega}_F(W)^ζ$ the cycle $\mcM_C^{0}$ gets identified with depends on the choice of the comparison isomorphism in Theorem \ref{thm:Drinfeld}. We do not need to be more precise about this identification, however, because the definitions of $\mcI(g)$ and $\Int(g)$ in \S\ref{ss:intersection_numbers} are purely in terms of spaces with group actions and because \eqref{eq:auto_Drinfeld} allows to interchange the two flat components. So we will henceforth assume that the map $\mcM_C^{0} \to \mcM_D^{0}$ is given by the morphism from \eqref{eq:U_L_over_E} and \eqref{eq:U_L_over_E_to_F_II}.

\begin{rmk}\label{rmk:Drinfeld_iso}
In fact, this is also the result one would obtain from Drinfeld's construction during his proof of Theorem \ref{thm:Drinfeld}. Namely, his construction is such that the line bundles $\mcL_0$ and $\mcL_1$ in \eqref{eq:U_L_over_E_to_F_II} occur as direct summands of the Lie algebra of the corresponding special $O_D$-module. Demanding that $ζ$ acts strictly on the Lie algebra in the sense of Definition \ref{def:pi_div_group} precisely means to single out the component $\Omega_F(W)^{ζ, +}$.
\end{rmk}

\begin{rmk}\label{rmk:basic_construction}
The map $\Omega_E(W) \to \Omega_F(W)$ from Proposition \ref{prop:Drinfeld_fixed_points} was already considered by Drinfeld and called by him the ``basic construction''. His \cite[Proposition 3.1 (1)]{Drinfeld} is similar to our Proposition \ref{prop:Drinfeld_fixed_points}. It seems, however, that the flatness condition in Proposition \ref{prop:Drinfeld_fixed_points} cannot be omitted.
\end{rmk}

\subsection{Conormal Bundle}
\label{ss:conormal_Drinfeld_1_4}

Let $Λ = Λ_0 \subset W$ be an $O_E$-lattice. We write $U_Λ$ instead of $U_{π^\mbZ Λ}$. A similar convention will apply to $U_{Λ_1\subset Λ_0}$, $U_{E, Λ}$ and $U_{E, Λ_0\subset Λ_1}$. The special fiber $\mbF \tensor_{O_{\breve F}} \breve U_{E,Λ}$ of $\breve U_{E,Λ}$ is $GL_{O_E}(Λ)$-equivariantly isomorphic to
$$\mbF\tensor_{O_E} \mbP(Λ) \setminus \mbP(Λ)(\mbF_{q^2}) \iso \mbP^1_{\mbF} \setminus \mbP^1(\mbF_{q^2}).$$
Let $P_Λ$ denote its closure in $\breve{\Omega}_E(W)$. It is isomorphic to the projective line $\mbF\tensor_{O_E} \mbP(Λ)$.

\begin{prop}\label{prop:conormal_1_4}
Let $\mcC$ denote the conormal bundle of $\breve{\Omega}_E(W)\subset \breve{\Omega}_F(W)$. Then
$$\deg(\det \mcC\vert_{P_Λ}) = q^2 - 1.$$
\end{prop}
\begin{proof}
Our strategy is to choose a suitable generator of $(\det \mcC)\vert_{\breve U_{E,Λ}}$ and to determine the divisor of its meromorphic extension to $P_Λ$.

(1) Fix an $O_E$-basis $e_1,e_2$ for $Λ$. Write $O_E\tensor_{O_F} Λ = Λ^+ \oplus Λ^-$ as before. For an element $e\in Λ$, put
\begin{equation}\label{eq:basis_splitting}
e^+ = ζ\tensor e - 1 \tensor \ob{ζ} e\in Λ^+,\ \ \ e^- = ζ\tensor e - 1 \tensor ζe\in Λ^-.
\end{equation}
Then $(e_1^\pm,e_2^\pm)$ forms an $O_E$-basis of $Λ^\pm$. Let $φ:\mcO_{U_Λ} \tensor_{O_F} Λ \to \mcL$ be the universal quotient. Using that we are working over $O_{\breve F}$ which contains $O_E$, write $φ = (φ^+, φ^-)$ as in \eqref{eq:U_L_over_E_to_F_I}. Comparing \eqref{eq:U_L_over_E_to_F_I} with \eqref{eq:U_L_over_E_to_F_II}, we see that $\breve U_{E, Λ}\subset \breve U_Λ$ is defined by the condition $φ^- = 0$. Since $e_1^-$, $e_2^-$ is a basis of $Λ^-$, this is the same as the two conditions $φ(e_1^-) = φ(e_2^-) = 0$.

For every $λ\in Λ\setminus πΛ$, the image $φ(λ)\in \mcL$ is a generator. In particular,
$$φ(e_i^+ - e_i^-) = φ((ζ-\ob{ζ})\tensor e_i) \neq 0.$$
Since $φ(e_i^-)$ vanishes along $\breve U_{E, Λ}$ as seen before, $φ(e_i^+)$ is invertible near $\breve U_{E, Λ}$. Thus the two functions $φ(e_1^-)/φ(e_1^+)$ and $φ(e_2^-)/φ(e_2^+)$ are defined on a Zariski open neighborhood of $\breve U_{E,Λ}$ and generate the ideal defining $\breve U_{E,Λ}\subset \breve U_Λ$. Their wedge product
$$c_{(e_1,e_2)} := \frac{φ(e_1^-)}{φ(e_1^+)} \wedge \frac{φ(e_2^-)}{φ(e_2^+)}$$
is then a generator of $\det \mcC\vert_{\breve U_{E,Λ}}$.

(2) We next determine the behaviour of $c_{(e_1,e_2)}$ under change of basis. Let $f_1 = ae_1 + ce_2$ and $f_2 = be_1 + de_2$ for some $A = \left(\begin{smallmatrix} a & b \\ c & d\end{smallmatrix}\right) \in GL_2(O_E)$. Then
\begin{equation}\label{eq:trafo_gen_conormal}
\begin{aligned}
c_{(f_1,f_2)} & = \frac{φ(e_1^+)φ(e_2^+)}{φ(f_1)^+φ(f_2)^+} \cdot \frac{\big(\ob{a} φ(e_1^-) + \ob{c}φ(e_2^-)\big) \wedge \big(\ob{b}φ(e_1^-) + \ob{d}φ(e_2^-)\big)}{φ(e_1^+)φ(e_2^+)}\\
& = \det(\bar {A})\,\frac{φ(e_1^+)φ(e_2^+)}{φ(f^+_1)φ(f^+_2)}\, c_{(e_1,e_2)}.
\end{aligned}
\end{equation}
Poles and zeroes of the proportionality factor (when restricted to $P_Λ$) are described as follows. For two elements $e, f\in Λ$, the ratio $φ(e^+)/φ(f^+)$ is a scalar if and only if $e^+ \equiv f^+$ mod $π$. Otherwise, it is the rational function with simple zero at the line $\langle e^+\rangle$ and simple pole at $\langle f^+\rangle$.

(3) Let $Λ_1 \subset Λ$ be the $O_E$-lattice generated by $πe_1, e_2$. We claim that $c_{(e_1,e_2)}$ extends to a generator of $\det \mcC\vert_{\breve U_{E, Λ_1\subset Λ}}$. Consider for this the universal point of $\breve U_{Λ_1\subset Λ}$, say
\begin{equation}\label{eq:U_L_universal}
\xymatrix{
\mcO_{\breve U_{Λ_1\subset Λ}}\tensor_{O_F} Λ \ar[r]^{\mr{id}\tensor π} \ar@{->>}[d]^{φ} &
\mcO_{\breve U_{Λ_1\subset Λ}}\tensor_{O_F} Λ_1 \ar[r] \ar@{->>}[d]^{φ_1} &
\mcO_{\breve U_{Λ_1\subset Λ}}\tensor_{O_F} Λ \ar@{->>}[d]^{φ}\\
\mcL \ar[r]^{α_1} &
\mcL_1 \ar[r]^{α_0} &
\mcL.
}
\end{equation}
Since $φ(λ)$ is a generator of $\mcL$ for every $λ\in Λ\setminus Λ_1$ and since similarly $φ_1(λ_1)$ is a generator of $\mcL_1$ for every $λ_1\in Λ_1\setminus πΛ_0$, the ideal defining $\breve U_{E, Λ_1\subset Λ}\subset \breve U_{Λ_1\subset Λ}$ is generated by
$$\frac{φ(e_1^-)}{φ(e_1^+)} \quand \frac{φ_1(e_2^-)}{φ_1(e_2^+)}$$
on a Zariski open neighborhood of $\breve U_{E, Λ_1\subset Λ}$. The map $α_0$ becomes an isomorphism when restricting the diagram \eqref{eq:U_L_universal} to the open subset $\breve U_Λ\subset \breve U_{Λ_1\subset Λ}$. Since also $φ = α_0\circ φ_1$, we see that
$$\left.\frac{φ(e_1^-)}{φ(e_1^+)}\wedge\frac{φ_1(e_2^-)}{φ_1(e_2^+)} \right\vert_{\breve U_{E, Λ_1\subset Λ}} = c_{(e_1,e_2)}$$
as claimed.

This argument applies symmetrically to the lattice $\langle e_1, πe_2\rangle \subset Λ$. So we have shown that the element $c_{(e_1,e_2)}$, which is a meromorphic section of the line bundle $\det \mcC \vert_{P_Λ}$, has neither a zero nor a pole at the points $\langle e_1\rangle, \langle e_2\rangle \in \mbP(Λ)(\mbF_{q^2})$.

(4) It is left to show that $c_{(e_1, e_2)}$ extends with a simple zero over all other $\mbF_{q^2}$-rational points $\langle e_1\rangle, \langle e_2\rangle \neq \langle f\rangle \in \mbP(Λ)(\mbF_{q^2})$. We know from Step (3) that $c_{(f, e_2)}$ is a generator of $\det \mcC\vert_{P_Λ}$ at $\langle f\rangle$. From Step (2), we have that
$$c_{(e_1, e_2)} / c_{(f, e_2)}  \in O_E^\times \cdot φ(f^+)/φ(e_1^+).$$
Moreover, the function $φ(f^+)\vert_{P_Λ}$ vanishes with simple zero at $\langle f\rangle$ while $φ(e_1^+)$ is regular in $\langle f\rangle$ because $\langle f\rangle \neq \langle e_1 \rangle$. Thus we have proved that
$$\mr{div}(c_{(e_1,e_2)}) = P_Λ(\mbF_{q^2}) \setminus \{ \langle e_1 \rangle, \langle e_2 \rangle\}$$
and obtain the claimed identity $\deg(\det \mcC\vert_{P_Λ}) = q^2 - 1.$
\end{proof}

\subsection{Intersection numbers}
\label{ss:intersection_numbers_1_4}

Let $g = 1 + z_g \in GL_F(W)$ be a regular semi-simple element; set $z = z_g$. Recall that $\mcI(g) \neq \emptyset$ only for topologically nilpotent $z$ (Proposition \ref{prop:nilpotent_reduction}), so we also impose this condition on $z$. Then $g$ lies in $GL_F(W)^0$. Let $(L = F[z^2], r, d)$ be the numerical invariant of $z$. Let $\bOm_E(W) \to \bOm_F(W)$ be the closed immersion defined by \eqref{eq:U_L_over_E} and \eqref{eq:U_L_over_E_to_F_II}. Our aim is to determine the intersection locus
$$\bOmega_E(W) \cap g\cdot \bOmega_E(W).$$
Let $Λ\subseteq W$ be an $O_E$-lattice such that $zΛ\subseteq Λ$. Define $O_E\tensor_{O_F} Λ = Λ^+ \oplus Λ^-$ as in \eqref{eq:def_Lambda_pm}. Then $z$ satisfies $zΛ^+ \subseteq Λ^-$ and $zΛ^- \subseteq Λ^+$ because it is $E$-conjugate linear.
\begin{defn}\label{def:fixed_points_U_E_y}
Let $\mcL$ be a chain of $O_E$ lattices in $W$, represented by a single lattice $Λ$ or a pair $πΛ_0\subset Λ_1 \subset Λ_0$. We define $\breve U_{E, \mcL}^z \subseteq \breve U_{E, \mcL}$ as the closed formal subscheme of all those $S$-valued points
$$\begin{minipage}{4cm}
\xymatrix{\mcO_S\tensor_{O_E} Λ^+ \ar@{->>}[d]^{φ} \\ \mcL}
\end{minipage}
\ \ \ \ \ \mr{resp.} \ \ \ \ \ 
\begin{minipage}{15cm}
\xymatrix{
\mcO_S\tensor_{O_E} Λ^+_0 \ar[r]^{\mr{id}\tensor π} \ar@{->>}[d]^{φ_0} &
\mcO_S\tensor_{O_E} Λ^+_1 \ar[r] \ar@{->>}[d]^{φ_1} &
\mcO_S\tensor_{O_E} Λ^+_0 \ar@{->>}[d]^{φ_0}\\
\mcL_0 \ar[r]^{α_1} &
\mcL_1 \ar[r]^{α_0} &
\mcL_0
}
\end{minipage}
$$
that satisfy $[φ\circ z: Λ^- \to \mcL] = 0$, resp. $[φ_i\circ z: Λ_i^- \to \mcL_i] = 0$ for both $ i = 0,1$.
\end{defn}

\begin{prop}\label{prop:stratif_intersection_Drinfeld}
Let $\mcW^g$ denote the set of $g$-stable chains of $O_E$-lattices. Then
$$\bOmega_E(W) \cap g\cdot \bOmega_E(W) = \bigcup_{\mcL \in \mcW^g} \breve U_{E, \mcL}^z.$$
\end{prop}
\begin{proof}
Consider a chart $\breve U_{E, \mcL}\subset \bOmega_E(W)$. Its image under $g$ is contained in $\breve U_{g\mcL}$, which can only intersect $\bOmega_E(W)$ non-trivially if $g\mcL$ is again a chain of $O_E$-lattices. This is equivalent to $zΛ\subseteq Λ$ because $z$ is topologically nilpotent by assumption (compare Lemma \ref{lem:lattice_combinatorics} (1)). Then we obtain that $g\mcL = \mcL$. Thus we find
$$\bOmega_E(W) \cap g\cdot\bOmega_E(W) = \bigcup_{\mcL \in \mcW^g} \breve U_{E, \mcL} \cap g\cdot \breve U_{E, \mcL}.$$
Recall that $zΛ^\pm \subseteq Λ^\mp$. So given an $S$-valued point $(\mcL, (φ,0))$ resp. $(\mcL_i, (φ_i,0))_{i = 0,1}$ of $\breve U_{E, \mcL}$ as in \eqref{eq:U_L_over_E_to_F_II}, we obtain that
$$g(\mcL, (φ, 0)) = (\mcL, (φ, φ\circ z)),\quad \text{resp.}\quad g\cdot (\mcL_i, (φ_i, 0))_{i = 0,1} = (\mcL_i, (φ_i, φ_i \circ z))_{i = 0,1}.$$
This point lies again in $\breve U_{E, \mcL}$ if and only if $φ\circ z$ vanishes, resp. $φ_i\circ z$ for both $i = 0,1$ vanishes.
\end{proof}

Recall that we defined the function $n(z, Λ) = \max\{k\in \mbZ \mid zΛ\subseteq π^kΛ\}$ in \S\ref{s:multiplicity_functions}. Denote by $m(z, Λ) := \max\{0, n(z,Λ)\}$ its non-negative cut-off.
\begin{prop}\label{prop:multiplicities_1_4}
Let $g = 1+z_g \in GL_F(W)$ be regular semi-simple with $z = z_g$ topologically nilpotent. Then $m(z, Λ)$ equals the multiplicity of $P_Λ$ in $\bOm_E(W)\cap g\cdot \bOm_E(W)$ in the sense that
$$(\bOmega_E(W) \cap g\cdot \bOmega_E(W))^{\mr{pure}} \ =\  \sum_{\{Λ \subset W\text{ $O_E$-lattice}\}/E^\times} m(z,Λ) \cdot [P_Λ]$$
as $1$-cycles on $\bOm_E(W)$. Here, the pure locus is meant in the sense of Definition \ref{def:purification}.
\end{prop}
\begin{proof}
By Proposition \ref{prop:stratif_intersection_Drinfeld}, the multiplicity of $P_Λ$ can only be positive if $zΛ\subseteq Λ$. In this situation, it equals the maximal integer $k$ such that
\begin{equation}\label{eq:vanishing_condition_open_chart}
π^k \mid [φ\circ z: Λ^- \lr \mcL],
\end{equation}
where $(\mcL, φ)$ denotes the universal point over $\breve U_{E, Λ}$. This integer is evidently equal to $n(z, Λ)$.
\end{proof}

Definition \ref{def:purification} also provides a definition of the artinian locus $(\bOm_E(W)\cap g\cdot \bOm_E(W))^{\mr{art}}$. Furthermore, recall that we defined $\mcT(z)$ as the set of homothety classes of $O_E$-lattices in which $n(z,-)$ takes its maximum (Definition \ref{def:Tz}). Also recall the following terminology for points on $\Omega_E(W)$:

\begin{defn}\label{def:superspecial}
A closed point of $\bOmega_E(W)$ is called
superspecial if it is defined over $\mbF_{q^2}$. The superspecial points are hence precisely the intersection points $P_Λ\cap P_{Λ'}$ for lattice chains $πΛ\subset Λ' \subset Λ$ and in bijection with the edges of $\mcB$.
\end{defn}

\begin{prop}\label{prop:embedded_components_1_4}
The artinian part $(\bOmega_E(W) \cap g\cdot\bOmega_E(W))^{\mr{art}}$ is non-empty if and only if $\mcT(z)$ is an edge and $r \in 4\mbZ + 2$. In this case, the artinian part is of length one and located in the superspecial point of that edge.
\end{prop}
\begin{proof}
We first reconsider the situation from \eqref{eq:vanishing_condition_open_chart}. Write $z = π^{m(z, Λ)}z_0$. Then, by definition of the artinian part, $φ\circ z_0$ is a defining equation for $(\bOm_E(W)\cap g\cdot \bOm_E(W))^{\mr{art}}$ on $\breve U_{E, Λ}$. If the kernel of $z_0$ is non-zero, however, then it defines an $\mbF_{q^2}$-point of $P_Λ$ which, in particular, does not lie in $\breve U_{E, Λ}$. It follows that the support of $(\bOmega_E(W) \cap g\cdot \bOmega_E(W))^{\mr{art}}$ is contained in the superspecial points. We next compute the local equations in such a point with Proposition \ref{prop:stratif_intersection_Drinfeld}.

Let $πΛ_0\subset Λ_1\subset Λ_0$ be a representative of a chain of $O_E$-lattices. Assume that $zΛ_i\subseteq Λ_i$, otherwise $\breve U_{E, \mcL}^z = \emptyset$.
Pick a compatible basis, say $Λ_0 = O_Ee_1 + O_Ee_2$ and $Λ_1 = πO_Ee_1 + O_Ee_2$. Then $Λ_0^{\pm}$ and $Λ_1^{\pm}$ have the bases $(e_1^\pm,e_2^\pm)$ and $(πe_1^\pm, e_2^\pm)$, see \eqref{eq:basis_splitting}. In these coordinates, the universal point over $\breve U_{E, \mcL}$ may be written as
\begin{equation}\label{eq:U_L_universal_2}
\xymatrix{
\mcO_S\tensor_{O_E} Λ^+_0
\ar[r]^{\left(\begin{smallmatrix} 1 & \\ & π\end{smallmatrix}\right)}
\ar@{->>}[d]^{φ_0 = (1\ \bu)} &
\mcO_S\tensor_{O_E} Λ^+_1
\ar[r]^{\left(\begin{smallmatrix} π & \\ & 1\end{smallmatrix}\right)}
\ar@{->>}[d]^{φ_1 = (\bv\ 1)} &
\mcO_S\tensor_{O_E} Λ^+_0 \ar@{->>}[d]^{φ_0 = (1\ \bu)}\\
\mcO_{U_{E,\mcL}} \ar[r]^{\bv} &
\mcO_{U_{E,\mcL}} \ar[r]^{\bu} &
\mcO_{U_{E,\mcL}}
}
\end{equation}
where $\breve U_{E,\mcL} \subset \Spf O_E\langle \bu, \bv\rangle/(\bu\bv-π)$ is an open that contains the superspecial point $P_{Λ_0}\cap P_{Λ_1} = V(\bu, \bv)$. We have already seen that $(\bOm_E(W)\cap g\cdot \bOm_E(W))^{\mr{art}}$ is supported in superspecial points. So we henceforth work over the formal completion $\Spf O_E[\![\bu, \bv]\!]/(\bu\bv - π)$. Write $z = \left(\begin{smallmatrix} a & b\\ πc & d\end{smallmatrix}\right)σ \in M_2(O_E)σ$ with respect to the basis $(e_1,e_2)$. Here, $σ\in \mr{Gal}(E/F)$ denotes the Galois conjugation. Note that $σ(e_i^+) = e_i^-$ and $σ(e_i^-) = e_i^+$. Thus the map from $Λ_0^-$ to $Λ_0^+$ defined by $z$ is given by $\left(\begin{smallmatrix} a & b\\ πc & d\end{smallmatrix}\right)$ with respect to the bases $(e_1^+, e_2^+)$ and $(e_1^-, e_2^-)$. The vanishing conditions defining $\breve U_{E, \mcL}^z \cap \Spf O_E[\![\bu, \bv]\!]/(\bu\bv-π)$ then become
\begin{equation}\label{eq:embedded_comp_1_4}
\left(\begin{matrix} \bu & 1\end{matrix}\right)
\left(\begin{matrix} a & b\\ πc & d \end{matrix}\right) = 0,\quad\mr{and}\quad
\left(\begin{matrix} 1 & \bv\end{matrix}\right)
\left(\begin{matrix} a & πb\\ c & d \end{matrix}\right) = 0.
\end{equation}
Note that
$$
\left(\begin{matrix} 1 & \bv\end{matrix}\right)
\left(\begin{matrix} a \\ c \end{matrix}\right) = 0 \ \ \implies \ \ 
\left(\begin{matrix} \bu & 1\end{matrix}\right)
\left(\begin{matrix} a \\ πc\end{matrix}\right) = 0.
$$
and
$$
\left(\begin{matrix} \bu & 1\end{matrix}\right)
\left(\begin{matrix} b \\ d \end{matrix}\right) = 0 \ \ \implies \ \ 
\left(\begin{matrix} 1 & \bv\end{matrix}\right)
\left(\begin{matrix} πb \\ d \end{matrix}\right) = 0.
$$
Therefore, \eqref{eq:embedded_comp_1_4} is equivalent to just
\begin{equation}\label{eq:local_large_open_Drinfeld}
\begin{cases}
a+c \bv & = 0 \\
b\bu+d & =0.
\end{cases}
\end{equation}
Write
$$
\left(\begin{matrix} a & b\\ πc & d \end{matrix}\right) =
\pi^m \left(\begin{matrix} a' & b' \\ πc' & d'\end{matrix}\right),$$
where $m = m(z,Λ_0)$ is chosen maximally. We claim that the ideal $(a+c\bv, b\bu+d)$ is principal unless $a',d' \in πO_E$ and $b',c' \in O_E^\times$. We furthermore claim that if the ideal is not principal, then it equals $π^m(\bu, \bv)$.

Note that $(a+c\bv, b\bu+d)$ is principal if and only if $P_{Λ_0}\cap P_{Λ_1} \notin (\bOm_E(W)\cap g\cdot \bOm_E(W))^{\mr{art}}$. Moreover, if it equals $π^m(\bu, \bv)$, then
$$(\bOm_E(W)\cap g\cdot \bOm_E(W))^{\mr{art}} \cap \Spf O_E[\![\bu, \bv]\!]/(\bu\bv - π) = V(\bu, \bv).$$
In order to prove the claim, observe that
$$a + c\bv \in π^rR^\times\ \ \ \text{or}\ \ \ a + c\bv\in π^r\bv R^\times$$
and
$$b\bu + d \in π^sR^\times\ \ \ \text{or}\ \ \ b\bu + d\in π^s\bu R^\times$$
for a uniquely determined pair of integers $(r,s)$. The only possibility for \eqref{eq:local_large_open_Drinfeld} giving a non-principal ideal is $r = s$, in which case $r = s = m$ and
$$\begin{cases}
a + c\bv \in π^m \bv R^\times\\
b\bu + d \in π^m \bu R^\times.
\end{cases}$$
This is equivalent to $a',d'\in πO_E$ and $b',c' \in O_E^\times$, which proves our claim. The property $a',d'\in πO_E$ and $b',c' \in O_E^\times$ implies that $m = \max n(z,-)$ and that $r - 4m = v(\det_E(π^{-2m}z^2)) = 2$. Since $r\in 4\mbZ$ whenever $L$ is split (see Remark \ref{rmk:occuring_cases_GL4}), this shows that we are in cases (2) or (4) of Theorem \ref{thm:classification_multiplicity_function} as claimed. (Being in one of these two cases is equivalent to $r\in 4\mbZ + 2$ and $\mcT(z)$ being an edge.)

Conversely, assume that $r\in 4\mbZ +2$ and that $\mcT(z)$ is an edge. Let $πΛ_0\subset Λ_1\subset Λ_0$ be the lattices representing that edge. Choose a compatible basis $e_1, e_2$ of $Λ_0$ as above. We have $m(z, Λ_0) = m(z, Λ_1) =: m$ because $\mcT(z) = \{Λ_0, Λ_1\}$ by assumption. In other words there are $a', b', c', d'\in O_E$ such that $π^{-m}z$ is given by the matrices
$$\begin{pmatrix}
a' & b'\\ πc' & d'
\end{pmatrix}\quad \text{and}\quad \begin{pmatrix}
a' & πb'\\ c' & d'
\end{pmatrix}$$
with respect to the bases $e_1,e_2\in Λ_0$ and $πe_1, e_2 \in Λ_1$. Each of these two matrices has an invertible entry because $m$ was chosen maximally. Furthermore, both matrices are still topologically nilpotent because $v(\det_E(π^{-2m}z^2)) = 2$. Thus $a', d' \in πO_E$ and $b', c'\in O_E^\times$. The previous calculation now shows that $P_{Λ_0}\cap P_{Λ_1} \in (\bOm_E(W)\cap g\cdot \bOm_E(W))^{\mr{art}}$ with local ring of length $1$ as claimed.
\end{proof}

We define the following auxiliary intersection number. Write $L^\times = Γ\times O_L^\times$ for some subgroup $Γ\subset L^\times$ as in Definition \ref{def:intersection_number} and let $Γ_0 = Γ\cap GL_F(W)^0$. Then $Γ_0 = \{1\}$ if $L$ is a field or $Γ_0 \iso \mbZ$ if $L\iso F\times F$. Note that the action of $L^\times$ preserves both $\bOm_E(W)$ and $g\cdot \bOm_E(W)$, so we can define
$$\Int_0(g) = \langle Γ_0\backslash \bOmega_E(W),\ Γ_0\backslash g\cdot \bOmega_E(W)\rangle_{Γ_0\backslash \bOmega_F(W)}.$$
\begin{prop}\label{prop:main_1_4_aux}
This intersection number is given by the following formula.
\begin{enumerate}[wide, labelindent=0pt, labelwidth=!, label=(\arabic*), topsep=2pt, itemsep=2pt]
\item If $L$ is a field, then $\Int_0(g) = r/2$.
\item If $L \iso F\times F$, then $\Int_0(g) = 0$.
\end{enumerate}
\end{prop}
\begin{proof}
Proposition \ref{prop:multiplicities_1_4} states that the multiplicity of $P_Λ$ in $\bOmega_E(W)\cap g\cdot \bOmega_E(W)$ is $m(z, Λ)$. Define
$$p_Λ := -m(z, Λ)\left[(q^2 - 1) + \langle P_Λ,\ (\bOmega_E(W)\cap g\cdot \bOmega_E(W))^{\mr{pure}}\rangle\right].$$
Here, the term $q^2-1$ is the degree of the conormal bundle (Proposition \ref{prop:conormal_1_4}) and the intersection pairing is that of divisors on $\bOmega_E(W)$. By Corollary \ref{cor:intersection_simplified}, we have
\begin{equation}\label{eq:intersection_number_1_4_aux_formula}
\Int_0(g) = \mr{len}\left(\mcO_{Γ_0\backslash \left(\bOmega_E(W) \cap g\cdot \bOmega_E(W)\right)^{\mr{art}}}\right) \ + \ \sum_{Λ\ \in\ Γ_0\backslash \{\text{$O_E$-lattices in $W$}\}/E^\times} p_Λ.
\end{equation}
We next compute the summands $p_Λ$ for all $Λ$ with $m(z, Λ) \geq 1$. Put $m = \max m(z, -)$ and $\mcT = \mcT(z)$; assume that $m\geq 1$. By \cite[Lemma 4.7]{KR}, the intersection numbers of the curves $P_Λ$ are given by
$$\langle P_\Lambda, P_{\Lambda'}\rangle_{\bOmega_E(W)} = \begin{cases}0 & P_\Lambda \cap P_{\Lambda'} = \emptyset\\
1 & P_\Lambda \cap P_{\Lambda'} = \{\mr{pt}\}\\
-(q^2+1) & P_\Lambda = P_{\Lambda'}.\end{cases}$$
\begin{enumerate}[wide, labelindent=0pt, labelwidth=!, label=(\roman*), topsep=2pt, itemsep=2pt]
\item First assume that $Λ\notin \mcT$. Then $Λ$ has some multiplicity $i=m(z, Λ)$ with $1\leq i < m$. Precisely $q^2$ of its neighbors have multiplicity $i-1$ and a single neighbor has multiplicity $i+1$ (Proposition \ref{prop:multiplicities}). Thus
$$p_Λ = -i\left[(q^2-1) - i (q^2+1) + q^2(i-1) + (i+1)\right] = 0.$$
\item Now assume that $Λ$ has multiplicity $m$, i.e. lies in $\mcT$. By Theorem \ref{thm:classification_multiplicity_function}, the valency $v_Λ$ of $Λ$ in $\mcT$ is $0$, $1$, $2$ or $q+1$. Then $Λ$ has $v_Λ$ many neighbors of multiplicity $m$ and $q^2 + 1 - v_Λ$ many neighbors with multiplicity $m-1$. It follows that
$$p_Λ = -m\left[(q^2-1) - m(q^2+1) + mv_Λ + (m-1)(q^2+1 - v_Λ)\right] = (2-v_Λ)m.$$
\end{enumerate}
We now evaluate \eqref{eq:intersection_number_1_4_aux_formula} for the six possible shapes of $\mcT$ from Theorem \ref{thm:classification_multiplicity_function}. An observation that applies in all cases is that $p_Λ \neq 0$ only for $Λ\in \mcT$, see (i) above, so the discussion will only involve the set $\mcT$. Moreover, Proposition \ref{prop:embedded_components_1_4} states that the artinian part $(\bOm_E(W)\cap g\cdot \bOm_E(W))^{\mr{art}}$ is of length $1$ precisely in cases (2) and (4), and $0$ otherwise. We will also use Lemma \ref{lem:maximum_of_multiplicity_function} in every case to relate $m$ with $r$.

\begin{enumerate}[wide, labelindent=0pt, labelwidth=!, label=(\arabic*), topsep=2pt, itemsep=2pt]
\item Assume that $L/F$ is inert and that $r \in 4\mbZ$. Then $\mcT$ is a $(q+1)$-regular ball of radius $d$ around a single vertex, $Γ_0 = \{1\}$, there are no embedded components, and $4m = r$. If $d = 0$, then (ii) above shows that
$$\Int_0(g) = 2m = r/2.$$
For $d >1$, let $A$ be the number of vertices of $\mcT$ with valency $1$ and let $B$ be the number of those with valency $q+1$. It is easy to check that $A - (q-1)B = 2$ for every $d\geq 1$. Applying (ii) again, we find
$$\Int_0(g) = m (A - (q-1)B) = 2m = r/2.$$

\item Assume that $L/F$ is inert and that $r \in 4\mbZ + 2$. Then $\mcT$ is an edge, $Γ_0 = \{1\}$, there is a single embedded component of length $1$, and $4m + 2 = r$. We obtain from (ii) that
$$\Int_0(g) = 1 + 2m = r/2.$$

\item Assume that $L/F$ is ramified and that $r \in 4\mbZ$. Then $\mcT$ is a $(q+1)$-regular ball of radius $d$ around an edge, $Γ_0 = \{1\}$, there are no embedded components, and $4m = r$. If $d = 0$, then (ii) immediately shows
$$\Int_0(g) = 2m = r/2.$$
For $d\geq 1$, let again $A$ be the number of vertices of $\mcT$ with valency $1$ and let $B$ be the number of those with valency $q+1$. It is again checked that $A - (q-1)B = 2$ for every $d\geq 1$. Applying (ii) again, we find
$$\Int_0(g) = m (A - (q-1)B) = 2m = r/2.$$

\item Assume that $L/F$ ramified and that $r \in 4 \mbZ + 2$. Just like in case (2), we obtain
$$\Int_0(g) = r/2.$$

\item Assume that $L = F\times F$ and that $z^2 = (z_1, z_2)$ has the property $v(z_1) = v(z_2)$. Then $\mcT$ is a $(q+1)$-regular ball of radius $d$ around an apartment. The action of the group $Γ_0 \iso \mbZ$ on this apartment is by a translation with two orbits. Moreover, there is no artinian contribution.

Assume first that $d = 0$. Then every $Λ\in \mcT$ has valency $2$ and hence $p_Λ = 0$ by (ii) above. It follows that
$$\Int_0(g) = 0.$$
Assume now that $d\geq 1$. Let $A$ be the number of vertices of $Γ_0\backslash \mcT$ of valency $1$ and let $B$ denote those of valency $q+1$. One checks that $A - (q-1)B = 0$ for all $d \geq 1$, so
$$\Int_0(g) = m[A - (q-1)B] = 0.$$

\item Assume finally that $L \iso F\times F$ and that $z^2 = (z_1, z_2)$ has the property $v(z_1) \neq v(z_2)$. Then $\mcT$ is an apartment on which $Γ_0$ acts with two orbits. There is no artinian contribution and one obtains just as before that
$$\Int_0(g) = 0.$$
\end{enumerate}
\end{proof}

We can now determine the intersection numbers $\Int(g)$ for $D = D_{1/4}$ and prove our arithmetic transfer conjecture (Conjecture \ref{conj:ATC_explicit}) in this situation. Let the notation be as in \S\ref{s:main_results}; in particular, $G' = GL_4(F)$ and $G_b$ denote the two groups that intervene in the formulation of the AT conjecture. Let $f'_\Par$ and $f'_\Iw$ denote the two test functions from \S\ref{s:main_analytic}.

\begin{thm}\label{thm:main_1_4}
Let $g\in G_{b, \mr{rs}}$ be a regular semi-simple element with numerical invariants $(L, r, d)$. The intersection number $\Int(g)$ is non-zero only if $r > 0$. In this case, it is given by
\begin{equation}\label{eq:int_number_formula_1_4}
\Int(g) = \begin{cases} r & \text{$L/F$ unramified}\\
r/2 & \text{$L/F$ ramified}\\
0 & \text{$L/F$ split}.\end{cases}
\end{equation}
In particular, Conjecture \ref{conj:ATC_explicit} holds for $D = D_{1/4}$ with correction function $- 4q\log(q) \cdot f'_\Par.$ In other words, for every regular semi-simple $γ\in G'_{\mr{rs}}$,
\begin{equation}\label{eq:AT_1_4_main}
\partial O(γ,\, f'_D) - 4q \Orb(γ,\, f'_\Par)\log(q) = \begin{cases} 2\,\Int(g)\log(q) & \text{if there exists a matching $g\in G_b$}\\
0 & \text{otherwise.}
\end{cases}
\end{equation}
\end{thm}
\begin{proof}[Proof of Identity \eqref{eq:int_number_formula_1_4}.]
The statement about the vanishing of $\Int(g)$ for $r \leq 0$ follows from Proposition \ref{prop:nilpotent_reduction}. We henceforth assume that $r > 0$ and even that $z_g$ is topologically nilpotent.

Recall that $\mcM_D^i\subset \mcM_D$ and $\mcM_C^i \subset \mcM_C$ denote the connected components triples $(Y, ι, ρ)$ resp. $(X, κ, ρ)$ where the height of $ρ$ is $i$. Also recall that $\mcM_D^i$ and $\mcM_C^i$ non-empty precisely if $i\in 4\mbZ$. Moreover, an element $h\in G_b \iso GL_4(F)$, resp. $h\in H_b \iso GL_2(E)$, has the property
$$h:\mcM_D^i \overset{\iso}{\lr} \mcM_D^{i + 4v(\det_F(h))},\quad \mr{resp.}\quad h:\mcM_C^i \overset{\iso}{\lr} \mcM_C^{i + 4v(\det_E(h))}.$$
By definition, the Serre tensor construction doubles the height, i.e. is such that $\mcM_C^i = \mcM_C \cap \mcM_D^{2i}$.

Recall that we wrote $L^\times = Γ \times O_L^\times$ and $Γ_0 = Γ\cap GL_F(W)^0$ before. Let $Γ_1\subseteq Γ$ be a complement to $Γ_0$ and let $θ\in Γ_1 \cap O_L$ be a generator. Then
$$θ\mcM_D^i = \begin{cases} \mcM_D^{i + 8} & \text{if $L/F$ is ramified or split}\\
\mcM_D^{i + 16} & \text{if $L/F$ is unramified.}\end{cases}$$
In other words, $Γ_1 \backslash π_0(\mcM_C) = \{0\}$ or $\{0,8\}$, depending on the case. Thus if $L$ is ramified or split, then
$$\Int(g) = \Int_0(g)$$
and we are done by Theorem \ref{prop:main_1_4_aux}. If $L$ is inert however, then we obtain
$$\Int(g) = \Int_0(g) + \langle Γ_0 \backslash \mcM_C^4,\ Γ_0\backslash (g\cdot \mcM_C^4)\rangle_{Γ_0\backslash \mcM_D^8}.$$
Let $h\in H_b$ be any with $v_E(\det(h)) = 1$. Then $h:\mcM_D^0\overset{\sim}{\to} \mcM_D^8$ as well as $h^{-1}(\mcM_C^4) = \mcM_C^0$ and $h^{-1}(g\mcM_C^4) = hgh^{-1} \mcM_C^0$. (Recall that the $G_b$-action is a right action.) We obtain that
$$\Int(g) = \Int_0(g) + \Int_0(hgh^{-1}).$$
But $hgh^{-1}$ and $g$ lie in the same $H_b$ double coset, so have the same numerical invariant $(L, r, d)$. Proposition \ref{prop:main_1_4_aux} shows that $\Int_0(g)$ only depends on the numerical invariant, so we obtain $\Int_0(hgh^{-1}) = \Int_0(g)$ and then $\Int(g) = r$ as claimed.
\end{proof}

\begin{proof}[Proof of Identity \eqref{eq:AT_1_4_main}.]
Let $γ\in G'_{\mr{rs}}$ be a regular semi-simple element with numerical invariant $(L, r, d)$. First consider the case that $r$ is odd. Then there is no matching element $g\in G_b$, see rows 2 and 5 of Table \ref{table:matching}, so we need to show that the left hand side of \eqref{eq:AT_1_4_main} vanishes.

The sign of the functional equation of $f'_\Par$ is $(-1)^r$ and hence negative if $r$ is odd. This shows $\Orb(γ, f'_\Par) = 0$. Proposition \ref{prop:derivative} for odd $r$ moreover states that $\del(γ, f'_D) = 0$ which is the desired vanishing.

Now we consider the case where $r$ is even. There exists a matching element $g\in G_b$ for $γ$ if and only if $L$ is a field or if $L\iso F\times F$ and $r\in 4\mbZ$, see rows 1 and 3 of Table \ref{table:matching}. No matter which case, \eqref{eq:derivative_main_teaser} shows the equality of the two sides in \eqref{eq:AT_1_4_main}.
\end{proof}

\section{Invariant $3/4$}
\label{s:3_4}
The aim of this section is to prove Theorem \ref{thm:main_ATC} for Hasse invariant $3/4$. It will turn out, however, that the geometry for invariant $3/4$ is closely related to the one for invariant $1/4$. So we will, in fact, consider the two intersection problems for $D\in \{D_{1/4}, D_{3/4}\}$ simultaneously. For this reason we introduce the following notation: We write $\mcM_λ$, $λ\in \{1/4, 3/4\}$, for the RZ space for $D = D_λ$. We similarly write $G_λ = D_λ^{\mr{op}, \times}$ for the group $G$ and $G_{λ, b}$ for the group $G_b$.

We also choose compatible presentations of $D_{1/4}$ and $D_{3/4}$: Let $F_4/F$ denote an unramified field extension of degree $4$ and let $σ\in \mr{Gal}(F_4/F)$ be its Frobenius. For both choices of $λ$, we fix an embedding $F_4 \to D_λ$ and a uniformizer $Π\in D_λ$ that normalizes $F_4$ and satisfies $Π^4 = π \in F$. Then $Πa = σ(a)Π$ if $λ = 1/4$ and $Πa = σ^3(a)Π$ if $λ = 3/4$, for $a\in F_4$. We assume that the embedding $E\to D$ is such that $E\subset F_4$. Then $\varpi = Π^2$ is a uniformizer of $O_C$ and we obtain the presentation $C = F_4[\varpi]$.

\subsection{Conormal Bundle}
\label{ss:conormal_3_4}

\begin{prop}\label{prop:conormal_3/4}
Let $P\subseteq \mcM_C$ be any irreducible component of the special fiber. The degree of the conormal bundle $\mcC$ of $\mcM_C \to \mcM_{3/4}$ on $P$ is the same as in the case of invariant $1/4$,
$$\deg(\det \mcC\vert_P) = q^2 - 1.$$
\end{prop}
\begin{proof}
Our proof is by showing that the degrees of the conormal bundles for $\mcM_C\to \mcM_{1/4}$ and $\mcM_C \to \mcM_{3/4}$ agree. Then Proposition \ref{prop:conormal_1_4} yields that the degree is $q^2-1$ in both cases. So let $\mcM = \mcM_{1/4}$ or $\mcM = \mcM_{3/4}$ and let $\mcI\subseteq \mcO_{\mcM}$ be the ideal sheaf such that $\mcM_C = V(\mcI)$. The conormal bundle is $\mcI/\mcI^2$.

Let $(Y, ι, ρ)$ be the universal point over $\mcM_C$ and let $\mcD$ be the covariant $O_F$-Grothendieck--Messing crystal of $Y$ evaluated at the thickening $V(\mcI^2)$, viewed with trivial PD-structure. It is endowed with an $O_C = O_{F_4}[\varpi]$-action $ι$ by functoriality. This provides a $\mbZ/4$-grading $\mcD = \bigoplus \mcD_i$ where
$$\mcD_i = \{x\in \mcD \mid ι(a)(x) = σ^i(a)x,\ \ a\in O_{F_4}\}.$$
Then $\varpi$ is homogeneous of degree $2$ and each graded piece is a vector bundle of rank $2$.

Write $\ob{\mcD} = \mcD/\mcI\mcD$, $\ob{\mcD_i} = \mcD_i/\mcI \mcD_i$ and denote by $\ob{\mcF_i} \subset \ob{\mcD_i}$ the Hodge filtration of $Y$. Recall that $\mcD/\mcF = \Lie(Y)$ is the Lie algebra. The special condition (see Definition \ref{def:special}) in particular requires that $O_E \subset O_{F_4}$ acts via the natural map $O_E\to \mcO_{M_C}$ on $\Lie(Y)$ which implies that $\ob{\mcF_i} = \ob{\mcD_i}$ for $i = 1,3$.

Next, consider $X := O_D \tensor_{O_C} Y$ with its natural $O_D$-action. The evaluation of its $O_F$-Grothendieck--Messing crystal at $V(\mcI^2)$ is $\mcP := O_D\tensor_{O_C} \mcD$ by functoriality. The action of $O_{F_4}\subset O_D$ again provides a $\mbZ/4$-grading $\mcP = \bigoplus \mcP_i$. It may be refined as follows: Write $O_D = O_B \oplus ΠO_B$, where $Π$ is the previously chosen uniformizer of $D$. We denote by $ΠY$, $Π\mcD$ etc. the summands $Π\tensor Y$, $Π\tensor \mcD$ etc. Then $\mcP$ is a direct sum of eight terms:
\begin{equation}\label{eq:Grothendieck_Messing_3/4}
\mcP = \begin{array}{ccccccc}
\mcD_0 & \oplus & \mcD_1 & \oplus & \mcD_2 & \oplus & \mcD_3 \\
\oplus &        & \oplus &        & \oplus &        & \oplus \\
Π\mcD_{4λ} & \oplus & Π\mcD_{1 + 4λ} & \oplus & Π\mcD_{2 + 4λ} & \oplus & Π\mcD_{3 + 4λ},
\end{array}
\end{equation}
where $\mcP_i = \mcD_i \oplus Π\mcD_{i+4λ}$. The operator $Π$ acts homogeneously of degree $-4λ$. Let $\mcQ\subset \mcP$ denote the Hodge filtration of the restriction to $V(\mcI^2)$ of the universal point of $\mcM$. It is $O_D$-stable, meaning it is $Π$-stable and graded ($\mcQ = \bigoplus \mcQ_i$ with $\mcQ_i \subset \mcP_i$).

The ideal $\mcI/\mcI^2$ tautologically defines the closed subscheme $\mcM_C\subset V(\mcI^2)$. This subscheme is also characterized by the three properties from Proposition \ref{prop:closed_immersion}. Consider the first one, $\mcM_C\subseteq \mcZ(ι(O_E))$. The vertical grading $\mcP = \mcD\oplus Π\mcD$ in \eqref{eq:Grothendieck_Messing_3/4} also equals the decomposition into the two eigenspaces of $\mcP$ under the $κ(O_E)\tensor_{O_F} ρι(O_E)ρ^{-1}$-action. (This action exists on $O_D\tensor_{O_C}Y$ and lifts to the crystal evaluated at $V(\mcI^2)$.) Thus the intersection $V(\mcI^2) \cap \mcZ(ι(O_E))$ as closed subscheme of $V(\mcI^2)$ is defined by the condition that $\mcQ$ is vertically graded in the sense $\mcQ = (\mcQ\cap \mcD) \oplus (\mcQ\cap Π \mcD)$. As $\mcQ$ is already $\mbZ/4$-graded, this is equivalent to
\begin{equation}\label{eq:grading_Q}
\mcQ_i = (\mcQ_i\cap \mcD_i) \oplus (\mcQ_i \cap Π\mcD_{i+4λ})\quad \forall i = 0,\ldots,3.
\end{equation}
We claim that in fact $V(\mcI) = V(\mcI^2) \cap \mcZ(ι(O_E))$. For this we need to check that the further conditions (1) and (2) from Proposition \ref{prop:closed_immersion} are implied by \eqref{eq:grading_Q}.

Condition (1) just says that the rank of $\mcQ_i \cap \mcD_i$ is $1$ for $i = 0,2$ and $2$ for $i = 1,3$. This already holds on $V(\mcI)$ and extends to any infinitesimal thickening. (The rank of a locally free module is locally constant.)

Condition (2) states that $κ(Π)$ defines an isomorphism $Π:X_+\simto X_-$, where $X = X_+\oplus X_-$ is the decomposition into eigenspaces of $X$ defined on $V(\mcI^2)\cap \mcZ(ι(O_E))$. Just like (1) above, this condition can be checked over $V(\mcI)$.

In summary, we see that $V(\mcI)\subset V(\mcI^2)$ is defined by \eqref{eq:grading_Q}. This condition is further equivalent to $\mcD_i \subset \mcQ_i$ and $Π\mcD_{i+1+4λ} \subset \mcQ_{i+1}$ for $i = 1,3$, because these inclusions hold over $V(\mcI)$. Since $\mcQ$ is $Π$-stable, it is equivalent to only require $\mcD_i\subset \mcQ_i$ for $i = 1,3$. So we see that $\mcI/\mcI^2$ is defined by the vanishing of the two maps
$$\mcD_1 \lr \mcL_1,\quad \mcD_3 \lr \mcL_3,\quad\quad \mcL_i := \mcP_i/\mcQ_i.$$
These two maps are known to vanish modulo $\mcI$, so they factor over $\ob{\mcD_1}$ and $\ob{\mcD_3}$. We thus obtain an exact sequence of vector bundles on $\mcM_C$,
\begin{equation}\label{eq:ex_seq_conormal}
0 \lr \mcK \lr \ob{\mcD_1 \tensor \mcL_1^{-1}} \oplus \ob{\mcD_3 \tensor \mcL_3^{-1}} \lr \mcI/\mcI^2 \lr 0.
\end{equation}
Denote its middle term by $\mcE$. Note that $Π:\ob{\mcL_{i+4λ}} \iso \ob{\mcL_i}$ for $i = 1,3$ and that $\ob{\mcL_{1+4λ}} \oplus\ob{\mcL_{3+4λ}}$ is the Lie algebra of $Y$. It follows that the determinant of $\mcE$ is independent of whether $λ = 1/4$ or $3/4$. What is left to show is that the determinant of $\mcK\vert_{V(π)}$ is also independent. (Here, $V(π)$ denotes the special fiber of $\mcM_C$.) This relies on the commutative diagram
\begin{equation}\label{eq:diagram_conormal}
\begin{minipage}{5cm}
\xymatrix{
\ob{\mcD_1} \ar[r]^{\varpi} \ar[d]_{φ_1} & \ob{\mcD_3} \ar[r]^{\varpi} \ar[d]_{φ_3} & \ob{\mcD_1} \ar[d]_{φ_1}\\
\ob{\mcL_1} \ar[r]^{\varpi} & \ob{\mcL_3} \ar[r]^{\varpi} & \ob{\mcL_3}.
}
\end{minipage}
\end{equation}
Write $\varpi^{\vee}:\ob{\mcL_{i+2}}{}^{-1}\to \ob{\mcL_i}{}^{-1}$ for the dual map on inverse line bundles. We claim that $\mcK$ is generated by all sections of the form
\begin{equation}\label{eq:generators_conormal}
(u_1 \tensor \varpi^{\vee} s_3, - \varpi u_1 \tensor s_3),\ (\varpi u_3\tensor s_1, - u_3 \tensor \varpi^{\vee}s_1),\ \ \ u_i \in \ob{\mcD_i},\ s_{i+2}\in \ob{\mcL_{i+2}}{}^{-1}.
\end{equation}
To prove this, it is sufficient to locally exhibit elements of the form \eqref{eq:generators_conormal} that generate a rank $2$ direct summand of $\mcE$. The top row of \eqref{eq:diagram_conormal} has a normal form, meaning there locally exist bases $e_1,f_1$ of $\mcD_1$ and $e_3,f_3$ of $\mcD_3$ such that $\varpi$ is given by
$$\varpi \begin{pmatrix} e_1 \\ f_1 \end{pmatrix} = \begin{pmatrix}e_3 \\ πf_3\end{pmatrix} \quad \text{and}\quad
\varpi \begin{pmatrix}
e_3 \\ f_3
\end{pmatrix} = \begin{pmatrix}
πe_1 \\ f_1
\end{pmatrix}.$$
In particular, $\varpi e_1$ and $\varpi f_3$ are nowhere vanishing sections. Also assume that $s_i\in \ob{\mcL_i}{}^{-1}$, with $i = 1,3$, are local generators. Then
$$(e_1 \tensor \varpi^\vee s_3, -\varpi e_1 \tensor s_3)\quad \text{and}\quad (\varpi f_3 \tensor s_1, -f_3 \tensor \varpi^\vee s_1)$$
lie in $\mcK$ and are fiberwise linearly indpendent, hence generate $\mcK$.

From now on we restrict to the special fiber $V(π)\subset \mcM_C$. The above normal form statement implies that the outer terms in the following canonical exact sequences are line bundles:
$$0 \lr \ker(\varpi\vert_{\mcD_i}) \lr \mcD_i \lr \mr{Im}(\varpi\vert_{\mcD_i})\lr 0.$$
The previous computation specializes to the fact that
\begin{equation}\label{eq:determinant_key}
\mcK/π\mcK = \ker(\varpi\vert_{\mcD_1})\tensor \mcL_1^{-1} \oplus \ker(\varpi\vert_{\mcD_3})\tensor \mcL_3^{-1}
\end{equation}
as subsheaf of $\mcE/π\mcE$. The determinant of $\mcK/π\mcK$ is then the tensor product of all four line bundles that occur on the right hand side of \eqref{eq:determinant_key}. This product is independent of $λ$, as was to be shown.
\end{proof}

\subsection{Intersection Locus (Simplified Formulation)}
\label{ss:simplified_intersection_3_4}

Let $g\in G_{λ, b, \mr{rs}}$ be a regular semi-simple element. Write $z = z_g$. Our next aim is to rephrase the definition of $\mcI(g) = \mcM_C\cap g\cdot \mcM_C$ in a simpler way.

Since the framing object $(\mbY, ι)$ that goes into the definition of $\mcM_C$ has no étale part, Lemma \ref{lem:nilpotent_reduction} states that $\mcM_C\cap g\cdot \mcM_C = \emptyset$ unless $z$ is toplogically nilpotent. So we assume for the following discussion that $z$ is topologically nilpotent. Then Proposition \ref{prop:nilpotent_reduction} states that
\begin{equation}\label{eq:intersection_locus_description}
\mcM_C \cap g\cdot \mcM_C = \mcM_C \cap \mcZ(z).
\end{equation}
(We recall that $\mcZ(z)$ denotes all those $(X, κ, ρ)\in \mcM_λ$ such that $ρzρ^{-1} \in \End(X)$, see \eqref{eq:def_cycle}.) In terms of the element $Π\in D_λ$ we chose at the beginning of \S\ref{s:3_4}, we have $O_{D_λ} = O_C \oplus Π O_C$ and obtain a presentation of $(\mbX, κ)$ as
\begin{equation}\label{eq:presentation_Serre_tensor}
\mbX = \mbY \oplus Π \mbY,\quad
κ(Π) = \left(\begin{matrix} & ι(\varpi)\\ 1 & \end{matrix}\right),\quad
κ(c) = \left(\begin{matrix} ι(c) & \\ & ι(Π^{-1}cΠ) \end{matrix}\right)\ \ \ c\in C.
\end{equation}
The endomorphism ring of $(\mbX, κ)$ is then
\begin{equation}\label{eq:endomorphism_ring_3/4}
\End^0_D(\mbX, κ) = \left.\left\{
\left(\begin{matrix} x & y\varpi \\ y & x \end{matrix}\right)\right\vert\ x\in \End^0_C(\mbY),\ y\in \End_F^0(\mbY) \text{ s.th. } yc = Π^{-1}cΠy \text{ for all }c\in C
\right\}.
\end{equation}
Description \eqref{eq:presentation_Serre_tensor} applies to every Serre tensor construction, not just to framing objects. Thus, writing $z = \left(\begin{smallmatrix} & y\varpi \\ y & \end{smallmatrix}\right)$, we obtain that
\begin{equation}\label{eq:final_description_intersection_locus}
\mcM_C \cap \mcZ(z) = \mcZ(y) = \{(Y, ι, ρ)\in \mcM_C \mid ρyρ^{-1} \in \End(Y)\}.
\end{equation}
The automorphism $c\mapsto Π^{-1}cΠ$ of $C$ satisfies $Π^{-1}\varpi Π = \varpi$ but its effect on $F_4$ depends on $λ$: It is given by $Π^{-1}a Π = σ^{-1}(a)$ if $λ = 1/4$ and by $Π^{-1}aΠ = σ^{-3}(a)$ if $λ = 3/4$. For both choices of $λ$ we define, with $Π = Π_λ\in D_λ$,
\begin{equation}\label{eq:def_set_S_of_endomorphisms}
S_λ = \{ y\in \End^0_F(\mbY)^\times\ \vert\ ycy^{-1} = Π_λ^{-1}cΠ_λ\ \text{for}\ c\in C\}.
\end{equation}
Let $S$ be the union $S_{1/4}\sqcup S_{3/4}$. Then $\varpi S_λ = S_{λ + 1/2}$ and, for every $y\in S$, we have inclusions of closed subschemes of $\mcM_C$
$$\mcZ(\varpi^{-1} y) \subseteq \mcZ(y) \subseteq \mcZ(\varpi y).$$
This relates the intersection loci for the two different invariants. Note that for every $y\in S$, the element $\varpi y^2$ lies in the centralizer $\End^0_C(\mbY)$ which is isomorphic to $M_2(E)$. Moreover, if $z =\left(\begin{smallmatrix} & y \varpi \\ y & \end{smallmatrix}\right)$ with $y\in S$ such that $1 + z\in G_λ$, then the following relation of invariant polynomials holds:
\begin{equation}\label{eq:rel_z_y_invariant}
\Inv(1 + z; T) = \Inv(y; T) := \mr{charred}_{\End^0_C(\mbY)/E}(\varpi y^2; T).
\end{equation}
(Here, the right hand side will always lie in $F[T]$.) We call $y\in S$ regular semi-simple if $\Inv(y;T)$ is regular semi-simple in the sense of \S\ref{s:invariants}. Let $S_{\mr{rs}}$ and $S_{λ,\mr{rs}}$ denote the regular semi-simple elements of $S$ and $S_λ$. The main task in the following sections is to determine the formal scheme $\mcZ(y)\subseteq \mcM_C$ for $y\in S_{3/4, \mr{rs}}$.

\subsection{Intersection Locus (Set-theoretic Support)}
\label{ss:inter_set_theoretic}
Given $y\in S_{\mr{rs}}$, our first result describes the support $\mcZ(y)(\mbF)$ in terms of Dieudonné theory. To this end, we first recall from \cite{BC} some more details on Drinfeld's isomorphism.

\begin{construction}
Let $(M, F, V, ι)$ be the covariant $O_F$-Dieudonné module of a special $O_C$-module $(Y, ι)$ over $\mbF$. Fix an embedding $F_4\to \breve F$. Then the contained ring of integers $O_{F_4}\subset O_C$ induces a $\mbZ/4\mbZ$-grading $M = M_0\oplus \ldots \oplus M_3$. Each summand is free of rank $2$ as $O_{\breve F}$-module and the operators $F$, $\varpi$ and $V$ are all homogeneous of degrees
$$\deg F = -1,\ \ \deg V = 1, \ \ \deg \varpi = 2.$$
It follows from the special condition that
\begin{equation}\label{eq:degree_V}
[M_i: VM_{i-1}] = \begin{cases} 1 & \text{if }i \equiv 0 \text{ mod }2\\
0 & \text{if }i \equiv 1 \text{ mod }2\end{cases}
\end{equation}
and that $[M_i:\varpi M_{i-2}] = 1$ for all $i$. Since $\varpi^2 = π$ vanishes on $\Lie(Y) = M/VM$, there exists an index $i\in \{0,2\}$ such that $\varpi M_i = V^2M_i$. Such indices are called critical.

The existence of critical indices implies that the operator $τ = V^{-2}\varpi$ is homogeneous of degree $0$ and $σ^2$-linear with all slopes $0$. We put $Λ_i = M_i^{τ = \mr{id}}$, which is a free rank $2$ module over $O_E = O_{\breve F}^{σ^2 = \mr{id}}$. There are two cases:
\begin{equation}\label{eq:cases_tau_invariants}
O_{\breve F} \tensor_{O_E} Λ_i = \begin{cases} M_i & \text{if $i$ critical}\\
\varpi M_{i-2} & \text{if $i$ not critical}.
\end{cases}
\end{equation}
Assume that $i$ is critical. We obtain a line $\ell = \varpi M_{i-2}/πM_i \in \mbP(Λ_i/πΛ_i)(\mbF)$ and the triple $(i, Λ_i, \ell)$ allows for a unique (up to isomorphism) reconstruction of $(M, F, V, ι)$.

Let us now bring the framing object $(\mbY, ι)$ into play. Denote its isocrystal by $(N, F, V, ι)$. We have seen that $τ := V^{-2}\varpi$ is of degree $0$ and $σ^2$-linear with all slopes $0$. Put $W = N_0^{τ = \mr{id}}$, which is a $2$-dimensional $E$-vector space. (Recall that this is a general statement: If $N$ is an $n$-dimensional $\breve F$-vector space and $τ:N\to N$ a $σ^t$-linear automorphism with all slopes $0$, then $N^{τ = \mr{id}}$ is an $n$-dimensional $F_t$-vector space where $F_t = \breve F^{σ^t = \mr{id}}$ is the degree $t$ unramified field extension of $F$. Moreover, we have $(\breve F\tensor_{F_t} N^{τ = \mr{id}}, σ^t \tensor \mr{id}) \simto (N, τ)$.)

We may define a map $\mcM_C(\mbF)\to \bOmega_E(W)(\mbF)$ by the following construction: An $ι(O_C)$-stable and special Dieudonné lattice $M\subseteq N$ with $Λ_i = M_i^{τ = \mr{id}}$ as above is sent to
\begin{equation}\label{eq:Drinfeld_iso_explicit}
\begin{cases} \varpi M_2 / πM_0 \in \mbP(Λ_0/πΛ_0)(\mbF) & \text{if $0$ is critical}\\
M_0 / \varpi^{-1}(πM_2) \in \mbP(\varpi^{-1}Λ_2/\varpi^{-1}(πΛ_2))(\mbF) & \text{if $2$ is critical}.\end{cases}
\end{equation}
It may happen, of course, that both indices are critical. In this case, the two lines in \eqref{eq:Drinfeld_iso_explicit} coincide as points of $\bOmega_E(W)$ and correspond to the diagram
$$\xymatrix{
\varpi Λ_2 \ar[r] \ar@{->>}[d] & Λ_0 \ar[r] \ar@{->>}[d] & \varpi^{-1} Λ_2 \ar@{->>}[d]\\
\varpi M_2 / πM_0 \ar[r]^0 & M_0/\varpi M_2 \ar[r]^0 & \varpi^{-1}M_2 / M_0,
}$$
where the lower outer terms have to be identified along $π:\varpi^{-1}M_2/M_0 \iso \varpi M_2/πM_0$. The restriction of the map $\mcM_C(\mbF)\to \bOmega_E(W)(\mbF)$ to the height $0$ connected component $\mcM_C^0(\mbF)$ agrees with the map from Drinfeld's isomorphism in Theorem \ref{thm:Drinfeld}.
\end{construction}
\begin{defn}\label{def:operator_w}
Given $y\in S_λ$, consider its action on the isocrystal $(N, F, V, ι)$ of $(\mbY, ι)$. Then $y$ is homogeneous of degree
$$\deg y = \begin{cases} -1 & \text{if } λ = 1/4\\
1 & \text{if } λ = 3/4.\end{cases}$$
Define $w(y) = Vy$ if $λ = 1/4$ and $w(y) = V^{-1}y$ if $λ = 3/4$. Then $w(y)$ is of degree $0$ and commutes with $V$ and $\varpi$. It hence commutes with $τ$ and acts as a $E$-conjugate linear endomorphism on $W = N_0^{τ = \mr{id}}$.
\end{defn}

In the following, we will also formulate some results for the invariant $1/4$ case. We will not use these again but hope that they clarify why and in which sense the two possibilities for $λ$ are different.

\begin{lem}\label{lem:descript_Z_y_Dieudonne}
Let $(Y, ι, ρ) \in \mcM_C(\mbF)$ be a point with Dieudonné lattice $M = M_0\oplus\ldots \oplus M_3 \subset N$. Let $i\in \{0,2\}$ be a critical index of $(Y,ι)$ and let $Λ = Λ_0$ (if $i = 0$) or $Λ = \varpi^{-1}Λ_2$ (if $i = 2$) be the resulting lattice $Λ\subset W$. Let $\ell \in P_Λ(\mbF)$ be the line defined by \eqref{eq:Drinfeld_iso_explicit}. Then $(Y,ι,ρ)\in \mcZ(y)$ if and only if
\begin{equation}\label{eq:properties_Dieudonne}
\begin{cases}
\text{$w(y)Λ\subseteq Λ$ and $w(y)\ell = 0$ and $\mr{Im}(w(y))\subseteq \ell$} & \text{if $λ = 1/4$}\\
\text{$w(y)Λ\subseteq Λ$ and $w(y)\ell \subseteq \ell$} & \text{if $λ = 3/4$}.
\end{cases}
\end{equation}
\end{lem}
\begin{proof}
Assume first that $λ = 1/4$ so that $\deg y = -1$. Using that $VM_0 = M_1$ and $VM_2 = M_3$, see \eqref{eq:degree_V}, we check that
\begin{enumerate}[wide, labelindent=0pt, labelwidth=!, label=(\arabic*), topsep=2pt, itemsep=2pt]
\item $yM_{i+1} \subseteq M_i$ if and only if $yVM_i\subseteq M_i$, i.e. $w(y) Λ \subseteq Λ$,
\item $yM_{i+2}\subseteq M_{i+1}$ if and only if $y \varpi M_{i+2} \subseteq \varpi V M_i$, meaning $Vy \varpi M_{i+2} \subseteq πM_i$, i.e. $w(y) \ell = 0$,
\item $yM_{i+3}\subseteq M_{i+2}$ if and only if $Vy \varpi M_{i+2} \subseteq \varpi M_{i+2}$, i.e. $w(y)\ell \subseteq \ell$, and
\item $yM_i \subseteq M_{i+3}$ if and only if $Vy V^{-2}\varpi M_i \subseteq \varpi M_{i+2}$, i.e. $\mr{Im}(w(y)) \subseteq \ell$.
\end{enumerate}
If $λ = 3/4$ however, then $\deg y = 1$ and we obtain slightly different conditions:
\begin{enumerate}[wide, labelindent=0pt, labelwidth=!, label=(\arabic*), topsep=2pt, itemsep=2pt]
\item $yM_i \subseteq M_{i+1}$ if and only if $V^{-1}yM_i \subseteq M_i$, i.e. $w(y) Λ\subseteq Λ$,
\item $yM_{i+1}\subseteq M_{i+2}$ if and only if $V^{-1}y M_i \subseteq V^{-2}M_{i+2}$ which is redundant after (1),
\item $yM_{i+2} \subseteq M_{i+3}$ if and only if $V^{-1}y \varpi M_{i+2} \subseteq \varpi M_{i+2}$, i.e. $w(y) \ell \subseteq \ell$, and
\item $yM_{i+3} \subseteq M_i$ if and only if $V^{-1}y \varpi M_{i+2} \subseteq V^{-2}\varpi M_i = M_i$ which is redundant after (3).
\end{enumerate}
These two lists of properties are precisely what was claimed in \eqref{eq:properties_Dieudonne}.
\end{proof}

For a homothety class of lattices $Λ\subset W$ and $y\in S$, we define
\begin{equation}\label{eq:def_nm}
n(y, Λ) := n(w(y), Λ),\ \ \ m(y, Λ) := \max\{0, n(y, Λ)\}.
\end{equation}
Also let $\mcZ(y)^0 := \mcZ(y)\cap \mcM^0_C$. Lemma \ref{lem:descript_Z_y_Dieudonne} shows that, under the isomorphism $\mcM_C^{0}\iso \bOmega_E(W)$,
$$\mcZ(y)^0(\mbF)\subseteq \bigcup_{Λ\subseteq W,\ n(y, Λ)\geq 0} P_Λ(\mbF).$$
Recall that $\mcT(w(y)) \subset \mcB$ denotes the set of those homothety classes of lattices $Λ\subseteq W$ in which $n(y, Λ)$ takes its maximum and that the shape of $\mcT(w(y))$ was described in Theorem \ref{thm:classification_multiplicity_function}. The next corollary combines this result with Lemma \ref{lem:descript_Z_y_Dieudonne}.

\begin{cor}\label{cor:support_intersection_set_theoretic}
Assume that $y\in S_{\mr{rs}}$ is regular semi-simple. The set $\mcZ(y)^0(\mbF)$ has the following description, in dependence on $λ$ and the maximum $n(y) = \max_{Λ\subseteq W} n(y, Λ)$ of the multiplicity function.
\begin{enumerate}[wide, labelindent=0pt, labelwidth=!, label=(\arabic*), topsep=2pt, itemsep=2pt]
\item Assume $n(y)<0$. Then $\mcZ(y)^0 = \emptyset$.
\item Assume $n(y) = 0$ and $λ = 1/4$. Then $\mcZ(y)^0 \neq \emptyset$ if and only if $w(y)$ is topologically nilpotent. In this case, $\mcT(w(y))$ is an edge and $\mcZ(y)^0(\mbF)$ the corresponding superspecial point.
\item Assume $n(y) = 0$ as well as $λ = 3/4$ and $\det(w(y)^2) \in O_E^\times$. Then $\mcZ(y)^0 \cap P_Λ(\mbF) \neq \emptyset$ if and only if $Λ\in \mcT(w(y))$. In the non-empty case,
\begin{equation}\label{eq:counting_fps}
|\mcZ(y)^0\cap P_Λ(\mbF)| = q+1.
\end{equation}
Moreover, for every edge of $\mcT(w(y))$, the corresponding superspecial point lies in $\mcZ(y)^0(\mbF)$.
\item Assume $n(y) = 0$ as well as $λ = 3/4$ and $\det(w(y)^2) \in O_E\setminus O_E^\times$. Then $\mcZ(y)^0(\mbF)$ consists of the superspecial points that correspond to edges of $\mcT(w(y))$.
\item Assume $n(y) \geq 1$. Then
$$\mcZ(y)^0(\mbF) = \bigcup_{Λ \in \{ Λ \mid n(y, Λ)\geq 1\}} P_Λ(\mbF).$$
\end{enumerate}
\end{cor}
\begin{figure}[h!]
\centering
\begin{tikzpicture}[scale=0.5]
    \draw[ultra thick] (-7,0) -- (7,0) node[right]{$\mathbb P^1$}; 
    \foreach \a in {-4,0,4} {
        \draw[ultra thick] plot[domain=-4:4, samples=100] (0.03*\a*\x*\x + \a, \x); 
	\pgfmathtruncatemacro\labelcurve{\a/4+2}
	\draw (0.03*\a*4*4+\a,4) node[above] {};
	\foreach \y in {-2,0,2}{
            \filldraw[black] (0.03*\a*\y*\y + \a,\y) circle (6pt); 
	}
	\foreach \y in {-3,-1,...,3}{
		\draw[black,thick,dashed] (0.03*\a*\y*\y + \a-0.8,\y)--(0.03*\a*\y*\y + \a+0.8,\y);
	}
    }
    \foreach \a in {-2,2} {
        \draw[dashed,black, thick] plot[domain=-3.5:3.5, samples=100] (0.03*\a*\x*\x + \a, \x); 
    }
\end{tikzpicture}
\caption{Illustration of case (3) of Corollary \ref{cor:support_intersection_set_theoretic} for $L/F$ inert and $q = 2$. Each line represents a curve of the special fiber of $\mcM^0_C$. The four thick lines correspond to the homothety classes $Λ$ with $n(y, Λ) = 0$. Their dual graph is depicted on the left in Figure \ref{fig:T}. The scheme $\mcZ(y)^0$ consists of $q+1$ points on each thick line. For the central curve, these points are all superspecial. For the remaining $q+1$ curves, one point is superspecial and the other $q$ are non-superspecial.}
\label{fig:embedded}
\end{figure}
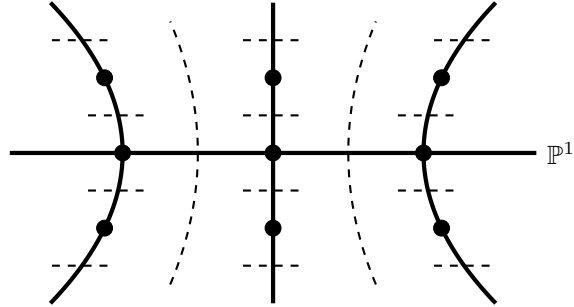
\begin{proof}
Cases (1) and (5) follow immediately from Lemma \ref{lem:descript_Z_y_Dieudonne}. For case (2), observe that \eqref{eq:properties_Dieudonne} implies that $\mcZ(y)^0\neq \emptyset$ can only hold if $w(y)$ is topologically nilpotent. Then Proposition \ref{prop:lattice_set_center_classification_nonunit} (1) ensures that $\mcT(w(y))$ is an edge and Lemma \ref{lem:classif_emanating_edges} shows that the unique $w(y)$-stable line in $P_Λ(\mbF)$ is the corresponding superspecial point. By Lemma \ref{lem:descript_Z_y_Dieudonne}, this point lies in $\mcZ(y)^0(\mbF)$ .

We turn to cases (3) and (4). Lemma \ref{lem:descript_Z_y_Dieudonne} states that $\mcZ(y)^0\cap P_Λ(\mbF)$ is non-empty only if $w(y)Λ \subseteq Λ$. Under the assumption $n(y) = 0$, this is equivalent to $Λ\in \mcT(w(y))$. Moreover, if $w(y)Λ \subseteq Λ$, then it states that $\mcZ(y)^0\cap P_Λ(\mbF)$ equals the set of $w(y)$-stable lines $\ell \in P_Λ(\mbF)$.

In case (3), if $Λ\in \mcT(w(y))$, then $w(y)$ defines a $σ$-linear automorphism of $\mbF \tensor_{\mbF_{q^2}}Λ$. It is a simple fact that every $σ$-linear automorphism of a two-dimensional $\mbF$-vector space has precisely $q+1$ fixed lines, so we obtain Identity \eqref{eq:counting_fps}. Moreover, edges emanating from $Λ$ in $\mcT(w(y))$ are in bijection with the $\mbF_{q^2}$-rational $w(y)$-fixed points in $P_Λ(\mbF)$. In particular, each such edge defines a point of $\mcZ(y)^0\cap P_Λ(\mbF)$.

The arguments for case (4) are the same. The only difference is that $w(y)\vert_Λ$, for $Λ\in \mcT(w(y))$, is not invertible anymore. This implies that every $w(y)$-stable line $\ell \in P_Λ(\mbF)$ is defined over $\mbF_{q^2}$ and hence comes from an edge of $Λ$ in $\mcT(w(y))$ as claimed.
\end{proof}

\subsection{Cartier Theory}
Let $R$ be an $O_F$-algebra in which $π$ is nilpotent. Then formal $π$-divisible groups over $R$ are equivalent to reduced $O_F$-Cartier modules over $R$, cf. \cite[\S1]{Drinfeld}. We will use this equivalence to compute the scheme structure of $\mcZ(y)$ and, to this end, collect some general results in this section.

We denote by $W_{O_F}(R)$ the ring of $O_F$-Witt vectors of $R$ with respect to the chosen uniformizer $π$ as in \cite[\S1]{Drinfeld}. For $x\in W_{O_F}(R)$, we write ${}^Fx$ and ${}^Vx$ for Frobenius and Verschiebung. They satisfy the relations
$${}^F [r] = [r^q],\quad {}^{FV}x = πx,\quad \left({}^Vx\right) \cdot y = {}^V\left(x\cdot {}^Fy\right).$$
We denote by $E(R) = W_{O_F}(R)[F, V]$ the $O_F$-Cartier ring over $R$. It is the non-commutative ring generated over $W_{O_F}(R)$ by two elements $F$ and $V$ that satisfy the relations (where $x\in W_{O_F}(R)$)
$$VxF = {}^Vx,\quad Fx = {}^FxF,\quad xV = V{}^Fx,\quad FV = π.$$
The third relation implies that if $M$ is an $E(R)$-left module, then the action of $W_{O_F}(R)$ on $M/VM$ factors through $W(R)/\,{}^VW(R) = R$.
\begin{defn}\label{def:Cartier}
A reduced $O_F$-Cartier module over $R$ is an $E(R)$-left module $M$ such that $V:M\to M$ is injective, such that $M = \lim_{i\geq 0} M/V^iM$ and such that $M/VM$ finite locally free over $R$.
\end{defn}
In the following, we will simply say ``Witt vectors'' and ``Cartier module'' instead of ``$O_F$-Witt vectors'' and ``reduced $O_F$-Cartier module''. This shall never lead to confusion.

Let $M$ be a Cartier module over $R$ such that $M/VM$ is free over $R$. Let $γ_1,\ldots,γ_d \in M$ be elements that reduce to an $R$-basis of $M/VM$. Such a tuple is called a $V$-basis for $M$. Then for every element $m\in M$ there are unique coefficients $c_{n,i}\in R$, $n\geq 0$, $i = 1,\ldots, d$, such that
\begin{equation}\label{eq:Cartier_unique}
m = \sum_{n\geq 0}\,\sum_{i = 1}^d V^n[c_{n,i}] γ_i.
\end{equation}
Conversely, by the $V$-completeness of $M$, every tuple of coefficients $c_{n,i}\in R$ defines an element of $M$ by \eqref{eq:Cartier_unique}. In particular, there are unique elements $r_{i, j, n}\in R$ such that
\begin{equation}\label{eq:Cartier_Frobenius}
Fγ_i = \sum_{n\geq 0}\,\sum_{j = 1}^d V^n[r_{i, j, n}] γ_j.
\end{equation}
These $r_{i,j,n}$ are called the structure constants of $(M, γ_1,\ldots, γ_d)$ and they uniquely determine the $E(R)$-module structure on $M$. More precisely, by \cite[Theorem 4.39]{Z_Cartier}, the $E(R)$-linear map $E(R)^d\to M$, $e_i \mapsto γ_i$ is surjective with kernel generated by the elements
\begin{equation}\label{eq:Cartier_mod_relation_general}
Fe_i - \sum_{n\geq 0}\, \sum_{j = 1}^d V^n [r_{i, j, n}] e_j,\qquad i = 1,\ldots, d.
\end{equation}

By \eqref{eq:Cartier_unique}, an $E(R)$-linear endomorphism $f:M\to M$ is uniquely determined by the images $m_i = f(γ_i)$. These may be (non-uniquely) lifted to elements $\wt m_i \in E(R)^d$. Conversely, let $(\wt m_1,\ldots,\wt m_d)\in E(R)^d$ be a tuple of elements with images $(m_1,\ldots,m_d)\in M^d$. Then the $E(R)$-linear map $E(R)^d\to E(R)^d$, $e_i \mapsto \wt m_i$ descends to $M$ if and only if it preserves the relations \eqref{eq:Cartier_mod_relation_general}, meaning that for all $i = 1, \ldots, d$, the following relation holds in $M$:
\begin{equation}\label{eq:Cartier_mod_endo_general_F}
F m_i = \sum_{n\geq 0}\,\sum_{j = 1}^d V^n [r_{i,j,n}] m_i.
\end{equation}
Recall that $π - [π] = {}^Vξ$ for a unit $ξ\in W_{O_F}(O_F)^\times$. Using the relation ${}^Vξ = VξF$ in $E(R)$, we can multiply \eqref{eq:Cartier_Frobenius} with $Vξ$ to obtain a description of multiplication by $π$, say
\begin{equation}\label{eq:Cartier_pi}
π γ_i = [π]γ_i + \sum_{n \geq 1}\,\sum_{j = 1}^d V^n[s_{i,j,n}]γ_j. 
\end{equation}
Multiplication by $V$ on a Cartier module is injective by definition so the coefficients $s_{i,j,n}$ determine the structure constants $r_{i,j,n}$ uniquely. Moreover, the relation \eqref{eq:Cartier_mod_endo_general_F} holds for a tuple $(m_1, \ldots, m_d)$ if and only if the analogous relation for multiplication by $π$ holds,
\begin{equation}\label{eq:Cartier_mod_endo_general}
π m_i = [π] m_i + \sum_{n\geq 1}\,\sum_{j = 1}^d V^n [s_{i,j,n}] m_i.
\end{equation}

\begin{prop}\label{prop:endos_general}
Let $M$ be a Cartier module over $R$ such that $M/VM$ is free over $R$. Let $γ_1,\ldots,γ_d \in M$ be a $V$-basis and let $(m_1,\ldots,m_d)\in M^d$ be a tuple of elements. Let $f:M\to M$ be the map of sets defined by
\begin{equation}\label{eq:set_map}
\sum_{n\geq 0}\,\sum_{i = 1}^d V^n[c_{n,i}] γ_i \longmapsto \sum_{n\geq 0}\,\sum_{i = 1}^d V^n[c_{n,i}] m_i.
\end{equation}
Then $f$ defines a Cartier module endomorphism of $M$ if and only if $(f\circ π)(γ_i) = (π\circ f)(γ_i)$ for all $i = 1,\ldots,d$.
\end{prop}
\begin{proof}
This is a reformulation of what was said in conjunction with \eqref{eq:Cartier_mod_endo_general}.
\end{proof}

\begin{defn}\label{def:V_series_substitution}
Let $γ_1,\ldots, γ_d$ be a $V$-basis of a Cartier module $M$ and let $(m_1,\ldots, m_d)\in M^d$ be a tuple of elements. We call the map in \eqref{eq:set_map} the $V$-series substitution map defined by $γ_i\mapsto m_i$, $i = 1,\ldots, d$.
\end{defn}

Assume that $R$ is equipped with the structure of an $O_{\breve F}$-algebra and let $M$ be the Cartier module of a special $O_C$-module $(Y, ι)$ over $R$. Then the $O_C$-action provides a $\mbZ/4\mbZ$-grading
\begin{equation}\label{eq:Cartier_mod_grading}
M = M_0 \oplus M_1 \oplus M_2 \oplus M_3,\ \ \ \deg V = 1,\ \deg F = -1,\ \deg \varpi = 2.
\end{equation}
We assume that each component of $\Lie(Y) = M_0/VM_3 \oplus M_2/VM_1$ is free as $R$-module and fix a homogeneous $V$-basis $γ_0\in M_0$ and $γ_2\in M_2$. Then the general descriptions in \eqref{eq:Cartier_pi} and \eqref{eq:Cartier_mod_endo_general} simplify drastically because of the grading. Namely, a homogeneous element $m\in M_i$ has a $V$-series expansion of the form
\begin{equation}\label{eq:generic_element_Cartier_module}
m = \begin{cases} [r_0]γ_i + V^2[r_2]γ_{i+2} + V^4[r_4]γ_i + \ldots & \text{if $i$ is even}\\
V[r_1]γ_{i-1} + V^3[r_3]γ_{i+1} + V^5[r_5]γ_{i-1} + \ldots & \text{if $i$ is odd}
\end{cases}
\end{equation}
with unique coefficients $r_i\in R$. So an endomorphism $f$ of $M$ that is homogeneous of some odd degree $i$, say, is uniquely described by two $V$-series
\begin{equation}\label{eq:generic_odd_degree_endo}
f:\begin{cases}
γ_0 \longmapsto V[a_1]γ_{i-1} + V^3[a_3]γ_{i+1} + V^5[a_5]γ_{i-1} + \ldots,\\
γ_2 \longmapsto V[b_1]γ_{i+1} + V^3[b_3]γ_{i-1} + V^5[b_5]γ_{i+1} + \ldots.
\end{cases}
\end{equation}
We ultimately care about the endomorphisms $y\in S_{3/4}$. These are precisely those endomorphisms that are homogeneous of degree $i = 1$ and commute with $\varpi$. Let
\begin{equation}\label{eq:expansion_varpi}
\varpi:\begin{cases}
γ_0 \longmapsto [x_0]γ_2 + V^2[x_2]γ_0 + V^4[x_4]γ_2 + \ldots,\\
γ_2 \longmapsto [y_0]γ_0 + V^2[y_2]γ_2 + V^4[y_4]γ_0 + \ldots
\end{cases}
\end{equation}
be the $V$-series expansion of $ι(\varpi)$. Note that $x_0y_0 = π$ holds because $\varpi^2 = π$ acts as $π$ on $\Lie(M) = M/VM$.
\begin{cor}\label{cor:char_endos}
Let $i \in \{1, 3\}$ and let $a_1,a_3,a_5,\ldots,b_1,b_3,b_5,\ldots \in R$ be any elements. Let $f:M\to M$ be the $V$-series substitution map defined by \eqref{eq:generic_odd_degree_endo}. Then $f$ defines a $\varpi$-linear endomorphism of $M$ if and only if $(\varpi \circ f)(γ_j) = (f\circ \varpi)(γ_j)$ for both $j = 0,2$.
\end{cor}
\begin{proof}
The ``only if'' direction is clear. Conversely, the assumption implies that $(π\circ f)(γ_j) = (f\circ π)(γ_j)$ for both $j = 0,2$ because $\varpi^2 = π$. Then apply Proposition \ref{prop:endos_general}.
\end{proof}

\subsection{Intersection Locus (Superspecial Points)}
\label{ss:superspecial}
Throughout this section, let $y\in S_{3/4,\mr{rs}}$ be a regular semi-simple element. Our aim is to determine the scheme structure of $\mcZ(y)\subset \mcM_C$. Let $L = F[\varpi y^2]$ be the quadratic étale extension of $F$ defined by $\Inv(y; T)$.

Recall from Definition \ref{def:operator_w} that $w(y) = V^{-1}y$ was defined as a $σ$-linear endomorphism of the isocrystal of $\mbY$. We now change this notation slightly and only consider the restriction $w(y)\vert_{W}$ which we still denote by $w(y)$. Here, $W = N_0^{V^2 = \varpi}$ as in \S\ref{ss:inter_set_theoretic}. The same change of notation applies to $w(\varpi y)$. Then there are the identities
\begin{equation}\label{eq:squares_w}
w(\varpi y) = πw(y),\quad w(y)^2 = π^{-1}\cdot \varpi y^2\quand w(\varpi y)^2 = π\cdot \varpi y^2.
\end{equation}
In particular, we can view both $w(y)^2$ and $w(\varpi y)^2$ as elements of $L$. By Lemma \ref{lem:descript_Z_y_Dieudonne}, $\mcZ(y) \neq \emptyset$ only if $w(y)^2 \in O_L$. We from now on impose this assumption.

The arguments in this section will exploit the inclusions
$$\mcZ(\varpi^{-1} y) \subseteq \mcZ(y) \subseteq \mcZ(\varpi y).$$
Here, $\varpi y$ and $\varpi^{-1}y$ both lie in $S_{1/4}$, and $\varpi y$ is regular semi-simple under the assumption $w(y)\in O_L$. In terms of \eqref{eq:endo_rings_CDA_deg_4}, $\varpi y$ defines the element $g = 1 + \left(\begin{smallmatrix} & \varpi\cdot \varpi y \\ \varpi y\end{smallmatrix}\right) \in G_{1/4, b, \mr{rs}}$. By \eqref{eq:squares_w}, its invariant is
\begin{equation}\label{eq:ident_invariants}
\Inv(g; T) = \mr{charred}_{M_2(E)/E}(π\cdot\varpi y^2) = \Inv(1 + w(\varpi y); T).
\end{equation}
All results of \S\ref{s:1_4} apply to the elements $g\in G_{1/4, b, \mr{rs}}$ and $1 + w(\varpi y) \in GL_F(W)_{\mr{rs}}$. In particular, the equality of invariants \eqref{eq:ident_invariants} shows that the following three schemes are all isomorphic
\begin{equation}\label{eq:isom_1_4}
\mcZ(\varpi y)^i \ \ \iso\ \ \mcM_C^i\cap g\cdot \mcM_C^i\ \ \underset{\text{\S\ref{s:1_4}}}{\iso}\ \ \bOmega_E(W)\cap (1 + w(\varpi y))\cdot \bOmega_E(W).
\end{equation}
Here, $i\in 4\mbZ$ and $\mcZ(\varpi y)^i$ as well as $\mcM_C^i$ again denote the loci where the height of $ρ$ equals $i$.

After these preparations, we now formulate and prove our results. The following three propositions will respectively concern the case where $w(y)^2$ lies in $O_L^\times$, in $O_L\setminus (O_L^\times \cup π^2O_L)$, or in $π^2O_L$. This matches the three cases (3), (4) and (5) in Corollary \ref{cor:support_intersection_set_theoretic}.

\begin{prop}\label{prop:artinian_case_3}
Let $y\in S_{3/4, \mr{rs}}$ be regular semi-simple with $w(y)^2 \in O_L^\times$. Then $\mcZ(y)$ is artinian and each connected component is of length $1$.
\end{prop}
\begin{proof}
Let $z\in \mcZ(y)(\mbF)$ be any point. Using the action of $H_b$, we may assume $z\in \mcZ(y)^0$. Let $R$ be the complete local ring $\widehat{\mcO}_{\mcM_C,z}$, let $\mfm$ be its maximal ideal, and let $I\subset R$ be the ideal defining $\mcO_{\mcZ(y),z}$.

By Corollary \ref{cor:support_intersection_set_theoretic}, if $z\in P_Λ(\mbF)$, then $n(w(y), Λ) = 0$. It follows that $n(w(\varpi y), Λ) = 1$. Moreover, by \eqref{eq:squares_w}, the assumption $w(y)^2 \in O_L^\times$ implies that $w(\varpi y)^2 \in π^2O_L^\times$. By Propositions \ref{prop:multiplicities_1_4} and \ref{prop:embedded_components_1_4} as well as \eqref{eq:isom_1_4}, $\mcZ(\varpi y)$ equals the vanishing locus $V(π)$ in a neighborhood of $z$. Since $\mcZ(y)\subseteq \mcZ(\varpi y)$, this implies $π\in I$ and it is left to show that the structure map $O_{\breve F} \to R/I$ is formally unramified.

Consider for this a square-zero thickening $S\twoheadrightarrow R/\mfm$ endowed with trivial PD-structure. Denote by $\mcD = \mcD_0\oplus \ldots \oplus \mcD_3$ the evaluation of the Grothendieck--Messing $O_F$-crystal of the special $O_C$-module $(Y, ι)$ over $R/\mfm$. Then $y$ lifts to a degree $1$ homomorphism $y:\mcD\to \mcD$. We claim that if $w(y)^2\in O_L^\times$, then there is at most one possibility for a $y$-stable and $O_C$-stable Hodge filtration $\mcF\subset \mcD$. Namely any such filtration would be graded $\mcF = \mcF_0\oplus\ldots\oplus \mcF_3$ with $\mcF_i\subseteq \mcD_i$ and
$$\mr{rk}_S(\mcF_i) = \begin{cases} 1 & \text{if $i = 0,2$}\\
2 & \text{if $i = 1,3$}.\end{cases}$$
So we already have $\mcF_i = \mcD_i$ if $i = 1,3$. Furthermore, $w(y)^2\in O_L^\times$ implies that the map $y:\mcD_{i-1}\to \mcD_i$ has rank $1$ mod $\mfm$ if $i = 0,2$. (This can be read off from the Dieudonné module.) Then we have at most the possibility $\mcF_i = y\mcD_{i-1}$ for $i = 0,2$, proving both the claim and the proposition. (The images $y\mcD_{i-1}$ need not be line bundles in general which provides an additional obstruction to deforming $(Y, ι, y)$. This does not matter for the argument however.)
\end{proof}

\begin{prop}\label{prop:artinian_case_4}
Let $y\in S_{3/4, \mr{rs}}$ be regular semi-simple with $w(y)^2\in O_L \setminus (O_L^\times \cup π^2O_L)$. Then $\mcZ(y)$ is artinian and has the following properties:
\begin{enumerate}[wide, labelindent=0pt, labelwidth=!, label=(\arabic*), topsep=2pt, itemsep=2pt]
\item If $L$ is a field, then each connected component of $\mcZ(y)$ has length $2 + 2q$.
\item If $L \iso F\times F$ is split, then each connected component of $\mcZ(y)$ has length $q$.
\end{enumerate}
\end{prop}
\begin{proof}
Let $z\in \mcZ(y)$ be any point. Using the action of $H_b$, we may assume that $z\in \mcZ(y)^0$. Corollary \ref{cor:support_intersection_set_theoretic} (4) states that $\mcZ(y)^0(\mbF)$ consists only of superspecial points, so $z$ is superspecial and the complete local ring $\widehat{\mcO}_{\mcZ(y),z}$ artinian. 

Let $R = \widehat{\mcO}_{\mcM_C,z}$ be the complete local ring of $\mcM_C$ in $z$. It is isomorphic to $O_{\breve F}[\![\bu, \bv]\!]/(\bu\bv-π)$. The elements $\bu,\bv\in R$ are uniquely determined up to interchanging them and/or scaling them in the way $(\bu, \bv)\mapsto (t\bu, t^{-1}\bv)$ for a unit $t\in R^\times$. The following auxiliary result allows to make a matching choice of coordinates for both $R$ and the Cartier module $M$ of the universal special $O_C$-module over $R$.

\begin{lem}\label{lem:Cartier_module_special_point}
For a suitable choice of coordinates $\bu, \bv\in R$ and a suitable $V$-basis $γ_0\in M_0$, $γ_2\in M_2$, the $V$-series presentation of $\varpi$ is
\begin{equation}\label{eq:varpi_simple_special}
\varpi:\begin{cases}
γ_0 \longmapsto [\bu]γ_2 + V^2γ_0\\
γ_2 \longmapsto [\bv]γ_0 + V^2γ_2.
\end{cases}
\end{equation}
\end{lem}
\begin{proof}
Given a $V$-basis $γ_0\in M_0$ and $γ_2\in M_2$ the $V$-series presentation of $\varpi$ has the form
\begin{equation}
\varpi:\begin{cases}
γ_0 \longmapsto [a_0]γ_2 + V^2[a_2]γ_0 + V^4[a_4]γ_2 + \ldots\\
γ_2 \longmapsto [b_0]γ_0 + V^2[b_2]γ_2 + V^4[b_4]γ_0 + \ldots
\end{cases}
\end{equation}
with $a_0b_0 = π$. If $a_2 = b_2 = 1$ and if all higher coefficients vanish, then we may put $\bu = a_0$ and $\bv = b_0$, and are done. So our aim is to arrange this situation for all coefficients in degree $\geq 2$.

Write $M^{(n)}$ for the Cartier module obtained by base change to $R/\mfm^n$. Then $M^{(0)}$ is precisely the Dieudonné module of the $π$-divisible group over the closed point. The point $z$ in question is superspecial, meaning both indices are critical, so $V^2M^{(0)} = \varpi M^{(0)}$. It follows that there is a $V$-basis $γ^{(0)}_0\in M^{(0)}_0$, $γ^{(0)}_2\in M^{(0)}_2$ with $\varpi γ^{(0)}_i = V^2γ^{(0)}_{i+2}$. We prove by induction on $n$, using the $\mfm$-adic completeness of $R$, that such a $V$-basis may be lifted to one as required. This is quite standard:

Assume we already found a $V$-basis $γ^{(n)}_0$, $γ^{(n)}_2$ of $M^{(n)}$ such that $\varpi γ^{(n)}_0 = [a_0]γ^{(n)}_2 + V^2 γ^{(n)}_0$ and $\varpi γ^{(n)}_2 = [b_0]γ^{(n)}_0 + V^2 γ^{(n)}_2$. Let $\wt{γ}^{(n)}_i \in M^{(n+1)}_i$ be any lifts and write
$$\varpi:\begin{cases}
\wt{γ}^{(n)}_0 \longmapsto [a_0]\wt{γ}^{(n)}_2 + V^2\wt{γ}^{(n)}_0 + V^2 δ_0\\
\wt{γ}^{(n)}_2 \longmapsto [b_0]\wt{γ}^{(n)}_0 + V^2\wt{γ}^{(n)}_2 + V^2 δ_2
\end{cases}$$
with $δ_i \in \mr{ker}(M^{(n+1)} \to M^{(n)})$. Any element $ε$ in this kernel satisfies $\varpi ε = 0$ and $[r]ε = 0$, for $r\in \mfm$. So we obtain our desired lifting as $γ^{(n+1)}_i = \wt{γ}^{(n)}_i + δ_i$.
\end{proof}

Let $\mfm\subset R$ be the maximal ideal and let $I\subseteq R$ be the ideal with $R/I = \mcO_{\mcZ(y), z}$. Our aim is to prove that the length of $R/I$ is $2+2q$ if $L$ is a field and $q$ if $L$ is split. Since $2+2q < q^3$, it suffices to prove that $R/(I + \mfm^{q^3})$ is of the desired length. Moreover, by Theorem \ref{thm:main_1_4}, the ideal in $R$ that defines $\mcO_{\mcZ(\varpi y),z}$ is $π(\bu, \bv)$ if $L$ is a field or $(π)$ if $L$ is split. So we know a priori that $π(\bu,\bv) \subseteq I$. 

Let $\ob{R} := R/(\mfm^{q^3}+π(\bu, \bv))$ and let $\ob I = I \ob R$ be the ideal that defines $V(\mfm^{q^3}) \cap \mcZ(y)$. We have reduced to the problem to showing that the length of $\ob{R}/\ob{I}$ is $2 + 2q$ resp. $q$. Let $\ob M = E(\ob R)\tensor_{E(R)} M$ be the reduction of $M$ to $\ob R$.

We assume from now on that $\bu,\bv\in R$ and $γ_0,γ_2\in M$ are chosen as in Lemma \ref{lem:Cartier_module_special_point}. The special fiber $\mbM := E(R/\mfm)\tensor_{E(R)} M$ of the Cartier module is the Dieudonné module of the special fiber of the special $O_C$-module over $R$. In particular, it has the property that $V^4\mbM = π\mbM$ because, for a superspecial point, both indices are critical. Let $a_1,b_1,a_3,b_3,\ldots \in \mbF$ be the coefficients of the $V$-series that define $y\in \End(\mbM)$ in the sense of \eqref{eq:generic_odd_degree_endo}, i.e.
\begin{equation}\label{eq:V_series_y_generic_residue_field}
y:\begin{cases}
γ_0 \longmapsto V[a_1]γ_0 + V^3[a_3]γ_2 + V^5[a_5]γ_0 + \ldots,\\
γ_2 \longmapsto V[b_1]γ_2 + V^3[b_3]γ_0 + V^5[b_5]γ_2 + \ldots.
\end{cases}
\end{equation}

\begin{lem}\label{lem:coefficients_of_y}
The coefficients in \eqref{eq:V_series_y_generic_residue_field} have the following properties.
\begin{enumerate}[wide, labelindent=0pt, labelwidth=!, label=(\arabic*), topsep=2pt, itemsep=2pt]
\item If $L$ is a field, then $a_1, b_1 = 0$ while $a_3, b_3 \in \mbF^\times$.
\item If $L = F\times F$, then one out of $a_1$, $b_1$ vanishes and the other lies in $\mbF^\times$.
\end{enumerate}
\end{lem}
\begin{proof}
Assume first that $L$ is a field. Then $w(y)^2\in O_L\setminus O_L^\times$ implies that $w(y) = V^{-1}y$ is topologically nilpotent. This provides $a_1 = b_1 = 0$. As $V^2\mbM = \varpi \mbM$ because $z$ is superspecial, we have that $γ_i, V^2γ_{i-2}$ provide an $O_{\breve F}$-basis for $\mbM_i/π\mbM_i$. Since by assumption $w(y)^2\notin π^2O_L$, the maps $y:\mbM_i/π\mbM_i \to \mbM_{i+1}/π\mbM_{i+1}$, for $i = 0,2$, are non-zero. So we find that at least one out of the pair $a_1, a_3$, resp. $b_1, b_3$ is non-zero. We have already seen that $a_1 = b_1 = 0$, so we necessarily have $a_3, b_3\in \mbF^\times$ as claimed.

Assume now that $L \iso F\times F$ and write $w(y)^2 = (y_1, y_2)$. Then $v(y_1)$ and $v(y_2)$ both lie in $2\mbZ$ because $1 + w(y) \in GL_4(F)$, compare with row 3 of Table \ref{table:matching}. The assumption $w(y)^2 \notin π^2O_L$ thus implies that $w(y) = V^{-1}y$ is not topologically nilpotent. Equivalently, at least one out of $\{a_1,b_1\}$ is invertible. It cannot happen that both coefficients are invertible, however, because this would imply $w(y)^2\in O_L^\times$ which is excluded by assumption. This finishes the proof the lemma.
\end{proof}

Consider now any sequence of elements $a_1,b_1,a_3,b_3,\ldots \in \ob R$ that lift the coefficients in \eqref{eq:V_series_y_generic_residue_field}. (It will not lead to confusion that we denote them by the same symbols.) Let $\wt y: \ob M \to \ob M$ be the map of sets that is given by the $V$-series substitution
\begin{equation}\label{eq:V_series_y_generic}
\wt y:\begin{cases}
γ_0 \longmapsto V[a_1]γ_0 + V^3[a_3]γ_2 + V^5[a_5]γ_0 + \ldots,\\
γ_2 \longmapsto V[b_1]γ_2 + V^3[b_3]γ_0 + V^5[b_5]γ_2 + \ldots.
\end{cases}
\end{equation}
It lifts the map $y\in \End(\mbM)$ by definition. By Corollary \ref{cor:char_endos}, $\wt y \in \End(E(\ob{R}/J)\tensor_{E(R)} M)$ for the ideal $J\subseteq \ob{R}$ that is generated by all coefficients of the two $V$-series $(\wt y\circ \varpi)(γ_0) - (\varpi \circ \wt y)(γ_0)$ and $(\wt y\circ \varpi)(γ_2) - (\varpi \circ \wt y)(γ_2)$.

We next make these $V$-series more explicit. Here it will pay off that we are working with $\ob R$ instead of $R$: In $E(\ob R)$, it holds that $[\bu]V^n = V[\bu^{q^n}] = 0$ and $[\bv]V^n = V^n[\bv^{q^n}] = 0$ whenever $n \geq 3$. The $V$-series expansions of $\wt y\circ \varpi$ and $\varpi \circ \wt y$ then become
\begin{equation}\label{eq:y_varpi}
\wt y\circ \varpi:\begin{cases}
γ_0 \longmapsto V[b_1\bu^q]γ_2 + V^3[a_1]γ_0 + V^5[a_3]γ_2 + V^7[a_5]γ_0 + \ldots,\\
γ_2 \longmapsto V[a_1\bv^q]γ_0 + V^3[b_1]γ_2 + V^5[b_3]γ_0 + V^7[b_5]γ_2 + \ldots
\end{cases}
\end{equation}
and
\begin{equation}\label{eq:varpi_y}
\varpi \circ \wt y:\begin{cases}
\begin{array}{lllll}
γ_0 \longmapsto V[a_1\bu]γ_2 & +V^3[a_3\bv]γ_0 & +V^5[a_5\bu]γ_2 & +V^7[a_7\bv]γ_0 & +\ldots\\
& +V^3[a_1^{q^2}]γ_0 & +V^5[a_3^{q^2}]γ_2 & +V^7[a_5^{q^2}]γ_0 & +\ldots,\\
γ_2 \longmapsto V[b_1\bv]γ_0 & +V^3[b_3\bu]γ_2 & +V^5[b_5\bv]γ_0 & +V^7[b_7\bu]γ_2 & +\ldots\\
& +V^3[b_1^{q^2}]γ_2& +V^5[b_3^{q^2}]γ_0 & +V^7[b_5^{q^2}]γ_2 & +\ldots.
\end{array}
\end{cases}
\end{equation}

\begin{lem}\label{lem:less_equal}
Let $J\subset \ob R$ be an ideal such that $y\in \End(E(\ob R/J)\tensor_{E(R)}M)$. Then the length of $\ob R/J$ is $\leq 2 + 2q$ if $L$ is a field and $\leq q$ if $L$ is split.
\end{lem}
\begin{proof}
Note the following two properties of sums in the ring of Witt vectors and in the Cartier ring:
\begin{enumerate}[wide, labelindent=0pt, labelwidth=!, label=(\alph*), topsep=2pt, itemsep=2pt]
\item In $W_{O_F}(R)$, a sum of Teichmüller lifts $[a] + [b]$ lies in $[a + b] + W_{O_F}(Ra + Rb)$.

\item In the Cartier ring $E(R)$, we have
$$V^n({}^{V^k}ε) = V^{n+k} ε F^k \in V^{n+k}E(εR).$$
\end{enumerate}
These two properties imply that
$$\begin{aligned}
V^3([a_3\bv] + [a_1^{q^2}])\ &\in V^3[a_3\bv + a_1^{q^2}] + V^5E(R)\\
V^3([b_3\bu] + [b_1^{q^2}])\ &\in V^3[b_3\bu + b_1^{q^2}] + V^5E(R).
\end{aligned}$$
Thus, comparing the $V$-coefficients and $V^3$-coefficients of \eqref{eq:y_varpi} and \eqref{eq:varpi_y}, we obtain that the following identities hold in $\ob R/J$:
\begin{equation}\label{eq:leading_terms_original}
\begin{array}{cc}
\begin{aligned} a_1\bu & = b_1\bu^q\\
b_1\bv & = a_1\bv^q
\end{aligned} &
\begin{aligned}
\hspace{5mm}a_1 & = a_3\bv + a_1^{q^2}\\
b_1 & = b_3\bu + b_1^{q^2}.
\end{aligned}
\end{array}
\end{equation}
The next lemma simplifies these identities. If $L$ is split, then we will only consider the case $a_1\in \ob R^\times$ and $b_1\in \mfm$ from now on (see Lemma \ref{lem:coefficients_of_y} above). The reverse situation is the same by symmetry.
\begin{lem}\label{lem:leading_term_solution}
Consider the two situations from Lemma \ref{lem:coefficients_of_y}.
\begin{enumerate}[wide, labelindent=0pt, labelwidth=!, label=(\arabic*), topsep=2pt, itemsep=2pt]
\item If $a_1$, $b_1\in \mfm$ and $a_3$, $b_3\in \ob{R}^\times$, then \eqref{eq:leading_terms_original} is equivalent to
\begin{equation}\label{eq:leading_terms_simplified}
\begin{aligned}
π & = b_3^{-1}a_3\bv^{q+1}\\
π & = a_3^{-1}b_3\bu^{q+1}
\end{aligned}
\begin{aligned}
\hspace{10mm}a_1 & = a_3\bv\\
b_1 & = b_3\bu.
\end{aligned}
\end{equation}

\item If $a_1 \in \ob{R}^\times$ and $b_1 \in \mfm$, then \eqref{eq:leading_terms_original} is equivalent to
\begin{equation}\label{eq:leading_terms_simplified_x}
\bu = \bv^q = b_1 = 0,\quad\quad a_1 = a_3\bv+a_1^{q^2}.
\end{equation}
\end{enumerate}
\end{lem}
\begin{proof}
We begin with the case $a_1$, $b_1\in \mfm$ and $a_3$, $b_3\in \ob{R}^\times$. If the two identities on the right hand side of \eqref{eq:leading_terms_original} hold, then $a_1 = s\bv$ and $b_1 = t\bu$ for units $s,t\in \ob{R}^\times$. The left hand side identities then imply $t^{-1}sπ = \bu^{q+1}$ and $s^{-1}tπ = \bv^{q+1}$. This forces
$$\bu^{q+2} = \bv^{q+2} = a_1^{q+2} = b_1^{q+2} = 0$$
because we are working modulo $π(\bu, \bv)$. Then we obtain $a_1 = a_3\bv$ and $b_1 = b_3\bu$, meaning $s = a_3$ and $t = b_3$. It follows that\eqref{eq:leading_terms_original} implies \eqref{eq:leading_terms_simplified}. The converse direction is immediate.

We consider the case $a_1 \in \ob R^\times$ and $b_1\in \mfm$. First, $a_1\bu = b_1\bu^q$ is equivalent to $\bu = 0$. Then $b_1 = b_3\bu + b_1^{q^2}$ is equivalent to $b_1 = 0$. Then $b_1\bv = a_1\bv^q$ is equivalent to $\bv^q = 0$, and we have arrived at \eqref{eq:leading_terms_simplified_x}.
\end{proof}

Let $J' \subset \ob R$ be the ideal generated by the relations in \eqref{eq:leading_terms_simplified} and \eqref{eq:leading_terms_simplified_x}. Then $J\subseteq J'$ and $\ob R/J'$ has length $\leq 2 + 2q$ resp. $\leq q$, proving Lemma \ref{lem:less_equal}.
\end{proof}

In particular, $\ob R/\ob I$ has length $\leq 2 + 2q$ resp. $\leq q$ because Lemma \ref{lem:less_equal} applies to $\ob I$ and the deformation of $y$ to $\End(E(\ob R/\ob I)\tensor_{E(R)} M)$. It remains to prove the converse inequality.

\begin{defn}\label{def:a1}
Given any two $a_3,b_3\in \ob{R}$ as in Lemma \ref{lem:leading_term_solution}, we from now on choose $a_1, b_1\in \ob{R}$ and the ideal $J_{a_3,b_3}\subset \ob{R}$ in the following way.
\begin{enumerate}[wide, labelindent=0pt, labelwidth=!, label=(\arabic*), topsep=2pt, itemsep=2pt]
\item If $L$ is a field, then $a_3$, $b_3\in \ob{R}^\times$ are units. We set $a_1 = a_3\bv$ and $b_1 = b_3\bu$, as well as
$$J_{a_3, b_3} = (π-b_3^{-1}a_3 \bv^{q+1},\ π - a_3^{-1}b_3 \bu^{q+1}).$$
\item If $L \iso F\times F$ and if $a_3$, $b_3 \in \ob{R}$ are any two elements, then we set $b_1 = 0$ and we let $a_1$ be the unique solution to the equation $a_1 = a_3\bv+a_1^{q^2}$ that lifts the given solution mod $\mfm$. (Existence and uniqueness follow from the fact that this equation is étale.) We define $J_{a_3, b_3} = (\bu, \bv^q)$.
\end{enumerate}
\end{defn}

With choices as in Definition \ref{def:a1}, the quotient $\ob{R}/J_{a_3,b_3}$ has length $2+2q$ if $L$ is a field and length $q$ if $L\iso F\times F$. Moreover, the elements $a_1, b_1, a_3$ and $b_3$ satisfy Identity \eqref{eq:leading_terms_simplified} resp. Identity \eqref{eq:leading_terms_simplified_x} modulo $J_{a_3, b_3}$.

\begin{lem}\label{lem:lifting_coefficients}
There exist choices for the coefficients $a_3, b_3, a_5, b_5, \ldots \in \ob{R}$ in \eqref{eq:V_series_y_generic} such that, with $a_1,b_1$ and $J_{a_3,b_3}$ as in Definition \ref{def:a1}, $\wt y\circ \varpi = \varpi \circ \wt y$ over $\ob{R}/J_{a_3,b_3}$.
\end{lem}
\begin{proof}
Using the previous properties (a) and (b) of addition in $E(R)$, we see that the $V^{2k+1}$-coefficients of \eqref{eq:varpi_y} take the form
$$\begin{cases}
a_{2k+1} + a_{2k-1}^{q^2} + p_{2k+1}(a_3\bv, a_5\bu, a_7\bv, \ldots, a_{2k-1}\bs, a_1^{q^2}, a_3^{q^2},\ldots, a_{2k-3}^{q^2})\\
b_{2k+1} + b_{2k-1}^{q^2} + q_{2k+1}(b_3\bu, b_5\bv, b_7\bu, \ldots, b_{2k-1}\bt, b_1^{q^2}, b_3^{q^2},\ldots, b_{2k-3}^{q^2})
\end{cases}$$
for certain polynomials $p_{2k+1}, q_{2k+1}$ with coefficients in $R$. Here, $(\bs,\bt) = (\bu, \bv)$ or $(\bs, \bt) = (\bv, \bu)$ in dependence on the parity of $k$. Thus \eqref{eq:y_varpi} equals \eqref{eq:varpi_y} if and only if the following identities hold:
\begin{equation}\label{eq:main_etale}
\begin{aligned}
a_3 & = a_5\bu + a_3^{q^2} + p_3(a_3\bv, a_1^{q^2})\\
a_5 & = a_7\bv + a_5^{q^2} + p_5(a_3\bv, a_5\bu, a_1^{q^2}, a_3^{q^2})\\
a_7 & = a_9\bu + a_7^{q^2} + p_7(a_3\bv, a_5\bu, a_7\bv, a_1^{q^2}, a_3^{q^2}, a_5^{q^2})\\
&\cdots&\\
&&\\
b_3 & = b_5\bv + b_3^{q^2} + q_3(b_3\bu, b_1^{q^2})\\
b_5 & = b_7\bu + b_5^{q^2} + q_5(b_3\bu, b_5\bv, b_1^{q^2}, b_3^{q^2})\\
b_7 & = b_9\bv + b_7^{q^2} + q_7(b_3\bu, b_5\bv, b_7\bu, b_1^{q^2}, b_3^{q^2}, b_5^{q^2})\\
&\cdots&.
\end{aligned}
\end{equation}
Assume we have shown that these two systems of equations have a unique solution in $\ob{R}/\mfm^n\ob{R}$ that lifts the coefficients of $y \in \End(\mbM)$. We claim that this solution lifts uniquely to $\ob{R}/\mfm^{n+1}\ob{R}$. Consider for this the truncations of the above two systems of equations in some degree $m = 2k+1$. The summands $a_{2k+3}\bs$ resp. $b_{2k+3}\bt$ (with $\bs, \bt\in \{\bu,\bv\}$ in dependence on $k$) in the two last lines are already uniquely determined by the given solution over $\ob{R}/\mfm^n\ob{R}$. In all remaining variables, the two truncated systems of equations are étale, because their Jacobian matrices are upper triangular up to nilpotent entries. The given solution thus lifts uniquely to a solution over $\ob{R}/\mfm^{n+1}\ob{R}$, proving the lemma.
\end{proof}

The proof of Proposition \ref{prop:artinian_case_4} is now complete. Namely, Lemma \ref{lem:lifting_coefficients} shows the existence of an ideal $J = J_{a_3, b_3}\subset \ob{R}$ such that $\ob I\subseteq J$ and such that the length of $\ob R/J$ is $2 + 2q$ if $L$ is a field and $q$ if $L$ is split. Lemma \ref{lem:leading_term_solution} on the other hand showed that the length of $R/I$ is $\leq 2 + 2q$ resp. $\leq q$. These two statements together imply the proposition.
\end{proof}

The third type of embedded component arises in superspecial points whenever $\max_{Λ\subset W} n(y, Λ) \geq 1$.

\begin{prop}\label{prop:artinian_case_5}
Let $y\in S_{3/4, \mr{rs}}$ be regular semi-simple and let $\{z\} = P_1\cap P_2$ be a superspecial point on $\mcM_C$, where $P_1,P_2\subseteq \mcM_C$ denote two irreducible components of the special fiber. Assume that $P_1\not\subseteq \mcZ(y)$ but $P_2\subseteq \mcZ(y)$. Then $z\in \mcZ(y)^{\mr{art}}(\mbF)$ and the length of $\mcO_{\mcZ(y)^{\mr{art}}, z}$ is $q$.
\end{prop}
\begin{proof}
The proof is analogous to the one of Proposition \ref{prop:artinian_case_4}. Let again $R$ be the complete local ring $\widehat{\mcO}_{\mcM_C, z} = O_{\breve F}[\![\bu, \bv]\!]/(\bu\bv - π)$ be the complete local ring in $z$. Let $\mfm\subset R$ be the maximal ideal and let $I\subset R$ be the ideal defining $\widehat{\mcO}_{\mcZ(y),z}$. Assume that $\bv$ corresponds to $P_2$, i.e. assume that $I\subseteq (\bv)$.

\emph{Claim: The ideal $I$ is given by $I = (π, \bv^{q+1})$.} This claim immediately implies the proposition: The maximal Cartier divisor dividing $I$ is $\bv$, so we obtain $\mcO_{\mcZ(y)^{\mr{art}}, z} = R/(\bu, \bv^q)$ which has length $q$.

In order to prove the claim, we first use our results about the case of invariant $1/4$. Let $Λ_1, Λ_2\subseteq W$ be the two lattices defined by $P_1$ and $P_2$ in \eqref{eq:Drinfeld_iso_explicit}. Then $n(y, Λ_1) = 0$ and $n(y, Λ_2) = 1$ by our assumptions on $P_1$ and $P_2$. It follows from $w(\varpi y) = πw(y)$ that $n(\varpi, Λ_1) = 1$ and $n(\varpi, Λ_2) = 2$. Thus $\widehat{\mcO}_{\mcZ(\varpi y), z}$ is defined by the ideal $(π\bv)$ and we obtain
$$(π\bv) \subseteq I \subseteq (\bv).$$
\begin{lem}\label{lem:ideal_juggling}
Let $I\subseteq R$ be an ideal such that $(π\bv) \subseteq I \subseteq (\bv)$ and such that $I + \mfm^{q^2} = (π, \bv^{q+1})$. Then $I = (π, \bv^{q+1})$.
\end{lem}
\begin{proof}
Since $(π\bv) \subseteq I$ by assumption, $I + \mfm^{q^2} = I + (\bu^{q^2}, \bv \bu^{q^2-1})$. Since $I \subseteq (\bv)$ and $π\in I + \mfm^{q^2}$ by assumption, we may thus write
$$π = a\bv + b\bu^{q^2} + c\bv \bu^{q^2-1},\quad a\in \mfm,\ a\bv\in I,\ b,c\in R.$$
Then $b$ has to be divisible by $\bv$. So after modifying $c$, we may assume $b = 0$. Then we obtain that $a\bv = π(1-c\bu^{q^2-2}) = (\mr{unit})\cdot π \in I$. It follows that $π\in I$ and hence, in particular, that $\bv\bu^{q^2-1} \in I$.

Since $\bv^{q+1}\in I + \mfm^{q^2}$ by assumption, we can now write
$$\bv^{q+1} = a\bv + b\bu^{q^2},\quad a\bv\in I,\ b\in R.$$
Then $b$ is divisible by $\bv$, so $b\bu^{q^2}\in I$ by the previous results, and hence $\bv^{q+1} \in I$. This finishes the proof.
\end{proof}

Let $\ob{R} = R/(\mfm^{q^2} + (π\bv))$ and let $\ob I = I\ob{R}$. By Lemma \ref{lem:ideal_juggling}, it suffices to show that $(\Spec \ob R)\cap \mcZ(y) = V(\ob I)$. Let $M$ be the Cartier module of the universal point over $R$, and assume from now on that $\bu,\bv\in R$ as well as $γ_0,γ_2\in M$ are chosen as in Lemma \ref{lem:Cartier_module_special_point}. Let $\mbM = E(R/\mfm)\tensor_{E(R)} M$ be the Dieudonné module of the special fiber. Let $a_1, b_1, a_3, b_3, \ldots \in \ob{R}$ be coefficients such that the $V$-series datum $\wt y$ from \eqref{eq:V_series_y_generic} lifts the given endomorphism $y \in \End(\mbM)$.

Just as in the proof of Proposition \ref{prop:artinian_case_4}, we now extract information on the coefficients from our assumptions. By way of symmetry, we assume that $0$ is the index corresponding to $Λ_1$, meaning that $V^{-1}y \mbM_0 \not\subset π\mbM_0$ while $V^{-1}y\mbM_2 \subseteq π\mbM_2$. The second condition means that $a_1$, $b_1$ and $b_3$ all lie in $\mfm \ob R$. The first then implies that $a_3$ lies in $\ob R^\times$.

The compositions $\varpi\circ \wt y$ and $\wt y\circ \varpi$ are again given by the identities \eqref{eq:y_varpi} and \eqref{eq:varpi_y}. This time, since we are working over $\ob{R}$, we even obtain that $a_1^{q^2} = b_1^{q^2} = b_3^{q^2} = 0$. The system of leading term identities of $\varpi \circ \wt y = \wt y\circ \varpi$ is then given by
\begin{equation}\label{eq:leading_terms_II}
\begin{aligned}
a_1\bu & = b_1\bu^q\phantom{a_3^{q^2}}\\
b_1\bv & = a_1\bv^q
\end{aligned}
\begin{aligned}
\hspace{10mm}a_1 & = a_3\bv\phantom{a_3^{q^2}}\\
b_1 & = b_3\bu
\end{aligned}
\begin{aligned}
\hspace{10mm}a_3 & = a_5\bu + a_3^{q^2}\\
b_3 & = b_5\bv.
\end{aligned}
\end{equation}
Performing the direct substitutions for $a_1$, $b_1$ and $b_3$ leaves the three equations
$$
\begin{aligned}
a_3π & = b_5π\bu^q\phantom{a_3^{q^2}}\\
b_5π\bv & = a_3\bv^{q+1}
\end{aligned}
\begin{aligned}
\hspace{10mm}a_3 & = a_5\bu + a_3^{q^2}\\
\phantom{b_3} & 
\end{aligned}
$$
Since $a_3\in \ob R^\times$, the upper left identity is equivalent to $π = 0$. Now for given $a_3\in \ob R^\times$ and $b_3\in \ob R$, we define $a_1$ and $b_1$ by \eqref{eq:leading_terms_II} and let $J$ be the ideal $(π, \bv^{q+1})$. Then we argue precisely as in Lemma \ref{lem:lifting_coefficients} and obtain that $\ob I = J$ as claimed. The proof is now complete.
\end{proof}

We may now extend Propositions \ref{prop:artinian_case_3}, \ref{prop:artinian_case_4} and \ref{prop:artinian_case_5} by a general argument.

\begin{prop}\label{prop:simple_extension}
Let $y\in S_{3/4, \mr{rs}}$ be any element and let $z \in \mcZ(y)$ be a closed point. Let $R = \widehat{\mcO}_{\mcM_C, z}$ be the complete local ring in $z$ and let $I \subset R$ be the ideal such that $R/I = \widehat{\mcO}_{\mcZ(y),z}$. Then
$$(\Spf R)\cap \mcZ(πy) = \Spf (R/πI).$$
\end{prop}
\begin{proof}
For $y'\in S$ such that $z\in \mcZ(y')$, we denote by $I(y')\subseteq R$ the ideal with
$$(\Spf R) \cap \mcZ(y') = \Spf (R/I(y')).$$
In this notation, our aim is to prove that $I(πy) = πI(y)$ for the given element $y$. Consider the Cartier module $M$ of the universal special $O_C$-module over $R$ and write
\begin{equation}\label{eq:expansion_pi}
π:\begin{cases}
γ_0 \longmapsto [π]γ_0 + V^2[u_2]γ_2 + V^4[u_4]γ_0 + \ldots,\\
γ_2 \longmapsto [π]γ_2 + V^2[v_2]γ_0 + V^4[v_4]γ_2 + \ldots
\end{cases}
\end{equation}
for the $V$-series expansion of multiplication by $π$ on $M$. Write $y \in \End(E(R/I(y))\tensor_{E(R)} M)$ as in \eqref{eq:generic_odd_degree_endo}. Choose a lifting of all its coefficients to $R$, say $a_1, b_1, a_3, b_3, \ldots\in R$, and denote by $\wt y$ the resulting $V$-series substitution map $\wt y:M\to M$. Define the obstruction $\mr{ob} := \mr{ob}(\wt y) := \varpi\circ \wt y - \wt y \circ \varpi$. (Recall that this is just an endomorphism of $M$ as set and need not come by $E(R)$-linear extension from $γ_i\mapsto \mr{ob}(γ_i)$, $i = 0,2$.) Still, by Proposition \ref{prop:endos_general}, $\wt y$ defines an endomorphism modulo some ideal $J\subset R$ if and only if $\mr{ob}(γ_i) = 0$ mod $J$. Write
\begin{equation}\label{eq:obstruction}
\mr{ob}:\begin{cases}
γ_0 \longmapsto V[c_1]γ_2 + V^3[c_3]γ_0 + V^5[c_5]γ_2 + \ldots,\\
γ_2 \longmapsto V[d_1]γ_0 + V^3[d_3]γ_2 + V^5[d_5]γ_0 + \ldots.
\end{cases} 
\end{equation}
By definition of $I(y)$ and by Proposition \ref{prop:endos_general}, $I(y)$ is precisely the ideal $(c_1,c_3,\ldots,d_1,d_3,\ldots)$. Now $π\circ \wt y$ defines a lift to $R$ of the $V$-series expression for $π\circ y$. Its obstruction, evaluated on $γ_0,γ_2$, is given by
\begin{equation}\label{eq:ob_times_pi}
\begin{aligned}
\mr{ob}(π\circ \wt y) & = \varpi \circ (π\circ \wt y) - (π\circ \wt y) \circ \varpi\\
& = π \circ \mr{ob}
\end{aligned}
\end{equation}
because $π = \varpi^2$ and $\varpi$ commute as (set-theoretic) endomorphisms of $M$. The crucial observation now is that
\begin{equation}\label{eq:obstruction_times_pi}
π \circ \mr{ob}:\begin{cases}
γ_0 \longmapsto V[c_1π]γ_2 + V^3[c_3π]γ_0 + V^5[c_5π]γ_2 + \ldots,\\
γ_2 \longmapsto V[d_1π]γ_0 + V^3[d_3π]γ_2 + V^5[d_5π]γ_0 + \ldots
\end{cases}
\end{equation}
modulo $I(y)^{q^2}$. Namely, $[a]V^{2k} = V^{2k}[a^{q^{2k}}]$. So when substituting \eqref{eq:expansion_pi} into \eqref{eq:obstruction}, all the higher terms
$$V^i[c_i]V^{2k}[u_{2k}]γ_ε,\quad V^i[d_i]V^{2k}[u_{2k}]γ_ε$$
vanish modulo $I(y)^{q^2}$. Thus we obtain that
\begin{equation}\label{eq:ideal_mod_power_m}
I(πy) = πI(y) \mod I(y)^{q^2}.
\end{equation}
It is left to show that this already implies $I(πy) = πI(y)$.

\emph{The case that $I(y) \subseteq (π)$.} Let $n\geq 1$ be such that $I(y) \subseteq (π)^n$ but $I(y)\not\subseteq (π)^{n+1}$. This integer can also be characterized by
$$n = \max_{z\in P_Λ} n(y, Λ).$$
Here, there are either one or two curves $P_Λ$ that contain $z$. Since $n(\varpi y, Λ) = n(y, Λ) + 1$ for every $y\in S_{3/4}$,
$$\max_{z \in P_Λ} n(\varpi y, Λ) = n + 1.$$
By Propositions \ref{prop:multiplicities_1_4} and \ref{prop:embedded_components_1_4}, the only possibilities for $I(y)$ are $(π)^n, π^n\bu, π^n\bv$ or $π^n(\bu, \bv)$ where the last three are meant for a superspecial point. Moreover, these propositions also show $I(π^2y) = π^2I(y)$. So we find that
$$(π)^{n+3} \subset I(\varpi π y) \subseteq I(π y).$$
It follows that $I(y)^{q^2} \subseteq (π)^{nq^2} \subseteq I(πy)$. The inclusion $I(y)^{q^2} \subseteq πI(y)$ is immediately clear, so we deduce from \eqref{eq:ideal_mod_power_m} that $I(πy) = πI(y)$ as desired.

\emph{The case that $I(y)\not\subseteq (π)$.} In this situation, the point $z$ is of one of the types considered in Propositions \ref{prop:artinian_case_3}, \ref{prop:artinian_case_4} and \ref{prop:artinian_case_5}. Also taking into account Propositions \ref{prop:multiplicities_1_4} and \ref{prop:embedded_components_1_4} to determine $I(\varpi y)$, the possibilities up to isomorphism are given by the following table:
\begin{table}[h!]
\centering
\def\arraystretch{1.3}
\begin{tabular}{|c|c|c|c|}
\hline
& $R$ & $I(y)$ & $I(\varpi y)$\\
\hline
(1) & $O_{\breve F}[\![t]\!]$ & $(π, t)$ & $(π)$\\
\hline
(2) & \multirow{4}{*}{$O_{\breve F}[\![\bu, \bv]\!]/(\bu\bv - π)$} & $(\bu, \bv)$ & $(π)$\\
\cline{1-1} \cline{3-4}
(3) & & $(\bu, \bv^q)$ & $(π)$\\
\cline{1-1} \cline{3-4}
(4) & & $(π - \bu^{q+1}, π - \bv^{q+1})$ & $π(\bu, \bv)$\\
\cline{1-1} \cline{3-4}
(5) & & $(π, \bv^{q+1})$ & $(π\bv)$\\
\hline
\end{tabular}
\medskip
\caption{The possible embedded components for Hasse invariant $3/4$.}
\label{table:cases_I}
\end{table}

We furthermore have the inclusions
$$I(π \varpi y) \subseteq I(π y) \subseteq I(\varpi y) \subseteq I(y).$$
By Propositions \ref{prop:multiplicities_1_4} and \ref{prop:embedded_components_1_4}, we know that $I(π\varpi y) = πI(\varpi y)$. It can also be verified from the above table that $πI(y) \subseteq I(\varpi y)$ in each case, so both $I(πy)$ and $πI(y)$ are contained in $I(\varpi y)$. It then follows from \eqref{eq:ideal_mod_power_m} that
$$I(πy) = πI(y) \mod I(y)^{q^2} \cap I(\varpi y).$$
By Nakayama's Lemma, the identity $I(πy) = πI(y)$ follows if we can show that
\begin{equation}\label{eq:ideal_identity_to_show}
I(y)^{q^2}\cap I(\varpi y) \subseteq π\mfm I(y).
\end{equation}
The reader will have no difficulty checking this relation in cases (1), (2) and (3) of Table \ref{table:cases_I} above and we only treat the cases (4) and (5):

\begin{enumerate}[wide, labelindent=0pt, labelwidth=!, label={Case (\arabic*)}, topsep=2pt, itemsep=2pt]
\item[Case (4).] First observe that
$$(π-\bu^{q+1})(π-\bv^{q+1}) \in π(\bu,\bv) \subseteq I(y).$$
A general element
$$\sum_{n = 0}^{q^2} a_n (π - \bu^{q+1})^n(π - \bv^{q+1})^{q^2-n} \in I(y)^{q^2}$$
thus lies in $π(\bu,\bv)$ if and only if $\bu\mid a_0$ and $\bv\mid a_{q^2}$. In this situation, already $a_0(π - \bv^{q+1})$ and $a_{q^2}(π - \bu^{q+1})$ lie in $π(\bu,\bv)$. This shows that
$$I(y)^{q^2}\cap π(\bu,\bv) \subseteq π(\bu,\bv)I(y)^{q^2-2} \subseteq π(\bu, \bv)\mfm I(y)^{q^2 - 3}$$
which veries \eqref{eq:ideal_identity_to_show} since $q^2 \geq 4$.

\item[Case (5).] Here we can directly verify \eqref{eq:ideal_identity_to_show} by
$$(π,\bv^{q+1})^{q^2} \cap (\bv π) \subseteq π I(y)^{q^2-2} \subseteq π\mfm I(y)^{q^2-3}.$$
\end{enumerate}
The proof of Proposition \ref{prop:simple_extension} is now complete.
\end{proof}

\subsection{Intersection Locus (Non-special Points)}
\label{ss:displays}

We come to our final argument. The following proposition is essentially a converse to Proposition \ref{prop:simple_extension} and shows that there are no embedded components beyond the ones found in the previous section.

\begin{prop}\label{prop:non_special_embedded}
Assume that $z\in \mcZ(y)$ is a closed point that is not superspecial. If $π\mcO_{\mcZ(y),z} \neq 0$, then $z\in \mcZ(π^{-1}y)$.
\end{prop}

We would have liked to prove this with Cartier theory as before, but the $V$-series expressions for non-superspecial points are more complicated than in Lemma \ref{lem:Cartier_module_special_point}. Instead, we use the $O_F$-display theory from \cite{ACZ} which requires us to restrict to the $p$-adic setting. The proof in the function field setting would be analogous but in terms of local $O_F$-shtukas. These are equivalent to strict $O_F$-modules by \cite[Theorem 8.3]{HS}.

\begin{proof}[Proof for $p$-adic $F$.] By assumption, $z$ is a smooth point of $\mcM_C$ whose complete local ring $\widehat{\mcO}_{\mcM_C,z}$ is isomorphic to $R = O_{\breve F}[\![t]\!]$. Let $I(y)\subseteq R$ be the ideal defining $\widehat{\mcO}_{\mcZ(y),z}$. The assumption $I(y) \varsubsetneq (π)$ is equivalent to the statement that for all continuous rings maps $φ:R \to O_{\breve F}$, equivalently one such map $φ$, it holds that $φ(I(y)) = (π)^n$ with $n\geq 2$. Thus we need to see
\begin{equation}\label{eq:pi_square_point}
\Hom_{O_{\breve F}}(\widehat{\mcO}_{\mcZ(y),z},\ O_{\breve F}/(π)^2)\neq \emptyset\quad \Longrightarrow\quad z\in \mcZ(π^{-1}y).
\end{equation}
Let $(Y, ι, ρ)$ be the triple defining the point $z\in \mcM_C(\mbF)$. Recall that we have the unramified quadratic extension $E\subset F_4\subset C$ which is normalized by $\varpi$. It will suffice for our arguments to consider the coarser datum $(Y, j) := (Y, ι\vert _{O_{F_4}})$. Our first aim is to compute the $O_F$-display of a universal deformation of $(Y, j)$. Let $(M, F, V)$ be the $O_F$-Dieudonné module of $(Y, ι)$. We choose the grading on $M$ such that $0$ is the critical index. Let
$$e_0,f_0 \in M_0^{\varpi = V^{-2}}$$
be an $O_E$-basis. One out of $V^2e_0$, $V^2f_0$ does not lie in $πM_2$. We choose our ordering such that this holds for $e_2 := V^2e_0$. Pick a complementary basis vector $f_2\in M_2$ that can be written as
\begin{equation}\label{eq:intermediate}
f_2 = π^{-1}(λe_2 + V^2f_0)
\end{equation}
for a suitable $λ\in O_{\breve F}$. Rewriting \eqref{eq:intermediate} also provides
$$V^2f_0 = -λe_2 + πf_2.$$
In summary, $\varpi$ and $V$ are now given by the following identities:
\begin{equation}\label{eq:varpi_V_in_basis}
\begin{array}{ccccccc}
\varpi e_0 &=& e_2 &\quad& \varpi e_2 &=& πe_0\\
\varpi f_0 &=& -λe_2 + πf_2 &\quad& \varpi f_2 &=& λe_0 + f_0,\\
&&&&&&\\
V^2 e_0 &=& e_2 &\quad& V^2 e_2 &=& πe_0\\
V^2 f_0 &=& -λe_2 + πf_2 &\quad& V^2 f_2 &=& σ^{-2}(λ)e_0 + f_0.
\end{array}
\end{equation}
We next rewrite this in the terminology of $O_F$-displays. Consider the following eight elements
\begin{equation}\label{eq:basis_display}
\begin{array}{rclcrclcrclcrcl}
t_0 &=& e_0 &\quad& m_1 &=& Vt_0 &\quad& t_2 &=& f_2 &\quad& m_3 &=& Vt_2\\ 
l_0 &=& σ^{-2}(λ)e_0 + f_0 &\quad& n_1 &=& Vl_0 &\quad& l_2 &=& e_2 &\quad& n_3 &=& Vl_2.
\end{array}
\end{equation}
Then $M = L\oplus T$ where $L = \mr{span}\{l_0, m_1, n_1, l_2, m_3, n_3\}$ and $T = \mr{span}\{t_0, t_2\}$. This is a normal decomposition of $M$ meaning $VM = L \oplus πT$. Let $\dot F := V^{-1}\vert_{L\oplus πT}$ be the display variant of the Verschiebung. Then $F\vert_T$ and $\dot F\vert_L$ are given by the following identities:
$$\begin{array}{rclcrcl}
\dot F(l_0) &=& m_3 &\quad& \dot F(l_2) &=& m_1\\
\dot F(m_1) &=& t_0 &\quad& \dot F(m_3) &=& t_2\\
\dot F(n_1) &=& l_0 &\quad& \dot F(n_3) &=& l_2.
\end{array}$$
$$\begin{array}{rcl}
F(t_0) &=& πV^{-1}(e_0) = V\varpi e_0 = Ve_2 = n_3,\\
F(t_2) &=& πV^{-1}(f_2) = V^{-1}(λe_2 + V^2f_0)\\
 &=& σ(λ)Vt_0 + V(l_0 - σ^{-2}(λ)t_0) = (σ(λ) - σ^{-3}(λ))m_1 + n_1.
\end{array}$$
Write $μ_1 = σ(λ) - σ^{-3}(λ)$. Order the chosen basis as $(t_0, l_0, m_1, n_1, t_2, l_2, m_3, n_3)$. The structure matrix of $M$ as $O_F$-display is then
\begin{equation}
S = \left[\begin{matrix}
&&1&&&&&\\
&&&1&&&&\\
&&&&μ_1&1&&\\
&&&&1&&&\\
&&&&&&1&\\
&&&&&&&1\\
&1&&&&&&\\
1&&&&&&&
\end{matrix}\right].
\end{equation}
It is known that a universal deformation of $M$ as $O_F$-display can be defined as follows, cf. \cite[Equation (87)]{Zink}.\footnote{The reference applies directly if $O_F = \mbZ_p$. The general statement can be reduced to that case because $O_F$-displays are equivalent to $\mbZ_p$-displays with strict $O_F$-action, cf. the functor $Γ(\mcO, \mcO')$ in \cite[(1.1)]{ACZ}.} Let $μ\in W_{O_F}(O_{\breve F})$ be any lift of $μ_1$. Consider the ring $A = O_{\breve F}[\![s_{01},\ldots,s_{06},s_{21},\ldots,s_{26}]\!]$. (The strict $O_F$-module $Y$ is of height $8$ and dimension $2$, so one knows a priori that $A$ is isomorphic to its universal deformation ring.) Put $L = W_{O_F}(A)^{\oplus 6}$, $T = W_{O_F}(A)^{\oplus 2}$ and $P = L\oplus T$. Denote and order their basis vectors just as before. Then a universal deformation of $M$ can be defined by declaring $P = L \oplus T$ to be a normal decomposition and by taking the $O_F$-display for the structure matrix
\begin{equation}\label{eq:display_main}
\arraycolsep=3pt
\left[\begin{array}{cccccccc}
1&[s_{01}]&[s_{02}]&[s_{03}]&&[s_{04}]&[s_{05}]&[s_{06}]\\
&1&&&&&&\\
&&1&&&&&\\
&&&1&&&&\\
&[s_{21}]&[s_{22}]&[s_{23}]&1&[s_{24}]&[s_{25}]&[s_{26}]\\
&&&&&1&&\\
&&&&&&1&\\
&&&&&&&1\\
\end{array}\right]
\cdot S =
\left[\begin{array}{cccccccc}
[s_{06}]&[s_{05}]&1&[s_{01}]&μ[s_{02}]+[s_{03}]&[s_{02}]&&[s_{04}]\\
&&&1&&&&\\
&&&&μ&1&&\\
&&&&1&&&\\
{[s_{26}]}&[s_{25}]&&[s_{21}]&μ[s_{22}] + [s_{23}]&[s_{22}]&1&[s_{24}]\\
&&&&&&&1\\
&1&&&&&&\\
1&&&&&&&
\end{array}\right]
\end{equation}
Consider now the universal deformation of $(M,j)$, i.e. as $O_F$-display with action by $O_{F_4}$. It is not difficult to check that this deformation space is described by the quotient
$$O_{\breve F}[\![s_{01}, s_{24}]\!] \iso \ob A = A/(s_{ij},\ (i,j) \notin \{(0,1), (2,4)\}).$$
Namely, $O_{F_4}$-actions on a display over an $O_{\breve F}$-algebra are equivalent to $\mbZ/4$-gradings such that $F$ and $\dot F$ are homogeneous of degree $-1$. At this point, one could go even further and also determine the relation between $s_{01}$ and $s_{24}$ that defines the deformation space of $M$ with special $O_C$-action $ι$, but this will not be necessary for our arguments.

We explained at the beginning, see \eqref{eq:pi_square_point}, that we only care about $\ob R$-points of $\mcZ(y)$, where $\ob R = O_{\breve F}/(π)^2$. So we pick any specialization map $φ:O_{\breve F}[\![s_{01}, s_{24}]\!] \to \ob R$. We put $s_0 = φ(s_{01})$ and $s_2 = φ(s_{24})$. Base changing the $O_F$-display over $\ob A$ given by \eqref{eq:display_main} along $φ$ defines an $O_F$-display $(\ob P, \ob Q, F, \dot F)$ over $\ob R$. It is defined by the normal decomposition $\ob P = \ob L \oplus \ob T$ obtained by base change from the previous normal decomposition, and the structure matrix 
\begin{equation}
\arraycolsep=3pt
\left[\begin{array}{cccccccc}
&&1&[s_0]&&&&\\
&&&1&&&&\\
&&&&μ&1&&\\
&&&&1&&&\\
&&&&&&1&[s_2]\\
&&&&&&&1\\
&1&&&&&&\\
1&&&&&&&
\end{array}\right].
\end{equation}
The claim is that every homogeneous degree $1$ endomorphism $y$ of $\ob P$ is divisible by $π$ after base change to $\ob R/(π)$. (Homogeneity is meant with respect to the $j(O_{F_4})$-grading.) Any such endomorphism is, in particular, a $W_{O_F}(\ob R)$-linear endomorphism of the $W_{O_F}(\ob R)$-module $\ob P$ that preserves the submodule $\ob Q = L\oplus {}^VW_{O_F}(\ob R)\cdot T$. It is hence given by four matrices of the form
\begin{equation}\label{eq:generic_endo_display}
\xymatrixcolsep{23mm}
\xymatrix{
y:\quad \ob P_0 \ar[r]^{\begin{pmatrix} a_{11} & a_{12} \\ a_{21} & a_{22}\end{pmatrix}} &
\ob P_1 \ar[r]^{\begin{pmatrix} {}^Vb_{11} & {}^Vb_{12} \\ b_{21} & b_{22}\end{pmatrix}}
&
\ob P_2 \ar[r]^{\begin{pmatrix} c_{11} & c_{12} \\ c_{21} & c_{22}\end{pmatrix}}
&
\ob P_3 \ar[r]^{\begin{pmatrix} {}^Vd_{11} & {}^Vd_{12} \\ d_{21} & b_{22}\end{pmatrix}}
&
\ob P_0}
\end{equation}
where the $16$ coefficients lie in $W_{O_F}(\ob R)$. (The matrix presentation is meant with respect to the above basis $(t_0, l_0, m_1, n_1, t_2, l_2, m_3, n_3)$.) Then $y$ being an endomorphism of $(\ob P, \ob Q, F, \dot F)$ is equivalent to $\dot F \circ y = y \circ \dot F$. We now express this condition in terms of the $16$ coefficients. We write $a$, $b$, $c$ and $d$ for the four matrices in \eqref{eq:generic_endo_display} and let $ξ\in W_{O_F}(\ob R)$ be any element. The compositions $\dot F \circ x$ and $x\circ \dot F$, for $x\in \{a, b, c, d\}$, are given as follows:
\begin{equation}
\begin{array}{lrcl}
(0)\ &\dot F(a({}^Vξt_0)) &=& πξ\{σ(a_{11})t_0 + σ(a_{21})[s_0]t_0 + σ(a_{21})l_0\}\\ 
&d(\dot F({}^Vξt_0)) &=& ξ\{{}^Vd_{12}t_0 + d_{22}l_0\} \\ 
&\dot F(a(l_0)) &=& σ(a_{12})t_0 + σ(a_{22})[s_0]t_0 + σ(a_{22})l_0\\ 
&d(\dot F(l_0)) &=& {}^Vd_{11}t_0 + d_{21}l_0\\ 
\\
(1)\ &\dot F(b(m_1)) &=& μb_{11}m_1 + σ(b_{21})m_1 + b_{11}n_1\\ 
&a(\dot F(m_1)) &=& a_{11}m_1 + a_{21}n_1\\ 
&\dot F(b(n_1)) &=& μb_{12}m_1 + σ(b_{22})m_1 + b_{12}n_1\\ 
&a(\dot F(n_1)) &=& [s_0]a_{11}m_1 + a_{12}m_1 + [t_0]a_{21}n_1 + a_{22}n_1\\ 
\\
(2)\ &\dot F(c({}^Vξt_2)) &=& πξ\{σ(c_{11})t_2 + [s_2]σ(c_{21})t_2 + σ(c_{21})l_2\}\\
&b(\dot F({}^Vξt_2)) &=& ξ\{μ{}^Vb_{11}t_2 + {}^Vb_{12}t_2 + μb_{21}l_2 + b_{22}l_2\}\\
&\dot F(c(l_2)) &=& σ(c_{12})t_2 + [s_2]σ(c_{22})t_2 + σ(c_{22})l_2\\
&b(\dot F(l_2)) &=& {}^Vb_{11}t_2 + b_{21}l_2\\
\\
(3)\ &\dot F(d(m_3)) &=& σ(d_{21})m_3 + d_{11} n_3\\
&c(\dot F(m_3)) &=& c_{11}m_3 + c_{21}n_3\\
&\dot F(d(n_3)) &=& σ(d_{22})m_3 + d_{12}n_3\\
&c(\dot F(n_3)) &=&  [s_2]c_{11}m_3 + c_{12}m_3 + [s_2]c_{21}n_3 + c_{22}n_3\\
\end{array}
\end{equation}
Thus $\dot F \circ y = y \circ \dot F$ if and only if the following identities hold.
\begin{equation}\label{eq:display_final_identities}
\begin{array}{lrclcrcl}
(α)\ & a_{11} &=& μb_{11} + σ(b_{21}) &\quad& a_{12} &=& μ b_{12} + σ(b_{22}) - [s_0]a_{11}\\
& a_{21} &=& b_{11} &\quad& a_{22} &=& b_{12} - [s_0]a_{21}\\
\\
(β)\ &{}^Vb_{11} &=& σ(c_{12}) + [s_2] σ(c_{22}) &\quad& {}^Vb_{12} &=& π(σ(c_{11}) + [s_2]σ(c_{21})) - μ{}^Vb_{11}\\
& b_{21} &=& σ(c_{22}) &\quad& b_{22} &=& πσ(c_{21}) - μb_{21}\\
\\
(γ)\ &c_{11} &=& σ(d_{21}) &\quad& c_{12} &=& σ(d_{22}) - [s_2]c_{11}\\
& c_{21} &=& d_{11} &\quad& c_{22} &=& d_{12} - [s_2]c_{21}\\
\\
(δ)\ &{}^Vd_{11} &=& σ(a_{12}) + [s_0]σ(a_{22})&\quad& {}^Vd_{12} &=& π(σ(a_{11}) + [s_0]σ(a_{21}))\\
& d_{21} &=& σ(a_{22}) &\quad& d_{22} &=& σ(a_{21})\\
\end{array}
\end{equation}
Assuming that all these relations hold, we claim that none of the $16$ variables is a unit. (This is equivalent to claiming that all $16$ coefficients are divisible by $π$ after base change to $W_{O_F}(R/π) = O_{\breve F}$ which means that $π^{-1}y$ defines an endomorphism of the Dieudonné module of the closed point. This precisely means that $z\in \mcZ(π^{-1}y)$.) The proof of the claim is as follows. We use the matrix notation $(α_{ij})$, $(β_{ij})$ etc. to refer to the individual identities in \eqref{eq:display_final_identities}.
\begin{enumerate}[wide, labelindent=0pt, labelwidth=!, label=(\arabic*), topsep=2pt, itemsep=2pt]
\item First, $(δ_{11})$ and $ (β_{11})$ imply that $a_{12}$ and $c_{12}$ cannot be units. Then $(γ_{12})$ implies that $d_{22}$ is no unit. Then $(δ_{22})$ shows that $a_{21}$ is no unit. By $(α_{21})$, also $b_{11}$ is no unit.

\item Next, specializing $(δ_{12})$ along the projection map $s:W_{O_F}(\ob R) \to \ob R$, we obtain that $0 = π s(σ(a_{11})) + 0$ because $πs_0 = 0$ in $\ob R$. Since $π\neq 0$ in $\ob R$, it follows that $s(σ(a_{11}))$ is no unit and hence that $a_{11}$ is no unit. The same argument but for $(β_{12})$ shows that $c_{11}$ is no unit.

\item An easy chain of substitutions now shows that all remaining variables, except for possibly $c_{21} = d_{11}$ cannot be units.

\item For the remaining two variables, we consider identity $(δ_{11})$. Consider the leading terms of the Witt vector expressions for $a_{12}$ and $a_{22}$:
$$a_{12} = [u_0] + {}^Vu_1 \quand a_{22} = [v_0] + {}^Vv_1.$$
We already know that $u_0,v_0\notin \ob R^\times$. In particular, $σ([u_0]) = [u_0]^q = 0$ and $σ([v_0]) = [v_0^q] = 0$. Recall that $σ({}^V(x)) = πx$ for every $x\in W_{O_F}(\ob R)$. Thus we obtain from $(δ_{11})$ that
$${}^Vd_{11} = σ(a_{12}) + [s_0]σ(a_{22}) = πu_1 + π[s_0]v_1.$$
Looking at the image of this expression under $s:W_{O_F}(\ob R)\to \ob R$ and using that $π\neq 0$ in $\ob R$, it follows that $u_1 \notin W_{O_F}(\ob R)^\times$.

Let $\bar a_{12}$ and $\bar d_{11}$ be the images of $a_{12}$ and $d_{11}$ under the reduction map $W_{O_F}(\ob R) \to W_{O_F}(\mbF) = O_{\breve F}$. Identity $(δ_{11})$ implies that ${}^V\bar d_{11} = σ(\bar a_{12})$. The above showed that $\bar a_{12}$ is divisible by $π^2$, so we obtain that $\bar d_{11}$ is divisible by $π$. We deduce that $d_{11}$ is not a unit. The proof of the proposition is now complete.
\end{enumerate}
\end{proof}

\subsection{Intersection Numbers}
\label{ss:intersection_numbers_3_4}

We first summarize the results of the previous sections. For a lattice $Λ\subset W$, we have previously defined $m(y, Λ) = \max \{0, n(y, Λ)\}$, see \eqref{eq:def_nm}.

\begin{prop}\label{prop:structure_3_4}
Let $y\in S_{3/4, \mr{rs}}$ be a regular semi-simple element and let $m = \max_{Λ\subset W} n(y, Λ)$. Put $\mcC(y) = \mcZ(π^{-m}y)$.
\begin{enumerate}[wide, labelindent=0pt, labelwidth=!, label=(\arabic*), topsep=2pt, itemsep=2pt]
\item The formal scheme $\mcZ(y)$ is non-empty if and only if $m\geq 0$. In particular, $\mcC(y) \neq \emptyset$.
\item For a stratum $P\subseteq \mcM_C$, let $Λ(P)\subset W$ be the lattice defined by it. The pure locus of $\mcZ(y)$ is given by
$$\mcZ(y)^{\mr{pure}} = \sum_{P \subseteq \mcM_{C, \mr{red}}} m(y, Λ(P))\cdot [P].$$
\item The formal scheme $\mcC(y)$ is artinian. Moreover,
\begin{equation}\label{eq:artinian_identity}
\mcZ(y)^{\mr{art}} = \mcC(y) \sqcup \coprod_{z \in |\mcZ(y)|\setminus |\mcC(y)|,\ z\ \text{superspecial}} \mcO_{\mcZ(y)^{\mr{art}}, z}
\end{equation}
and each local ring in the disjoint union on the right hand side has length $q$.
\end{enumerate}
\end{prop}
\begin{proof}
(1) This follows directly from Corollary \ref{cor:support_intersection_set_theoretic}.

(2) Corollary \ref{cor:support_intersection_set_theoretic} shows that the multiplicity of $P$ in $\mcZ(y)$ is indeed $0$ if $n(y,Λ(P))\leq 0$. By the same corollary, if $n(y, Λ(P)) = 0$, then there exists a point $z\in \mcZ(y)\cap P(\mbF)$ because every $σ$-linear endomorphism of a $2$-dimensional $\mbF$-vector preserves some point of $P(\mbF)$. Proposition \ref{prop:simple_extension} applies to that point and shows that the multiplicity of $P$ in $\mcZ(π^ay)$ equals $a = m(π^ay, Λ(P))$ for every $a\geq 0$. This reasoning applies to all pairs $(π^\mbZ y,Λ(P))$, and statement (2) follows.

(3) For all $y$, by Proposition \ref{prop:simple_extension}, if $z\in \mcZ(y)^{\mr{art}}$ then $z\in \mcZ(πy)^{\mr{art}}$ and there is an equality of local rings
$$\mcO_{\mcZ(πy)^{\mr{art}}, z} = \mcO_{\mcZ(y)^{\mr{art}}, z}.$$
We know from Corollary \ref{cor:support_intersection_set_theoretic} that $\mcC(y)$ is artinian, so this shows $\mcC(y) \subseteq \mcZ(y)^{\mr{art}}$. Moreover, by Proposition \ref{prop:artinian_case_5} combined with (again) Proposition \ref{prop:simple_extension}, every superspecial point $z = P \cap P'$ such that $m(y, Λ(P)) > m(y, Λ(P')) \geq 0$ lies in $\mcZ(y)^{\mr{art}}$ and has a local ring of length $q$. This shows that the right hand side in \eqref{eq:artinian_identity} is an open and closed subscheme of the left hand side.

By Corollary \ref{cor:support_intersection_set_theoretic}, every superspecial point of $\mcZ(y)$ already lies in the right hand side of \eqref{eq:artinian_identity}. Let $z\in \mcZ(y)^{\mr{art}} \cap P(\mbF)$ be a non-superspecial point. The multiplicity of $P$ in $\mcZ(π^{-m(y, Λ(P))+1}y)$ is $1$. By Proposition \ref{prop:non_special_embedded}, $z$ even lies in $\mcZ(π^{-m(y, Λ(P))}y)$. By Corollary \ref{cor:support_intersection_set_theoretic}, the only possibility is $m(y, Λ(P)) = m$ and $z\in \mcC(y)$, and the proof of (3) is complete.
\end{proof}

Let $G^0_{3/4, b}\subset G_{3/4, b}$ denote the subgroup of elements $g$ with reduced norm $\mr{Nrd}(g) \in O_F^\times$. Then $g\mcM_{3/4}^i = \mcM_{3/4}^i$ for every $i\in \mbZ$ and we may define the connected component intersection number
$$\Int_0(g) := \langle \mcM_C^0,\ g\cdot \mcM_C^0\rangle_{\mcM_{3/4}}.$$

\begin{thm}\label{thm:main_3_4}
Let $g\in G^0_{3/4, b, \mr{rs}}$ be a regular semi-simple element with numerical invariant $(L, r, d)$. The intersection number $\Int_0(g)$ only depends on the triple $(L, r, d)$. It is related to $\Int(g)$ by
\begin{equation}\label{eq:rel_conn_comp_int}
\Int(g) = \begin{cases}
\Int_0(g) & \text{if $L/F$ is split or ramified}\\
2\,\Int_0(g) & \text{if $L/F$ is an unramified field extension.}
\end{cases}
\end{equation}
Moreover, the arithmetic transfer conjecture (Conjecture \ref{conj:ATC_explicit}) holds with correction function $f'_{\mr{corr}} = 0$. That is, for every $γ\in G'_{\mr{rs}}$,
\begin{equation}\label{eq:ATC_3_4}
\del (γ, f'_D) = \begin{cases} 2\,\Int(g)\log(q) & \text{if there is a matching $g\in G_{3/4, b}$}\\
0 & \text{otherwise.}
\end{cases}
\end{equation}
\end{thm}
\begin{proof}
We first determine the intersection number $\Int(g)$ for $g\in G_{3/4, b, \mr{rs}}$. We may assume that $g = 1 + z$ with $z = z_g$. Since $\mcM_C\cap g\cdot \mcM_C = \emptyset$ whenever $z$ is not topologically nilpotent by Lemma \ref{lem:nilpotent_reduction}, we may assume that $z$ is topologically nilpotent.

We work in the coordinates of \S\ref{ss:simplified_intersection_3_4}. Let $y\in S_{3/4, \mr{rs}}$ be such that $z = \left(\begin{smallmatrix} & y\varpi \\ y & \end{smallmatrix}\right)$. In particular, $L = F[\varpi y^2]$ with $\varpi y^2 \in O_L$ and
$$r = v(N_{L/F}(\varpi y^2)),\quad d = [O_L : O_F[\varpi y^2]] - r/2.$$
The element $\wt{y} = \varpi^{-1}y$ lies in $S_{1/4}$. Since $\varpi \wt y^2 = π^{-1} \varpi y^2$, its numerical invariant is given by
\begin{equation}\label{eq:numerical_triples}
(\wt L, \wt r, \wt d) = (L, r - 2, d).
\end{equation}
Let $\wt z = \left(\begin{smallmatrix} & \wt y\varpi \\ \wt y & \end{smallmatrix}\right)$ and define $\wt g = 1 + \wt z \in \End^0_F(\mbY)$. Then $\wt g$ lies in $G_{1/4, b, \mr{rs}}$ unless the following exceptional case occurs: $L \iso F\times F$ and one of the eigenvalues of $\varpi y^2$ is $π$. All statements that follow also apply in this exceptional case if one sets $\mcI(\wt g) = \emptyset$ and $\mr{Int}_0(\wt g) = 0$.

Let $Γ\subset L^\times$ be a subgroup such that $L^\times = Γ\times O_L^\times$ and set $Γ_0 = Γ \cap G^0_{3/4, b}$.

\begin{lem}\label{lem:strata_contribution}
The following identity of intersection numbers holds:
$$\Int_0(g) = \Int_0(\wt g) + q\cdot\#\big(Γ_0E^\times \backslash \{Λ \subseteq W \mid m(y, Λ) \geq 0\}\big) + \begin{cases} 1 & \text{if $L$ is a field}\\
0 & \text{if $L\iso F\times F$.}\end{cases}$$
\end{lem}
\begin{proof}
First consider the divisors $\mcZ(y)^{0, \mr{pure}},\, \mcZ(\wt y)^{0, \mr{pure}} \subseteq \mcM^0_C$. Because
$$w(\varpi^{-1} y) = V\cdot \varpi^{-1}y = (V^2\varpi^{-1})\cdot V^{-1}y = (V^2\varpi^{-1})\cdot w(y)$$
by Definition \ref{def:operator_w}, the restrictions of $w(\varpi^{-1}y)$ and $w(y)$ to the $V^2\varpi^{-1}$-invariants $W = N_0^{τ = \mr{id}}$ agree. The multiplicity formula for the invariant $1/4$ case (Proposition \ref{prop:multiplicities_1_4}) and the analogous formula for invariant $3/4$ (Proposition \ref{prop:structure_3_4}) hence give that
\begin{equation}\label{eq:divisors_equal}
\mcZ(y)^{0, \mr{pure}} = \mcZ(\wt y)^{0, \mr{pure}}.
\end{equation}
By Proposition \ref{prop:conormal_3/4}, the degree $\deg(\mcC\vert_P)$ of the restriction of the conormal bundle for $\mcM_C\to \mcM_D$ to an irreducible component $P\subseteq \mcM_{C, \mr{red}}$ equals $q^2-1$ in all situations. The equality in \eqref{eq:divisors_equal} and the general intersection number formula from Corollary \ref{cor:intersection_simplified} then imply that
\begin{equation}\label{eq:difference_int_numbers}
\Int_0(g) - \Int_0(\wt g) = \mr{len}(\mcO_{Γ_0\backslash \mcZ(y)^{0, \mr{art}}}) - \mr{len}(\mcO_{Γ_0\backslash \mcZ(\wt y)^{0,\mr{art}}}).
\end{equation}
By Proposition \ref{prop:embedded_components_1_4}, the length of $Γ_0\backslash \mcZ(\wt y)^{0, \mr{art}}$ is given by
\begin{equation}\label{eq:embed_comp_1_4}
\mr{len}(\mcO_{Γ_0\backslash \mcZ(\wt y)^{0, \mr{art}}}) = \begin{cases} 1 & \text{if $L$ is a field and $\wt r\in 2 + 4\mbZ_{\geq 0}$}\\
0 & \text{in all other cases.}
\end{cases}
\end{equation}
We determine the length of $Γ_0\backslash \mcZ(y)^{0, \mr{art}}$: By \eqref{eq:artinian_identity}, it is given as
\begin{equation}\label{eq:eq1}
\mr{len}(\mcO_{Γ_0 \backslash \mcC(y)^0}) + q \cdot \#\big(Γ_0E^\times \backslash \{Λ\subset W \mid m > m(y, Λ)\geq 0\}\big).
\end{equation}
First assume that $\wt r \in 4\mbZ_{\geq 1}$. Then we are in case (3) of Corollary \ref{cor:support_intersection_set_theoretic} and obtain from Proposition \ref{prop:artinian_case_3} that
\begin{equation}\label{eq:eq2}
\mr{len}(\mcO_{Γ_0 \backslash \mcC(y)^0}) = q\cdot \#Γ_0\backslash\mcT(w(y)) + \begin{cases}1 & \text{if $L$ is a field}\\
0 & \text{if $L$ is split.}
\end{cases}
\end{equation}
Now assume that $\wt r \in 2 + 4\mbZ_{\geq 0}$. Then we are in case (4) of Corollary \ref{cor:support_intersection_set_theoretic} and Proposition \ref{prop:artinian_case_4} shows that
\begin{equation}\label{eq:eq3}
\mr{len}(\mcO_{Γ_0 \backslash \mcC(y)^0}) = q\cdot \#Γ_0\backslash\mcT(w(y)) + \begin{cases}2 & \text{if $L$ is a field}\\
0 & \text{if $L$ is split.}
\end{cases}
\end{equation}
In both cases, the set $\mcT(w(y))$ is precisely the set $\{Λ\subseteq W\mid m(y, Λ) = m\}/E^\times$. So the three identities \eqref{eq:eq1}, \eqref{eq:eq2} and \eqref{eq:eq3} together give
$$
\mr{len}(\mcO_{Γ_0\backslash \mcZ(y)^{0, \mr{art}}}) = q\cdot \#\big(Γ_0 E^\times \backslash \{Λ\subset W \mid m(y, Λ)\geq 0\}\big) + \begin{cases} 1 & \text{if $L$ is a field and $\wt r\in 4\mbZ_{\geq 2}$}\\
2 & \text{if $L$ is a field and $\wt r\in 2+4\mbZ_{\geq 0}$}\\
0 & \text{if $L$ is split.}
\end{cases}
$$
Substituting this result and Identity \eqref{eq:embed_comp_1_4} in \eqref{eq:difference_int_numbers} proves the lemma.
\end{proof}

We can now prove the first part of Theorem \ref{thm:main_3_4}: We already know from Proposition \ref{prop:main_1_4_aux} that $\Int_0(\wt g)$ only depends on $(L, r, d)$. (Recall that the numerical invariant of $\wt g$ is $(L, r - 2, d)$.) By Theorem \ref{thm:classification_multiplicity_function} and Proposition \ref{prop:multiplicities}, the number of lattices $\#(Γ_0E^\times\backslash \{Λ\subset W\mid m(y, Λ) \geq 0\})$ in Lemma \ref{lem:strata_contribution}  also only depends on $(L, r, d)$. So we obtain that $\Int_0(g)$ only depends on $(L, r, d)$ as claimed. The claimed identity
$$\Int(g) = \begin{cases} \Int_0(g) & \text{if $L/F$ ramified or split}\\
2\,\Int_0(g) & \text{if $L/F$ inert}\end{cases}$$
follows by the same argument as in the case of invariant $1/4$, see the first part of the proof of Theorem \ref{thm:main_1_4}. By the intersection number formula in the invariant $1/4$ case \eqref{eq:int_number_formula_1_4}, which we apply to $\wt r = r -2$, we obtain from Lemma \ref{lem:strata_contribution} that
\begin{equation}\label{eq:int_number_formula_3_4}
\Int(g) = δ \cdot q\cdot \#\big(Γ_0E^\times \backslash\{Λ\subset W \mid n(y, Λ) \geq 0\}\big) + \begin{cases}
r & \text{$L/F$ inert}\\
r/2 & \text{$L/F$ ramified}\\
0 & \text{$L/F$ split}\end{cases}
\end{equation}
where $δ = 2$ if $L/F$ is inert and $δ = 1$ otherwise. It is left to verify the arithmetic transfer identity \eqref{eq:ATC_3_4}.

Let $γ\in G'_{\mr{rs}}$ be a regular semi-simple element with numerical invariant $(L, r, d)$. If the sign of the functional equation $ε_D(γ)$ is positive, equivalently $r$ odd, then the left hand side of \eqref{eq:ATC_3_4} vanishes by Proposition \ref{prop:derivative_teaser}. There is no matching element $g\in G_{3/4, b}$ by Proposition \ref{prop:char_isogeny_class}, so the right hand side vanishes as well.

Assume from now on that the sign $ε_D(γ)$ is negative, equivalently that $r$ is even. Then $\Orb(γ, f'_D) = 0$, so the left hand side of \eqref{eq:ATC_3_4} equals (by Proposition \ref{prop:derivative_teaser})
\begin{equation}\label{eq:derivative_recap}
\del(γ, f'_D) = 4q\log(q)\, \Orb(γ, f'_\Par) + \log(q) \begin{cases}
2r & \text{if $L$ inert}\\
r & \text{if $L$ ramified}\\
0 & \text{if $L$ split.}
\end{cases}
\end{equation}
By Table \ref{table:matching}, there exists no matching element $g\in G_{3/4, b}$ if and only if $L\iso F\times F$ and $z_γ = (z_1, z_2)$ with $v(z_1), v(z_2)$ both even. (The components $z_1$ and $z_2$ always have the same parity because $r = v(z_1) + v(z_2)$ was assumed even.) In this case, the derivative \eqref{eq:derivative_recap} vanishes by Proposition \ref{prop:orb_int_para_teaser} (5).

It remains to consider the case where there exists a matching element $g\in G_{3/4, b, \mr{rs}}$. Then, $(L, r, d)$ is the numerical invariant of both $g$ and $γ$. We assume that $r >0$ because otherwise all involved terms vanish. The desired equality of \eqref{eq:derivative_recap} and the $2\log(q)$-multiple of \eqref{eq:int_number_formula_3_4} is then a direct consequence of the following lemma:

\begin{lem}\label{lem:strata}
Assume that $γ\in G'_{\mr{rs}}$ matches $g\in G_{3/4, b, \mr{rs}}$. Then
\begin{equation}\label{eq:formula_lattice_count}
2 \Orb(γ, f'_\Par) = \#\big(Γ_0E^\times \backslash\{Λ\subset W \mid n(y, Λ) \geq 0\}\big)\cdot \begin{cases} 2 & \text{$L/F$ inert}\\
1 & \text{$L/F$ ramified or split.}
\end{cases}
\end{equation}
\end{lem}
\begin{proof}
The left hand side only depends on the triple $(L, r, d)$ and is given by Proposition \ref{prop:orb_int_para_teaser}. Let $m = \max_{Λ\subset W} n(y, Λ)$ and assume that $m \geq 0$. The right hand side of the lemma is described by Proposition \ref{prop:multiplicities} and our proof is by going through the cases of that proposition. (Every $g\in G_{3/4,b,\mr{rs}}$ has a matching element $γ$, so this condition will not come up anymore.)

Proposition \ref{prop:multiplicities} describes the right hand side of \eqref{eq:formula_lattice_count} in terms of a ball of radius $m$ around the center $\mcT(w(y))$,
$$B(\mcT(w(y)), m) = \{Λ \in \mcB \mid d(Λ, \mcT(w(y))) \leq m\}.$$
The center $\mcT(w(y))$ in turn has been determined in Theorem \ref{thm:classification_multiplicity_function}. Note that $w(y)^2 = π^{-1} \varpi y^2$, so the numerical invariant of $w(y)$ is $(L, \wt r = r-2, d)$. We now count $Γ_0\backslash B(m, \mcT(w(y)))$ in dependence on $(L, r, d)$. The group $Γ_0$ is trivial if $L$ is a field and will only come up for $L = F\times F$.

\begin{enumerate}[wide, labelindent=0pt, labelwidth=!, label=(\arabic*), topsep=2pt, itemsep=2pt]
\item Assume that $L/F$ is inert or ramified, and that $r \in 4\mbZ_{\geq 1}$. Then $\wt r \in 2 + 4\mbZ_{\geq 0}$ and Theorem \ref{thm:classification_multiplicity_function} states that $\mcT(w(y))$ is a single edge. Then
\begin{equation}\label{eq:ball_inert}
\# B(\mcT(w(y)), m) = 2(1 + q^2 + \ldots + q^{2m}).
\end{equation}
Lemma \ref{lem:maximum_of_multiplicity_function} applied to $\wt r$ states that $2m = r/2 - 2$. The expression in \eqref{eq:ball_inert} equals $\Orb(γ, f'_\Par)$ if $L$ is inert or $2\,\Orb(γ, f'_\Par)$ if $L$ is ramified which follows from Proposition \ref{prop:orb_int_para_teaser} (1) and (3). This proves \eqref{eq:formula_lattice_count} in these two cases.

\item Assume that $L/F$ is inert and that $r\in 2 + 4\mbZ_{\geq 0}$. Then $\wt r \in 4\mbZ_{\geq 0}$ and Theorem \ref{thm:classification_multiplicity_function} states that $\mcT(w(y))$ is a $(q+1)$-regular ball of radius $d$ around a vertex. If $d = 0$, then we obtain
$$\# B(\mcT(w(y)), m) = 1 + (1+q^2)(1 + q^2 + \ldots + q^{2(m-1)}) = 2(1 + q^2 + \ldots + q^{2(m-1)}) + q^{2m}.$$
By Lemma \ref{lem:maximum_of_multiplicity_function} applied to $\wt r$, we find $2m = r/2-1$. The equality $\# B(\mcT(w(y)), m) = \Orb(γ, f'_\Par)$ follows from Proposition \ref{prop:orb_int_para_teaser} (4). Assume now that $d \geq 1$. Let $A$ be the number of vertices of valency $q+1$ of $\mcT(w(y))$ and let $B$ be the number of vertices of valency $1$. Then
$$\# B(\mcT(w(y)), m) = (1 + (q^2-q)(1 + q^2 + \ldots + q^{2m-2}))A + (1 + q^2(1 + q^2 + \ldots + q^{2m-2}))B.$$
Evaluating this with $d \geq 1$ and $2m = r/2 -1$, one obtains the expression for $\Orb(γ, f'_\Par)$ in Proposition \ref{prop:orb_int_para_teaser} (4).
\end{enumerate}
The verification in the remaining cases works in exactly the same way. We do not spell this out in full detail but say here which further cases one has to consider precisely.
\begin{enumerate}[wide, resume, labelindent=0pt, labelwidth=!, label=(\arabic*), topsep=2pt, itemsep=2pt]
\item Assume that $L/F$ is ramified and that $r \in 2 + 4\mbZ_{\geq 0}$. Then $\wt r \in 4\mbZ_{\geq 0}$ and Theorem \ref{thm:classification_multiplicity_function} states that $\mcT(w(y))$ is a $q+1$-regular ball of radius $d$ around an edge. Moreover, $2m = r/2 - 1$ as in (2) above. Direct inspection shows that the cardinality of $B(\mcT(w(y)), m)$ is given by the expression for $2\Orb(γ, f'_\Par)$ in Proposition \ref{prop:orb_int_para_teaser} (2).

\item Assume that $L/F$ is split which implies that $r \in 2 + 4\mbZ_{\geq 0}$ and hence $\wt r \in 4\mbZ_{\geq 0}$. First assume that $d < 0$. Then Theorem \ref{thm:classification_multiplicity_function} states that $\mcT(w(y))$ is an apartment. The quotient $Γ_0\backslash \mcT(w(y))$ has two elements because $Γ_0$ is generated by an element of the form $(π_1, π_2^{-1})$ with $v(π_1) = v(π_2) = 1$. It follows that
\begin{equation}\label{eq:formula_split}
\# (Γ_0\backslash B(\mcT(w(y)), m)) = 2[1 + (q^2-1)(1 + q^2 + \ldots + q^{2m-2})] = 2q^{2m}.
\end{equation}
Lemma \ref{lem:maximum_of_multiplicity_function} states that $2m = r/2 +d-1$.
Proposition \ref{prop:orb_int_para_teaser} (6) states that $2\Orb(γ, f'_\Par) = 2q^{r/2 + d -1}$ which equals \eqref{eq:formula_split}. In the case $d\geq 0$, we have that $2m = r/2 -1$ and that $\mcT(w(y))$ is a $(q+1)$-regular ball of radius $m$ around an apartment. The verification of $\#\big(Γ_0\backslash B(\mcT(w(y)), m)\big) = 2\Orb(γ, f'_\Par)$ works just as before.
\end{enumerate}
\end{proof}
The proof of Theorem \ref{thm:main_3_4} is now complete.
\end{proof}

\begin{rmk}\label{rmk:better_counting}
We have proved Lemma \ref{lem:strata} by explicitly comparing its two sides. There is, however, a more conceptual explanation for this identity: The maximal reduced subscheme of $\mcM_{3/4}$ can be shown to admit a Bruhat--Tits type stratification indexed by $O_C$-lattices in $C^2$ which, under the Serre--Tensor construction, translates the counting problem in Lemma \ref{lem:strata} into an orbital integral on $GL_2(B)$. The fundamental lemma for Hasse invariant $1/2$ from \cite{HM} expresses this as $O(γ, 1_{\mr{Par}})$.
\end{rmk}

\end{document}